\documentclass[12pt]{amsart}
\usepackage{latexsym,amsfonts,amsmath,amssymb,amsthm,url,amsbsy,amscd,mathrsfs}
\usepackage[english]{babel}
\usepackage[latin1]{inputenc}
\usepackage{graphicx,psfrag,epsfig}
\usepackage{enumerate}
\usepackage{bbm}

\usepackage[colorlinks=true,urlcolor=blue, citecolor=red,linkcolor=blue,
linktocpage,pdfpagelabels, bookmarksnumbered,bookmarksopen]{hyperref}

\numberwithin{equation}{section}
\newtheorem{theorem}{Theorem}[section]
\newtheorem{proposition}[theorem]{Proposition}
\newtheorem{corollary}[theorem]{Corollary}
\newtheorem{lemma}[theorem]{Lemma}
\newtheorem{definition}[theorem]{Definition}
\theoremstyle{definition}
\newtheorem{remark}[theorem]{Remark}

\usepackage{a4wide}

\usepackage{color}
\definecolor{darkred}{rgb}{0.8,0,0}
\definecolor{darkblue}{rgb}{0,0,0.7}
\definecolor{darkgreen}{rgb}{0,0.4,0}

\newcommand{\EEE}{\color{black}} 
\newcommand{\eps}{\varepsilon}
\newcommand{\id}{{ I}}
\renewcommand{\d}{{\rm d}}
\newcommand{\dd}{\d}

\newcommand{\R}{{\mathbb R}}

\newcommand{\N}{{\mathbb N}}

\newcommand{\PP}{{\mathscr P}}

\newcommand{\II}{{\mathcal I}}
\newcommand{\BB}{{\mathcal B}}

\newcommand{\C}{{\mathbb C}}
\newcommand{\FF}{{\mathcal F}}

\newcommand{\EE}{{\mathcal E}}

\newcommand{\VV}{{\mathcal V}}
\renewcommand{\ln}{\log}

\newcommand{\Leb}[1]{{\mathcal L}^{#1}}
\newcommand{\un}{{\rm 1\kern -2.5pt l}}

\newcommand{\de}{\partial}

\renewcommand{\div}{\mathrm{div}}

\newcommand{\AC}{{\mathrm{AC}}}
\newcommand{\mom}{{\rm m}_2}

\newcommand{\Space}{\PP_2^M(\R^d)}

\newcommand{\bC}{{\mathbbm{1}}}

\usepackage[normalem]{ulem}     

\newcommand{\RRR}{\color{black}}

\usepackage{hyperref}
\AtBeginDocument{
  \label{CorrectFirstPageLabel}
  
}

\newenvironment{proofad1}{\removelastskip\par\medskip
\noindent{\textit {\bf Proof of Theorem  \ref{th:main}}.}
\rm}{\penalty-20\null\hfill$\square$\par\medbreak} 

\newenvironment{proofad2}{\removelastskip\par\medskip
\noindent{\textbf {\bf Proof of Theorem  \ref{th:main2}}.}
\rm}{\penalty-20\null\hfill$\square$\par\medbreak} 

\usepackage{etoolbox}

\makeatletter

\renewcommand{\tocsection}[3]{%
  \indentlabel{\@ifnotempty{#2}{\bfseries\ignorespaces#1 #2\quad}}\bfseries#3}
\renewcommand{\tocsubsection}[3]{%
  \indentlabel{\@ifnotempty{#2}{\ignorespaces#1 #2\quad}}#3}

\def\@tocline#1#2#3#4#5#6#7{\relax
  \ifnum #1>\c@tocdepth 
  \else
    \par \addpenalty\@secpenalty\addvspace{0.03cm}%
    \begingroup \hyphenpenalty\@M
    \@ifempty{#4}{%
      \@tempdima\csname r@tocindent\number#1\endcsname\relax
    }{%
      \@tempdima#4\relax
    }%
    \parindent\z@ \leftskip#3\relax \advance\leftskip\@tempdima\relax
    \rightskip\@pnumwidth plus4em \parfillskip-\@pnumwidth
    #5\leavevmode\hskip-\@tempdima
      \ifcase #1
       \or\or \hskip 2em \or \hskip 3em \else \hskip 4em \fi%
      #6\nobreak\relax
    \dotfill\hbox to\@pnumwidth{\@tocpagenum{#7}}\par
    \nobreak
    \endgroup
  \fi}
  
\makeatother

\begin{document}

\title[]{
Smoothing effect and uniqueness for aggregation diffusion models}
\begin{abstract}
We consider aggregation-diffusion models for a density of mass in $\R^d$, $d\ge 2$, where the diffusion can be either linear or of porous medium type, 
while the aggregation effect is governed by the attractive Newtonian or Bessel potential. 
It is well known that the dynamics can be interpreted as Wasserstein gradient flow of the natural free energy of the system, and that,
 under suitable assumptions on the diffusion exponent and  the mass of the initial datum, \EEE solutions exist globally in time.  
 In these regimes, we perform the analysis of the discrete variational approach by means of the JKO scheme.
   We establish a sharp $L^\infty$ smoothing effect, which proves to be of same rate as that of the porous medium equation. 
  Thanks to this estimate, we obtain uniqueness of global  gradient flow solutions \EEE for initial data of finite energy and, in the diffusion dominated regime, 
   for measure data having finite second moment.  
  We prove the energy dissipation equality and characterize the gradient flow in terms of  suitable  evolution variational inequalities. 
In the fair competition regime, we show uniform extinction of solutions for large time under smallness conditions on the mass.  
  
\end{abstract}
\author{S. Lisini}
\address{Stefano Lisini -- Universit\`a degli Studi di Pavia, Dipartimento di Matematica "F. Casorati", via Ferrata 5, I-27100 Pavia, Italy}
\email{stefano.lisini@unipv.it}
\author{E. Mainini}
\address{Edoardo Mainini -- Universit\`a degli Studi di Genova, 
Dipartimento di Ingegneria meccanica, energe\-tica, gestionale e dei trasporti (DIME), via all'Opera Pia 15, I-16145 Genova, Italy}
\email{edoardo.mainini@unige.it}
%
%
%

\thanks{}

\maketitle
		\tableofcontents

\section{Introduction}

We consider the following aggregation diffusion model in space dimension $d\ge 2$:
\begin{equation}\label{PE}
\begin{cases}
\partial_t u= \Delta u^m-\chi\mathrm{div}(u\nabla v ) &\mbox{in }\mathbb{R}^d\times (0,+\infty), \vspace{0.2cm}\\ \vspace{0.2cm}
-\Delta v +\alpha v=u &\mbox{in }\mathbb{R}^d\times(0,+\infty),\\
u(\cdot,0)=u_0(\cdot),&\mbox{in }\mathbb{R}^d,
\end{cases}
\end{equation}
where $\chi>0$ and $\alpha \geq 0$ are given constants. We let 
\begin{equation}\label{mc}
	m\geq m_c:=2-\frac{2}{d},
\end{equation}
and $m_c$ is called critical exponent.  Therefore the diffusion can be linear (only in space dimension $2$) or nonlinear of porous medium type. The second equation is understood in terms of fundamentals solution, 
i.e., $v$ is the Bessel potential of $u$ if $\alpha>0$, and the Newtonian potential of $u$ if $\alpha=0$.
The initial datum $u_0$ is a nonnegative measure in $\R^d$ with total mass $M\in(0,+\infty)$  and with finite second moment.
   
       \subsection{Overview on aggregation diffusion models}
Aggregation diffusion models provide macroscopic description of particles/agents  that evolve according to linear or nonlinear diffusion processes, in competition with the effect of mutual attraction forces described by a mean-field potential. 
Among the most notable examples of diffusion models with nonlocal drift of the form \eqref{PE} is the renowned Patlak-Keller-Segel model in chemotaxis \cite{A, H1, H2, KS1, KS2, P, Pe, Pe2}, describing the dynamics of bacteria whose motion is oriented by a chemical signal. In its basic formulation, it takes the form \eqref{PE} with $m=1$, $d=2$, where $u$ represents the density of bacteria, while $v$ represents the density of the chemoattractant produced by the bacteria, and $\chi$ denotes the sensitivity constant measuring the aggregation strength.
Nonlinear diffusion of porous medium type often appears for modeling localized repulsion taking into account volume effects and non interpenetration of cells \cite{BCL, CC, LSV, O, TBL}.
 In the Newtonian interaction case $\alpha=0$, the system \eqref{PE} also appears in the theory of gravitational collapse and it is known as the Smoluchowski-Poisson model \cite{BN,C0,CLL,CM}.

  \medskip
For a given mass $M>0$, we denote by $\Space$ the space of Borel non-negative measures with total mass $M$ and finite second moment.
The internal energy functional
$$\FF_m: \Space\to (-\infty,+\infty],$$
 accounting for diffusion, is defined by
\[
\FF_1(u)=\int_{\R^d}u(x)\log u(x)\,\d x,\quad\qquad \FF_m(u)=\frac{1}{m-1}\int_{\R^d}u^m(x)\,\d x,\quad m>1,
\]
if $u$ is absolutely continuous with respect to the Lebesgue measure on $\R^d$, 
where we identify a measure with its density, and
 $\FF_m(u)=+\infty$ if $u$ is not absolutely continuous with respect to the Lebesgue measure on $\R^d$. 
 We set $$D(\mathcal F_m):=\left\{u\in\Space: \mathcal F_m(u)<+\infty\right\}.$$
For $\alpha\geq0$,  the interaction energy functional
$$\BB_\alpha: \Space\to (-\infty,+\infty]$$ is defined by
\[
\BB_\alpha(u)=\frac12\int_{\R^d\times\R^d}B_\alpha(x-y)\,\d u(x)\,\d u(y),
\]
where
$B_\alpha$ denotes the Bessel kernel in $\mathbb{R}^d$, i.e., the fundamental solution for the operator $-\Delta+\alpha I$, reduced to the Newtonian kernel if $\alpha=0$.  
The free energy of the system \eqref{PE} is  the functional  $$\EE_m: \Space\to (-\infty,+\infty]$$
 defined by 
\begin{equation}\label{E_m}
	\EE_m(u)=\begin{cases}
		\FF_m(u)-\chi\BB_\alpha(u) &\mbox{if }u\in D(\mathcal F_m)\\
		+\infty 	& \mbox{otherwise.}
		\end{cases}
\end{equation} 
We notice that, since $m\ge m_c$, the finiteness of $\mathcal E_m(u)$ for every $u\in D(\mathcal F_m)$ is a consequence of the Hardy-Littlewood-Sobolev inequalities (recalled in Section \ref{Sec:notation}). 
In particular, 
 \begin{equation}\label{DE_m}
 	D(\mathcal F_m)=D(\mathcal E_m):=\left\{u\in\Space: \mathcal E_m(u)<+\infty\right\}.
 \end{equation}
It is well known that the system \eqref{PE} has a gradient flow structure. It can be interpreted as gradient flow of the energy functional $\mathcal E_m$ in the space $\Space$ endowed with the Wasserstein distance.
This is due to the fact that \eqref{PE} has the form of a continuity equation where the velocity vector field is the gradient of a scalar function, the pressure, which is in turn the functional derivative of $\mathcal E_m$. That is, \eqref{PE} is in the form, for $m>1$,
\[
\partial_tu=\mathrm{div}(u\nabla \mathfrak p_u),\qquad \mathfrak p_u=\frac{\delta \mathcal E_m}{\delta u}(u)=\frac m{m-1}u^{m-1}-\chi B_\alpha\ast u.
\]
In the case $m=1$, the first term in the right hand side is replaced by $1+\log u$.
The free energy $\EE_m$ is decreasing along the flow according to the formal dissipation identity
\[
\frac{d}{dt}\mathcal E_m(u(t))=-\int_{\R^d}\left|\nabla \mathfrak p_u(t)\right|^2\,u(t)\,\d x.
\]
Taking advantage to this interpretation, one way to prove existence of solutions for the system \eqref{PE} is the Jordan-Kinderlehrer-Otto  (JKO) scheme, introduced in the  the seminal paper \cite{JKO}. 
In the case of the Keller-Segel model, this approach has been introduced in \cite{BCC} and then exploited in \cite{B,B2,  BCKKLL, BL}.
It is also standardly applied to many other models featuring nonlocal drift terms with singular kernels, including for instance vortex dynamics in superconductivity \cite{AMS,AS}, the Poisson-Nernst-Planck system modeling ionic transport \cite{KMX}, and Riesz potential interactions \cite{HMVV,LMS}.


	\medskip
Standard scaling considerations  about the energy show that the critical exponent
 $m_c=2-2/d$
governs the competition between diffusion and aggregation, 
 and show that diffusion and aggregation effects are in balance in the case $m=m_c$, which is the so-called {\it fair competition regime}.
 On the other hand, if $m>m_c$, the diffusion dominates and concentrating the mass is never energetically favorable, therefore we are in the {\it diffusion-dominated regime}. The balance of the effects for $m=m_c$ is also seen from the PDE by taking into account that the mass invariant rescaling $u_\lambda(x,t)=\lambda^du(\lambda x,\lambda^d t)$ preserves the equation, up to rescaling $\alpha$, precisely when $m=m_c$: in this case, the competition is ruled by the value of the mass, and it is well known that there exists a critical mass $M_c$ whose value is
 \begin{equation}\label{defMc}
 \displaystyle M_c:=\frac{8\pi}{\chi}\quad \text{ if }d=2,\qquad\qquad
M_{c}:=\left(\displaystyle\frac{2}{\chi(m_c-1)C_{H}}\right)^{d/2}\quad \text{ if }d\geq3
\end{equation}
where $C_H$ denotes the best constant in a suitable Hardy-Littlewood-Sobolev inequality, to be recalled in Section \ref{Sec:notation} along with 
the main properties of the energy functional $\EE_m$. Indeed,  solutions do exists globally in time if $\chi M<8\pi$ in dimension $2$, see \cite{JL,DP,BDP}, while in the case $\chi M>8\pi$ blow up in finite time occurs \cite{JL, BDP, N}. Stability of the blow up has also been analyzed \cite{RS,CGMN1,CGMN2}. If $m>m_c$ the diffusion dominates and prevents blow up, so that again solutions exist globally in time \cite{CC,BCL, S, S3}. In the case $m=m_c$, $d\ge 3$, a critical mass also appears, whose value from \eqref{defMc} has been determined in \cite{BCL}, again discriminating between global existence and blow up \cite{BRB, BK, BL, S2}.
A lot of interest has been also devoted to the study of stationary states  for aggregation-diffusion models.
	In the two-dimensional case     with linear diffusion $m=1$ and  Newtonian interaction kernel $\alpha=0$, stationary states exist only if $M=M_c$, see \cite{DP,BCC2,BCM, CF}, they are related to optimizers of the logarithmic HLS inequality and have infinite second moment. For initial data of finite second moment, blow up occurs in infinite time \cite{BCM}, and blow up profile is studied in \cite{DDDMW, GMa}.
	Newtonian fair competition models $m=m_c$ in higher dimension also admit stationary states only in the case $M=M_c$, and these are extremals of suitable HLS inequalities \cite{BCL}.  
	On the other hand, in the diffusion dominated regime stationary states exist for every choice of the mass $M>0$ and can be obtained as minimizers of $\mathcal E_m$ over the set of mass densities in $\Space\cap L^m(\R^d)$. Moreover, for given mass $M$, stationary states are unique up to translation, radially decreasing and compactly supported \cite{BLiu,CS,LO,KY,CCV,CHVY}. 
	 Similar properties are very well established also for more general potentials and have been generalized to Riesz potential interaction  \cite{CCH,CCH2, CHMV, CCH3, CGHMV, BMai,CLLS,DYY}.

\medskip
In this paper we are mostly interested in uniqueness, smoothing effect and  $L^\infty(\R^d)$ hypercontractivity estimates for aggregation diffusion models, i.e., estimates of the form 
$\|u(t)\|_{L^\infty(\R^d)}\le C t^{-\sigma}$ for small time, with constants $C,\sigma$ depending on the parameters,  and, possibly, on the initial datum.
Also in this regard, the most studied problem is the classical subcritical two dimensional Keller Segel system with $m=1$ and $\alpha=0$. Estimating $L^p$ norms is in fact a classical step towards the proof of existence of global solutions \cite{JL,BDP,BCL,CC,CPZ, CPZ04, K}. Refined $L^p$ hypercontractivity estimates for finite $p$ are obtained in \cite{EM16, LW} and are a crucial step towards the uniqueness of solutions proved therein. Still in the setting $d=2$, $m=1$, $\alpha=0$, with subcritical mass, existence, uniqueness and 
 optimal $L^\infty(\R^d)$ hypercontractivity rate, i.e., the estimate $\|u(t)\|_{L^\infty}\le C t^{-1}$ for small $t$, is obtained in \cite{BM}, also for measure data. The same estimate has been recently obtained in \cite{EJS} (based on the proof of Aronson-B\'enilan estimates) for measure data of small mass, and for data with finite energy
 in the general subcritical range $M<M_c$. 
 In \cite{EJS}, this result is shown also  for the fair competition regime in higher dimension.  Of course the $L^\infty(\R^d)$ estimates become uniform even for small $t$ under the additional assumption that $u_0\in L^\infty(\R^d)$, and interestingly such uniform estimates have been also obtained
for discrete solutions of the JKO scheme, \cite{CS18,DMS22}. In the diffusion dominated regime, hypercontractivity estimates are obtained with suboptimal rates in \cite{BLZ}.
 Going back to the uniqueness issue, a general result covering fair competition and diffusion dominated regimes has been obtained in \cite{CLM} for bounded data. Other general uniqueness results are found in \cite{KS0,LW2}, still for bounded data.

\medskip
The main contributions of this paper are in the diffusion dominated regime $m>m_c$ and are given in Theorem \ref{th:main} below. We provide sharp $L^\infty(\R^d)$ smoothing effect of solutions. Most importantly, in combination with the evolution variational inequality formulation of the dynamics introduced in \cite{CLM}, this allows to obtain existence and uniqueness of the solution for general initial data in $\Space$. This settles a problem that was left open in \cite{CLM}, where uniqueness crucially required $L^\infty(\R^d)$ initial data, and to obtain a full parallel with the theory of gradient flows of $\lambda$-convex functionals along Wasserstein geodesics, providing semigroup generation in $\Space$, smoothing effect, evolution variational inequalities, energy dissipation equality, and convergence of the minimizing movements scheme. Furthermore, building on the results on stationary states of the dynamics from \cite{CHVY}, we obtain as a corollary that for every initial datum in $\Space$ the solution converges as $t\to+\infty$ to the unique (up to translations) minimizer $u_M$ of the free energy $\mathcal E_m$ among densities with given total mass $M$, at least in dimension $d=2$ with $\alpha=0$ (see Corollary \ref{maincoro} below). 
Concerning the fair competition regime $m=m_c$, our results still apply at least if $u_0\in D(\mathcal E_m)$, see Thereom \ref{th:main2} below. On the other hand, in this case  sharp $L^\infty(\R^d)$ smoothing effect is already available in the literature as previously discussed, as well as uniqueness of solutions in dimension $2$.

\medskip
Before moving towards the statement of the main results, 
 recalling that $M_c$ is defined in \eqref{defMc}, \EEE
let us fix the basic assumption on the mass $M$  and on the parameters of the equation:
\begin{equation}\label{assumption}\tag{{\bf{\bf A}}}
\begin{aligned}
&d\ge 2, \quad \alpha\geq0,\quad \chi>0, \quad m\ge m_c,\quad M>0,&\\
&\qquad\quad\mbox{if $m=m_c$  assume also $M<M_c$.}&
\end{aligned}
\end{equation}
After the introduction, without further mention,  \eqref{assumption} is assumed to be always valid.

 \subsection{The JKO scheme}
 In this paper we consider the gradient flow approach to \eqref{PE} and deduce all the main results from estimates on the JKO scheme. We denote by $W$ the Wasserstein distance on $\Space$ (see Section \ref{Sec:notation} for definition and properties). The narrow convergence on $\Space$ is defined in duality with continuous and bounded functions on $\R^d$, and the convergence in the Wasserstein distance $W$ coincides with narrow convergence plus convergence of second moments.
 
Given $u_0\in \Space $,
we consider a uniform partition of size $\tau>0$ of the time interval $[0,+\infty)$ and we let
$u_\tau^0$ the following approximation of the initial datum $u_0$:
\begin{equation}\label{utau0}
 u_\tau^0:=\Gamma_{\varpi(\tau)}\ast u_0,\qquad\text{where}\quad \varpi(\tau):=\begin{cases} -1/\ln \tau &\mbox{if }\tau\in(0,1/2) \\
-1/\ln (1/2) & \mbox{if }\tau\in[1/2,+\infty),
\end{cases}
 \end{equation}
where $\Gamma_t$ denotes the heat kernel defined by
\begin{equation}\label{gaussian}
\Gamma_t(x):=\frac{1}{(4\pi t)^{d/2}}e^{-|x|^2/4t},\quad x\in\mathbb{R}^d,\;t>0.
\end{equation}
We clearly have $u_\tau^0\to u_0$ in $(\Space,W)$, i.e., $W(u_\tau^0,u_0)\to0$ as $\tau\to0$,
and $u_\tau^0\in D(\mathcal F_r)$ for every $r\ge 1$. Moreover, thanks to the assumption \eqref{assumption}, the energy $\mathcal E_m(u_\tau^0)$ is finite and there holds   $\mathcal E_m(u_\tau^0)\to \mathcal E_m(u_0)$ as $\tau\to0$, even if $\mathcal E_m(u_0)=+\infty$ (see also Proposition \ref{initialdatum} below).
We notice that by Young convolution inequality, for every $p\in[1,+\infty]$ we have
\begin{equation}\label{heatcontraction}
	\|u_\tau^0\|_{L^p(\R^d)}\le M\|\Gamma_{\varpi(\tau)}\|_{L^p(\R^d)}\le M p^{-\tfrac d{2p}}\,(4\pi\varpi(\tau))^{-\tfrac{d}{2}\left(1-\tfrac1p\right)},
\end{equation}
where,   for $p=+\infty$, we use the convention $p^{-\tfrac d{2p}}=1$ and $\frac1p=0$. \EEE
Therefore, $\|u^0_\tau\|_{L^p(\R^d)}$ may diverge at most logarithmically as $\tau\to0$.

For every integer $k\ge1$, we recursively  select \RRR
\begin{equation}\label{minmov1}
u_\tau^k\in{\rm argmin}_{u\in \Space}\left\{ \EE_m(u)+\frac1{2\tau}\,W^2(u,u^{k-1}_\tau)\right\}.
\end{equation}
A  sequence of minimizers $\{u_\tau^k: k=1,2,\ldots\}$ satisfying \eqref{minmov1} does exist under assumption \eqref{assumption}, by means of very standard arguments recalled in Section \ref{Sec:notation}. 
Given any sequence $\{u_\tau^k:{k}=0,1,2,\ldots\}$ satisfying \eqref{utau0} and \eqref{minmov1}, we define
the piecewise constant interpolation 
$u_\tau:[0,+\infty)\to \Space $  by
\begin{equation}\label{floor}
u_\tau(t):=u^k_{\tau}, \qquad t\in ((k-1)\tau,k\tau], \qquad k=0,1,2,3,\ldots,
\end{equation}
i.e., $u_\tau(t)=u_\tau^{\lceil t/\tau\rceil},$ $t\ge 0$, where $\lceil x\rceil:=\min\{z\in\mathbb Z: x\le z\}$.
 Any piecewise constant curve defined as above 
will be called a \emph{discrete
 solution} of the JKO scheme. 

\begin{definition}\label{GFdefinition}
We say that a curve $u\in C([0,+\infty);(\Space,W))$ is a 
 {\it gradient flow} \RRR in $(\Space, W)$ of the energy functional $\mathcal E_m$, starting from $u_0$,   if 
\begin{equation}\label{0conv}
u_{\tau_n}(t) \to u(t)\quad\mbox{ narrowly as $n\to+\infty$}, \qquad\mbox{for every $t\ge 0$},
\end{equation}
 for some sequence of discrete solutions $(u_{\tau_n})_n$ of the JKO scheme, associated to some vanishing sequence $(\tau_n)_n$ of time steps. \end{definition}
 The above definition is related to the notion of generalized minimizing movement in the space $(\Space,W)$ according to \cite[Section 2]{AGS}. 
 In our main results, we will prove uniqueness of the gradient flow of functional $\mathcal E_m$, starting from $u_0$, so that given any family of discrete solutions $(u_\tau)_{\tau>0}$ of the JKO scheme, the above convergence will hold as $\tau\to 0$ 
  and not only along some sequence $(\tau_n)_n$.\EEE
 

%

\EEE

Before stating the main result, we recall the notion of $AC^2$ curve in the Wasserstein space $(\Space,W)$. 
We say that a curve $u:[0,+\infty)\to \Space$ is locally absolutely continuous with finite energy, and we use the notation
$u\in \AC^2_{loc}((0,+\infty);\Space)$, if there exists
$g\in L^2_{loc}((0,+\infty);\R)$ such that $W(u(t_1),u(t_2))\leq \int_{t_1}^{t_2} g(r)\,\d r$ for every $0<t_1<t_2<+\infty$.
In this case we can take $g$ to be the Wasserstein {\it metric derivative} of the curve, defined by the $a.e.$ existing limit $|u'|(t):=\lim_{s\to t}\tfrac{W(u(s),u(t))}{|s-t|}$.


\subsection{The main result in the diffusion dominated regime}
We now state our first main result, concerning the case $m>m_c$. We prove existence and uniqueness of gradient flows according to Definition \ref{GFdefinition} along with several more properties.

%

\begin{theorem}\label{th:main}
Under assumption \eqref{assumption}, let $m>m_c$
and $u_0\in \Space$. \smallskip

{\rm\textbf{(I)}} {\bf Existence and regularity \EEE of discrete minimizers.} 
 Let $\tau>0$.   There exists a sequence $\{u_\tau^k: k=0,1,2,\ldots\}$ satisfying \eqref{utau0} and \eqref{minmov1}.
Moreover, for any such sequence, letting $v_\tau^k:=B_\alpha\ast u_\tau^k$, there holds
\begin{equation*}
\begin{aligned}&(u_\tau^k)^q\in H^1(\R^d)\;\; \mbox{for every $q\in[\tfrac m2,+\infty)$ and every  $k\ge 0$},\\& \nabla v_{\tau}^k\in W^{1,p}(\R^d)\;\;\mbox{for every $p\in(\tfrac d{d-1},+\infty)$  and every $k\ge 0$.}\end{aligned}
\end{equation*}

{\rm\textbf{(II)}} {\bf Existence and uniqueness of the gradient flow.} 
  There exists a  unique \EEE curve $u\in C([0,+\infty);(\Space,W))$ such that for every family $(u_\tau)_{\tau>0}$ of  discrete solutions of the JKO scheme there holds 
\begin{equation}\label{wholetau}
u_{\tau}(t) \to u(t)\quad\mbox{ narrowly as $\tau\to0$}, \qquad\mbox{for every $t\ge 0$}.
\end{equation}
 \RRR
Moreover, $u$ satisfies
\[\begin{aligned}&u\in \AC^2_{loc}((0,+\infty);\Space),\\
&u\in C((0,+\infty);L^p(\R^d))\quad\mbox{for every $p\in[1,+\infty)$},\\
& u^{q}\in L^2_{loc}((0,+\infty);H^1(\R^d)), \quad\mbox{for every $q\in[m/2,+\infty)$},\RRR\\
& v:=B_\alpha\ast u \in C((0,+\infty); W^{2,p}_{loc}(\R^d)) \quad\mbox{for every $p\in[1,+\infty)$}.\RRR
\end{aligned}\]

{\rm\textbf{(III)}} {\bf Convergence properties of the scheme.}
For every family $(u_\tau)_{\tau>0}$ of discrete solutions of the JKO scheme, letting $v_\tau=B_\alpha\ast u_\tau$,  there hold as $\tau\to 0$
 \begin{equation*}
\begin{aligned}
&\mbox{$u_{\tau}(t) \to u(t)$ weakly in $L^p(\R^d)$, for every $p\in[1,+\infty)$ and every $t>0$},\\
&\mbox{$u_{\tau}\to u$ strongly in $L^p_{loc}((0,+\infty);L^q(\R^d))$ for every $1\le p,q<+\infty$,}\\
&\mbox{$\nabla v_{\tau}\to \nabla v$ strongly in $L^p_{loc}((0,+\infty);L^q(\R^d))$ for every $p\in[1,+\infty)$ and $q\in(\frac{d}{d-1},+\infty].$}
\end{aligned}
\end{equation*}

 \RRR

{\rm\textbf{(IV)}} {\bf Solution of the equation.}   $u$ is a global solution of the PDE in \eqref{PE} \RRR in the following weak form
\[
\int_0^{+\infty}\int_{\R^d} \left( u\,\partial_t\varphi  - \nabla\varphi\cdot (\nabla u^m - \chi u \nabla v )\right)\, \dd x \, \dd t=0, \quad
\forall\,\varphi\in C^\infty_c((0,+\infty)\times\R^d).
\]


{\rm\textbf{(V)}} {\bf Energy dissipation identity.} 
The map $t\mapsto \mathcal E_m(u(t))$ is locally absolutely continuous on $(0,+\infty)$
and the following energy equality holds
\begin{equation*}
\EE_m(u(t_2))+\int_{t_1}^{t_2}\int_{\mathbb{R}^d}\left | \frac{\nabla u^{m}(t) - \chi u(t) \nabla v(t)} {u(t)}\right|^2u(t)\,\d x\,\d t = \EE_m(u(t_1)), \qquad \forall \;t_2>t_1 \ge 0.
\end{equation*}

{\rm\textbf{(VI)}} {\bf Smoothing effect.} There exists a constant $C_{\chi,d,m}$, only depending on $\chi,d,m$, such that
\[
\|u(t)\|_{L^\infty({\R}^d)} \le C_{\chi,d,m} \,\left(M^{\textstyle\frac{2}{2+d(m-1)}}\,\left(\frac{1}{t}\right)^{\textstyle\frac{d}{2+d(m-1)}}+M^{\textstyle\frac{2}{2+d(m-2)}}\right)\qquad\mbox{for every $t>0$.}
\]
If $u_0\in L^\infty(\R^d)$, then there exists  a constant $\bar C_{\chi,d,m}$, only depending on $\chi,d,m$, such that
\[
\|u(t)\|_\infty\le \bar C_{\chi,d,m}\left(M^{\textstyle\frac{2}{2+d(m-2)}}\vee\|u_0\|_{L^\infty(\R^d)}\right)\qquad\mbox{for every $t>0$}.
\]
\end{theorem}
\bigskip

For $m>m_c$,  every stationary state
 of the dynamics is radially decreasing and compactly supported, as shown in \cite{KY,CCV,CHVY}, and there is a unique (up to translations) stationary state of mass $M$, coinciding with the unique (up to translations) solution to the minimization problem
\begin{equation}\label{minprob}
\min\left\{\mathcal E_m(u): u\in L^m(\R^d),\; u\ge 0,\; \int_{\R^d} u=M \right\}.
\end{equation}
The energy functional $\mathcal E_m$ generates a nonlinear semigroup $S_t:\Space\to\Space$ that associates to every initial datum $u_0\in\Space$ the corresponding unique  curve $u$ given by Theorem \ref{th:main} evaluated at time $t>0$. \EEE The regularizing effect provided by Theorem \ref{th:main} implies that $S_t(u_0)\in D(\mathcal E_m)\cap L^\infty(\R^d) $ for every $t>0$. Concerning the long time asymptotics of the dynamics, it is still an open problem to determine whether  $S_t(u_0)$ converges for every $u_0\in \Space$ to the unique stationary state with same mass and center of mass of $u_0$, as $t\to+\infty$.  This result has been proven only in space dimension $2$ with $\alpha=0$, see \cite[Theorem 4.12]{CHVY}, for $L^\infty(\R^d)$ initial data. The semigroup property and the $L^\infty$ smoothing effect that we have proven in Theorem \ref{th:main} allow to  immediately extend the result of \cite{CHVY} to any initial data in $\Space$. 


\begin{corollary}\label{maincoro} 
Under assumption \eqref{assumption}, let $d=2$, $\alpha=0$, $u_0\in \mathscr P_2^M(\R^2)$  and $m>1$. Let $S_t(u_0)$ the unique gradient flow for functional $\mathcal E_m$, starting from $u_0$, provided by {\rm Theorem \ref{th:main}}. 
Then 
\[
\lim_{t\to+\infty} \|S_t(u_0)-u_M^0\|_{L^p(\R^2)}=0\qquad \mbox{ for every $p\in[1,+\infty)$},
\]
 where $u_M^0$ is the unique solution of the minimum problem in \eqref{minprob} having the same center of mass of $u_0$. \EEE
\end{corollary}

Let us briefly comment on Theorem \ref{th:main}.  
A key result is the hypercontractivity estimate of  point {\textbf{(VI)}, in which the rate of explosion of the $L^\infty$ norm as $t\to 0$ is sharp, since it is the same rate  obtained for the classical porous medium equation $\partial_t u=\Delta u^m$, which corresponds to the case $\chi=0$. \EEE 
Indeed, it is well known that for the porous medium equation, the following estimate holds for a suitable constant $C_{d,m}$ depending only on $d$ and $m$, see \cite[Theorem 2.1]{V2}, for every initial datum in $\mathscr P^M(\R^d)$,
\begin{equation}\label{smoothingporous}
\|u(t)\|_{L^\infty({\R}^d)} \le C_{d,m} \,M^{\textstyle\frac{2}{2+d(m-1)}}\,\left(\frac{1}{t}\right)^{\textstyle\frac{d}{2+d(m-1)}}\qquad\mbox{for every $t>0$.}
\end{equation}
For the heat equation $m=1$, this is just \eqref{heatcontraction}.
A summary of the key steps towards the proof of  point {\textbf{(VI)}} is  given at the beginning of Section \ref{hypersection}.
%
We also prove in Section \ref{hypersection} that the hypercontractivity rate $\tfrac{d}{2+d(m-1)}$ improves to $\tfrac{d}{2p+d(m-1)}$ if we assume that $u_0\in L^p(\R^d)$ for some $p>1$, and that the uniform bound of point ${\textbf{(VI)}}$ of Theorem \ref{th:main}  holds if $u_0\in L^\infty(\R^d)$.

 In Section \ref{lastsection},
 once the $L^\infty$ estimate has been established, we  combine it with the evolution variational inequality (EVI) formulation  of the gradient flow \EEE
  that we have proved in \cite{CLM}:  for every $\bar u\in\PP_2^M(\mathbb R^d)$ there holds  
\begin{equation*}\begin{aligned}
   & \frac{1}2\frac{d}{dt} W_2^2(u(t),\bar u) \le\,
\EE_m(\bar u) - \EE_m(u(t))\\&\qquad+ \chi C_d(M+\|u(t)\|_{\infty}+\|\bar u\|_\infty)\,
\omega(W_2^2(u(t),\bar u))  \quad \text{for a.e. }
t>0,\end{aligned}
\end{equation*}
where $C_d\ge0$ is a constant only depending on $d$. This is a weaker form of the classical evolution variational inequality formulation of gradient flows for convex functionals along Wasserstein geodesics, as the remainder features a function $\omega(x)$ that is not linear for small $x$ but behaves like $x|\log x|$. This corresponds to a $\omega$-convexity property along geodesics as discussed in \cite{CLM}, see also \cite{Craig}. It is crucial to notice that the coefficient in front of $\omega$ linearly depends on $\|u(t)\|_\infty$. Thanks to the EVI formulation and to the the sharp $L^\infty$ estimate at our disposal, we obtain a stability inequality and, as a consequence, the uniqueness of the gradient flow. Uniqueness also implies that the whole family $(u_\tau)_{\tau>0}$ of  discrete solutions \EEE converge to $u$ as $\tau$ to zero as noticed in point {\textbf{(II)}}. Finally, suitable adaptations to the standard theory of convex functionals along Wasserstein geodesics allow to obtain the energy dissipation equality from point {\textbf{(V)}}, as we do in the last part of Section \ref{lastsection}.

%
%
%
%
%

%
%
%


\vspace{3mm}

\subsection{The main result in the fair competition regime}
 We define the {\it subcritical mass} 
\begin{equation}\label{defMsc} 
M_{sc}:=\displaystyle\frac4{\chi G_1}\quad \text{if }d=2,\qquad\quad
M_{sc}:=\displaystyle\left(\frac{4}{\chi C_{2,d}\,m_c}\right)^{d/2}\;\; \text{if }d\geq3,
\end{equation}
where $C_{2,d}=S_{2,d}^2$ and  $S_{2,d}$ is the optimal constant in the classical Sobolev inequality \eqref{sobolev}
%
 and $G_{1} $ is the optimal constant in the
Gagliardo-Nirenberg-Sobolev inequality \eqref{GNS}.
The explicit expression of $C_{2,d}$  will be recalled in Section \ref{Sec:notation}, along with the the approximate value of $G_1$.
\RRR
 The relation $M_{sc}<M_c$ holds in any dimension $d\ge 2$ and   will also be proved in Section \ref{Sec:notation}  (see Remark \ref{remarchino} and Remark \ref{remarcone}). \EEE

\begin{theorem}\label{th:main2}
Under assumption \eqref{assumption},
let $m=m_c$ and
$u_0\in D(\EE_m)$. 
Then {\rm{\textbf{(I), (II), (III), (IV), (V)}}} of {\rm Theorem \ref{th:main}} hold true and $u\in AC^2_{loc}([0,+\infty);(\Space,W))$. Moreover, the following properties hold:

\medskip
{\rm\textbf{(VI)}} {\bf Smoothing effect.} There exist  constants $C> 0$, $q\ge 0$,  depending on $M$, $\chi$, $d$, $\alpha$, 
$\mathcal E_m(u_0)$ and the second moment of $u_0$, with $q=0$ if $d\ge 3$, such that
\[
\|u(t)\|_{L^\infty(\R^d)}\le  C\left(\frac1t\right)^{\textstyle\frac {d^2}{d(d+2)-4}}+Ct^q\qquad\mbox{for every $t> 0$}.
\]
Under the additional assumption $u_0\in L^\infty(\R^d)$, 
there exists  $ \bar C>0$ depending on $M$, $\chi$, $d$, $\alpha$, 
$\|u_0\|_{L^\infty(\R^d)}$ 
and the second moment of $u_0$, such that
$$\|u(t)\|_{L^\infty(\R^d)}\le \bar C+\bar Ct^q \qquad\mbox{for every } t>0.$$

\medskip
{\rm\textbf{(VII)}} {\bf Uniform decay for small mass.}
If  $M<M_{sc},$ 
then there exists a constant $ C^*= C^*_{\chi,d,M}$  depending only on ${\chi,d,M}$ such that \EEE
\[
\|u(t)\|_\infty\le  \frac{C^*}{t}\qquad\mbox{for every  $t>0$.}
\]
\end{theorem}
\bigskip

The proof of Theorem \ref{th:main2} follows the same line as that of Theorem \ref{th:main}. Some extra care is needed in dimension $2$, where $m=1$, in which case the estimates of discrete minimizers of the JKO scheme require suitable use of equiintegrability properties. The obtained hpercontractivity rates are again sharp for $d\ge 3$, taking into account the assumption $u_0\in D(\mathcal E_m)$,  while for $d=2$ the expected rate would be $1/|t\log t|$, consistently with the heat equation, but we will not pursue this improvement  (a logarithmic correction to the $L^p$ hypercontractivity rate is obtained in \cite{LW} but only for finite $p$). 

Asymptotic results for $m=m_c$, in the subcritical two-dimensional model with $\alpha=0$, show that the solution gets close to a self similar solution for large time \cite{BDP, CC2, CD, EM16}. On the other hand, uniform extinction of solutions for large $t$, at the optimal rate $1/t$, is often  proven under smallness conditions on the mass, see for instance \cite{BDEF}. See also   \cite{EJS}  where the authors treat the case $m=m_c$ even in higher dimension. It is natural to conjecture the validity of such an estimate in all the subcritical regime $M<M_c$, which we do not prove, but we have an improvement of the range of values of the mass for which $L^\infty$ decay holds. For instance in dimension $2$ we reach the mass threshold $4/{(\chi G_1)}$ which is classically obtained as the formal threshold under which the logarithmic entropy $\mathcal F_1$ decreases along the solution, see \cite{BDP}.

Finally, let us mention that in dimension $d=2$, the uniqueness part in Theorem \ref{th:main2} does not directly follow from our analysis (unless we assume in addition that $u_0\in L^p(\R^d)$ for some $p>1$, as we will  explain in Section \ref{lastsection} and Section \ref{sectionproof}), but can be directly inferred by combining our results with the ones by \cite{EM16}.



\section{Notation and preliminary results} \label{Sec:notation}
\subsection{Wasserstein distance}
We denote by $\PP(\mathbb{R}^d)$ the set of Borel probability measures on $\mathbb{R}^d$ and by 
$\PP^M(\R^d):=\{Mu:u\in\PP(\R^d)\}$ the set of non-negative Borel measures of mass $M>0$.
We say that a sequence $\{u_n\}_n\subset \PP^M(\R^d)$ narrowly converges to $u\in  \PP^M(\R^d)$ if
$$\displaystyle\int_{\R^d} \phi \, \d u_n \to \int_{\R^d} \phi \, \d u\qquad\mbox{for every} \quad \phi\in C_b(\R^d),$$
where  $C_b(\R^d)$ is the set of continuous and bounded functions on $\R^d$.
The second moment and the logarithmic moment of $u\in\mathscr P^M(\R^d)$ are defined by $$\mathrm m_2(u)=\int_{\R^d}|x|^2\,\d u(x),\qquad\mathrm m_{log}(u)=\int_{\R^d}\log(1+|x|^2)\,du(x),$$ 
noticing that $\mathrm m_{log}(u)\le \mom(u)$.
We use the notation
$$\displaystyle\PP^M_2(\R^d):=\left\{u\in\PP^M(\R^d):  \mathrm m_2(u)<+\infty \right\},$$
for the subset of $\PP^M(\R^d)$ made by measures with finite second moment.
The Kantorovich-Rubinstein-Wasserstein distance (in the sequel only Wasserstein distance) {\RRR W} in $\PP^M_2(\R^d)$ is defined by
\begin{equation}\label{Kanto}
W(u,v):=\min_{\gamma\in\PP^M({\R^d\times\R^d})}\left\{\left(\int_{\R^d\times\R^d}\!\!\!\!\!|x-y|^2\,\d\gamma(x,y)\right)^{1/2}:
\,(\pi_1)_\#\gamma=u,\,(\pi_2)_\#\gamma=v
\right\}
\end{equation}
where $\pi_i:\R^d\times\R^d\to\R^d$, $i=1,2$, denote the canonical projections on the first and second factor respectively.
Denoting by $\id$ the identity map in $\R^d$, when $u$ is absolutely continuous with respect to the Lebesgue measure, 
the minimum problem~\eqref{Kanto}
has a unique solution $\gamma$ induced by a transport map $T_u^v$ in the following way:
$\gamma=(\id,T_u^v)_\#u$.
In particular, $T_u^v$ is the unique solution of the Monge optimal transport problem
$$
\min_{S:\R^d\to\R^d}\left\{\int_{\R^d}|S(x)-x|^2\d u(x):\ S_\#u=v\right\}
$$
of which~\eqref{Kanto} is the Kantorovich relaxed version. The push forward notation $S_\#u=v$ means that $v(A)=u(S^{-1}(A))$ for every Borel set $A$ of $\R^d$.
The function $W:\PP^M_2(\R^d)\times\PP^M_2(\R^d)\to [0,+\infty)$ is a distance on $\PP^M_2(\R^d)$ and $(\PP^M_2(\R^d),W)$ is a complete and separable metric space.
Moreover the distance $W$ is sequentially lower semi continuous with respect to the narrow convergence, i.e.,
\begin{equation}\label{lscWass}
	 u_n\to u,\quad v_n\to v, \mbox{ narrowly } \quad\Longrightarrow \quad\liminf_{n\to+\infty}W(u_n,v_n)\geq W(u,v),
\end{equation}
and bounded sets in $(\PP^M_2(\R^d),W)$ are narrowly sequentially relatively compact.

 For a detailed treatment of the above topics, see \cite{AGS}, \cite{Vil},  \cite{Sant}. \EEE
\subsection{Gagliardo-Nirenberg-Sobolev inequalities}
We recall the basic Sobolev inequality: for $1\le q<d$ and $q^*=\frac{dq}{d-q}$ there exists an optimal constant $S_{q,d}$ such that
\begin{equation}\label{sobolev}
\left(\int_{\mathbb{R}^d} |v|^{q^*}\,\d x\right)^{1/{q^*}} \le S_{q,d}\left(\int_{\mathbb{R}^d} |\nabla v|^q\,\d x \right)^{1/q}, \qquad \forall\, v\in W^{1,q}(\R^d),
\end{equation}
see for instance  \cite[Theorem 9.9]{Brezis}. 
The optimal constant in \eqref{sobolev} is
\begin{equation}\label{talenti}
 S_{1,d}=\frac{\Gamma(1+d/2)^{1/d}}{\sqrt{\pi}d}, \quad 
 S_{q,d}=\frac{1}{\sqrt{\pi}d^{1/q}}\left(\frac{q-1}{d-q}\right)^{1-1/q}\left(\frac{\Gamma(1+d/2)\Gamma(d)}{\Gamma(d/q)\Gamma(1+d-d/q)}\right)^{1/d},
\end{equation}
see \cite{T}, where $\Gamma$ is the Euler Gamma function.
 For notational convenience we also let
\begin{equation}\label{sq}
C_{q,d}:= S_{q,d}^2,
\end{equation}
and we stress that $C_{2,d}$ is the constant in the definition of  $M_{sc}$, see  \eqref{defMsc}. 
We observe that $C_{1,2}=1/(4\pi)$.

In dimension $d=2$, we will use this version of the Gagliardo-Nirenberg-Sobolev inequality:
 given  $p\ge 1$ there exists an optimal constant $G_p$ such that  \EEE
\begin{equation}\label{GNS}
\int_{\mathbb{R}^2}u^{p+1}\,\d x \le G_p\int_{\mathbb{R}^2} u \,\d x \int_{\mathbb{R}^2}|\nabla u^{p/2}|^2\,\d x,
 \end{equation}
 for any $u\in L^1(\R^2)$ such that $u\geq 0$ and $u^{p/2} \in H^{1}(\R^2)$.
  The sharp constant $G_p$ is however not explicitly known. 
A simple estimate of the constant $G_p$ is
\begin{equation}\label{badupper}
G_p\le\frac{(p+1)^2}{p^2}\, S_{1,2}^2=\frac{(p+1)^2}{4{\pi}p^2},
\end{equation}
which  can be  obtained from \eqref{sobolev} for $q=1$ and $v=u^{(p+1)/2}$ and the Cauchy-Schwarz inequality applied to $u^{1/2}$ and $|\nabla u^{p/2}|$. An easy estimate from below for $G_p$ can also be obtained, by testing the inequality \eqref{GNS} with $w_\lambda(x):=e^{-|x|^\lambda}$, $\lambda>0$, i.e.,
$$
G_p\ge \sup_{\lambda>0}\frac{\int_{\mathbb R^2}w_\lambda^{p+1}\,\d x}{\int_{\mathbb R^2} w_\lambda\,\d x\,\int_{\mathbb R^2}|\nabla w_{\lambda}^{p/2}|^2\,\d x}=\sup_{\lambda>0}\frac{2}{\pi\lambda(p+1)^{2/\lambda}}=\frac1{\pi\,e\log(p+1)}.
$$\RRR

In the next Lemma we state a generalization of \eqref{GNS}, needed in the sequel.\EEE
\begin{lemma} Let $p>s\ge 1$ and $m\ge 1$. 
Then, there exists a constant $K_{p,s}$ such that
\begin{equation}\label{nuovagns}
	\left(\int_{\R^2}u^{p+1}\,\d x\right)^{\textstyle\frac{p+m-1}{p-s+1}}\le K_{p,s}\left(\int_{\R^2}u^s\,\d x\right)^{\textstyle\frac{p+m-1}{p-s+1}}\int_{\R^2}\left|\nabla u^{\frac{p+m-1}{2}}\right|^2\,\d x,
\end{equation}
for any $u\in L^s(\R^d)$ such that $u\geq 0$ and $u^{\frac{p+m-1}{2}} \in H^1(\R^2)$.
A possible value for the constant $K_{p,s}$  is 
\begin{equation}\label{kps} K_{p,s}=\frac1{4\pi}\left(\frac{p+1}{p-s+1}\right)^2.\end{equation}
\end{lemma}

\begin{proof}
The case $m=s=1$ is obtained from \eqref{GNS} and the estimate \eqref{badupper}. 

Assume that $m+s>2$ and let $b:=\frac{p+m-1}{p+1-s}$.
We apply \eqref{sobolev}  with $d=2$, $q=1$, to $v= u^{\frac{b(p+1)}2}$. Therefore, with the notation $\bar u:=u^b$, we have
\[
\int_{\R^2}\bar u^{p+1}\,\d x\le S_{1,2}^2\left(\int_{\mathbb R^2}\left|\nabla \bar u^{\frac{p+1}{2}}\right|\,\d x\right)^2.
\]
Since $\nabla \bar u^{\frac{p+1}{2}}=\frac{p+1}{p+1-s}\,\bar u^{\frac s2}\,\nabla \bar u^{\frac{p+1-s}{2}}$, by Cauchy-Schwarz inequality we get
\begin{equation*}
\int_{\R^2}\bar u^{p+1}\d x\le S_{1,2}^2\left(\frac{p+1}{p+1-s}\right)^2\left(\int_{\R^2}\bar u^{\frac{s}{2}}\,|\nabla \bar u^{\frac{p+1-s}{2}}|\,\d x\right)^2\le 
K_{p,s}\int_{\R^2}\bar u^s\,\d x\,\int_{\R^2}|\nabla \bar u^{\frac{p+1-s}{2}}|^2\,\d x,
\end{equation*}
that is,
\begin{equation}\label{thekappa}
\int_{\R^2} u^{b(p+1)}\d x\le 
K_{p,s}\int_{\R^2} u^{bs}\,\d x\,\int_{\R^2}|\nabla u^{\frac{p+m-1}{2}}|^2\,\d x,
\end{equation}
having used \eqref{sq} and \eqref{kps}.  Since $m+s>2$ we have $b>1$ and we make use of the following $L^s(\R^2)-L^{sb}(\R^2)-L^{(p+1)b}(\R^2)$ and $L^s(\R^2)-L^{p+1}(\R^2)-L^{(p+1)b}(\R^2)$ interpolation inequalities
\[
\|u\|_{sb}\le\|u\|_s^{1-\theta}\|u\|_{(p+1)b}^{\theta}\qquad\theta:=\frac{(p+1)(b-1)}{b(p+1)-s},
\]
\[
\|u\|_{p+1}\le \|u\|_s^{1-\xi}\|u\|_{(p+1)b}^\xi,\qquad\xi:=\frac{b(p+1)-bs}{b(p+1)-s},
\]
which can be inserted into \eqref{thekappa} to get
\[
\left(\int_{\R^2}u^{p+1}\,\d x\right)^{\textstyle\frac{b(p+1)-sb\theta}{(p+1)\xi}}\le K_{p,s}\left(\int_{\R^2}u^s\,\d x\right)^{\textstyle\frac{sb\xi(1-\theta)+(bp+b-sb\theta)(1-\xi)}{s\xi}}\int_{\R^2}\left|\nabla u^{\frac{p+m-1}{2}}\right|^2\,\d x,
\]
and a computation using the definitions of $b,\theta,\xi$ shows that
$$
\frac{sb\xi(1-\theta)+(bp+b-sb\theta)(1-\xi)}{s\xi}=\frac{b(p+1)-sb\theta}{(p+1)\xi}=b=\frac{p+m-1}{p+1-s}
$$
so that we recover \eqref{nuovagns}.
\end{proof}

 Going back to the optimal constant $G_p$ in \eqref{GNS},
 we next obtain further estimates, that will be useful later on for the characterization of the two-dimensional subcritical mass $M_{sc}=4/(\chi G_1)$ from \eqref{defMsc}, in terms of the values of $G_p$ for $p>1$. 
 \begin{lemma}\label{GpG1}
 There holds
 $G_p\le G_1\le p\,G_p$ for every $p\ge 1$, where $G_p$ denotes the best constant of the inequality \eqref{GNS}.  In particular,
 \[
 G_1=\lim_{p\downarrow 1} p\,G_p=\inf_{p>1}p\,G_p.
 \]
 \end{lemma}
 \begin{proof}
  Letting $p\in(1,+\infty)$,  by $L^1(\R^2)-L^{p+1}(\R^2)-L^{2p}(\R^2)$ interpolation  we have
\[
\int_{\R^2}u^{p+1}\,\d x
\le
\left(\int_{\R^2}u\,\d x\right)^{\frac{p-1}{2p-1}}
\left(\int_{\mathbb R^2}u^{2p}\,\d x\right)^{\frac{p}{2p-1}},
\]
where we include the inequality \eqref{GNS} with $p=1$, applied to $u^p$  in place of $u$, to get
\begin{equation}\label{2p-1}
\int_{\mathbb R^2} u^{p+1}\,\d x
\le
\left(\int_{\mathbb R^2} u\,\d x\right)^{\frac{p-1}{2p-1}}
\left(G_1\left(\int_{\mathbb R^2} u^p\,\d x\right)
\left(\int_{\mathbb R^2} |\nabla u^{p/2}|^2\,\d x\right)
\right)^{\frac{p}{2p-1}}.
\end{equation}
Moreover, by the $L^1(\R^2)-L^p(\R^2)-L^{p+1}(\R^d)$ interpolation inequality
\[
\int_{\mathbb R^2} u^p\,\d x
\le
\left(\int_{\mathbb R^2} u\,\d x\right)^{1/p}
\left(\int_{\mathbb R^2} u^{p+1}\,\d x\right)^{(p-1)/p},
\]
to be inserted into \eqref{2p-1}, we get
\[
\int_{\mathbb R^2} u^{p+1}\,\d x
\le
G_1
\left(\int_{\mathbb R^2} u\,\d x\right)
\left(\int_{\mathbb R^2} |\nabla u^{p/2}|^2\,\d x\right),
\]
so that by definition of $G_p$ as the optimal constant in \eqref{GNS}, we obtain $G_p\le G_1$ for every $p\in(1,+\infty)$.
%

On the other hand, the optimal constant $G_1$ of the inequality \eqref{GNS} with $p=1$  can be expressed as
  $G_1=2/\|g\|_2^2$, where $g$ is the unique ground state solution to \begin{equation}\label{groundstate}-\Delta u=u^3-u\qquad \mbox{in $\R^2$},\end{equation} with approximate value $G_1\simeq 0.171$, see \cite{Kw,W}. 
  The ground state $g$ is smooth, positive, radially decreasing. Moreover, $g$ and $g'$ decay exponentially fast at infinity (the decay rate of $g$ at infinity is the same as that of the Bessel kernel $B_1$). 
Identifying $g$ with its radial profile we have $g'(r)\le0$ for every $r>0$ and \begin{equation}\label{radialground}-g''-\frac1rg'+g-g^3=0,\qquad r>0.\end{equation} Let \begin{equation}\label{calcoloderivata}\Phi(r):=r^2\left(g(r)^2-g'(r)^2-\tfrac12 g(r)^4\right)-rg(r)g'(r),\qquad r\ge0.\end{equation}
 We compute the derivative
\begin{equation*}\begin{aligned}
\Phi'(r)&=2r\left(g(r)^2-g'(r)^2-\tfrac12 g(r)^4\right)+2r^2\left(g(r)g'(r)-g'(r)g''(r)-g(r)^3g'(r)\right)\\&\qquad\qquad-g(r)g'(r)-rg'(r)^2-rg(r)g''(r).
\end{aligned}\end{equation*}
But \eqref{radialground} entails  $$g(r)g'(r)-g'(r)g''(r)-g(r)^3g'(r)=\frac1r\,g'(r)^2,$$ $$rg(r)g''(r)=rg(r)^2-rg(r)^4-g(r)g'(r),$$ so that
\begin{equation}\label{derPhi}\Phi'(r)=r(g(r)^2-g'(r)^2).\end{equation} 
We claim that $\Phi(r)\ge0$ for every $r>0$. 
Arguing by contradiction, we assume that $\Phi$ takes some negative value on $(0,+\infty)$. Since $\Phi$ is smooth with $\Phi(0)=0$, and since $\Phi$ vanishes at infinity (due to its definition and to the exponential decay of $g,g'$) we deduce the existence of $r_0>0$ such that \begin{equation}\label{itsamin}\Phi(r_0)<0,\qquad \Phi'(r_0)=0,\qquad \Phi''(r_0)\ge 0,\end{equation} and then \eqref{derPhi} entails 
\begin{equation}\label{+o-}
g(r_0)=-g'(r_0).
\end{equation}
We compute $\Phi''(r)=g(r)^2+2rg(r)g'(r)-g'(r)^2-2rg'(r)g''(r)$, and taking \eqref{+o-} into account we deduce $\Phi''(r_0)=2r_0\left(g(r_0)g'(r_0)-g'(r_0)g''(r_0)\right)$. From \eqref{radialground} and using again \eqref{+o-} we get therefore $\Phi''(r_0)=2g(r_0)^2\left(1-r_0 g(r_0)^2\right)$.
Since $\Phi''(r_0)\ge 0$ as seen in \eqref{itsamin}, we conclude that
$
r_0g(r_0)^2\le 1.
$
By evaluating $\Phi(r_0)$ from \eqref{calcoloderivata} and by using \eqref{+o-}, we have 
\[
\Phi(r_0)=r_0g(r_0)^2\left(1-\frac12 r_0g(r_0)^2\right)
\]
and since we have just obtained $
r_0g(r_0)^2\le 1,
$ we deduce that $\Phi(r_0)\ge 0$ thus contradicting \eqref{itsamin}. The claim is proven and it implies that for every $p\in(1,+\infty)$
\[
\int_0^{+\infty}g(r)^{2p-2}\Phi'(r)\,\d r=-\int_0^{+\infty}\Phi(r)g_p'(r)\,\d r\ge 0
\]
having used that $g_p(r):=g(r)^{2p-2}$ is nonincreasing on $(0,+\infty)$, and then \eqref{derPhi} yields
\begin{equation}\label{radie}
\int_0^{+\infty}g(r)^{2p-2}(g(r)^2-g'(r)^2)\,r\,\d r\ge 0,\quad\mbox{i.e.,\quad}\int_{\R^2}\left(g^{2p-2}|\nabla g|^2-g^{2p}\right)\,\d x\le 0
\end{equation}
for every $p\in(1,+\infty)$.  By multiplying \eqref{groundstate} by $g^{2p-1}$ and integrating by parts we get
\[
(2p-1)\int_{\R^2}g^{2p-2}|\nabla g|^2\,\d x+\int_{\R^2}g^{2p}\,\d x=\int_{\R^2} g^{2p+2}\,\d x,
\]
thus \eqref{radie} implies
\begin{equation}\label{radie2}
\int_{\R^2}g^{2p+2}\,\d x\ge 2p\int_{\R^2}g^{2p-2}|\nabla g|^2\,\d x.
\end{equation}
Form the definition of $G_p$, since  $G_1=2/\|g\|_2^2$,  by making use of  \eqref{radie2} we get
\[
p G_p\ge \frac{p\int_{\R^2}g^{2p+2}\,\d x}{\int_{\R^2}|\nabla g^p|^2\,\d x\;\int_{\R^2}g^2\,\d x}=\frac{\int_{\R^2}g^{2p+2}\,\d x}{p\int_{\R^2}g^{2p-2}|\nabla g|^2\,\d x\int_{\R^2}g^2\,\d x}\ge\frac2{\int_{\R^2}g^2\,\d x}=G_1
\]
for every $p\in(1,+\infty)$,
thus concluding the proof.
\end{proof}
\begin{remark}\label{remarchino}\rm
Lemma \ref{GpG1} implies, since $G_1\simeq 0.171$, that for $d=2$ there holds
\[
M_{sc}=\sup_{p>1}\frac4{p\chi G_p}= \frac{4}{\chi G_1}\simeq\frac{23.391}{\chi}.
\]
We also notice that $M_{sc}<M_c=8\pi/\chi\simeq25.133/\chi$.
\end{remark}

\RRR

\RRR

%
%

\subsection{Bessel potentials and Hardy-Littlewood-Sobolev inequalities}
For $\alpha\geq 0$, the Bessel kernel $B_\alpha:\mathbb R^d\setminus\{0\}\to \mathbb R$ is defined  by
\[
B_\alpha(x):=\int_0^{+\infty} \frac{1}{(4\pi s)^{d/2}}\exp\left(-\frac{|x|^2}{4s}-\alpha s\right)\,\d s, \qquad \mbox{for } \alpha>0\quad\mbox{and}\quad d\ge 2,
\]
\[
B_0(x)=\frac{\Gamma(d/2)}{2(d-2)\pi^{d/2}}\,|x|^{2-d}, \qquad \mbox{for } d\geq 3,
\]
and 
\[
B_0(x)=-\frac1{2\pi}\log|x|, \qquad \mbox{for } d=2.
\]
For $\alpha=0$ we have therefore the Newtonian kernel.
It is not difficult to check that
\begin{equation}\label{alpha2}
	B_\alpha(x)\leq B_0(x), \qquad \forall\,x\in\R^d\setminus\{0\}, \quad \alpha \geq 0, \quad d\geq3
\end{equation}
and that 
\begin{equation}\label{alpha1}
B_\alpha(x)=\alpha^{\frac{d-2}{2}}\,B_1(\sqrt{\alpha}x), \qquad \forall\,x\in\R^d\setminus\{0\}, \quad \forall\alpha > 0. 
\end{equation}



In the next statements we recall Carleman and Hardy-Littlewood-Sobolev type inequalities.
 These inequalities are useful for proving boundedness and equintegrability of the sublevels of the energy functional  $\mathcal E_m$.


\begin{proposition}\label{carl}
Let $u\in L^1(\R^2)$ such that $u\geq 0$ and $\int_{\R^2}u(x)\,\d x =M>0$.  
If $\mathrm m_{log}(u) <+\infty$ and  $\FF_1(u)<+\infty$, 
then $\int_{\R^2}u\log^-u\,dx<+\infty$, where $\log^-\ge 0$ denotes the negative part of the natural logarithm, and
\begin{equation}\label{Carleman}
	\int_{{\R}^2}u|\log u|\,\d x\leq \FF_1(u) + 2M\ln\pi+\frac{2}{e} + 4 \mathrm m_{log}(u).
\end{equation}
and there exists a constant  $C_{M,\alpha}$,  depending only on $M$ and $\alpha$, \EEE such that
\begin{equation}\label{logbessel}
\frac{4\pi}M\int_{\mathbb{R}^2\times\mathbb{R}^2}B_\alpha(x-y) u(x)u(y)\,\d x\d y 
\leq \mathcal F_1(u) + 2\mathrm m_{log}(u) +C_{M,\alpha}.
\end{equation}


\end{proposition}
 The proof of \eqref{Carleman} can be carried out as a simple variant of the proof of \cite[Lemma 2.2]{BCC}. 
 The proof of \eqref{logbessel} is given  \cite{CL} in the case $\alpha=0$, where the constant is explicit, and in \cite[Lemma 4.2]{CaCo} for the case $\alpha>0$.

\begin{proposition}\label{HLS}
Let $d\ge 3$. Then there exists an optimal constant $C_H$, depending only on $d$,
such that
\begin{equation}\label{HLS2}
	\int_{\mathbb{R}^d\times \mathbb{R}^d} B_{\alpha}(x-y) u(x)u(y)\,\d x \,\d y
	\leq C_H M^{2/d}\int_{\mathbb{R}^d} u^{m_c}(x)\,\d x, \qquad \forall \, u \in L^{m_c}(\R^d).
\end{equation}
for any  $u \in L^{m_c}(\R^d)$ such that $u\geq 0$ and  $\int_{\R^2}u(x)\,\d x =M$, where $m_c$ is defined in \eqref{mc}. 
\end{proposition}

The proof of \eqref{HLS2} can be found in
\cite{BCL} and \cite{BL}, where it is shown that the optimal constant $C_H$ does not depend on $\alpha$.

A direct consequence of Proposition \ref{carl} and \ref{HLS}, taking into account the validity of the basic assumption $m\ge m_c$, 
is that the energy functional $\mathcal E_m$ defined in \eqref{E_m} is finite for every $u\in D(\mathcal F_m)$.

\begin{remark}\label{remarcone}
From the proof of  \cite[Lemma 3.2]{BCL} we deduce that, for $d\ge 3$, the optimal constant $C_H$ in \eqref{HLS2} can be estimated 
as $C_H\le \frac{\Gamma(d/2)}{2(d-2)\pi^{d/2}}\, C_{HLS}$, where $C_{HLS}$ is the best constant of the classical Hardy-Littlewood-Sobolev inequality
\[
\iint_{\mathbb R^d\times \mathbb R^d}|x-y|^{2-d}u(x)u(y)\,\d x\,\d y\le C_{HLS} \|u\|^2_{L^\frac{2d}{d+2}(\mathbb R^d)},
\]
whose explicit value is $C_{HLS}=\frac{\pi^{-1+d/2}}{\Gamma(1+d/2)}\left(\frac{\Gamma(d/2)}{\Gamma(d)}\right)^{-2/d},$ see \cite{L}. 
Therefore, we have $$C_H\le \frac{1}{\pi d(d-2)}\left(\frac{\Gamma(d/2)}{\Gamma(d)}\right)^{-2/d}.$$ The latter estimate, along with  \eqref{defMsc}, shows that in dimension $d\ge 3$ there holds 
$$
M_c\ge \left(\frac{2\pi d^2}{\chi}\right)^{d/2}\frac{\Gamma(d/2)}{\Gamma(d)}.
$$
On the other hand we notice that by \eqref{defMsc},
 \eqref{talenti} and \eqref{sq} we get
\begin{equation*}
M_{sc}=\left(\frac4{\chi m_c C_{2,d}}\right)^{d/2}=\left(\frac{2\pi d^2}{\chi}\right)^{d/2}\left(\frac{d-2}{d-1}\right)^{d/2}\frac{\Gamma(d/2)}{\Gamma(d)}, 
\end{equation*}
showing that indeed
$M_{sc}<M_c$  also in dimension $d\ge 3$.
\end{remark}

\EEE

Some useful properties of Bessel potentials are gathered in the next proposition.
%
%
%
%
%
%
%

\begin{proposition}\label{loggrowth}
Let $d\ge 2$, $\alpha\ge 0$. 
\begin{itemize}
\item[{\bf(i)}] If $p\in(d,+\infty]$, then there exist constants $c_1$, $c_2$ such that
\[
\|\nabla B_\alpha\ast u\|_{L^\infty(\mathbb R^d)}\le c_1\|u\|_{L^p(\mathbb R^d)}+c_2\|u\|_{L^1(\mathbb R^d)}, \qquad \forall u\in L^1\cap L^p(\R^d).
\] 
\item[{\bf(ii)}]
If  $p\in(d/2,+\infty]$ and either $d\geq 3$ or $\alpha>0$, then
there exist constants $c_3$, $c_4$ such that
\[
\| B_\alpha\ast u\|_{L^\infty(\mathbb R^d)}\le c_3\|u\|_{L^p(\mathbb R^d)}+c_4\|u\|_{L^1(\mathbb R^d)}, \qquad \forall u\in L^1\cap L^p(\R^d).
\] 
\item[{\bf(iii)}]
If $p\in(1,+\infty]$, $d=2$ and $\alpha=0$,  then  there exists a constant $c_5$ such that 
\[\begin{aligned}
|(B_0\ast u)(x)|\le c_5\|u\|_{L^p(\mathbb R^2)}+\frac1{2\pi}\mathrm m_{log}(|u|)&+\frac{1}{2\pi}\|u\|_{L^1(\mathbb R^2)}\,\log(1+|x|^2),\\&\qquad \quad \forall u\in L^1\cap L^p(\R^d), \,\;\;\;\forall x\in \R^2.
\end{aligned}\]
\item[{\bf (iv)}] If $p=+\infty$, then there exists a constant $c_d>0$ depending only on $d$ such that,
the following log-Lipschitz estimate  holds:
\begin{equation*}\label{logmeno}
\begin{aligned}
	|\nabla (B_\alpha\ast u)(x)-\nabla (B_\alpha\ast u)(y)|\le & c_d(\|u\|_{L^1(\R^d)}+\|u\|_{L^\infty(\R^d)})|x-y|(1+\log^-|x-y|), \\
	 &\forall u\in L^1\cap L^\infty(\R^d), \quad \forall\,x,y\in \R^d,
\end{aligned}
\end{equation*}
where $\log^-\ge 0$ is the negative part of the natural logarithm.
\end{itemize}
\end{proposition}
\begin{proof} {\bf(i)}. Notice that  $|\nabla B_\alpha|$ is radially decreasing and vanishing at infinity.
Let $p\in(d,+\infty]$ and $p'$ its conjugate.
It is simple to check that $|\nabla B_\alpha|^{p'}$  is integrable in $B_1(0)$. Using H\"older's inequality  and identifying $|\nabla B_\alpha|$ with its radial profile, we obtain 
\begin{equation}\begin{aligned}\label{vanishingatinfty}
	&|(\nabla B_\alpha\ast u)(x)|\le \int_{\mathbb R^d}|\nabla B_\alpha(x-y)||u(y)|\,\d y\\&\qquad=\int_{\{y:|y-x|\le R\}}|\nabla B_\alpha(x-y)||u(y)|\,\d y+\int_{\{y:|y-x|>R\}}|\nabla B_\alpha(x-y)||u(y)|\,\d y\\&\qquad
	\le \|\nabla B_\alpha\|_{L^{p'}(B_R(0))}\|u\|_{L^p(B_R(x))}+|\nabla B_\alpha(R)|\|u\|_{L^1(\mathbb R^d)}
\end{aligned}\end{equation}
for every $R>0$, and choosing for instance $R=1$ proves {\bf{(i)}}.
 
{\bf (ii)}.
Let $p\in(d/2,+\infty]$ and either $d\geq 3$ or $\alpha>0$.
Since $B_\alpha$ belongs to $L^{p'}(B_1(0))$ and it is vanishing at infinity, using H\"older's inequality, we have
\[\begin{aligned}
|(B_\alpha\ast u)(x)|&\le \int_{\mathbb R^d}B_\alpha(x-y)|u(y)|\,\d y\\&=\int_{\{y:|y-x|\le 1\}}B_\alpha(x-y)|u(y)|\,\d y+\int_{\{y:|y-x|>1\}}B_\alpha(x-y)|u(y)|\,\d y\\&
\le  \|B_\alpha\|_{L^{p'}(B_1(0))} \|u\|_{L^p(\mathbb R^d)}+B_\alpha(1)\|u\|_{L^1(\mathbb R^d)}.
\end{aligned}\]
\EEE

{\bf (iii)}.
Let $p\in(1,+\infty]$, $d=2$ and $\alpha=0$. Using  H\"older inequality and the elementary inequality $|x-y|\le(1+|x|^2)(1+|y|^2)$, we obtain
\[\begin{aligned}
|(B_0\ast u)(x)|&\le\frac1{2\pi}\int_{\{y:|y-x|\le 1\}}-\log|x-y|\,|u(y)|\,\d y+\frac{1}{2\pi}\int_{\{y:|y-x|>1\}}\log|x-y|\,|u(y)|\,\d y\\&
\le \|B_0\|_{L^{p'}(B_1(0))} \|u\|_{L^p(\mathbb R^2)}+\frac1{2\pi}\int_{\mathbb R^2}(\log(1+|x|^2)+\log(1+|y|^2))\,|u(y)|\,\d y.
\end{aligned}\]

{\bf(iv)}. The proof of the case $\alpha=0$ can be found in \cite[Lemma 4.1]{SV} (see also \cite{BLL}, \cite[Chapter 8]{MB} and 
 \cite{Y}).\\
Let $\alpha>0$. It is immediate to prove, using the definition of $B_\alpha$ and Fubini theorem, that $\|B_\alpha\|_{L^1(\R^d)}=\alpha^{-1}$.
Observing that $B_\alpha\ast u=B_0 \ast(u-\alpha B_\alpha\ast u)$ and consequently $\nabla B_\alpha\ast u=\nabla B_0 \ast(u-\alpha B_\alpha\ast u)$, 
the thesis follows applying the result for $\alpha=0$, taking into account that
 $ \|B_\alpha\ast u\|_{L^\infty(\R^d)}\le \|B_\alpha\|_{L^1(\R^d)}\|u\|_{L^\infty(\R^d)}=\alpha^{-1}\|u\|_{L^\infty(\R^d)}$ 
 and $\|B_\alpha\ast u\|_{L^1(\R^d)}\le \|B_\alpha\|_{L^1(\R^d)}\|u\|_{L^1(\R^d)}=\alpha^{-1}\|u\|_{L^1(\R^d)}$.
\end{proof}

\subsection{Lower bounds for the energy}

As a direct consequence of the Hardy-Littlewood-Sobolev inequalities of Propositions \ref{carl} and \ref{HLS}
we obtain the following basic estimates of the energy functional $\EE_m$ that yield useful equi-integrability properties for their sublevels. 
The constants appearing in the next result are explicit and given through the proof.
\begin{proposition}\label{F1boundprop} 
There exist constants $C_1, C_2,C_3 >0$ depending only on $M$, $\chi$, $m$ and $d$,  $C_2$ depending also on $\alpha$ in \eqref{F1bound}  \eqref{Fmbound12}, \EEE such that, for each $u\in D(\EE_m)$,
\begin{equation}\label{F1bound}
\EE_1(u)\ge C_1\int_{\mathbb{R}^2}u|\log u|\,\d x -C_2 -C_3\int_{\mathbb{R}^2}\log(1+|x|^2)u(x)\,\d x, \qquad \text{for }d=2,\, m=1,
\end{equation}
\begin{equation}\label{Fmbound12}
\EE_m(u)\ge C_1\int_{\mathbb{R}^2}u^m\,\d x -C_2 -C_3\int_{\mathbb{R}^2}\log(1+|x|^2) u(x)\,\d x, \qquad \text{for }d=2,\, m>1,
\end{equation}
\begin{equation}\label{Fmbound}
\EE_{m_c}(u) \ge C_1\int_{\mathbb{R}^d}u^{m_c}\,\d x, \qquad \text{for }d\geq 3,\, m=m_c,
\end{equation}
\begin{equation}\label{Fmbound2}
\EE_m(u)\ge C_1\int_{\mathbb{R}^d}u^m\,\d x -C_2. \qquad  \text{for }d\geq3,\, m>m_c.
\end{equation}

%
\end{proposition}

\begin{proof}
Let $d=2$ and $m=1$.
Using \eqref{logbessel} we obtain
\begin{equation*}\label{F1bound0}
\begin{split}
\mathcal{E}_1(u) &= 
\frac{8\pi-\chi M}{8\pi}\int_{\mathbb{R}^2}u\log u\,\d x + \frac{\chi M}{8\pi}\int_{\mathbb{R}^2}u\log u\,\d x
-\frac{\chi }{2} \iint_{\mathbb{R}^2\times\mathbb{R}^2}B_\alpha(x-y) u(x)u(y)\,\d x\d y \\
& \geq \frac{8\pi-\chi M}{8\pi}\int_{\mathbb{R}^2}u\log u\,\d x
 - \frac{\chi }{2} \left(\frac{ M}{2\pi} \int_{\mathbb{R}^2}u(x)\log(1+|x|^2)\,\d x +\frac{MC_{M,\alpha}}{4\pi}    \right). 
 \end{split}
\end{equation*}
Thus using \eqref{Carleman} we obtain \eqref{F1bound} with
\begin{equation*}\label{C123}
C_1=\frac{8\pi-\chi M}{8\pi}, \quad C_2=\frac{\chi M  C_{M,\alpha}\RRR}{8\pi}+\frac{8\pi-\chi M}{8\pi}\Big(2M\ln\pi+\frac{2}{e}\Big), \quad C_3=\frac{16\pi-\chi M}{4\pi}.
\end{equation*}
We observe that $C_1>0$ is equivalent to $M<M_c$, where $M_c$ is defined in  \eqref{defMc}.

Let $d=2$ and $m>1$. 
Taking into account that for $K\geq 1$ we have  $$\displaystyle\int_{\{u<K\}}u\log^+u\,\d x\le M\log K,$$
 by \eqref{logbessel} it holds
\[\begin{aligned} 
\EE_m(u) &\geq  \frac1{m-1}\int_{\mathbb{R}^2} u^m\,\d x-\frac{M\chi}{8\pi}\left(\int_{\mathbb{R}^2}u\log u\,\d x + 2\int_{\R^2}\log(1+|x|^2)u(x)\,\d x +  C_{M,\alpha}\right)\\
&= \frac{1}{2(m-1)}\int_{\mathbb{R}^2} u^m\,\d x
+\frac1{2(m-1)}\int_{\{u\ge K\}}u^{m}\,\d x -\frac{M\chi}{8\pi} \int_{\{u\ge K\}}u\log u\,\d x\\
&\qquad+\frac1{2(m-1)}\int_{\{u< K\}}u^{m}\,\d x -\frac{M\chi}{8\pi} \int_{\{u< K\}}u\log u\,\d x\\
&\qquad-\frac{M\chi}{4\pi}\int_{\R^2}\log(1+|x|^2)u(x)\,\d x  -\frac{M\chi C_{M,\alpha}}{8\pi}\\
&\geq \frac{1}{2(m-1)}\int_{\mathbb{R}^2} u^m\,\d x
+\int_{\{u\ge K\}} \left(\frac{u^{m}}{2(m-1)} -\frac{M\chi u\log u}{8\pi} \right)\,\d x\\
& \qquad-\frac{M^2\chi \log K}{8\pi} 
-\frac{M\chi}{4\pi}\int_{\R^2}\log(1+|x|^2)u(x)\,\d x  -\frac{M\chi C_{M,\alpha}}{8\pi}.
\end{aligned}
\]
Fixing $K\geq e^{1/(m-1)}$ such that $\frac{K^{m-1}}{\log K} \geq \frac{\chi M(m-1)}{4\pi}$ we have
$$\int_{\{u\ge K\}} \left(\frac{u^{m}}{2(m-1)} -\frac{M\chi u\log u}{8\pi} \right)\,\d x\geq 0$$
and \eqref{Fmbound12} holds with $C_1=\frac1{2(m-1)}$, $C_2=\frac{M^2\chi \log K+M\chi C_{M,\alpha}}{8\pi}$, $C_3=\frac{M\chi}{4\pi}.$

Let $d\geq3$ and $m=m_c$. In this case \eqref{Fmbound} holds with  $C_1=\frac{1}{m_c-1}-\frac\chi2 C_H M^{2/d}$ and follows by the definition of $\EE_{m_c}$ and \eqref{HLS2}.
We observe that $C_1>0$ is equivalent to $M<M_c$, where $M_c$ is defined in  \eqref{defMc}.

Let $d\geq3$ and $m>m_c$. For $K>0$ it holds  
\begin{equation*}\label{equisimple}
\begin{aligned}
	\int_{\{u\le K\}}u^{m_c}\,\d x \le K^{m_c-1}\int_{\mathbb{R}^d} u\,\d x = K^{m_c-1}M, 
	\qquad \int_{\{u\ge K\}}u^{m}\,\d x \ge K^{m-m_c}\int_{\{u\ge K\}}u^{m_c}\,\d x, 
\end{aligned}
\end{equation*}
 which together with \eqref{HLS2}  entails
\[\begin{aligned}
\mathcal{E}_m(u)\ge &\frac1{2(m-1)}\int_{\mathbb{R}^d} u^m\,\d x
+\frac1{2(m-1)}\int_{\{u\ge K\}}u^{m}\,\d x -\frac\chi2 C_H M^{2/d} \int_{\{u\ge K\}}u^{m_c}\,\d x\\
&+\frac1{2(m-1)}\int_{\{u< K\}}u^{m}\,\d x -\frac\chi2 C_H M^{2/d} \int_{\{u< K\}}u^{m_c}\,\d x\\
&\geq\frac1{2(m-1)}\int_{\mathbb{R}^d} u^m\,\d x
+\left(\frac{K^{m-m_c}}{2(m-1)}-\frac\chi2 C_H M^{2/d}\right) \int_{\{u\ge K\}} {u^{m_c}}\,\d x\\
&- \frac\chi2 C_H M^{2/d}K^{m_c-1}M.
\end{aligned}
\]
For $K= (\chi(m-1) C_H M^{2/d})^{1/(m-m_c)}$, we get \eqref{Fmbound2} with constants $C_1=\frac1{2(m-1)}$, $C_2=\frac\chi2 C_H M^{2/d}K^{m_c-1}M$.
\end{proof}


\subsection{Continuity properties of the energy}

We state basic properties of lower semi continuity for the functional $\FF_m$ and of continuity for the  functional $\BB_\alpha$ with respect to the narrow convergence.

\begin{proposition}\label{prop:FP}
Let $u_n,u\in\Space$ such that $u_n\to u$ narrowly. 
If
\begin{equation}\label{energymomentum}
\sup_{n\in\N}\FF_m(u_n)<+\infty\quad\mbox{and }\quad\sup_{n\in\N}\int_{\mathbb R^d}|x|^2u_n(x)\,\d x<+\infty,
\end{equation}
then
\begin{equation}\label{lscmoment}
	\liminf_{n\to+\infty}\int_{\mathbb R^d}|x|^2u_n(x)\,\d x\geq\int_{\mathbb R^d}|x|^2u(x)\,\d x,
\end{equation}
\begin{equation}\label{lscFm}
	\liminf_{n\to+\infty}\FF_m(u_n)\geq\FF_m(u),
\end{equation}
\begin{equation}\label{cBalpha}
	\lim_{n\to+\infty}\BB_\alpha(u_n)=\BB_\alpha(u),
\end{equation}
\begin{equation}\label{lscEm}
	\liminf_{n\to+\infty}\EE_m(u_n)\geq\EE_m(u)
\end{equation}
and $u\in D(\EE_m)$.
\end{proposition}

\begin{proof}
 \eqref{lscmoment} follows from \eqref{lscWass} choosing $v_n=\delta_0$. \EEE 
\eqref{lscFm} follows from the convexity and superlinearity of the integrand  and standard lower semi continuity of integral functionals (see for instance \cite{AFP}), \EEE taking into account Proposition \ref{carl} for bounding the negative part of the entropy in the case of $\mathcal F_1$.
We also notice that  \eqref{lscEm} is a direct consequence of \eqref{lscFm} and \eqref{cBalpha}. $u\in D(\EE_m)$ as a consequence of \eqref{lscEm} and \eqref{energymomentum}.

 The property \eqref{cBalpha} in the case $d=2$ and $\alpha=0$, is proved in \cite[Lemma 3.1]{BCC}.

We fix  $d=2$ and $\alpha>0$ and 
we prove \eqref{cBalpha}. \EEE
In view of \eqref{alpha1}, we may assume that $\alpha=1$. 
We have $B_1(x)=\frac1{2\pi}\,K_0(|x|)$, where $K_0$ is the modified Bessel function of second kind. By taking advantage of the series representations of $K_0$, see for instance \cite[pp 919]{GR}, there exists $\delta_0\in (0,1)$ such that for every $r\in(0,\delta_0)$ there holds
\[
K_0(r)\le\log\frac{2}{r}.
\] 
Thanks to the latter property and to the elementary inequality $ab\le e^a+b\log b$, $\forall \, a,b\in (0,+\infty)$, if $0<\eps<\delta_0$ we find
\[
\begin{aligned}
&2\pi\iint_{\{|x-y|<\eps\}}u_n(x)u_n(y)\,B_1(x-y)\,\d x\,\d y=\iint_{\{|x-y|<\eps\}}u_n(x)u_n(y)\,K_0(x-y)\,\d x\,\d y\\
&\quad\le\iint_{\{|x-y|<\eps\}}u_n(x)\left(u_n(y)\log u_n(y)+\frac2{|x-y|}\right)\,\d x\,\d y\\&\quad
\le \int_{\mathbb R^2}u_n(y)\log^+ u_n(y)\left(\int_{B_\eps(y)}u_n(x)\,\d x\right)\d y+ \int_{\mathbb R^2}u_n(x)\,\d x\int_{B_\eps(x)}\frac2{|x-y|}\,\d y,
\end{aligned}
\]
But letting $V_{\eps,n}:=\{x\in\mathbb R^2:u_n(x)\le 1/\eps\}$ we see that for every $y\in\mathbb R^d$
\[\begin{aligned}
\int_{B_\eps(y)}u_n(x)\,\d x&=\int_{B_\eps(y)\cap V_{\eps,n}}u_n(x)\,\d x+\int_{B_\eps(y)\cap{V}_{\eps,n}^C}u_n(x)\,\d x
\\&\le\pi\eps+\int_{B_\eps(y)\cap V_{\eps,n}^C}\frac{u_n\log^+ u_n}{\log u_n}\,\d x\le \pi\eps+\frac1{\log^+(1/\eps)}\int_{\mathbb R^2}u_n\log^+ u_n\,\d x,
\end{aligned}\]
and we notice that $Q:=\sup_{n\in\N}\int_{\mathbb R^2}u_n\log^+ u_n\,\d x<+\infty$
thanks to \eqref{energymomentum} (in case $m=1$, this is due to the Carleman estimate \eqref{Carleman}). Inserting the latter estimate in the previous one we get
\begin{equation}\label{Qn}
2\pi\iint_{\{|x-y|<\eps\}}u_n(x)u_n(y)\,B_1(x-y)\,\d x\,\d y\le \pi Q\eps+\frac{Q^2}{\log(1/\eps)}+M\int_{B_\eps(0)}\frac{2}{|y|}\,\d y.
\end{equation}
Since \eqref{lscFm} holds and since $u$ has finite second moment (by the lower semicontinuity of the second moment) the same argument can be applied with $u$ in place of $u_n$ to get
\begin{equation}\label{Q0}
2\pi\iint_{\{|x-y|<\eps\}}u(x)u(y)\,B_1(x-y)\,\d x\,\d y\le \pi Q_0\eps+\frac{Q_0^2}{\log(1/\eps)}+M\int_{B_\eps(0)}\frac{2}{|y|}\,\d y,
\end{equation}
where $Q_0:=\int_{\mathbb R^2}u\log^+u\,\d x$.
We notice that 
\begin{equation}\label{Q3}\lim_{n\to+\infty}\iint_{\{|x-y|\ge\eps\}}u_n(x)u_n(y)\,B_1(x-y)\,\d x\,\d y=\iint_{\{|x-y|\ge \eps\}}u(x)u(y)B_1(x-y)\,\d x\,\d y,\end{equation}
since $\bC_ {\{|x-y|\ge \eps\}}(x,y)\,B_\alpha(x-y)$ belongs to $L^\infty(\mathbb R^2\times\mathbb R^2)$ and since $u_n(x)u_n(y)$ converge to $u(x)u(y)$ weakly in $L^1(\mathbb R^d\times\mathbb R^d)$. For the latter convergence, see \cite[Lemma 2.3]{BCC}: it follows from the weak $L^1(\mathbb R^d)$ convergence of $u_n$ to $u$ and to the equiintegrability property which comes from \eqref{energymomentum}. 
By combining  \eqref{Qn},  \eqref{Q0} and \eqref{Q3} we deduce that
\[\begin{aligned}&
\limsup_{n\to\infty}\left|\iint_{\mathbb R^2\times\mathbb R^2}(u_n(x)u_n(y)-u(x)u(y))B_1(x-y)\,\d x\,\d y\right|\\&\qquad\qquad\le \frac{(Q+Q_0)\eps}{2}+\frac{Q_0^2+Q^2}{2\pi\log(1/\eps)}+\frac{2M}{\pi}\int_{B_\eps(0)}\frac{1 }{|y|}\,\d y.
\end{aligned}\]
Passing to the limit for $\eps \to 0$ we obtain \eqref{cBalpha}.

The proof of  \eqref{cBalpha} in the case $d\geq 3$ and $\alpha\ge0$ is analogous, taking into account \eqref{alpha2} (see also \cite{BL}).
\EEE
\end{proof}

We also  show the continuity of the energy functional $\EE_m$ with respect to the 
regularization of the initial datum given by \eqref{utau0}, i.e., the continuity of $\mathcal E_m$ along the heat flow. This property is a particular case of the analysis of Section \ref{interchange}.
\begin{proposition}\label{initialdatum}
Let
$u_0\in \mathscr P_2^M(\R^d)$ and $u_\tau^0$ given by \eqref{utau0}. 
Then
\begin{equation}\label{momevolve}
	 \mom(u_\tau^0)=\mom(u_0)+2dM\varpi(\tau)\le \mom(u_0)+3dM, \qquad \forall\, \tau\ge 0
\end{equation}
and
\begin{equation}\label{energytau}
\lim_{\tau\to 0}\EE_m(u_\tau^0)=\EE_m(u_0).
\end{equation}
 
If $u_0\in D(\mathcal E_m),$ then the map $\tau\mapsto \mathcal E_m(u_\tau^0)$ is continuous on $[0,+\infty)$,
\begin{equation}\label{uniformwithtau}
\mathcal E_m(u_\tau^0)\le \mathcal F_m(u_0)+\frac{\chi M}{\pi}(\mom(u_0)+6M), \qquad \forall\, \tau\ge 0,
\end{equation}
and in the case $\alpha>0$ or $d\ge 3$, it holds
$$\mathcal E_m(u_\tau^0)\le \mathcal F_m(u_0), \qquad \forall\, \tau\ge 0. 
$$\end{proposition}
\begin{proof} 
A direct computation on the Gaussian kernel $\Gamma_t$ (defined in \eqref{gaussian}) shows that $\mom(\Gamma_t\ast u_0)=\mom(u_0)+2dM t$, for every $t\ge 0$, so that \eqref{momevolve} follows from \eqref{utau0}.

We prove \eqref{energytau}.
Let $u_0\in D(\mathcal E_m).$
We observe that for a convex function $F:[0,+\infty)\to \R$, 
taking into account that $\Gamma_t$ is a density of probability measure,
by Jensen's inequality  we have 
$F(u_0*\Gamma_t)\leq F(u_0)*\Gamma_t$. 
Integrating this inequality and using Fubini's theorem
we obtain 
\begin{equation}\label{jens}
	\int_{\R^d}F(u_0*\Gamma_t(x))\,\d x \leq \int_{\R^d} F(u_0(x))\,\d x, \qquad \forall\, t\in (0,+\infty).
\end{equation} 
Using \eqref{jens} it follows that
\begin{equation*}\label{uscFm}
	\limsup_{t\to 0}\FF_m(\Gamma_t\ast u_0)\leq\FF_m(u_0).
\end{equation*}
Thus \eqref{energytau} follows using Proposition \ref{prop:FP}, \eqref{lscEm}. 

 If $u_0\notin D(\mathcal E_m),$ by \eqref{DE_m} $u_0\notin D(\mathcal F_m)$. By Fatou's Lemma we obtain that 
 $$\liminf_{t\to 0}\FF_m(\Gamma_t\ast u_0)\geq\FF_m(u_0)=+\infty.$$ 
 Taking into account assumption \eqref{assumption},   then  \eqref{energytau} follows by the previous limit inequality and Proposition \ref{F1boundprop}, 
 since $\Gamma_t\ast u_0 \in D(\EE_m)$ for any $t\in (0,+\infty)$ and \eqref{momevolve} holds (necessary for the case $d=2$).
 \RRR
 

Let $u_0 \in D(\mathcal E_m)$, by the semigroup property of the heat flow it follows that $t\mapsto \mathcal F_m(\Gamma_t\ast u_0)$ is continuous at every $t>0$, and then the continuity of $t\mapsto\mathcal E_m(\Gamma_t\ast u_0)$ follows from Proposition \ref{prop:FP} and \eqref{energytau}.

We notice that
$\mathcal B_\alpha\ge 0$ if either $\alpha>0$ or $d\ge 3$, else if $d=2$, $\alpha=0$ we have the straightforward estimate
\begin{equation}\label{-mom}
-\chi\mathcal B_0(u)\le \frac{\chi M}\pi\mom(u)\qquad \mbox{for every $u\in D(\FF_1),$}
\end{equation}
following from $\log^+|x-y|< |x-y|^2\le  2|x|^2+2|y|^2$. By combining the latter fact with \eqref{momevolve} and \eqref{jens}, we get \eqref{uniformwithtau}.
\end{proof}

\subsection{Basic estimates of the JKO scheme}
Existence of a solution for  the minimization problem \eqref{minmov1} is standard and it is stated in the following Proposition.
\begin{proposition}\label{prop:existenceMM}
Let 
$u_0\in \Space$.
For any $\tau>0$ there exists a sequence
$\{u_\tau^k:k=0,1,2,\ldots\}$
satisfying \eqref{utau0} and \eqref{minmov1}.
\end{proposition}
\begin{proof}
For $k=0$, $u_\tau^0$ is given by definition in \eqref{utau0} and we observe that $u_\tau^0\in D(\EE_m)$.
For $k\geq 1$ we consider the functional $\Phi_\tau^k:\Space\to(-\infty,+\infty]$ defined by
$$\Phi_\tau^k(u):=\EE_m(u)+\frac1{2\tau}\,W^2(u,u^{k-1}_\tau).$$
This functional is bounded from below over $\mathscr P_2^M(\R^d)$ by Proposition \ref{F1boundprop},
and the sub-levels of $\Phi_\tau^k$ are relatively weakly compact in $L^1$ by Proposition \ref{F1boundprop} and Dunford-Pettis theorem. 
In particular, a sequence on a sub-level of $\Phi_\tau^k$ has a  narrowly converging subsequence and uniformly bounded second moments.
By Proposition \ref{prop:FP} and the property \eqref{lscWass}, $\Phi_\tau^k$ is lower semicontinuous with respect to the narrow convergence.
The existence of a minimum of $\Phi_\tau^k$ follows by direct method in calculus of variations.
\end{proof}
Given a sequence 
 of discrete minimizers of \eqref{minmov1},
we state some uniform equi-integrability estimates in the next two propositions. 
We have the following
\begin{proposition}\label{logequi} 
Let $d=2$.
Let $\tau>0$ and $\{u_\tau^k:k=0,1,2,\ldots\}$ be a sequence from {\rm Proposition \ref{prop:existenceMM}}.
Let $C_1,C_2,C_3$  the constants of {\rm Proposition \ref{F1boundprop}}, and let 
\begin{equation}\label{forsecostante}
C_\tau^0:=\frac{(8C_1)\vee 1}{C_1}\left(\frac12 \mom(u_\tau^0)+\mathcal E_1(u_\tau^0)+C_2+C_3M\right).
\end{equation}
Then,
\begin{equation}\label{FondEst2}
\int_{{\R}^2}u_\tau^k|\log u_\tau^k|\,\d x \leq  C_\tau^0(1+\log^+(8C_3T)), \;\; \mom(u_\tau^k)  \leq  C_\tau^0 T(1+\log^+(8C_3T)),\;\; m=1,
\end{equation}
\begin{equation}\label{FondEst2m}
\int_{{\R}^2}(u_\tau^k)^m\,\d x \leq  C_\tau^0(1+\log^+(8C_3T)), \quad \mom(u_\tau^k)  \leq C_\tau^0 T(1+\log^+(8C_3T)),\quad m>1,
\end{equation}
for every $T\ge1$, $\tau>0$ and $k\ge 1$ such that $k\tau\le T$. \RRR 
Under the additional assumption $u_0\in D(\mathcal E_m)$, we can bound $C_\tau^0$ uniformly w.r.t. $\tau$, that is, for every $\tau\ge 0$ there holds
\begin{equation}\label{notau}
C^0_\tau\!\le C_0:=\frac{(8C_1)\vee 1}{C_1}\left(\!\mathcal F_m(u_0)+\frac{\pi+2\chi M}{2\pi}\mom(u_0)+C_2+C_3M+\!\frac32M+\frac{6\chi M^2}{\pi}\right)\!.\end{equation}
\RRR
\end{proposition}
\begin{proof}
The proof is adapted from \cite[Lemma 4.1]{BCKKLL}. We consider the case $m=1$.
Let $T\ge1$, $\tau>0$ and $k\in\mathbb{N}$ such that $k\tau\le T$.
By the triangle inequality of the Wasserstein distance and the  Cauchy-Schwarz inequality we get the basic moment estimate
\begin{equation}\label{A1}
\begin{aligned}
&\mom(u_\tau^k)=W^2(u_\tau^k,M\delta_0)\leq \left(\sum_{n=1}^kW(u_\tau^n,u_\tau^{n-1})+W(u_\tau^0,M\delta_0)\right)^2\\
&\quad\le 2k\sum_{n=1}^k W^2(u_\tau^n,u_\tau^{n-1})+2W^2(u_\tau^0,M\delta_0)\le \frac{2T}{\tau}\sum_{n=1}^k W^2(u_\tau^n,u_\tau^{n-1})
+2\mom(u_\tau^0).
\end{aligned}
\end{equation}
On the other hand, the basic estimate of the minimizing movements scheme is 
\begin{equation}\label{basicMM}
  \EE_1(u_\tau^k)+\frac{1}{2\tau}\sum_{n=1}^k W^2(u_\tau^n,u_\tau^{n-1}) \leq  \EE_1(u_\tau^0).
\end{equation}
Rewriting \eqref{A1} as
\begin{equation*}
  \frac{1}{4T}\mom(u_\tau^k) \leq \frac{1}{2\tau}\sum_{n=1}^k W^2(u_\tau^n,u_\tau^{n-1}) + \frac{1}{2T}\mom(u_\tau^0)
\end{equation*}
and inserting in the basic estimate \eqref{basicMM} we obtain
\begin{equation*}
 	\EE_1(u_\tau^k)+ \frac{1}{4T}\mom(u_\tau^k) \leq  \EE_1(u_\tau^0) + \frac{1}{2T}\mom(u_\tau^0).
\end{equation*}
Combining \eqref{F1bound} with the last inequality we have
\begin{equation*}
	C_1\int_{{\R}^2}u_\tau^k|\log u_\tau^k|\,\d x -C_2 -C_3\int_{{\R}^2}u_\tau^k(x)\log(1+|x|^2)\,\d x + \frac{1}{4T}\mom(u_\tau^k) \leq  \EE_1(u_\tau^0) + \frac{1}{2T}\mom(u_\tau^0).
\end{equation*}
Using the elementary inequality $\log(1+|x|^2)\leq \eps|x|^2+\log^-\eps$, holding for any $\eps>0$, we have
\begin{equation*}
	\frac{1}{4T}\mom(u_\tau^k) - C_3\int_{{\R}^2}u_\tau^k(x)\log(1+|x|^2)\,\d x  \geq  \left(\frac{1}{4T}-\eps C_3\right)\mom(u_\tau^k) - C_3M\log^-\eps
\end{equation*}
and 
\begin{equation*}
	C_1\int_{{\R}^2}u_\tau^k|\log u_\tau^k|\,\d x -C_2 + \left(\frac{1}{4T}-\eps C_3\right)\mom(u_\tau^k)  \leq  \EE_1(u_\tau^0) + \frac{1}{2T}\mom(u_\tau^0) + C_3M\log^-\eps
\end{equation*}
The choice of $\eps= 1/(8C_3 T)$ in the last inequality yields, since $T\ge 1$,
\begin{equation}\label{FondEst}
 \frac{1}{8T}\mom(u_\tau^k) + C_1\int_{{\R}^2}u_\tau^k|\log u_\tau^k|\,\d x \leq \frac{1}{2}\mom(u_\tau^0) + \EE_1(u_\tau^0) + C_2 + C_3M\log^+(8C_3T).
\end{equation}
 Then, \eqref{FondEst2} directly follows from  \eqref{FondEst}, for $C_\tau^0$ defined by \eqref{forsecostante}.  If $u_0\in D(\mathcal E_1)$, from \eqref{forsecostante} and Proposition \ref{initialdatum}, 
 the last statement follows.

The proof in the case $m>1$ is the same and uses \eqref{Fmbound12} instead of  \eqref{F1bound}.
\end{proof}

\begin{proposition}\label{equi} Let $d\ge 3$. 
Let $\tau>0$ and $\{u_\tau^k:k=0,1,2,\ldots\}$ be a sequence from {\rm Proposition \ref{prop:existenceMM}}.
Then, if  $C_1,C_2$ are the constants from \eqref{Fmbound}-\eqref{Fmbound2},  
for any $k\in\mathbb{N}$ there holds 
\[
\int_{\mathbb{R}^d}(u_\tau^k)^m\,\d x\le C_\tau^0:=\frac{\mathcal E_m{(u_\tau^0)}+C_2}{C_1}.
\]
Moreover, if $T\ge1$ and $k\ge 1$ is an integer such that $k\tau\le T$, we have 
\[\mom(u_\tau^k)\le 4TC_1C_\tau^0+2\mom(u_\tau^0).
\]
If  $u_0\in D(\mathcal E_m)$, we also have for every $\tau\ge 0$
\[
C_\tau^0\le C_0:=\frac1{C_1}\left( \mathcal F_m(u_0)+\frac{\chi M}{\pi}(\mom(u_0)+6M)+C_2\right), 
\]
and 
$\;
4TC_1C_\tau^0+2\mom(u_\tau^0)\le 4TC_1C_0+2\mom(u_0)+6dM.
$
\end{proposition}
\begin{proof} 
The proof of the first estimate follows from the basic estimate of the minimizing movements scheme 
\begin{equation}\label{basicjko}
  \EE_m(u_\tau^k)+\frac{1}{2\tau}\sum_{n=1}^k W^2(u_\tau^n,u_\tau^{n-1}) \leq  \EE_m(u_\tau^0),
\end{equation}
combined with the estimates \eqref{Fmbound} and \eqref{Fmbound2}. 
Moreover,  \eqref{A1} and \eqref{basicjko} yield
\[
\frac1{4T}\mom(u_\tau^k)\le \mathcal E_m(u_\tau^0)-\mathcal E_m(u_\tau^k)+\frac1{2T}\mom(u_\tau^0),
\]
so that the second estimate follows again from \eqref{Fmbound} and \eqref{Fmbound2}.
  Proposition \ref{initialdatum} yields the last statement.
\end{proof}

We show how to extract a limit curve in $\Space$ from the minimizing movements scheme.
This is a very basic result that assumes $u_0\in D(\mathcal E_m)$, and shows existence of gradient flows of $\mathcal E_m$ starting from $u_0\in D(\mathcal E_m)$,   in the sense of Definition \ref{GFdefinition}. \RRR  
Later on, after having developed refined estimates on discrete solutions, we will show how to extend the result to general initial data in $\mathscr P_2^M(\R^d)$, at least for $m>m_c$, see Theorem \ref{esistenzagenerale}.

\begin{theorem}[{\bf Basic convergence result}]\label{th:convergence1}
Let 
$u_0\in D(\EE_m)$. 
For every $\tau>0$, let $\{u_\tau^k:k=0,1,2,\ldots\}$ be a sequence from {\rm Proposition \ref{prop:existenceMM}} and 
let $u_\tau$ be the corresponding piecewise constant curve defined in \eqref{floor}.
Then, for every vanishing sequence $(\tau_n)_n$, there exists a (not relabeled) subsequence and a curve
$u\in AC^2_{loc}([0,+\infty);\Space)$ such that \eqref{0conv} holds
\end{theorem}
\begin{proof} Let $d=2$. Let $\tau\in(0,1)$ and $T>0$. Let $k\ge1$ be an integer such that $k\tau\le T+1$.
We use Proposition \ref{initialdatum} 
and combine it with \eqref{F1bound}-\eqref{Fmbound12}, with the elementary inequality $\mathrm m_{log}(u)\le \mom(u)$,
and with \eqref{FondEst2}-\eqref{FondEst2m}-\eqref{notau}, to get
\[\begin{aligned}
&\mathcal E_m(u_\tau^0)-\mathcal E_m(u_\tau^k)\le \mathcal E_m(u_\tau^0)+C_2+C_3\mom(u_\tau^k)\\
&\quad\le \mathcal F_m(u_0)+\frac{\chi M}{\pi}(\mom(u_0)+6M)+C_2+C_3\,C_0\,(T+1)(1+\log^+(8C_3(T+1)))=:C(T). 
\end{aligned}\]
By the latter and by the basic estimate of minimizing movements \eqref{basicjko}
we obtain 
\begin{equation}\label{Cest}
 	\frac{1}{2\tau}\sum_{n=1}^k W^2(u_\tau^n,u_\tau^{n-1}) \leq  C(T).
\end{equation}
But  
\eqref{Cest} and Cauchy-Schwarz inequality imply
\[\begin{aligned}
W^2(u_\tau^{k_1}, u_\tau^{k_2})&\le\left(\sum_{n=1+k_1}^{k_2}W(u_\tau^n,u_\tau^{n-1})\right)^2\le (k_2-k_1)\sum_{n=1+k_1}^{k_2}W^{2}(u_\tau^n,u_\tau^{n-1})\le 2\tau(k_2-k_1)C(T)
\end{aligned}\]
and choosing $k_1=\lceil t_1/\tau\rceil$, $k_2=\lceil t_2/\tau\rceil$, with $0\le t_1\le t_2\le T$, shows that  the family
$\mathcal U:=\{u_\tau(t): \tau\in(0,1),\,t\in[0,T]\}$
 is equicontinuous up to an error that vanishes as $\tau\to0$, in the sense that
 \begin{equation}\label{equicontinuityofWasserstein}
 W^2(u_\tau(t_1),u_\tau(t_2))\le 2C(T)(t_2-t_1+\tau)\qquad\mbox{for every $0\le t_1\le t_2\le T$.}
 \end{equation}
%
The latter also shows that $\mathcal U$ is equibounded in $\mathscr P_2^M(\R^d)$, choosing $t_1=0$ and taking \eqref{momevolve} into account.
By the generalized Ascoli-Arzel\'a theorem \cite[Proposition 3.3.1]{AGS}  and a diagonal argument \RRR we conclude that there for every vanishing sequence $(\tau_n)_n$ there exists a non relabelled subsequence and there exists a narrowly continuous curve $[0,+\infty)\ni t\mapsto u(t)\in\mathscr P_2^M(\R^d)$ such that \eqref{0conv} holds. 

The case $d\geq 3$ is proven in the same way, taking advantage of Proposition \ref{equi}.
Another standard argument based on \eqref{Cest} shows that the limit curve is in $AC^2([0,T];\mathscr P_2^M(\R^d))$ for every $T>0$.  
We omit the details here and we will be back on this argument in a more general setting later in Theorem \ref{esistenzagenerale}.
\end{proof}

\section{Energy estimates along auxiliary flows}\label{interchange}


In order to obtain regularity properties of discrete minimizers we will take advantage of the flow interchange technique introduced by Matthes, McCann  and Savar\'e \cite{MMS}. For, we shall introduce suitable auxiliary functionals. These are typical choices in the theory of gradient flows with respect to the Wasserstein distance.
Then, in the rest of this section we shall compute the time derivatives of the energy $\mathcal E_m$ along auxiliary flows. 

\begin{definition}[\textbf{Potential energy}]\label{pot}
Let $J\in C^1(\R^d)$ be a convex function, bounded from below. 
The associated potential energy of $u\in\mathscr P_2^M(\R^d)$ is the functional  $\mathcal J:\PP^M_2(\R^d)\to (-\infty,+\infty]$ defined by
\[
\mathcal J(u)=\int_{\R^d}J(x)\,\d u(x).
\]
\end{definition}


\begin{definition}[\textbf{Displacement convex entropy}]\label{entr} 
Let $U:[0,+\infty)\to\R$ be a continuous convex function with super linear growth at infinity, such that $U(0)=0$,  $U\in C^2(0,+\infty)$, \RRR
$ \liminf_{x\downarrow 0}\frac{U(x)}{x^\alpha}>-\infty$ for some $\alpha>\frac{d}{d+2}$ and
the following assumption (introduced by Mc Cann in \cite{Mc}) holds:
\begin{equation*}\label{McCann}
	r \mapsto r^dU(r^{-d}) \qquad \text{is convex and nonincreasing in }(0,+\infty).
\end{equation*}
Let 
 $\mathcal U:\PP^M_2(\R^d)\to (-\infty,+\infty]$ be defined by
 $$\mathcal U(u)=\int_{\R^d} U(u(x))\,\d x$$ if $u$
 has a density (still denoted by $u$) w.r.t. the Lebesgue measure and $\mathcal U(u)=+\infty$ otherwise. Then, we say that $\mathcal U$ is a displacement convex entropy 
and we say that $U$ is the density function of $\mathcal U$.
\end{definition}

Taking into account that $U$ is not necessarily a positive function, the condition on the behavior of $U$ at $0$ in Definition \ref{entr} is needed as usual to have the integrability of the negative part of {\RRR $U\circ u$},
 as soon as $u$ is mass density with finite second moment. 
 Moreover, if $u_0\in\PP_2(\mathbb{R}^d)$ and $u_\tau^0$ is the regularization defined by \eqref{utau0}, then $u_\tau^0\in D(\mathcal U)$ for any displacement convex entropy $\mathcal U$, 
 since $u_\tau^0$ is bounded.   
 We also recall that displacement convex entropies are lower semicontinuous along narrowly converging sequences having uniformly bounded second moments, as a consequence of the convexity and superlinearity of $U$ along with the uniform bound of the negative part of $U\circ u_n$ provided by the bound on second moments. Also $\mathcal J$ is narrowly lower semicontinuous, since $J$ is continuous and bounded from below.
 Finally, the functionals in the above two definitions are convex along Wasserstein geodesics and generalized Wasserstein geodesics, see \cite[Chapter 9]{AGS}. These functionals are also coercive in the sense that they are bounded from below over sets of measures with mass $M$ and uniformly bounded second moments.

A coercive and  $W$-lower semicontinuous functional $\VV:\PP^M_2(\R^d)\to (-\infty,+\infty]$ that is convex along generalized Wasserstein geodesics  generates a continuous semigroup $S_t:D(\VV)\to D(\VV)$, where $D(\VV):=\{u\in\mathscr P_M^2(\R^d):\mathcal V(u)<+\infty\}$, satisfying
the following family of \emph{Evolution Variational Inequalities} (see \cite[Theorem 11.2.5]{AGS})
\begin{equation}\label{EVI0}
       \frac12 W^2(S_t(\bar u),v)- \frac12 W^2(\bar u,v) \le t(\VV(v)-\VV(S_t(\bar u))) \quad \forall \bar u,v \in D(\VV), \quad  \forall t>0.
\end{equation}
Moreover, $u_t:=S_t(\bar u)$ belongs to $AC^2_{loc}([0,+\infty);\Space)$
 and it is the unique limit curve obtained from the JKO scheme \eqref{minmov1} applied to functional $\mathcal V$, starting from the initial datum $\bar u$. 
The semigroup $S_t$ satisfies the following contraction property: 
\begin{equation}\label{contr} 
	W(S_t(\bar u), S_t(\tilde u))\le W(\bar u,\tilde u)\qquad \forall\,\bar u,\tilde u \in D(\mathcal V), \quad \forall\, t>0.
\end{equation} 
By the property \eqref{contr}, the semigrup $S_t$ extends to $\overline{D(\VV)}=\PP_2^M(\mathbb{R}^d)$.
Moreover the following regularizing effect holds: $S_t(u)\in D(\VV)$ for any $u\in\PP_2(\mathbb{R}^d)$ and any $t>0$.
Then,  \eqref{EVI0} holds for any $u,v\in\PP_2(\mathbb{R}^d)$. 
$S_t(u)$ will be referred to as the {\it gradient flow} of the functional $\mathcal V$, starting from $u$.

The properties illustrated above are shared by the functionals $\mathcal J, \mathcal U$ of Definition \ref{pot} and Definition \ref{entr}, 
and by any linear combination $a\mathcal U+b\mathcal J$ with positive coefficients $a,b$. 
In the latter case, convexity along generalized geodesics 
 still holds and  $u_t:=S_t(\bar u)$ 
is the unique distributional solution to the Cauchy problem 
\begin{equation*}
\label{basicpde}\de_t u_t=a\,\Delta(L_U(u_t))+b\,\mathrm{div}(\nabla J\,u),\quad u(0)=\bar u,\end{equation*} where
$L_U(u):=uU'(u)-U(u)$ is the conjugated function of $U$. Taking $a=b=1$, $U(u)=u\log u$ and $J(x)=\frac12|x|^2$ 
we obtain one of the fundamental examples of Wasserstein gradient flows \cite{JKO}, i.e., $u_t=S_t(\bar u)$ solves the Cauchy problem for the linear Fokker-Planck equation
 \begin{equation*}\de_t u_t=\Delta u+\mathrm{div}(xu),\quad u(0)=\bar u.\end{equation*}
 
Among our main results, we will establish to which extent the above properties are still valid for the energy functional $\mathcal E_m$ itself, see Section \ref{lastsection}. 
At this stage, we are interested in the behavior of the functional $\mathcal E_m$ along the gradient flow of the above functionals $\mathcal J$ and $\mathcal U$.
The following result provides the basic flow interchange \cite{MMS} estimate. The statement is borrowed from \cite[Proposition 4.3]{LMS}. We include the quick proof.
\begin{proposition}[\textbf{Flow interchange}]\label{prop:FI}
Let $\tau>0$ and let $\{u_\tau^k:k=0,1,2,\ldots\}$ be a sequence given by {\rm Proposition \ref{prop:existenceMM}}. 
 Let $\VV:\PP^M_2(\R^d)\to (-\infty,+\infty]$ be coercive,  $W$-lower semicontinuous, convex along generalized Wasserstein geodesics, and such that $u_\tau^0\in D(\VV)$. \RRR
Suppose that
\begin{equation}\label{>-inf}
\limsup_{t \downarrow 0} \frac{\EE_m(u_\tau^k)-\EE_m(S_t(u_\tau^k))}{t}>-\infty\quad \text{for every } k\geq 1.
\end{equation} 
Then, for every  $k\ge 1$ there hold $u_\tau^k\in D(\VV)$  and
\begin{equation}\label{mms09}
 \limsup_{t \downarrow 0} \frac{\EE_m(u_\tau^k)-\EE_m(S_t(u_\tau^k))}{t}\le   \frac{\VV(u_{\tau}^{k-1}) -\VV(u_{\tau}^k)}{\tau}.
\end{equation}
\end{proposition}

\begin{proof}
For  $t>0$ and and integer $k\ge1$,
by definition of minimizer there holds
\[
    \EE_m(u_\tau^k) + \frac{1}{2\tau}W^2(u_\tau^k,u_\tau^{k-1})\leq \EE_m(S_t(u_\tau^k)) + \frac{1}{2\tau}W^2(S_t(u_\tau^k),u_\tau^{k-1}),
\]
that is,
\[
   \tau(\EE_m(u_\tau^k) -\EE_m(S_t(u_\tau^k)) ) \leq \frac{1}{2}W^2(S_t(u_\tau^k),u_\tau^{k-1}) - \frac{1}{2}W^2(u_\tau^k,u_\tau^{k-1}).
\]
By using \eqref{EVI0} we obtain
\[
   \tau\frac{\EE_m(u_\tau^k) -\EE_m(S_t(u_\tau^k))}{t} \leq \VV(u_\tau^{k-1}) - \VV(S_t(u_\tau^k)).
\]
As $u_\tau^0\in D(\VV)$, we may now recursively apply the above inequality: thanks to \eqref{>-inf}, by passing to the limit as $t\downarrow 0$ and using the lower semicontinuity of $\VV$ with respect to the narrow convergence  we obtain that $u_\tau^k\in D(\VV)$ and  \eqref{mms09} holds.
\end{proof}

\subsection{Nonlinear diffusion equations}

Given  a displacement convex entropy $\mathcal U$ with density function $U$, according to Definition \ref{entr}, the conjugated function $L_U:[0,+\infty)\to \mathbb R$ is defined by 
\begin{equation*}\label{L}
	L_U(0)=0,\qquad\qquad L_{U}(u)=uU'(u)-U(u),\quad u>0.
\end{equation*}
 The properties of $U$ directly imply that 
\begin{equation}\label{LU}
	L_U\in C([0,+\infty))\cap C^1((0,+\infty)), \qquad L_U\geq 0, \qquad (L_U)'(u)=uU''(u)\geq 0 \;\;\;\mbox{$\forall \,u>0$}.
\end{equation}
\EEE
We shall consider the modified displacement convex entropy 
\begin{equation}\label{defV}
\mathcal V:=\mathcal U+\eps\mathcal F_1, \quad \eps>0,
\end{equation}
having conjugated function defined by 
\begin{equation}\label{defPsi}
	\Psi:[0,+\infty)\to\mathbb R, \qquad \Psi(u)=L_U(u)+\eps u.
\end{equation}
The gradient flow of $\mathcal V$ is therefore a distributional solution to the nondegenerate parabolic equation
\begin{equation}\label{nonlinear}
	\de_t u_t= \Delta \Psi(u_t).
\end{equation}
 By \eqref{LU} we have 
\begin{equation}\label{newstrong}
	\Psi\in C([0,+\infty))\cap C^1((0,+\infty)), \qquad \Psi'(u)\ge \eps \quad \forall\,u> 0.
\end{equation}\RRR
  The condition \eqref{newstrong} guarantees that the distributional solution of \eqref{nonlinear} is a strong solution, 
  in the sense that $\partial_tu_t \in L^1((t_1,t_2)\times\mathbb R^d)$ for every $0<t_1<t_2<+\infty$ and \eqref{nonlinear} is satisfied pointwise a.e. in $(0,+\infty)\times\mathbb R^d$
  (see \cite{BG} and \cite{J}).
  Moreover, by \eqref{newstrong} the equation \eqref{nonlinear} is uniformly parabolic, 
 hence $u_t>0$ a.e. in $(0,+\infty)\times\mathbb R^d$. 
  We also notice that under the above assumptions the solution to \eqref{nonlinear} is strong but it is not expected to be classical in general, since $\Psi$ is not required to be more smooth than $C^1$. \EEE 
  
   We observe that the model cases $U(u)=u\log u$ and $U(u)=u^m$, $m>1$,  satisfy \eqref{newstrong}. Another example  is $U(u)=(u-K)_+^p$ for $p>2$ and $K>0$.

The following Proposition is a technical result about the behavior of functionals $\mathcal F_m$ and $\mathcal B_\alpha$ along the solution of \eqref{nonlinear}.
We state it in a general form, even if we shall only need particular choices of $\mathcal U$ later on.

\EEE

\begin{proposition}\label{thenewprop} 
Let $\mathcal U$ be a displacement convex entropy according to {\rm Definition \ref{entr}}, with density function $U$. Let $\VV$ and $\Psi$ defined in \eqref{defV} and \eqref{defPsi}.
Given $\bar u\in D(\mathcal V)$, we denote by 
$u_t$ the gradient flow generated by $\VV$ starting from $\bar u$.\EEE
Then the following hold:
\begin{itemize}
\item[\textbf{(I)}] 
 For every $r\ge1$, \EEE
the map  $t\mapsto\mathcal{F}_r(u_t)$ is nonincreasing, belongs to $AC_{loc}((0,+\infty);\R)$ (also $AC_{loc}([0,+\infty);\R)$   if $r=1$)  and 
\begin{equation}\label{PSIm}
\frac d{dt}\mathcal F_r(u_t)=-r\int_{\mathbb R^d}\Psi'(u_t(x))\,u_t(x)^{r-1}\,\frac{|\nabla u_t(x)|^2}{u_t(x)}\,\d x, \qquad \mbox{for a.e. }t>0. 
\end{equation}
\item[\textbf{(II)}]
If $G:[0,+\infty)\to[0,+\infty)$ is  convex nondecreasing and such that $G(0)=0$, then the map $t\mapsto \int_{\mathbb R^d}G(u_t)\,\d x$ is nonincreasing on $(0,+\infty)$. 
 Furthermore, the map $t\mapsto \mathcal U(u_t)$ is nonincreasing and belongs to $AC_{loc}([0,+\infty);\R)$.
\item[\textbf{(III)}]
The map $t\mapsto \int_{\mathbb R^d}u_t\Psi(u_t)\,\d x$ belongs to $L^1_{loc}((0,+\infty))$
 and the map $t\mapsto \mathcal B_\alpha(u_t)$ belongs to $AC_{loc}((0,+\infty);\R)$ (also $AC_{loc}([0,+\infty);\R)$ if $d=2$\RRR). 
 Moreover
\begin{equation}\label{flowB}
	\frac d{dt}\mathcal B_\alpha(u_{t})\ge-\int_{\mathbb R^d}u_t\Psi(u_t)\,\d x \qquad \mbox{for a.e. }t>0,
\end{equation}
and equality holds if $\alpha=0$.
\end{itemize}
\end{proposition}

\begin{proof} \textbf{Step 1: properties of the solution $u_t$}.
%
  First of all, since $u_t$ is a gradient flow of the energy functional $\VV$,  $u\in L^1_{loc}((0,+\infty);W^{1,1}_{loc}(\mathbb R^d))$ 
  and it satisfies the energy identity
 \begin{equation}\label{enid}
	\VV(u_{t_1})=\VV(u_{t_2})+\int_{t_1}^{t_2}\int_{\mathbb R^d}\frac{|\nabla \Psi(u_t)|^2}{u_t}\,\d x\,\d t, \quad \forall\, t_1,t_2\in[0,+\infty), \quad t_1<t_2,
\end{equation} 
 see \cite[Theorem 11.2.1, Theorem 11.2.5]{AGS}.
 In particular, $u_t^{-1}{|\nabla (\Psi\circ u_t)|^2}\in L^1((0,T)\times\mathbb R^d)$ for any $T>0$. \EEE 
Moreover,  
\begin{equation}\label{1infty}
u\in C^0([0,+\infty);L^1(\mathbb R^d))\cap L^\infty((\tau+\infty)\times \mathbb R^d), \qquad \forall\,\tau>0
\end{equation}
(see \cite[Theorem 9.12]{V} for the porous media equation
and \cite[Corollary 10.24]{V}, \cite{BB}, \cite[Section 2.5]{V2} for the general case).

 By \eqref{newstrong} $\Psi$ is strictly increasing and locally Lipschitz function on $[0,+\infty)$, then $\Psi^{-1}$ is Lipschitz on $[0,+\infty)$ with Lipschitz constant $L\leq \eps^{-1}$. 
 We deduce that,
  for $T>0$,
 \begin{equation}\label{fish}\begin{aligned}
 \int_0^{T}\int_{\mathbb R^d}\frac{|\nabla u_t|^2}{u_t}\,\d x\,\d t&= \int_0^{T}\int_{\mathbb R^d}\frac{|\nabla (\Psi^{-1}\circ\Psi\circ u_t)|^2}{u_t}\,\d x\,\d t\\
 &\le \frac1{\eps^2}\int_0^{T}\int_{\mathbb R^d}\frac{|\nabla \Psi( u_t)|^2}{u_t}\,\d x\,\d t<+\infty .
 \end{aligned}\end{equation}
 \EEE
 From the latter, by applying Young inequality we also deduce that 
 \begin{equation}\label{graduestimate}\begin{aligned}
 	\int_{0}^T\int_{\mathbb R^d}|\nabla u_t|\,\d x\,\d t&\le\frac12\int_0^T\left(\int_{\mathbb R^d}\frac{|\nabla u_t|^2}{u_t}\,\d x+\int_{\mathbb R^d}u_t\,\d x\right)\,\d t\\
	 &=\frac {MT}2+\frac12\int_0^T\int_{\mathbb R^d}\frac{|\nabla u_t|^2}{u_t}\,\d x\,\d t<+\infty, \quad \forall \,T>0,
	 \end{aligned}
 \end{equation}
so that the solution $u_t$ belongs to $L^1((0,T);W^{1,1}(\mathbb R^d))$. Moreover, from \eqref{1infty} and \eqref{graduestimate}, and since $\Psi$ is locally Lipschitz, we deduce
 \begin{equation}\label{doublebound}
 \Psi(u_t)\in L^1_{loc}((0,+\infty); W^{1,1}(\mathbb R^d)\cap L^\infty(\mathbb R^d)).
 \end{equation}

By applying \eqref{EVI0} with $v$ being the characteristic function of the unit ball in $\mathbb R^d$, we easily see that the second moment of $u_t$ grows at most linearly with respect to $t$, in particular it stays uniformly bounded on $[0,T]$, for any fiexd $T\in \mathbb R$. Further, \eqref{enid} implies that $t\mapsto \VV(u_t)$ belongs to $C^0([0,T])$. The bound on $\VV(u_t)$ and on the second moment of $u_t$ on $[0,T]$ implies the uniform bound on $[0,T]$ also for $\FF_1$, thanks to the Carleman estimate \eqref{Carleman}. 
Moreover, having already noticed that $[0,+\infty)\ni t\mapsto u_t$ is strongly continuous in $L^1(\mathbb R^d)$, we may now apply the lower semicontinuity properties of displacement convex entropies and deduce that both $[0,+\infty)\ni t\mapsto\mathcal U(u_t)$ and $[0,+\infty)\ni t\mapsto \FF_1(u_t)$ are lower semicontinuous. In fact, they are both continuous on $[0,+\infty)$ since $t\mapsto\VV(u_t)=\mathcal U(u_t)+\eps\FF_1(u_t)$ is continuous. The map $t\mapsto \mathcal B_\alpha(u_t)$ is also continuous on $[0,+\infty)$ if $d=2$, thanks to the bound on $\mathcal F_1$\RRR,  by Proposition \ref{prop:FP}.

\medskip

 \textbf{Step 2: proof of (I)}.
We have already observed that the assumption \eqref{newstrong} on  $\Psi$ implies that $u_t$ is a strong solution to \eqref{nonlinear},  \EEE i.e.,   its time distributional derivative $\partial_tu_t$ belongs to $L^1((t_1,t_2)\times\mathbb R^d)$ for every $0<t_1<t_2<+\infty$ and \eqref{nonlinear} is satisfied pointwise a.e. in $(0,+\infty)\times\mathbb R^d$.
 In particular, if $r>1$, by the Sobolev chain rule we have $\partial_t(u_t^r)=ru_t^{r-1}\partial_t u_t$ a.e. in $(0,+\infty)\times\mathbb R^d$. Since $u_t$ is $W^{1,1}$ in space, it admits classical space partial  derivatives a.e., implying that $\nabla(u_t(x)^\gamma)=\gamma u_t(x)^{\gamma-1}\,\nabla u_t(x)$ a.e. on the set $\{u_t>0\}$, thus a.e. in $(0,+\infty)\times\mathbb R^d$, for every $\gamma>0$.  
 
  Let us consider first the case $r>1$. Let $\zeta_n(x)=\zeta(x/n)$, $n\in\mathbb N$, where $\zeta:\mathbb R^d\to [0,1]$ is smooth, radially decreasing and such that $\zeta(x)=1$ if $|x|\le 1$ and $\zeta(x)=0$ if $|x|\ge 2$. Thus $\zeta_n\to 1$ pointwise and monotonically and $\nabla\zeta_n\to 0$ in $L^\infty(\mathbb R^d)$ as $n\to+\infty$.   Let also $\eta\in C^\infty_c((0,+\infty))$. We may multiply \eqref{nonlinear} by $ru_t^{r-1}\eta\zeta$ and integrate by parts to get
\begin{equation}\label{ffff}\begin{aligned}
&-\int_0^\infty\int_{\mathbb R^d}\eta'(t)u_t(x)^r\zeta_n(x)\,\d x\,\d t=\int_0^\infty\int_{\mathbb R^d}\partial_t (u_t(x)^r)\eta(t)\zeta_n(x)\,\d x\,\d t\\
&\qquad=r\int_0^\infty\int_{\mathbb R^d}\Delta\Psi(u_t(x))u_t(x)^{r-1}\eta(t)\zeta_n(x)\,\d x\,\d t\\
&\qquad=-r(r-1)\int_0^\infty\int_{\mathbb R^d}\Psi'(u_t(x))\frac{|\nabla u_t(x)|^2}{u_t(x)}u_t(x)^{r-1}\eta(t)\zeta_n(x)\,\d x\,\d t\\
&\qquad\qquad-r\int_0^\infty\int_{\mathbb R^d}\Psi'(u_t(x))\eta(t)u_t(x)^{r-1}\nabla u_t(x)\cdot\nabla\zeta_n(x)\,\d x\,\d t.
\end{aligned}\end{equation}
  The last term  is vanishing as $n\to+\infty$, since $u_t$ belongs to $L^\infty((t_1,t_2)\times\mathbb R^d)$ and $\nabla u_t$ belongs to  $L^1((t_1,t_2)\times\mathbb R^d)$ for every $0<t_1<t_2<+\infty$ by \eqref{graduestimate},  and since $\Psi$ is locally Lipschitz on $[0,+\infty)$.
  The same properties and \eqref{fish} allow to pass to the limit in \eqref{ffff} by dominated convergence and get
\[\begin{aligned}
\int_0^\infty\eta'(t)\int_{\mathbb R^d}u_t(x)^r\,\d x\,\d t=
r(r-1)\int_0^\infty\int_{\mathbb R^d}\Psi'(u_t(x))\frac{|\nabla u_t(x)|^2}{u_t(x)}u_t(x)^{r-1}\eta(t)\,\d x\,\d t,
\end{aligned}\]
along with the fact that
 the map $t\mapsto \int_{\mathbb R^d}\Psi'(u_t(x))|\nabla u_t(x)|^2u_t(x)^{r-2}\,\d x$ is in $L^1((t_1,t_2))$ for every $0<t_1<t_2<+\infty$.
  We deduce that the map $t\mapsto\int_{\mathbb R^d} u_t(x)^r\,dx$ is in $AC_{loc}((0,+\infty))$ and that 
\begin{equation}\label{mt1t2}
\int_{\mathbb R^d}\frac{u_{t_2}(x)^r}{r-1}\,\d x-\int_{\mathbb R^d}\frac{u_{t_1}(x)^r}{r-1}\,\d x=-r\int_{t_1}^{t_2}\int_{\mathbb R^d}\Psi'(u_t(x))\frac{|\nabla u_t(x)|^2}{u_t(x)}u_t(x)^{r-1}\,\d x\,\d t
\end{equation}
for every $0<t_1<t_2<+\infty$. We notice that the integrand in the right hand side is nonnegative, so that the map $t\mapsto\int_{\mathbb R^d} u_t(x)^r\,\d x$ is nonincreasing on $(0,+\infty)$. From this relation it follows \eqref{PSIm}.

The same result for $r=1$ may be obtained  by passing to the limit in \eqref{mt1t2} as $r\to1^+$. Indeed, in the right hand side we may pass to the limit since we can apply  the monotone convergence theorem (up to splitting the integrals in two regions, where $u_t>1$ and $u_t<1$). About the left hand side, since we have mass conservation so that the map $t\mapsto\int_{\mathbb R^d}u_t\,\d x$ is constant, we have, for $0<t_1<t_2<+\infty$,
\[
\int_{\mathbb R^d}\frac{u_{t_2}(x)^r}{r-1}\,\d x-\int_{\mathbb R^d}\frac{u_{t_1}(x)^r}{r-1}\,\d x=\int_{\mathbb R^d}\frac{u_{t_2}(x)^r-u_{t_2}(x)}{r-1}\,\d x-\int_{\mathbb R^d}\frac{u_{t_1}(x)^r-u_{t_1}(x)}{r-1}\,\d x
\]
and again we may pass to the limit with the monotone convergence theorem, getting
\begin{equation}\label{dissm=1}\begin{aligned}
\int_{\mathbb R^d}{u_{t_2}\log u_{t_2}}\,\d x-\int_{\mathbb R^d}{u_{t_1}\log u_{t_1}}\,\d x&=-\int_{t_1}^{t_2}\int_{\mathbb R^d}\Psi'(u_t(x))\frac{|\nabla u_t(x)|^2}{u_t(x)}\,\d x\,\d t.
\end{aligned}\end{equation}
Since $t_1$ and $t_2$ are arbitrary, it follows \eqref{PSIm} for $r=1$. 
The continuity of $t\mapsto\mathcal F_1(u_t)$ up to $t=0$ has been shown in the previous step. Its absolute continuity up to $t=0$ is then obtained by passing to the limit as $t_1\to 0$ in \eqref{dissm=1}, where the left hand side passes to the limit by the continuity of $\mathcal F_1(u_t)$ and the right hand side passes to the limit by the monotone convergence theorem.

\medskip

\textbf{Step 3: proof of (II)}.
Concerning the general nonlinearity $G$, we preliminarily observe that $\int_{\R^d} G(u_t)\,\d x<+\infty$ for every $t>0,$ since $G$ is convex with $G(0)=0$ and since $u_t\in L^\infty(\R^d)$ for every $t>0$. Let $(G_k)_{k\in\mathbb N}$ be a sequence of convex nonnegative nondecreasing $C^{2}((0,+\infty))$ functions such that $G_k(x)\le G(x)$ for every $x\ge 0$ and such that $G_k\uparrow$ G pointwise and monotonically on $[0,+\infty)$.  By the Sobolev chain rule we have
$\partial_tG_k(u_t)=G_k'(u_t)\partial_t u_t$ and $\nabla G'_k(u_t)=G''_k(u_t)\nabla u_t $ a.e. in $(0,+\infty)\times\mathbb R^d$ and by repeating the same arguments of the previous step starting from \eqref{nonlinear} we obtain
\begin{equation}\label{GoppureU}
\int_{t_1}^{t_2}\eta'(t)\int_{\mathbb R^d}G_k(u_t)\,\d x\,\d t=\int_{t_1}^{t_2}\eta(t)\int_{\mathbb R^d}G_k''(u_t)u_t\Psi'(u_t)\frac{|\nabla u_t|^2}{u_t}\,\d x\,\d t
\end{equation}
for every $0<t_1<t_2<+\infty$, the space integral in the right hand side being in $L^1(t_1,t_2)$ due to \eqref{1infty}, \eqref{fish} and to the properties of $\Psi$ and $G_k$. This shows that, for every $0<t_1<t_2<+\infty$ and every $k\in\mathbb N$, the map $t\mapsto \int_{\mathbb R^d}G_k(u_t)\,\d x$ is in $AC([t_1,t_2])$, and since  $G_k''\ge0$ and $\Psi'\ge 0$ we also have $\int_{\mathbb R^d}G_k(u_{t_2})\,\d x\le\int_{\mathbb R^d}G_k(u_{t_1})\,\d x$. The latter inequality holds for $G$ as well, by the monotone convergence theorem.  The assumptions on $U$ and $L_U$ imply $U\in C^2((0,+\infty))$ with $U''(x)\ge 0$ for every $x>0$, therefore \eqref{GoppureU} holds with $U$ in place of $G$. This shows that $t\mapsto\mathcal U(u_t)$ is nonincreasing on $(0,+\infty)$. It is also $AC([0,T])$ for every $T>0$, since this holds true for $\mathcal V$ and $\mathcal F_1$ as shown in Step 1 and Step 2, respectively.

\medskip

\textbf{Step 4: proof of (III)}.
Assume that either $d>2$ or $\alpha>0$. 
By Proposition \ref{loggrowth} {\bf(i)}, {\bf(ii)} and \eqref{1infty} we have $B_\alpha\ast u_t\in L^\infty((\tau,+\infty); W^{1,\infty}(\mathbb R^d))$ for every $\tau>0$.
 
By the symmetry of the kernel $B_\alpha$ is is immediate to prove that 
\[
	\int_{\R^d}v B_\alpha\ast w\,\d x = \int_{\R^d}w B_\alpha\ast v\,\d x.
\]
Using the previous property and recalling \eqref{nonlinear},  we have, for any $\eta\in C^\infty_c((0,+\infty))$
and $\zeta_n$ as defined in Step 2,
\[
\begin{aligned}
&\frac12\int_0^\infty\int_{\mathbb R^d}\eta(t)\zeta_n(x)\partial_t[u_t(x)(B_\alpha\ast u_t)(x)]\,\d x\,\d t
\\&\qquad\qquad=\frac12\int_0^\infty\int_{\mathbb R^d}\eta(t)\left[\zeta_n(x)(B_\alpha\ast u_t)(x)+(B_\alpha\ast(\zeta_n u_t))(x)\right]\Delta\Psi(u_t(x))\,\d x\,\d t.
\end{aligned}
\]
Since $-\Delta(B_\alpha\ast u_t)+\alpha B_\alpha\ast u_t=u_t$, and the same for $\zeta_nu_t$ in place of $u_t$,
after integrations by parts  we get
\begin{equation}\label{R_n}
\begin{aligned}
&-\frac12\int_0^\infty\int_{\mathbb R^d}\eta'(t)u_t(x)(B_\alpha\ast u_t)(x)\zeta_n(x)\,\d x\,\d t\\&\;\;\;=
\frac 12\int_0^\infty\int_{\mathbb R^d}\eta(t)\left[\zeta_n(x)\Delta(B_\alpha\ast u_t)(x)+\Delta (B_\alpha\ast(\zeta_n u_t))(x)\right]\,\Psi(u_t(x))\,\d x\,\d t+\mathcal R_n\\&\;\;\;=
 \frac12\int_0^\infty\int_{\mathbb R^d}\eta(t)\zeta_n(x)[\alpha (B_\alpha\ast u_t)(x)-u_t(x)]\,\Psi(u_t(x))\,\d x\,\d t\\
 &\qquad\quad+\frac12 \int_0^\infty\int_{\mathbb R^d}\eta(t)[\alpha (B_\alpha\ast (\zeta_n u_t))(x)-\zeta_n u_t(x)]\,\Psi(u_t(x))\,\d x\,\d t+\mathcal R_n
\end{aligned}
\end{equation}
where the remainder term is
\[\begin{aligned}\mathcal R_n&=
  \frac12\int_0^\infty\int_{\mathbb R^d}\eta(t)\Psi(u_t(x))\nabla\zeta_n(x)\cdot\nabla(B_\alpha\ast u_t)(x)\,\d x\,\d t
\\&\qquad- \frac12\int_0^\infty\int_{\mathbb R^d}\eta(t)(B_\alpha\ast u_t)(x)\nabla\zeta_n(x)\cdot\nabla\Psi(u_t(x))\,\d x\,\d t.
\end{aligned}
\]
We wish to pass to the limit in \eqref{R_n} as $n\to+\infty$. First of all, we claim that $\mathcal R_n\to 0$ as $n\to+\infty$. 
Indeed, recalling that $B_\alpha\ast u_t\in L^\infty_{loc}((0,+\infty); W^{1,\infty}(\mathbb R^d))$ and that  \eqref{doublebound} holds, the claim follows from the fact that $\nabla\zeta_n\to 0$ in $L^\infty(\mathbb R^d)$ as $n\to+\infty$.
 Thanks to the property $B_\alpha\ast u_t\in L^\infty_{loc}((0,+\infty); L^{\infty}(\mathbb R^d))$ along with \eqref{1infty} and \eqref{doublebound}, and since $\zeta_n\to 1$ pointwise an monotonically as $n\to+\infty$, implying also the pointwise convergence $(B_\alpha\ast(\zeta_n u_t))(x)\to (B_\alpha \ast u_t)(x)$, we may also
 pass to the limit as $n\to\infty$ in both sides of \eqref{R_n} by dominated convergence, obtaining
\begin{equation}\label{-1/2}\begin{aligned}
&-\frac12\int_0^\infty\int_{\mathbb R^d}\eta'(t)u_t(x)(B_\alpha\ast u_t)(x)\,\d x\,\d t\\&\qquad\qquad\qquad=\int_0^\infty\int_{\mathbb R^d}\eta(t)(\alpha (B_\alpha\ast u_t)(x)-u_t(x))\Psi(u_t(x))\,\d x\,\d t.
\end{aligned}\end{equation}

Let us show how to get the same formula if $\alpha=0$ and $d=2$ as well. 
In fact, in $\mathcal R_n$ the first term can be treated as before, because $\nabla B_0\ast u_t\in L^\infty_{loc}((0,+\infty);L^\infty(\mathbb R^2))$ by Proposition \ref{loggrowth} {\bf(i)}. 
We claim that the second term in $\mathcal R_n$ vanishes as well. 
Indeed, letting the interval $[t_1,t_2]$ include the support of $\eta$, we use Proposition \ref{loggrowth} {\bf(iii)} to
estimate $B_0\ast u_t$, and  we get  
\begin{equation}\label{redundancy}\begin{aligned}
	& \int_0^\infty\int_{\mathbb R^2}|(B_0\ast u_t)(x)|\eta(t)||\nabla\zeta_n(x)||\nabla\Psi(u_t(x))|\,\d x\,\d t\
	\\&\qquad \le \max|\eta|\, \|\nabla \zeta_n\|_{L^\infty(\mathbb R^2)}\,\int_{t_1}^{t_2}\left(\frac{\pi}{2}\|u_t\|_{L^\infty(\mathbb R^2)}+\frac{1}{2\pi}\mathrm m_{log}(u_t)\right)\int_{\mathbb R^2}|\nabla\Psi(u_t)|\,\d x\,\d t\\		&\qquad\quad+\frac{M}{2\pi}\,\max|\eta|\,\|\nabla\zeta_n(\cdot)\log(1+|\cdot|^2)\|_{L^\infty(\mathbb R^2)}\int_{t_1}^{t_2}\int_{\mathbb R^2}|\nabla\Psi(u_t)|\,\d x\,\d t.
 \end{aligned}
 \end{equation}
 We notice that $\mathrm m_{log}(u_t)\in L^\infty((t_1,t_2))$, because the same holds even for the second moments as previously observed, and then thanks to \eqref{1infty} and to \eqref{doublebound} we deduce that the first term in the right hand side of \eqref{redundancy} vanishes as $n\to+\infty$. Still in view of \eqref{doublebound}, the same conclusion is found for the second term in the right hand side of \eqref{redundancy},
 because $|\nabla \zeta_n(x)| \,\log(1+|x|^2)\le n^{-1}(\max |\nabla\zeta|)\log(1+4n^2)$ for every $x\in\mathbb R^2$ so that $\nabla \zeta_n(\cdot)\log(1+|\cdot|^2)$ is still vanishing in $L^\infty(\mathbb R^d)$ as $n\to+\infty$.
  The same estimate for $B_0\ast u_t$ can be employed for obtaining a dominant function for passing to the limit in the other terms of \eqref{R_n}, recalling that $\mathrm m_{log}(u_t)\in L^1_{loc}(0,+\infty)$.

Since the identity \eqref{-1/2} holds for every $\eta\in C^\infty_c((0,+\infty))$, and the map 
$$t\mapsto \int_{\mathbb R^d}(\alpha (B_\alpha\ast u_t)(x)-u_t(x))\Psi(u_t(x))\,\d x$$ belongs to
$L^1_{loc}(0,+\infty)$, then the map $t\mapsto\mathcal B_\alpha(u_t)$ is in $AC_{loc}((0,+\infty))$ and
\begin{equation}\label{flowforB}
\mathcal B_\alpha(u_{t_2})-\mathcal B_\alpha(u_{t_1})=\int_{t_1}^{t_2}\int_{\mathbb R^d}(\alpha (B_\alpha\ast u_t)-u_t)\Psi(u_t)\,\d x\,\d t, \qquad \forall\, t_1,t_2\in(0,+\infty).
\end{equation}

We claim that \eqref{flowforB} holds for $t_1=0$ as well if $d=2$\RRR.
Indeed, since $u_t$ is in $C^0([0,+\infty);L^1(\mathbb R^d))$,  
the map $t\mapsto\mathcal B_\alpha(u_t)$ is continuous up  to $t=0$ by Proposition \ref{prop:FP} (where we can take $m=1$ if $d=2$, and take advantage of the fact that $\mathcal F_1(u_t)$ is continuous up to $t=0$ as seen in the previous steps\RRR). Therefore, if $\alpha=0,$ we may pass to the limit in \eqref{flowforB} as $t_1\to 0$ by applying the monotone convergence theorem to the right hand side where the integrand $-u_t\Psi(u_t)$ is nonpositive, obtaining in particular that the map $t\mapsto u_t\Psi(u_t)$ is in $L^1(0,t_2)$, for every $t_2>0$. The latter fact, together with the continuity of $t\mapsto\mathcal B_\alpha(u_t)$ and the monotone convergence theorem, allows to pass to the limit in \eqref{flowforB} also in the case $\alpha>0$. The claim is proved and \eqref{flowB} follows.
\end{proof}

\begin{remark}\label{heatremark}\rm If $\mathcal U$ is given itself by $\mathcal F_1$, the above result holds true  (we can of course take $\eps=0$) since \eqref{nonlinear} is reduced to the heat equation.  
\end{remark}

\begin{corollary}\label{corofp} In {\rm Proposition \ref{thenewprop}},
 if $\bar u\in L^q(\R^d)$ for some $q>1$, and if there exists $c>0,p>1$ such that $\Psi'(x)\ge cx^{p-1}$ for every $x\ge 0$  (for instance if $\mathcal U=\mathcal F_p$, $p>1$), then  the map $t\mapsto\FF_q(u_t)$ is absolutely continuous and nonincreasing on $[0,T]$, for every $T>0$. 
If $q\ge m$, also the map $t\mapsto\mathcal B_\alpha (u_t)$ is absolutely continuous on $[0,T]$, for every $T>0$.
\end{corollary}
\begin{proof}
Since the initial datum $\bar u$ belongs to $L^q(\mathbb R^d)$, we obtain a uniform bound $\|u_t\|_{L^q(\mathbb R^d)}\le \|\bar u\|_{L^q(\mathbb R^d)}$ for every $t\ge0$. Indeed, the monotonicity of $L^q(\R^d)$ norms is a  standard property of the porous medium equation, see for instance \cite[Theorem 9.3]{V}, \cite[Section 1.2]{V2}. It holds more generally for equation \eqref{nonlinear}, by using comparison with the porous medium equation, if $c>0, p>1$ and   $\Psi'(x)\ge c p  x^{p-1}$ for every $x\ge 0$ (meaning that \eqref{nonlinear} is more diffusive than the porous media equation $\partial_t v_t=c\Delta v_t^p $). By applying \cite[Theorem 1.3]{V2} we deduce
$\|u_t\|_{L^q(\mathbb R^d)}\le \|v_t\|_{L^q(\mathbb R^d)}$ for every $t>0$, where $v_t$ solves   $\partial_t v_t=c\Delta v_t^p $ with initial datum $\bar u^*$, which is the spherical decreasing rearrangement of $\bar u$. Since $v_t$ solves a porous medium equation with initial datum $\bar u^*$, we conclude that 
$$\|u_t\|_{L^q(\mathbb R^d)}\le \|v_t\|_{L^q(\mathbb R^d)}\le \|\bar u^*\|_{L^q(\mathbb R^d)}=\|\bar u\|_{L^q(\mathbb R^d)}.$$
     On the other hand,  the narrow lower semicontinuity of the $L^q$ norm entails
$$\|\bar u\|_{L^q(\mathbb R^d)}^q\le \liminf_{t\downarrow 0}\|u_t\|^q_{L^q(\mathbb R^d)}\le \limsup_{t\downarrow 0}\|u_t\|^q_{L^q(\mathbb R^d)}\le \|\bar u\|_{L^q(\mathbb R^d)}^q,$$ hence we get continuity of the map  $t\mapsto \FF_q(u_t)$ up to $t=0$. Such a map then proves to be absolutely continuous up to $t=0$, by applying  \eqref{mt1t2} with $r=q$ and taking the limit as $t_1\to 0$, as done for $\mathcal F_1$ in the proof of Proposition \ref{thenewprop}.

Eventually, if $q\ge m$, by taking advantage of the uniform bound on the $L^m$ norms of $u_t$, we may apply Proposition \ref{prop:FP} and deduce that $t\mapsto\mathcal B_\alpha (u_t)$ is continuous up to $t=0$. Then we may pass to the limit in \eqref{flowforB} as $t_1\to0$ as done in the last part of the proof of Proposition \ref{thenewprop} and get the desired absolute continuity up to $t=0$.  
\end{proof}

\EEE

In the next corollary, we specialize the formulas of Proposition \ref{thenewprop}  to a  relevant choice of $\mathcal U$: 
we consider the displacement convex entropy $\mathcal{F}_{p,K}$, with density function $$F_{p,K}(x):=\frac1{p-1}(x-K)_+^p,$$ 
where $K> 0$ and $p\ge2$.
If $p=2$,  ${F}_{2,K}$ is not $C^2(0,+\infty)$   
as required by Definition \ref{entr} and Proposition \ref{thenewprop}. 
This can be easily circumvented by taking a   regularized version of $\FF_{p,K}$, namely   
\begin{equation}\label{baddelta}
F_{p,K,\delta}(x)=\frac1{p-1}\,(x-K)_\delta^p,
\end{equation}
where   $\{(x-K)_\delta\}_{\delta\in(0,K)}$ is the notation for a family of smooth, convex, strictly increasing functions on $[0,+\infty)$ with the following properties: for any  fixed $\delta\in (0,K)$, there holds  $(x-K)_\delta=(x-K)_+$ for every $0\le x\le K-\delta$ and every $x\ge K+\delta$; moreover, as $\delta\downarrow 0$, $(x-K)_\delta\to (x-K)_+$ pointwise and monotonically.
Indeed, letting
\[
z_\delta(t)=\left\{\begin{aligned}&\left(1+\exp\left(-\frac{4t\delta}{\delta^2-t^2}\right)\right)^{-1}\quad&\mbox{if $|t|\le \delta$}\\
&\bC_{(0,+\infty)}(t)\quad &\mbox{if $|t|>\delta$},
\end{aligned}\right. 
\]
we define $$(x-K)_\delta=\int_{-\infty}^{x-K}z_\delta(t)\,dt,\qquad x\ge 0.$$ 
Therefore, $F_{p,K,\delta}$ is a smooth function on $[0,+\infty)$. We also notice that $$D(\FF_{p,K,\delta})=D(\FF_{p,K})=D(\FF_p),$$
where $D(\FF_p)$ is the set of positive measures with finite second moment and $L^p(\mathbb R^d)$ density w.r.t. the Lebesgue measure.
 We shall consider the gradient flow associated with the displacement convex entropy  $\FF_{p,K,\delta}+\eps\FF_1$, satisfying equation \eqref{nonlinear} with $\Psi(x)=L_{p,K,\delta}(x)+\eps x$, where $L_{p,K,\delta}$ is the conjugated function of $F_{p,K,\delta}$.
We next apply the result of Proposition \ref{thenewprop} to this specific choices.
\begin{corollary}\label{dissipations} 
Let $p\ge 2$, $\eps>0$, $\delta>0$, $K> 0$, $r\ge 1$. 
Let $u_t$ denote the gradient flow, starting from  $\bar u\in D(\mathcal F_p)$, of  the displacement convex entropy $\VV=\mathcal{F}_{p,K,\delta}+\eps\mathcal{F}_1$. 
Then, for a.e. $t>0$ there hold 
\begin{equation}\label{mterm}
	\frac{d}{dt}\FF_r(u_t)\le-\frac{4pr}{(r+p-1)^2}\int_{\mathbb{R}^d}\left|\nabla (u_t-K)_\delta^{\frac{r+p-1}{2}}\right|^2\,\d x
	-\eps \frac{4}{r}\int_{\mathbb{R}^d} |\nabla u_t^{r/2}|^2\,\d x,
\end{equation}
\begin{equation}\label{interactionderivative}\begin{aligned}
	\frac{d}{dt}\mathcal{B}_\alpha(u_t)&\ge - \int_{\mathbb{R}^d}(u_t-K)_\delta^{p+1}\,\d x-\eps\int_{\mathbb R^d}u_t^2\,\d x\\
	&\quad-\frac{2p-1}{p-1}\, K\int_{\mathbb{R}^d}(u_t-K)_\delta^p\,\d x-\frac{p}{p-1}\, K^2\int_{\mathbb{R}^d}(u_t-K)_\delta^{p-1}\,\d x.\end{aligned}
\end{equation}
\end{corollary}
\begin{proof}
We define the smooth nondecreasing  function $S_{K,\delta}:[0,+\infty)\to\mathbb R$ as $S_{K,\delta}(x)=x-(x-K)_\delta$. A computation shows that the conjugated function of $F_{p,K,\delta}$, namely $L_{p,K,\delta}(x):=xF'_{p,K,\delta}(x)-F_{p,K,\delta}(x)$, can be expressed as
\begin{equation}\label{lpkd}L_{p,K,\delta}(x)=(x-K)_\delta^p+\frac {p}{p-1}\,S_{K,\delta}(x)\, (x-K)_\delta^{p-1}\end{equation}
and it is itself a smooth nondecreasing function on $[0,+\infty)$. Indeed,  the properties of $(x-K)_\delta$ ensure that its positive powers are smooth as well. We shall apply Proposition \ref{thenewprop} with $\mathcal U=\mathcal F_{p,K,\delta}$, so that the function $\Psi$ therein is given in this case by $\Psi(x)=L_{p,K,\delta}(x)+\eps x$. Since $S_{K,\delta}$ is nondecreasing we get
\[
\int_{\mathbb R^d} u_t^{r-2}(u_t-K)_\delta^{p-1}\nabla S_{K,\delta}(u_t)\cdot\nabla u_t\,\d x \ge 0.
\]
 Therefore, starting from  \eqref{PSIm}  we obtain
 if $r>1$
\begin{equation}\label{ateveryt}\begin{aligned}
     & \frac{d}{dt}\mathcal F_r(u_t)=-r\int_{\mathbb R^d}\Psi'(u_t)\,u_t^{r-1}\,\frac{|\nabla u_t|^2}{u_t}\,\d x=-\frac{r}{r-1}\int_{\mathbb R^d}\nabla u_t^{r-1}\cdot\nabla \Psi(u_t)\,\d x\\
	&\;\;\le-\frac{r}{r-1}\int_{\mathbb R^d}\left(\nabla (u_t-K)_\delta^p+\frac{p}{p-1}S_{K,\delta}(u_t)\nabla (u_t-K)_\delta^{p-1}+\eps\nabla u_t\right)\cdot\nabla u_t^{r-1}\,\d x
\end{aligned}
\end{equation}
for a.e. $t>0$, where
concerning the first term in the right hand side, since  $\nabla(u_t-K)_\delta^p=p(u_t-K)_\delta^{p-1}\nabla(u_t-K)_\delta$ vanishes on the set $\{x\in\mathbb R^d:u_t(x)\le K-\delta\}$, we deduce 
\[\begin{aligned}
&-\frac{r}{r-1}\int_{\mathbb R^d}\nabla (u_t-K)_\delta^p\cdot\nabla u_t^{r-1}\,\d x=-\int_{\{u_t>K-\delta\}}\frac1{u_t}\,\nabla u_t^r\cdot\nabla(u_t-K)_\delta^p\,\d x\\
&\qquad=\int_{\{u_t>K-\delta\}}\frac{S_{K,\delta}(u_t)}{u_t(u_t-K)_{\delta}}\,\nabla u_t^r\cdot\nabla (u_t-K)_\delta^p\,\d x-\int_{\{u_t>K-\delta\}}\frac{\nabla u_t^r\cdot\nabla (u_t-K)_\delta^p}{(u_t-K)_\delta}\,\d x\\
&\qquad=\frac{p}{p-1}\int_{\{u_t>K-\delta\}}\frac{{S_{K,\delta}(u_t)}}{u_t}\,\nabla(u_t-K)_\delta^{p-1}\cdot\nabla u_t^r\,\d x-\int_{\{u_t>K-\delta\}}\frac{\nabla u_t^r\cdot\nabla (u_t-K)_\delta^p}{(u_t-K)_\delta}\,\d x,
\end{aligned} 
\]
whereas for the term involving $\eps$ we readily have
\[
\eps\frac{r}{r-1}\int_{\mathbb R^d}\nabla u_t\cdot\nabla u_t^{r-1}\,\d x = \eps \frac{4}{r}\int_{\mathbb{R}^d} |\nabla u_t^{r/2}|^2\,\d x,
\]
which we insert into \eqref{ateveryt} to obtain
\begin{equation}\label{nexttime}
\frac{d}{dt}\mathcal F_r(u_t)\le -\int_{\{u_t>K-\delta\}}\frac{\nabla u_t^r\cdot\nabla (u_t-K)_\delta^p}{(u_t-K)_\delta}\,\d x-\eps \frac{4}{r}\int_{\mathbb{R}^d} |\nabla u_t^{r/2}|^2\,\d x
\end{equation}
for a.e. $t>0$, and very similar computations show that \eqref{nexttime} holds if $r=1$ as well.
In the right hand side of \eqref{nexttime} we write $u_t^r=(u_t-K)_\delta^r+u_t^r-(u_t-K)_\delta^r$ and similarly as before we take advantage of the fact that the maps
$x\mapsto (x-K)_\delta^p$, $x\mapsto x^r-(x-K)_\delta^r$ are nonincreasing to deduce
\[
\int_{\{u_t>K-\delta\}}\frac{\nabla(u_t^r-(u_t-K)_\delta^r)\cdot\nabla(u_t-K)_\delta^p}{(u_t-K)_\delta}\,\d x\ge 0,
\]
thus
\[\begin{aligned}
\frac{d}{dt}\mathcal F_r(u_t)&\le -\int_{\{u_t>K-\delta\}}\frac{\nabla (u_t-K)_\delta^r\cdot\nabla (u_t-K)_\delta^p}{(u_t-K)_\delta}\,\d x-\eps \frac{4}{r}\int_{\mathbb{R}^d} |\nabla u_t^{r/2}|^2\,\d x\\
&=-\frac{4pr}{(p+r-1)^2} \int_{\{u_t>K-\delta\}}\left|\nabla (u_t-K)_\delta^{\frac{p+r-1}{2}}\right|^2\,\d x-\eps \frac{4}{r}\int_{\mathbb{R}^d} |\nabla u_t^{r/2}|^2\,\d x,
\end{aligned}
\]
where the last equality is obtained with a direct computation. Moreover, the integral can be equivalently taken on the whole $\mathbb R^d$ since $\nabla (u_t-K)_\delta^{(p+r-1)/2}=0$ on the set $\{u_t\le K-\delta\}$. This proves \eqref{mterm}.


We next notice that since $\Psi(x)=(x-K)_\delta^p+\frac{p}{p-1}\,S_{K,\delta}(x)\,(x-K)_\delta^{p-1}$ and since $0\le S_{K,\delta}(x)\le K$ we have for every $x\ge 0$
\[\begin{aligned}
x\Psi(x)&=(x-K)_\delta\Psi(x)+S_{K,\delta}(x)\Psi(x)\\&=(x-K)_\delta^{p+1}+\frac{2p-1}{p-1}\,S_{K,\delta}(x)\,(x-K)_\delta^p+\frac{p}{p-1}\, S_{K,\delta}^2(x)\,(x-K)_\delta^{p-1}+\eps x^2\\
&\le (x-K)_\delta^{p+1}+\frac{2p-1}{p-1}\,K\,(x-K)_\delta^p+\frac{p}{p-1}\, K^2\,(x-K)_\delta^{p-1}+\eps x^2.
\end{aligned}\]
The latter can be inserted into \eqref{flowB} to get \eqref{interactionderivative}.
\end{proof}

\subsection{The Fokker-Planck equation}
We  compute the derivative of the energy along the solution of the Fokker-Planck equation.
\begin{proposition}\label{thenewprop2}   Let $\eps>0$. Let $u_t$ denote the gradient flow of $\mathcal V=\frac12\mom+\eps\mathcal{F}_1$, 
starting from $\bar u\in D(\mathcal V)$. 
Then the following hold.
\begin{itemize}
\item[\textbf{(I)}] 
The map  $t\mapsto\mathcal{F}_m(u_t)$ is smooth on $(0,+\infty)$, it is continuous up to $t=0$ if $\bar u\in D(\mathcal F_m)$,  and for every $t>0$ there holds 
\begin{equation*}
\frac d{dt}\mathcal F_m(u_t)=d\int_{\R^d}u_t^m\,\d x-\eps m\int_{\R^d} u_t^{m-2}|\nabla u_t|^2\,\d x.
\end{equation*}
\item[\textbf{(II)}]
The map $t\mapsto \mathcal B_\alpha(u_t)$ is smooth on $(0,+\infty)$, it is continuous up to $t=0$ if $\bar u\in D(\mathcal F_m)$,  and for every $t>0$ there holds 
\begin{equation*}
\frac d{dt}\mathcal B_\alpha(u_{t})=\eps\alpha\mathcal B_\alpha(u_t)-\eps\int_{\R^d}u_t^2\,\d x-\int_{\R^d}u_t\,x\cdot\nabla B_\alpha\ast u_t\,\d x 
\ge -\eps\int_{\R^d}u_t^2\,\d x.
\end{equation*}
\end{itemize}
\end{proposition}

\begin{proof}
The gradient flow of $\mathcal V$ is the Fokker Planck equation $\partial _tu=\eps\Delta u+\mathrm{div}(xu)$, see \cite{JKO}.
The solution to the initial value problem for the Fokker-Planck equation admits the representation formula
\[
u_t(x)=\int_{\R^d}H_t^\eps(x,y)\bar u(y)\,\d y, \qquad t>0,\;x\in\mathbb R^d,
\]
where
\[
H_t^\eps(x,y)=\left(2\eps\pi(1-e^{-2t})\right)^{-d/2}\exp\left\{\frac{|x-e^{-t}y|}{2\eps(1-e^{-2t})}\right\},
\]
see for instance \cite{C, KLan}.
In terms of the Gaussian kernel introduced in \eqref{gaussian}, it reads
\begin{equation}\label{convtrans}
u_t(x)=e^{dt}\left(\Gamma_{\eps\sigma(t)}\ast\bar u\right)\!(e^tx),\qquad\sigma(t):=\frac{e^{2t}-1}{2}.
\end{equation}
Therefore $u\in C^{\infty}((0,+\infty)\times\mathbb R^d)$, and if $\bar u \in D(\mathcal F_m)$, then the maps $t\mapsto\mathcal B_\alpha(u_t)$   and $t\mapsto\mathcal F_m(u_t)$ are continuous up to $t=0$, see Proposition \ref{initialdatum} and Proposition \ref{prop:FP}. Furthermore, if $m>1$,  for every $t>0$ there holds 
\begin{equation}\label{sottointegrale}
\frac{d}{dt}\mathcal F_m(u_t)=\frac m{m-1}\int_{\mathbb R^d}u_t^{m-1}\left(\eps\Delta u_t+\mathrm{div}(xu_t)\right)\,\d x.
\end{equation}
Letting $\zeta_n$ as in the proof of Proposition \ref{thenewprop}, we have for every $n\in\mathbb N$
\[\begin{aligned}
\int_{\mathbb R^d}\zeta_nu_t^{m-1}\mathrm{div}(xu_t)\,\d x&=d\int_{\mathbb R^d}\zeta_nu_t^m\,dx+\frac{1}{m}\int_{\mathbb R^d}\zeta_n \,x\cdot\nabla u_t^m\,\d x
\\&= d\int_{\mathbb R^d}\zeta_nu_t^m\,\d x-\frac{1}{m}\int_{\mathbb R^d}u_t^m\,\mathrm{div}(x\zeta_n)\,\d x\\&=\frac {(m-1)d}m\int_{\mathbb R^d}\zeta_nu_t^m\,\d x-\frac{1}{m}\int_{\mathbb R^d}u_t^m\,x\cdot\nabla \zeta_n\,\d x.
\end{aligned}\]
Since the representation formula \eqref{convtrans}  shows that for every $t>0$ there holds $u_t(\cdot)\in D(\mathcal V)\cap L^1(\R^d)\cap L^\infty(\R^d)$ and thus $x u_t^m \in L^1(\mathbb R^d)$, we can pass to the limit as $n\to+\infty$ to get the integration by parts formula
\begin{equation}\label{unaemme}
\int_{\mathbb R^d}u_t^{m-1}\mathrm{div}(xu_t)\,\d x=\frac {(m-1)d}m\int_{\R^d}u_t^m\,\d x.
\end{equation}
Similarly, since $\dfrac{|\nabla u_t|^2}{u_t}\in L^1(\mathbb R^d)$, we may pass to the limit as $n\to+\infty$ in the identity
\[
\int_{\mathbb R^d}\zeta_n u_t^{m-1}\Delta u_t\,\d x=-\int_{\R^d}u_t^{m-1}\nabla \zeta_n\cdot\nabla u_t\,\d x-\int_{\R^d}\zeta_n\nabla u_t\cdot\nabla u_t^{m-1}\,\d x
\]
and get by monotone convergence
\[
\int_{\mathbb R^d} u_t^{m-1}\Delta u_t\,\d x=-(m-1)\int_{\R^d} u_t^{m-2}|\nabla u_t|^2\,\d x.\]
The latter and \eqref{unaemme} can be inserted in\eqref{sottointegrale} to conclude the proof of point \textbf{(I)} in the case $m>1$. The proof for  $d=2, m=1$ is analogous, and it is obtained by making use of the following identities that holds for every $t>0$
\begin{equation}\label{twoidentities}\begin{aligned}
&\int_{\R^2}x\cdot\nabla(u_t\log u_t-u_t)\,\d x=-d\int_{\R^2}(u_t\log u_t-u_t)\,\d x,\\
&\int_{\R^2}\log u_t\,\Delta u_t\,\d x=-\int_{\R^2}\frac{|\nabla u_t|^2}{u_t}\,\d x
\end{aligned}
\end{equation}
and are proved in the same way. Indeed, by making use of \eqref{twoidentities} we compute for $t>0$
\[\begin{aligned}
\frac{d}{dt}\mathcal F_1(u_t)&=\frac{d}{dt}\int_{\R^2}(u_t\log u_t-u_t)\,\d x
=\int_{\R^2}\log u_t\,\partial_tu_t\,\d x=\int_{\R^2}\log u_t(\eps\Delta u_t+\mathrm{div}(xu_t))\,\d x\\&
=d\int_{\R^2}u_t\log u_t\,\d x+\int_{\R^2}x\cdot\nabla(u_t\log u_t-u_t)\,\d x-\eps\int_{\R^2}\frac{|\nabla u_t|^2}{u_t}\,\d x\\
&=dM -\eps\int_{\R^2}\frac{|\nabla u_t|^2}{u_t}\,\d x
\end{aligned}\]
thus proving \textbf{(I)} also in the case $d=2, m=1$.

On the other hand, 
  after two integrations by parts, we have for $t>0$ 
 \begin{equation}\label{zetarhs}
 \begin{aligned}
 &\int_{R^d}\zeta_n (B_\alpha\ast u_t)\Delta u_t\,\d x=-\int_{\R^d}B_\alpha\ast u_t\nabla\zeta_n\cdot\nabla u_t\,\d x\\&\qquad+\int_{\R^d}u_t \nabla \zeta_n\cdot\nabla B_\alpha\ast u_t \,\d x+\int_{\R^d}\zeta_n (\Delta B_\alpha\ast u_t)u_t\,\d x.
\end{aligned}
 \end{equation}
 We claim that the first two terms in the right hand side of \eqref{zetarhs} vanish as $n\to+\infty$, for every $t>0$. We have indeed $\nabla \zeta_n\to 0$ in $L^\infty(\R^d)$,  $u_t\in L^\infty(\R^d)\cap L^1(\R^d)$, $\nabla u_t\in L^\infty(\R^d)\cap L^1(\R^d)$. We also have finiteness of the second moment of $\nabla u_t$, thanks to \eqref{convtrans} and since $\bar u\in \mathscr P_2^M(\R^d)$. Thus,  by invoking Proposition \ref{loggrowth}, showing that  $\nabla B_\alpha\ast u_t\in L^\infty(\R^d)$ and that $B_\alpha\ast u_t$ grows at most logarithmically as $|x|\to+\infty$, 
 the claim follows. 
 By the same argument we also have $(B_\alpha\ast u_t)\Delta u_t\in L^1(\R^d)$ for every $t>0$. 
 Therefore, by passing to the limit as $n\to+\infty$ in \eqref{zetarhs}, by the monotone convergence theorem and by using that $-\Delta B_\alpha\ast u_t+\alpha B_\alpha\ast u_t=u_t$, we deduce
 \begin{equation}\label{partslaplaciano}
 \int_{R^d} (B_\alpha\ast u_t)\Delta u_t\,\d x=\int_{\R^d}(\Delta B_\alpha\ast u_t)u_t\,\d x=\int_{\R^d} (\alpha B_\alpha\ast u_t-u_t)u_t\,\d x.
 \end{equation}
Similarly, we have
\[
\int_{\R^d}\zeta_n B_\alpha\ast u_t\,\mathrm{div}(xu_t)\,\d x=-\int_{\R^d}x\cdot\nabla \zeta_n\,u_tB_\alpha\ast u_t\,\d x-\int_{\R^d}\zeta_n u_t\,x\cdot\nabla B_{\alpha}\ast u_t\,\d x,
\]
to be passed to the limit with analogous argument,  using that $xu_t \in L^1(\R^d)$, $x u_t B_\alpha\ast u_t\in L^1(\R^d)$ and $x \cdot\nabla u_t B_\alpha\ast u_t\in L^1(\R^d)$ for every $t>0$ (again as a consequence of the at most logarithmic growth of $B_\alpha\ast u$ from Proposition \ref{loggrowth}),  to get
\begin{equation}\label{consegno}
\int_{\R^d} B_\alpha\ast u_t\,\mathrm{div}(xu_t)\,\d x=-\int_{\R^d} u_t\,x\cdot\nabla B_{\alpha}\ast u_t\,\d x,
\end{equation}
where, since $B_\alpha$ is radially decreasing with radial profile denoted  by $\tilde B_\alpha$, there holds
\begin{equation}\label{radialdecreasing}\begin{aligned}
\int_{\R^d}u_t\, x\cdot \nabla B_\alpha\ast u_t\,\d x&=\int_{\R^d}\int_{\R^d}\frac{\tilde B_\alpha(|x-y|)}{|x-y|}\,x\cdot(x-y)\,u_t(x)u_t(y)\,\d x\,\d y\\
&=\frac12\int_{\R^d}\int_{\R^d}\tilde B'_\alpha(|x-y|)|x-y|\,u_t(x)u_t(y)\,\d x\,\d y\le 0.\end{aligned}
\end{equation}
 Since
 \[
 \frac{d}{dt}\mathcal B_\alpha(u_t)=\frac12\frac{d}{dt}\int_{\R^d}u_t\,B_\alpha\ast u_t\,\d x=\int_{\R^d}B_\alpha\ast u_t\,\partial_tu_t\,\d x
 =\int_{\R^d}B_\alpha\ast u_t\,(\eps\Delta u_t+\mathrm{div}(xu_t))\,\d x,
 \]
the result of point \textbf{(II)} follows from \eqref{partslaplaciano}, \eqref{consegno},  and \eqref{radialdecreasing} .
 \end{proof}

\EEE

\EEE

\section{Regularity of discrete minimizers}\label{discretemin}

The section is devoted to the properties of discrete minimizers from \eqref{minmov1}. We shall prove the Sobolev estimate $(u_\tau^k)^{m/2}\in H^1(\R^d)$ in Theorem \ref{generalp}, by invoking the flow interchange from Proposition \ref{prop:FI}. Furthermore, in Proposition \ref{psprop} we will derive estimates of the discrete derivatives of the $L^p$ norms of the $(u_\tau^k)$'s that will be crucially exploited to get  $L^\infty$ smoothing effect later in Section \ref{hypersection}. We shall eventually prove discrete hypercontractivity estimates for finite $p$ in Thereom \ref{lemmaLp}, Theorem \ref{theop} and Theorem \ref{newlpdge3}.\RRR


\subsection{$L^p$ and Sobolev estimates of discrete minimizers}

We start with a basic regularization property for the case $d=2$ and $m=m_c=1$. Its proof contains the main line of the arguments that shall be developed in the subsequent propositions.
We define the Fisher information functional
\begin{equation*}
	\II:\PP_2^M(\R^d)\to [0,+\infty], \qquad \II(v)=\int_{\R^d}\frac{|\nabla v|^2}{v}\,\d x, \quad \II(v)=+\infty \text{ if } v\not\ll \Leb{d},
\end{equation*}
and we recall that it is lower semi continuous with respect to the narrow convergence on sequence with uniformly bounded entropy (see \cite{AGS} Proposition 10.4.14). We shall take advantage of the next
\begin{proposition}
There exists a constant $C>0$ such that, for any $\eps>0$,
\begin{equation}\label{L2I}
	\int_{{\R}^2}v^2\,\d x\le \eps\left( \int_{{\R}^2}v|\log v|\,\d x\right)\II(v)\,+ Ce^{2C/\eps}\int_{{\R}^2}v\,\d x.
\end{equation} 
for any nonnegative $v\in L^1(\mathbb{R}^2)$ such that $\II(v)<+\infty$ and $v\log v\in L^1(\R^2)$.
\end{proposition}
\begin{proof} A proof of this inequality can be found on \cite[Corollary A.1]{BCKKLL}. \end{proof}

We include another simple estimate that will be useful later on. 
\begin{proposition}\label{trunc}
Let $p\in(1,2)$.
For every  $K>1$ there holds 
\[
\int_{\R^2}v^{p+1}\,\d x\le 2K\int_{\R^2}v^p\,\d x+\frac{A_p}{\log K}\left(\int_{\R^2}v|\log v|\,\d x\right)\int_{\R^2}|\nabla v^{p/2}|^2\,\d x,
\]
for every nonnegative $v \in L^p(\R^2)$ such that $v\log v\in L^1(\R^2)$ and $v^{p/2}\in H^1(\R^2)$, where $A_p:=(1-2^{-p/2})^{-2-2/p}\,G_p$ and $G_p$ is the best constant in the inequality \eqref{GNS}. 
\end{proposition}
\begin{proof}
Let us fix $K>1$ and define $w_K:=(v^{p/2}-K^{p/2})_+$. It is easy to verify that that on the set $\{v>2K\}$ there holds $v^{p+1}\le \tilde c_pw_K^{2+2/p}$, where $\tilde c_p:=(1-2^{-p/2})^{-2-2/p}$.   We may apply \eqref{GNS} to the function $w_K^{2/p}$ to get
\begin{equation}\label{trunc1}\begin{aligned}
\int_{\{v>2K\}}v^{p+1}\,\d x&\le \tilde c_p\int_{\R^2}w_K^{2+2/p}\,\d x\le \tilde c_pG_p\int_{\R^2}{w_K}^{2/p}\,\d x\int_{\R^2}|\nabla w_K|^2\,\d x\\&\le A_p \int_{\R^2}{w_K}^{2/p}\,\d x\int_{\R^2}|\nabla v^{p/2}|^2\,\d x.
\end{aligned}
\end{equation}
On the other and, since $w_K^{2/p}\le v$ if $v>K$ and $w_K=0$ otherwise, we have
\begin{equation}\label{trunc2}
\int_{\R^d} w_K^{2/p}\,\d x\le \frac{log K}{log K}\int_{\{v>K\}}v\,\d x\le \frac1{\log K}\int_{\R^d}v|\log v|\,\d x.
\end{equation}
Since
\[
\int_{\R^2}v^{p+1}\,\d x=\int_{\{v\le 2K\}}v^{p+1}\,\d x+\int_{\{v>2K\}}v^{p+1}\,\d x\le 2K\int_{\R^2}v^p+\int_{\{v>2K\}}v^{p+1}\,\d x,
\]
by invoking \eqref{trunc1} and \eqref{trunc2} we conclude.
\end{proof}
\RRR
\begin{proposition}\label{L2}
Let $d=2$, $m=1$.
Let  $\tau>0$ and $\{u_\tau^k: k=0,1,2,\ldots\}$ a sequence given by {\rm Proposition \ref{prop:existenceMM}}. 
Then, for any integer $k\ge 1$, we have $u_\tau^k\in L^2(\mathbb{R}^2)$, $u_\tau^k\in W^{1,1}_{loc}(\R^2)$, $\II(u_\tau^k)<+\infty$.
Moreover, for any $T\ge 1$ there exists a constant 
$C_\tau(T)>0$ depending only on $M$, $\chi$, $\alpha$ $\EE_1(u_\tau^0)$, $\mom(u_\tau^0)$ 
such that
\begin{equation}\label{EstimateIL2}\begin{aligned}
	&\frac{1}{2}\II(u_\tau^k) \leq \frac{\FF_1(u_\tau^{k-1})-\FF_1(u_\tau^{k})}{\tau} + C_\tau(T), \\&\frac{1}{2} \|u_\tau^k\|^2_{L^2(\R^2)} \leq \frac{\FF_1(u_\tau^{k-1})-\FF_1(u_\tau^{k})}{\tau} + C_\tau(T)\end{aligned}
\end{equation}
for each $k\ge 1$ such that $k\tau\leq T$.  If $u_0\in D(\mathcal E_1)$, then $C_\tau(T)\le C(T)$ for every $\tau\ge0$, where $C(T)$ depends only on $M$, $\chi$, $\alpha$ $\EE_1(u_0)$, $\mom(u_0)$. \RRR
\end{proposition}

\begin{proof}
We fix $T\ge 1$ and $k\ge 1$ such that $k\tau\leq T$. 
We consider the flow associated to the entropy $\FF_1$ starting from $u_\tau^k$.
We denote by $u_t:=S_t(u_\tau^k)$, $t\geq 0$, this flow which solve the heat equation starting from $u_\tau^k$, 
thus given by convolution of $u_\tau^k$ with the Gaussian kernel \eqref{gaussian}. 
Such a solution is smooth and the map $t\mapsto \EE_1(u_t)=\FF_1(u_t)-\chi\mathcal B_\alpha(u_t)$ is differentiable at every point $t>0$. 
By taking into account Proposition \ref{thenewprop} (see also Remark \ref{heatremark}) we have 
\[
\frac{d}{dt}\mathcal{E}_1(u_t) \le -\II(u_t)+\chi\int_{\mathbb{R}^d}u_t^2\,\d x \qquad \forall\,t>0.
\]
Using the Carleman type inequality \eqref{Carleman} and \eqref{L2I}
it holds that
\[
\int_{\R^2}|u_t|^2\,\d x\le \eps\left(\FF_1(u_t)+2M\log\pi+\frac{2}{e}+4\int_{\R^2}|x|^2 u_t(x)\,\d x\right)\II(u_t)\,+ Ce^{2C/\eps}M  \qquad \forall\,t>0
\]
for any $\eps>0$, where $C$ is the constant from \eqref{L2I}.
Taking into account the evolution of the second moment along the heat equation, see Proposition \ref{initialdatum}, and $\FF_1(u_t)\le \FF_1(u_\tau^k)$ it follows that for every $t>0$
\[
\int_{\R^2}u_t^2\,\d x\le \eps\left(\FF_1(u_\tau^k)+2M\log\pi+\frac{2}{e}+4\int_{\R^2}|x|^2 u_\tau^k(x)\,\d x + 16M t \right)\II(u_t)\,+ Ce^{2C/\eps}M.
\]
By Proposition \ref{logequi}, if $C_\tau^0$ denotes the constant from \eqref{forsecostante},  letting $\bar C_{\tau,T}:=2M\log\pi+2e^{-1}+16M+C_\tau^0+4C_\tau^0 T$, we then deduce 
\[
\int_{\R^2}u_t^2\,\d x\le \eps\left (\bar C_{\tau,T}(1+\log^+(8C_3T))\right)\II(u_t)\,+ Ce^{2C/\eps}M  \qquad \forall\,t\in(0,1).
\]
Choosing  $\eps=1/\left(2\chi \bar C_{\tau,T}(1+\log^+(8C_3T))\right)$ we get
\[
\frac{d}{dt}\mathcal{E}_1(u_t) \le -\frac{1}{2}\II(u_t)+ C_\tau(T)  \qquad \forall\,t\in(0,1)
\]
where $C_\tau(T):= MCe^{4C\chi \bar C_{\tau,T}(1+\log^+(8C_3T))}$.

We observe that the map $[0,+\infty)\ni t\mapsto\EE_1(u_t)$ is continuous at $t=0$. Indeed, the maps $t\mapsto \FF_1(u_t)$ and $t\mapsto\BB_\alpha(u_t)$ 
are continuous up to $t=0$ along the heat flow, as seen in Proposition \ref{initialdatum}.
By Lagrange mean value theorem, for any $t\in(0,1)$ there exists $\theta(t)\in (0,t)$ such that
\begin{equation*}
\frac{\mathcal{E}_1(u_\tau^k)-\mathcal{E}_1(S_t(u_\tau^k))}{t}=-\frac{d}{ds}(\mathcal{E}_1(S_s(u_\tau^k))){\Bigg|}_{s=\theta(t)}\!\ge 
\frac{1}{2}\II(u_{\theta(t)}) - C_\tau(T) 
\end{equation*}
Passing to the limit as $t\to 0$
\begin{equation}\label{basicinterchange}
\limsup_{t\to 0^+}\frac{\mathcal{E}_1(u_\tau^k)-\mathcal{E}_1(S_t(u_\tau^k))}{t} \ge 
\liminf_{t\to 0^+}\frac{1}{2}\II(u_{\theta(t)}) - C_\tau(T) 
\end{equation}
By the lower semi continuity of $\II$ with respect to the narrow convergence we obtain that
$u_\tau^k\in D(\II)$ and by \eqref{basicinterchange} and Proposition \ref{prop:FI} we get
\begin{equation}\label{firstchange}
\frac{1}{2}\II(u_\tau^k) \leq \frac{\FF_1(u_\tau^{k-1})-\FF_1(u_\tau^{k})}{\tau} + C_\tau(T)
\end{equation}
which is the first  inequality in \eqref{EstimateIL2}.

Applying \eqref{L2I} to $u_\tau^k$ we get $u_\tau^k\in L^2(\R^2)$ and
\begin{equation}\label{noname}
	\int_{{\R}^2}|u_\tau^k|^2\,\d x\le \eps\left( \int_{{\R}^2}u_\tau^k|\log u_\tau^k|\,\d x\right)\II(u_\tau^k)\,+ MCe^{2C/\eps}.
\end{equation} 
By applying \eqref{FondEst2},  $C_\tau^0$ being again the constant therein, by applying \eqref{noname} with  the choice of $\eps=1/\left(C_\tau^0(1+\log^+(8C_3T))\right)$, and by \eqref{firstchange}, up to redefining $C_\tau(T)$ we obtain the second inequality in \eqref{EstimateIL2}.
The last statement follows from Proposition \ref{logequi} as well.
\end{proof}

\RRR

\begin{proposition}\label{sscritical} 
Let $d=2$, $m=1$. 
Let  $\tau>0$ and $\{u_\tau^k: k=0,1,2,\ldots\}$ a  sequence given by {\rm Proposition \ref{prop:existenceMM}}.
For every $T\ge 1$ and $p\in(1,+\infty)$, define
\begin{equation}\label{kappat}
K_{p,\tau}(T):= \exp\left(\frac{1}{4}\chi p G_p C_\tau^0 (1+\log^+(8C_3T))\right),
\end{equation} 
where $C_\tau^0$ is the constant in {\rm Proposition \ref{logequi}} and $G_p$ is the optimal constant of the inequality \eqref{GNS}.
 Then, for every $T\ge 1$, for every integer $k\ge 1$ such that $k\tau\le T$  and every $p\in[2,+\infty)$ we have $u_\tau^k\in L^{p+1}(\R^2)$ and there holds \RRR
\begin{equation}\label{extraterms1}\begin{aligned}
	& \frac1{\tau(p-1)} \left(\int_{\mathbb{R}^2}(u_\tau^{k-1}-K)_+^p\,\d x-\int_{\mathbb{R}^2}(u_\tau^k-K)^p_+\,\d x\right)
	\ge C_*(K)\int_{\mathbb{R}^2}(u_\tau^k-K)_+^{p+1}\,\d x\\
	&\qquad\qquad -\frac{2p-1}{p-1}\,\chi K\int_{\mathbb{R}^2}(u_\tau^k-K)_+^p\,\d x-\frac{p}{p-1}\,\chi  K^2\int_{\mathbb{R}^2}(u_\tau^k-K)_+^{p-1}\,\d x,
\end{aligned}\end{equation}
for every $K > K_{p,\tau}(T)$,
 where $C_*(K)>0$ is defined by 
\begin{equation}\label{CKappa}
	C_*(K):= M^{-1}\left(\frac{4}{pG_p}-\frac{\chi C_\tau^0(1+\log^+(8C_3T))}{\log K}\right).
\end{equation}
In particular,  $u_\tau^k\in L^p(\R^2)$ for every $p\in(1,+\infty)$ and every integer $k\ge 1$. 
\end{proposition}

\begin{proof}


 Let $T\ge 1$ and $k\ge 1$ an integer such that $k\tau\le T$.  The argument is recursive with respect to $k$, as in the proof of  Proposition \ref{prop:FI}. Indeed, we assume that for every  $j=0,1,\ldots k-1$ there holds $u_\tau^j\in L^p(\R^d)$ for every $1\le p<+\infty$, a property which is clearly true for $j=0$, see \eqref{heatcontraction}. \RRR
The result will be achieved by proving the following implication:
\begin{equation}\label{recursivestep}
p\ge 2\quad\mbox{and}\quad(u_\tau^k)\in L^p(\R^2)\qquad\Longrightarrow\qquad(u_\tau^k)\in L^{p+1}(\R^2)\quad\mbox{and}\quad\mbox{\eqref{extraterms1} holds.}
\end{equation}
Indeed, this allows to recursively prove \eqref{extraterms1} for every $p\ge 2$, and we notice that we can start the recursive argument since we already know that $u_\tau^k\in L^2(\R^2)$ by Proposition \ref{L2}.

In order to prove \eqref{recursivestep}, we assume $(u_\tau^k)\in L^p(\R^2)$ for some $p\ge 2$.
We let $K> K_{p,\tau}(T)$ (in particular, $K>1$)
such that $\int_{\mathbb R^2}(u_\tau^k-K)_+\,\d x>0$ (otherwise $u_\tau^k\in L^\infty(\R^d)$ and \eqref{recursivestep} is satisfied).
We also let $\delta\in(0,(K-1)\wedge(K/2))$. 
We denote by $u_t:=S_t(u_\tau^k)$ the gradient flow of the
 functional $\FF_{p,K,\delta}+\eps\FF_1$, starting from $u_\tau^k$. Here, $\mathcal F_{p,K,\delta}$ is the displacement convex entropy whose density function is given by \eqref{baddelta}, and $u_t$ solves therefore \eqref{nonlinear} with $\Psi(x)=L_{p,K,\delta}(x)+\eps x$.
By Proposition \ref{thenewprop} and Corollary \ref{dissipations},  the map $t\mapsto \mathcal E_1(u_t)=\FF_1(u_t)-\mathcal B_\alpha(u_t)$ belongs to $AC_{loc}([0,+\infty))$ and, for any $t>0$, there holds
\begin{equation}\label{long}\begin{aligned}
&{\mathcal{E}_1(u_\tau^k)-\mathcal{E}_1(S_t(u_\tau^k))}=-\int_{0}^t\frac{\d}{\d s}(\mathcal{E}_1(S_s(u_\tau^k)))\,\d s\ge \frac4p\int_0^t\int_{\mathbb{R}^2}\left|\nabla (u_s-K)_\delta^{p/2}\right|^2\,\d x\,\d s\\&\qquad\quad-\chi\int_0^t\int_{\mathbb{R}^2}(u_s-K)_\delta^{p+1}\,\d x\,\d s-\frac{2p-1}{p-1}\,\chi K\int_0^t\int_{\mathbb{R}^2}(u_s-K)_\delta^p\,\d x\,\d s\\
&\qquad\quad-\frac{p}{p-1}\,\chi K^2\int_0^t\int_{\mathbb{R}^2}(u_s-K)_\delta^{p-1}\,\d x\,\d s-\chi\eps\int_0^t\int_{\mathbb R^2}u_s^2\,\d x\,\d s.
\end{aligned}
\end{equation}
 We estimate the gradient term in \eqref{long} with \eqref{GNS} obtaining
\begin{equation}\label{startgns}
	\frac4p\int_{\R^2}|\nabla(u_s-K)_\delta^{p/2}|^2\,\d x \ge \frac{4}{pG_p}\left(\int_{\R^2}(u_s-K)_\delta\,\d x \right)^{-1}\int_{\R^2} (u_s-K)_\delta^{p+1}\,\d x.
\end{equation} 
Since $K>1$ we have 
\begin{equation}\label{equiintegral0}\begin{aligned}
	\int_{\mathbb{R}^2} (u_s-K)_\delta\,\d x&\le \frac1{\log K}\int_{\R^2}(u_s-K)_\delta\log K\,\d x\\&\le\frac1{\log K}\int_{\{u_s>K-\delta\}}(u_s-K)_\delta\log(u_s+\delta)\,\d x.\end{aligned}
\end{equation}
We notice that for every $x\ge K-\delta$ there holds $$(x-K)_\delta\log(x+\delta)\le (x-(K-\delta))\log(x+\delta)
= x\log x +b_{K,\delta}(x),
$$
where $b_{K,\delta}:[K-\delta,+\infty)\to\mathbb R$ is given by $b_{K,\delta}(x):=x\log (x+\delta)-x\log x-(K-\delta)\log(x+\delta)$, and we claim that  $b_{K,\delta}\le 0 $. Indeed, Since $\delta<K-1$ we get $b_{K,\delta}(K-\delta)<0$, and since $b'_{K,\delta}(x)={(\delta^2-x(K-\delta))}/{(x^2+\delta x)}$ we get $b'_{K,\delta}\le 0$ as $K>2\delta$. This proves the claim, which together with \eqref{equiintegral0} entails  
\begin{equation}\label{equiintegral}
\int_{\mathbb{R}^2} (u_s-K)_\delta\,\d x\le\frac1{\log K}\int_{\mathbb R^2}u_s\log^+u_s\,\d x.
\end{equation}
We next claim that $\|u_t\|_2\le\|u_\tau^k\|_2<+\infty$ for every $t\ge0$, and more generally we claim that the four maps \begin{equation*}\label{3maps}t\mapsto \int_{\R^2}(u_t-K)_\delta^p\,\d x,\quad t\mapsto \int_{\R^2} (u_t-K)_\delta^{p-1}\,\d x,\quad t\mapsto \int_{\mathbb R^2}u_t\log^+u_t\,\d x,\quad t\mapsto\int_{\R^2}u_t^2\,\d x\end{equation*} are nonincreasing  and continuous up to $t=0$. 
Indeed, since $u_\tau^k\in L^p(\R^2)$, $p\ge 2$, and since it is not difficult to check that $\Psi'(x)\ge c x^{p-1}$ for every $x\ge 0$ (for some suitable $c>0$ depending on $p,K,\delta,\eps$), we may invoke Corollary \ref{corofp} to deduce that the map $t\mapsto \|u_t\|_p$ is nonincreasing and continuous on $[0,+\infty)$. The four maps above are in the form $t\mapsto \int_{\R^2}G(u_t)\,\d x$ for a nonnegative convex $G:[0,+\infty)\to[0,+\infty)$ with at most $p$ growth and with $G(0)=0$, so that there exists $a>0,b>0$ such that  $G(x)\le ax+bx^p$ for every $x\ge 0$. Their continuity on $[0,+\infty)$ then follows by an application of Vitali's convergence theorem. Their monotonicity is a consequence of point {\bf{(II)}} of Proposition \ref{thenewprop}.  \RRR
Then, from \eqref{equiintegral} and Proposition \ref{logequi},
\begin{equation*}\begin{aligned}\label{intermezzo}
\int_{\mathbb{R}^2} (u_s-K)_\delta\,\d x&\le\frac1{\log K}\int_{\mathbb R^2}u_s\log^+u_s\,\d x\le \frac{1}{\log K}\int_{\mathbb R^2}u_\tau^k\log^+u_\tau^k\,\d x
\le\frac{C_\tau^0(1+\log^+(8C_3T))}{\log K}\end{aligned}
\end{equation*}
where $C_\tau^0$ is the constant of Proposition \ref{logequi}. 
Therefore, from \eqref{startgns} we obtain
\begin{equation}\label{followgns}\begin{aligned}
	&\frac4p\int_{\R^2}|\nabla(u_s-K)_\delta^{p/2}|^2\,\d x-\chi\int_{\R^2}(u_s-K)_\delta^{p+1}\,\d x\\
	&\quad \ge\left(\int_{\R^2}(u_s-K)_\delta\,\d x \right)^{-1}\left(\frac{4}{pG_p}  -\chi \int_{\R^2}(u_s-K)_\delta\,\d x \right)\int_{\R^2}(u_s-K)_\delta^{p+1}\,\d x\\
	&\quad\geq M^{-1}\left(\frac{4}{pG_p}  -\chi \frac{C_\tau^0(1+\log^+(8C_3T))}{\log K}\right)\int_{\R^2}(u_s-K)_\delta^{p+1}\,\d x.\\
\end{aligned}\end{equation}
By \eqref{CKappa} we have $C_*(K)>0$ for $K>K_{p,\tau}(T)$, therefore the coefficient in the above left hand side is positive. 
Then the latter claim,
 \eqref{followgns},
 \eqref{long} and  \eqref{CKappa} entail that for every $t>0$
\begin{equation}\label{long2}\begin{aligned}
	&\frac{\mathcal{E}_1(u_\tau^k)-\mathcal{E}_1(S_t(u_\tau^k))}{t}=-\frac1t\int_0^t\frac{\d}{\d s}(\mathcal{E}_1(S_s(u_\tau^k)))\,\d s\\
	&\quad\ge C_*(K)\frac1t\int_0^t \int_{\mathbb{R}^2}(u_s-K)_\delta^{p+1}\,\d x\,\d s-\frac{2p-1}{p-1}\,\chi K\,\int_{\mathbb{R}^2}(u_\tau^k-K)_\delta^p\,\d x\\
	&\qquad-\frac{p}{p-1}\,\chi K^2\,\int_{\mathbb{R}^2}(u_\tau^k-K)_\delta^{p-1}\,\d x\,-\chi\eps \int_{\R^d}(u_\tau^k)^2\,\d x.
\end{aligned}
\end{equation}
  If  we  
 pass to the limit as $t\to 0$ in \eqref{long2},  we get
\[\begin{aligned}
	&\liminf_{t\downarrow 0}\frac{\mathcal{E}_1(u_\tau^k)-\mathcal{E}_1(S_t(u_\tau^k))}{t}
	\ge C_*(K)\liminf_{t\downarrow 0}\frac1t\int_0^t\int_{\mathbb{R}^2}(u_s-K)_\delta^{p+1}\,\d x\,\d s\\
	&\quad-\frac{2p-1}{p-1}\chi K\int_{\mathbb{R}^2}(u_\tau^k-K)_\delta^p\,\d x-\frac{p}{p-1}\chi K^2\int_{\mathbb{R}^2}(u_\tau^k-K)^{p-1}_\delta\,\d x -\chi \eps\int_{\R^2}(u_\tau^k)^2\,\d x.
\end{aligned}
\]
The flow interchange inequality from Proposition \ref{prop:FI} gives therefore
\begin{equation*}\label{iglu}\begin{aligned}
	&\frac{\FF_{p,K,\delta}(u_\tau^{k-1})+\eps\mathcal F_1(u_\tau^{k-1})-\FF_{p,K,\delta}(u_\tau^k)-\eps\mathcal F_1(u_\tau^k)}{\tau}
	+\frac{2p-1}{p-1}\chi K\int_{\mathbb{R}^2}(u_\tau^k-K)_\delta^p\,\d x\\&\quad+\frac{p}{p-1}\chi K^2\int_{\mathbb{R}^2}(u_\tau^k-K)^{p-1}_\delta\,\d x+\chi\eps\int_{\R^2}(u_\tau^k)^2\,\d x\\
	&\quad \ge C_*(K)\liminf_{t\downarrow0}\frac1t\int_0^t\int_{\mathbb{R}^2}(u_s-K)_\delta^{p+1}\,\d x\,\d s\ge C_*(K)\liminf_{t\downarrow0}\int_{\mathbb{R}^2}(u_t-K)_\delta^{p+1}\,\d x
\end{aligned}
\end{equation*}
where the last inequality is due to the fact that the map $s\mapsto\int_{\mathbb{R}^2}(u_s-K)_\delta^{p+1}\,\d x $ is nonincreasing by point \textbf{(II)} of Proposition \ref{thenewprop}. By a lower semicontinuity argument, we obtain 
\begin{equation}\label{iglu2}\begin{aligned}
	&\frac{\FF_{p,K,\delta}(u_\tau^{k-1})+\eps\mathcal F_1(u_\tau^{k-1})-\FF_{p,K,\delta}(u_\tau^k)-\eps\mathcal F_1(u_\tau^k)}{\tau}
	+\frac{2p-1}{p-1}\chi K\int_{\mathbb{R}^2}(u_\tau^k-K)_\delta^p\,\d x\\&\quad+\frac{p}{p-1}\chi K^2\int_{\mathbb{R}^2}(u_\tau^k-K)^{p-1}_\delta\,\d x+\chi\eps \int_{\R^2}(u_\tau^k)^2\,\d x \ge C_*(K)\int_{\mathbb{R}^2}(u_\tau^k-K)_\delta^{p+1}\,\d x.
\end{aligned}
\end{equation}
Taking into account that $u_\tau^k$ is independent of $\eps$ and $\delta$, we pass to the limit in \eqref{iglu2} for $\eps\to 0$ 
and for $\delta\to 0$ (using monotone convergence theorem) thus obtaining  $u_\tau^k\in L^{p+1}(\R^d)$ and the validity of \eqref{extraterms1}. This proves \eqref{recursivestep}.
 
 Eventually, the arbitrariness of $T$ also implies that $u_\tau^k\in L^p(\R^2)$ for every $p\in(1,+\infty)$ and every integer $k\ge 1$.
\end{proof}

\RRR

 \RRR
 
 The following is among the main result of this section. It provides a Sobolev estimate for discrete minimizers that is useful for proving  the subsequent discrete derivatives estimates for the $L^p$ norms.
 
\begin{theorem}\label{generalp}
Let  
 $\tau>0$ and let $\{u_\tau^k: k=0,1,2,\ldots\}$ a sequence given by {\rm Proposition \ref{prop:existenceMM}}.
 Then for every $k\ge 1$ and every $p\in[1,+\infty)$ there hold $$u_\tau^k\in L^p(\R^d),\qquad (u_\tau^k)^{	{\frac{p+m-1}{2}}}\in H^1(\R^d),$$ and
\begin{equation}\label{stimagirata}
\frac{4mp}{(p+m-1)^2}\|\nabla(u_\tau^k)^{\frac{p+m-1}{2}}\|^{2}_{L^{2}(\R^d)}\le 
\frac{\mathcal F_p(u_\tau^{k-1})-\mathcal F_p(u_\tau^{k})}{\tau} +\chi\|u_\tau^k\|_{L^{p+1}(\R^d)}^{p+1}.
\end{equation}
\end{theorem}

\begin{proof} 
{\textbf{Step 1.}}
 We start by proving that for $d\ge 3$ there holds $u_\tau^k\in L^p(\R^d)$ for every $p\in(1,+\infty)$ and every $k\ge 1$.  Again this is done with a recursive argument with respect to $k$: we prove this property for given $k\ge1$ assuming it is already known for all integers from $0$ to $k-1$\RRR. In order to do this, we first let $p\in(1,m]$. We notice that $(1,m]$ is nonempty because $m\ge m_c$ and $d\ge 3$. We  consider the functional $\FF_{p}+\eps\FF_1$, and $u_t:=S_t(u_\tau^k)$ its gradient flow
starting from ${u_\tau^k}$. 
Since  $u_\tau^k \in D(\FF_m)$ and $m\ge p$, we may apply  Proposition \ref{thenewprop} and Corollary \ref{corofp} to obtain that 
the maps $t\mapsto\FF_m(u_t)$ and  $t\mapsto\BB_\alpha(u_t)$ belong to $AC_{loc}([0,+\infty))$ and, for every $t>0$, there holds 
\begin{equation}\label{beta?}\begin{aligned}
&{\mathcal{E}_{m}(u_\tau^k)-\mathcal{E}_{m}(S_t(u_\tau^k))}=-\int_0^t\frac \d{\d s}\mathcal E_m(u_s)\,\d s\ge -\chi\eps\int_0^t\int_{\R^d}u_s^2\,\d x\,\d s
\\
&\;\;\;+ \frac{4pm}{(p+m-1)^2}\int_0^t\int_{\mathbb{R}^d}\left|\nabla u_s^{\frac{p+m-1}{2}}\right|^2\,\d x\,\d s
-\chi\int_0^t\int_{\mathbb{R}^d}u_s^{p+1}\,\d x\,\d s
\end{aligned}
\end{equation}
Then, applying the Sobolev inequality  \eqref{sobolev} to the gradient term in \eqref{beta?}, and setting $\lambda=\frac{d}{d-2}$ for notational ease, we deduce
\begin{equation}\label{fromsobolev}
C_{2,d}\int_{\R^d}\left|\nabla u_s^{\frac{p+m-1}{2}}\right|^2\,\d x \ge \|u_s\|_{L^{(p+m-1)\lambda}(\R^d)}^{p+m-1}.
\end{equation}
We notice that ${(p+m-1)\lambda}>p+1>2$ since $m\ge m_c$.
We also notice that setting $a=\frac{2p+d(m-1)}{2p+d(m-2)}$, $a'=\frac{a}{a-1}$, and by using $L^p-L^{p+1}-L^{(p+m-1)\lambda}$ interpolation, along with Young inequality, we get for every $\delta>0$
\begin{equation}\label{interp1}
\|\rho\|_{L^{p+1}(\R^d)}^{p+1}\le \frac{\delta^a}{a}\|\rho\|_{L^{p}(\R^d)}^{(1-\theta)(p+1)a}+\frac1{a'\delta^{a'}}\|\rho\|_{L^{(p+m-1)\lambda}(\R^d)}^{p+m-1},\qquad\theta=\frac{(p+m-1)\lambda}{(p+1)[(p+m-1)\lambda-p]}.
\end{equation}
Similarly by $L^1-L^2-L^{(p+m-1)\lambda}$ interpolation and Young inequality we get
\begin{equation}\label{interp2}
\|\rho\|_{L^{2}(\R^d)}^2\le \frac{M^{(2-2\varsigma)b'}}{b'}+\frac1b\,\|\rho\|_{L^{(p+m-1)\lambda}(\R^d)}^{p+m-1},\qquad2\varsigma=\frac{(p+m-1)\lambda}{(p+m-1)\lambda-1},
\end{equation}
where $b=\frac{(p+m-1)\lambda-1}{\lambda}$ and $b'=\frac{b}{b-1}$. We apply \eqref{interp1} and \eqref{interp2} to $\rho=u_s$, and insert them along with \eqref{fromsobolev} into \eqref{beta?} to obtain that for every $t>0$
\begin{equation}\label{dividebyt}\begin{aligned}
&\frac{\mathcal{E}_{m}(u_\tau^k)-\mathcal{E}_{m}(S_t(u_\tau^k))}{t}\ge\left(\frac{4pm}{C_{2,d}(p+m-1)^2}-\frac{\chi}{a'\delta^{a'}}-\frac{\chi\eps}{b}\right)\frac1t\int_0^t\|u_s\|_{L^{(p+m-1)\lambda}(\R^d)}^{p+m-1}\,\d s 
\\
&\qquad\qquad -\frac{\chi\delta^a}{a}\frac1t\int_0^t\|u_s\|_{L^{p}(\R^d)}^{(1-\theta)(p+1)a}\,ds-\frac{\chi\eps }{b'}\,M^{(2-2\varsigma)b'}.
\end{aligned}
\end{equation}
We may  fix $\delta$ large enough, and $\eps$ small enough, in order to have a positive coefficient in the first term in the above right hand side. Moreover, since $u_\tau^k\in L^p(\R^d)=D(\mathcal F_p)$,  we also have by Proposition \ref{thenewprop} and Corollary \ref{corofp} that $\|u_s\|_{L^{p}(\R^d)}\le \|u_\tau^k\|_{L^{p}(\R^d)}<+\infty$ for every $s>0$, thus from \eqref{dividebyt} it is clear that the map 
$t\mapsto \|u_t\|_{L^{(p+m-1)\lambda}(\R^d)}^{p+m-1}$ belongs to $L^1(0,t)$ for every $t>0$ 
and  we may also invoke the  flow interchange from Proposition \ref{prop:FI}, and we find
\begin{equation*}\begin{aligned}
&\frac{\FF_{p}(u_\tau^{k-1})+\eps\FF_1(u_\tau^{k-1})-\FF_{p}(u_\tau^k)-\eps\FF_1(u_\tau^k)}{\tau} +\frac{\chi\delta^a}{a}\|u_\tau^k\|_{L^p(\R^d)}^{(1-\theta)(p+1)a}+\frac{\chi\eps }{b'}\,M^{(2-2\varsigma)b'}
\\&\quad
\ge\limsup_{t\downarrow0}\left(\frac{4pm}{C_{2,d}(p+m-1)^2}-\frac{\chi}{a'\delta^{a'}}-\frac{\chi\eps}{b}\right)\,\frac1t\int_0^t\|u_s\|_{{L^{(p+m-1)\lambda}(\R^d)}}^{p+m-1}ds\\&\quad
 \ge\limsup_{t\downarrow0}\left(\frac{4pm}{C_{2,d}(p+m-1)^2}-\frac{\chi}{a'\delta^{a'}}-\frac{\chi\eps}{b}\right)\,\|u_t\|_{{L^{(p+m-1)\lambda}(\R^d)}}^{p+m-1},
\end{aligned}
\end{equation*}
where the last inequality is due to point {\bf{(II)}} of Proposition \ref{thenewprop}. Since $u_t$  narrowly converges to $u_\tau^k$ as $t\downarrow0$, by the narrow lower semicontinuity of the $L^{(p+m-1)\lambda}(\R^d)$ norm we conclude that $u_\tau^k\in L^{(p+m-1)\lambda}(\R^d)$. In particular, $u_\tau^{k}\in L^{p+1}(\R^d)$. Therefore, starting from $u_\tau^k\in L^{p}(\R^d)$ for some $p\in(1,m]$, we have obtained improved $p+1$ summability. By repeating the same argument with $p+1$ in place of $p$, we get $p+2$ summability. Recursively, we conclude that $u_\tau^k\in L^{p}(\R^d)$ for every $p\in(1,+\infty)$.

\textbf{ Step 2.} 
We prove the same property of the previous step in the case $d=2$, that is, $u_\tau^k \in L^p(\R^2)$ for every $p\in(1,+\infty)$ and every integer $k\ge 1$. This has already been done in Proposition \ref{sscritical} in the case $m=1$ so that now we prove this for $m>1$. 
The argument is similar to the previous step  and recursive in $k$ as usual\RRR, thus we shall skip some details.
 Let $k\ge 1$ and assume the desired property holds for every integer from $0$ to $k-1$\RRR. We start from $p\in(1,m]$, so that $u_\tau^k\in L^m(\R^2)$, and since $p\le m$, we have $u_\tau^k\in L^p(\R^2)$. If $u_t=S_t(u_\tau^k)$ is the gradient flow of $\mathcal F_p+\eps\mathcal F_1$, starting from ${u_\tau^k}$, we still have that for every $t>0$ the maps $t\mapsto\FF_m(u_t)$ and of $t\mapsto\BB_\alpha(u_t)$ are absolutely continuous on $[0,t]$ and \eqref{beta?} holds.
For the gradient term of \eqref{beta?} we use \eqref{GNS} with $p+m-1$ in place of $p$, that is,
\begin{equation}\label{startgnsm2}
\int_{\R^2}|\nabla u_s^{(p+m-1)/2}|^2\,\d x \ge \frac{1}{MG_{p+m-1}}\int_{\R^2} u_s^{p+m}\,\d x.
\end{equation} 
By $L^p-L^{p+1}-L^{p+m}$ interpolation and Young inequality we have for every $\delta>0$
\[
\|u_s\|_{L^{p+1}(\R^d)}^{p+1}\le \frac{\delta^{m'}}{m'}\|u_s\|_{L^{p}(\R^d)}^{(1-\vartheta)(p+1)m'}+\frac1{m\delta^m}\|u_s\|_{L^{p+m}(\R^d)}^{p+m},\qquad\vartheta=\frac{p+m}{m(p+1)},
\]
where $m'=\frac m{m-1}$, and in the same way by $L^1-L^2-L^{p+m}$ interpolation and Young inequality
\[
\|u_s\|_{L^{2}(\R^d)}^2\le \frac{1}{b'}M^{2-2\ell b'}+\frac1{b}\|u_s\|_{L^{p+m}(\R^d)}^{p+m},\qquad2\ell=\frac{p+m}{p+m-1}
\]
where $b=p+m-1$ and $b'=\frac b{b-1}$, which we insert into \eqref{beta?} along with \eqref{startgnsm2} to get
\begin{equation*}\begin{aligned}
&\frac{\mathcal{E}_m(u_\tau^k)-\mathcal{E}_m(S_t(u_\tau^k))}t
\ge \left(
\frac{4pm}{MG_{p+m-1}(p+m-1)^2}-\frac\chi{m\delta^m}-\frac{\eps\chi}{b}\right)\frac1t\int_0^t\int_{\mathbb{R}^2} u_s^{p+m}\,\d x\,\d s
\\&\qquad-\frac{\chi\delta^{m'}}{m'}\|u_\tau^k\|_{L^{p}(\R^d)}^{(1-\vartheta)(p+1)m'}-\frac{\chi\eps}{b'}M^{(2-2\ell)b'},
\end{aligned}
\end{equation*}
having also used $\|u_s\|_{L^{p}(\R^d)}\le \|u_\tau^k\|_{L^{p}(\R^d)}<+\infty$ for every $s>0$. By Proposition \ref{prop:FI}, the limsup as $t\downarrow0$ of the left hand side is finite  
and then the weak lower semicontinuity of $L^p$ norms shows that $u_\tau^k\in L^{p+m}(\R^2)$ thus improving the original $p$ summability of $u_\tau^k$. We can therefore repeat the argument with $p+1$ in place of $p$, and then recursively, to deduce finally that $u_\tau^k\in L^p(\R^2)$ for every $p\in(1,+\infty)$.

{\textbf{Step 3.}}
Let $p\in[1,+\infty)$ and let again $u_t=S_t(u_\tau^k)$ the gradient flow of $\mathcal F_p+\eps\mathcal F_1$, starting from $u_\tau^k$.  For every $t>0$ the map $s\mapsto \mathcal E_m(u_s)$ is in $AC([0,t])$ and \eqref{beta?} holds (invoking Proposition \ref{thenewprop} and Corollary \eqref{corofp} as in the previous steps). From \eqref{beta?}, already knowing by Corollary \ref{corofp} the finiteness and continuity up to $t=0$ of the map $t\mapsto \|u_t\|_{p+1}$, since we already know that $u_\tau^k\in L^r(\R^d)$ for every $r\in[1,+\infty)$,
we get
\begin{equation}\label{acca1}
\int_0^t\int_{\R^d}\left|\nabla u_s^{\frac{p+m-1}{2}}\right|^2\,\d x\,\d s<+\infty.\end{equation}
 We divide \eqref{beta?}   by $t$ and  apply again Proposition \ref{prop:FI}, thus obtaining
\begin{equation}\label{hopt}\begin{aligned}
&\frac{\FF_{p}(u_\tau^{k-1})+\eps\FF_1(u_\tau^{k-1})-\FF_{p}(u_\tau^k)-\eps\FF_1(u_\tau^k)}{\tau} +\chi\eps\|u_\tau^k\|_{L^{2}(\R^d)}^2
+\chi\|u_\tau^k\|_{L^{p+1}(\R^d)}^{p+1}
\\&\quad
\ge\frac{4pm}{(p+m-1)^2}\,\liminf_{t\downarrow0}\frac1t\int_0^t\|\nabla u_s^{\frac{p+m-1}{2}}\|_{L^{2}(\R^d)}^2\,\d s.
\end{aligned}
\end{equation}
By \eqref{acca1}, we have
 $\int_{\R^d}|\nabla u_t^{{(p+m-1)}/{2}}|^2\,\d x<+\infty$ for a.e. $t\in (0,1)$, and 
 we denote by $E\subset(0,1)$ the set of times where the latter property holds. We claim that
\begin{equation}\label{Eclaim}
\liminf_{t\downarrow0}\frac1t\int_0^t\|\nabla u_s^{\frac{p+m-1}{2}}\|_{L^{2}(\R^d)}^2\,\d s
\ge \liminf_{t\downarrow0\atop{t\in E}}\|\nabla u_t^{\frac{p+m-1}{2}}\|_{L^{2}(\R^d)}^2.
\end{equation}
Indeed, let $L\ge 0$ denote the liminf in the right hand side. If $L=+\infty$, then for every $R>0$ there exists $\tilde \delta>0$ such that 
for any $s\in E\cap (0,\tilde\delta)$, we have $\|\nabla u_s^{{(p+m-1)}/{2}}\|_{L^{2}(\R^d)}^2\ge R$. 
Therefore we have
\[
\liminf_{t\downarrow0}\frac1t\int_{0}^t\|\nabla u_s^{\frac{p+m-1}{2}}\|_{L^{2}(\R^d)}^2\,\d s\ge R.
\]
Since $R$ is arbitrary, then the left hand side of \eqref{Eclaim} is $+\infty$.
If $L<+\infty$,
 then for every $\tilde\eps>0$ there exists $\tilde\delta>0$ such that for any $s\in E\cap (0,\tilde\delta)$,
 we have  $\|\nabla u_s^{{(p+m-1)}/{2}}\|_{L^{2}(\R^d)}^2>L-\tilde\eps$. Thus  we have
\[
\liminf_{t\downarrow0}\frac1t\int_{0}^t\|\nabla u_s^{\frac{p+m-1}{2}}\|_{L^{2}(\R^d)}^2\,\d s\ge L-\tilde\eps.
\]
Since $\tilde\eps$ is arbitrary, then \eqref{Eclaim} holds. 
We notice that \eqref{hopt} and \eqref{Eclaim} imply $L<+\infty$. 
Let $(t_n)_n\subset E$ a strictly decreasing vanishing sequence such that $L=\lim_{n\to+\infty} \|\nabla u_{t_n}^{\frac{p+m-1}{2}}\|_{L^{2}(\R^d)}^2$.
Since $L$ is finite there exists a subsequence of $n\mapsto\nabla u_{t_n}^{\frac{p+m-1}{2}}$ weakly convergent in $L^2(\R^d;\R^d)$ to a limit point $z_*$. 
By the lower semicontinuity of the $L^2(\R^d;\R^d)$ norm we have that 
\begin{equation}\label{eccolastima}\begin{aligned}
\frac{\FF_{p}(u_\tau^{k-1})+\eps\FF_1(u_\tau^{k-1})-\FF_{p}(u_\tau^k)-\eps\FF_1(u_\tau^k)}{\tau} &+\chi\eps\|u_\tau^k\|_{L^{2}(\R^d)}^2
+\chi\|u_\tau^k\|_{L^{p+1}(\R^d)}^{p+1}
\\&\quad\ge \frac{4pm}{(p+m-1)^2}\|z_*\|_{L^{2}(\R^d)}^2.\end{aligned}
\end{equation}
We observe that $u_{t_n}$ strongly converges to $u_\tau^k$ in $L^p(\R^d)$ for every $p\in[1,+\infty)$. 
Indeed, it follows from the strong continuity in $L^1(\R^d)$ of the  gradient flow of $\mathcal F_p+\eps\mathcal F_1$, see \eqref{1infty}, 
and from the boundedness in every $L^{p}(\R^d)$ up to $t=0$ (which is again due to the fact that $u_\tau^k$ is in $L^r(\R^d)$ for every $r\in[1,+\infty)$). 
Therefore, possibly extracting a subsequence, we have that $u_{t_n}^{(p+m-1)/2}$ converge to $(u_{\tau}^k)^{(p+m-1)/2}$ a.e. in $\R^d$.
Then  the weak $L^2(\R^d;\R^d)$ limit of the sequence of the gradients is identified and $z_*=\nabla(u_\tau^k)^{(p+m-1)/2}$. 
Thanks to this identification, \eqref{stimagirata} follows by \eqref{eccolastima} by sending $\eps$ to $0$.
\EEE
 \end{proof}

 For $m=m_c$ we obtain improved estimates in the small mass regime. To this purpose, for $p\in(1,+\infty)$ we introduce the $p$-subcritical mass as
\begin{equation}\label{defMscp} M_{sc}(p):=\displaystyle\frac4{p\chi G_p}\quad \text{if }d=2,\qquad\quad
M_{sc}(p):=\displaystyle\left(\frac{4pm_c}{\chi C_{2,d}(p+m_c-1)^2}\right)^{d/2}\;\; \text{if }d\geq3.
\end{equation}
Recalling the definition of $M_{sc}$ from \eqref{defMsc}, for $d\ge 3$ it is clear that \begin{equation}\label{penonp}M_{sc}=\sup_{p>1} M_{sc}(p).\end{equation} The same is true also for $d=2$, as a consequence of Lemma \ref{GpG1}, see Remark \ref{remarchino}.

\RRR

 \begin{proposition} 
 Let $m=m_c$.
 Let $\tau>0$ and $\{u_\tau^k: k=0,1,2,\ldots\}$ be a sequence given by {\rm Proposition \ref{prop:existenceMM}}.
  If  $d=2$ and  $M<M_{sc}(p)$  for some $p\in(1,+\infty)$, then for every $k\ge 1$
\begin{equation}\label{subsub}
	\frac{\mathcal \|u_\tau^{k-1}\|_{L^{p}(\R^d)}^p-\|u_\tau^k\|_{L^{p}(\R^d)}^p}{\tau(p-1)} \ge \left(\frac4{pMG_p}-\chi\right)\,\|u_\tau^k\|_{L^{p+1}(\R^d)}^{p+1}.
\end{equation}
If $d\ge 3$ and  $M<M_{sc}(p)$ for some $p\in(1,+\infty)$, then for every $k\ge 1$
\begin{equation}\begin{aligned}\label{ssd}
&\frac{\mathcal \|u_\tau^{k-1}\|_{L^{p}(\R^d)}^p-\|u_\tau^k\|_{L^{p}(\R^d)}^p}{\tau(p-1)}\ge
 \left(\frac{4pm_c}{C_{2,d}(p+m_c-1)^2}-\chi  M^{2/d}\right)\left(\int_{\R^d}(u_\tau^k)^{\frac{(p+m_c-1)d}{d-2}}\,\d x\right)^{\frac{d-2}{d}}.
\end{aligned}
\end{equation}
\end{proposition}
\begin{proof}
 Let $d=2$. Then the result follows from \eqref{stimagirata} combined with with the Gagliardo-Nirenberg-Sobolev inequality \eqref{GNS}. 
Let $d\ge 3$. Then the  result follows by combining  the interpolation inequality
\begin{equation*}
\int_{\R^d}(u_\tau^k)^{p+1}\,\d x \le
\left(\int_{\R^d}u_\tau^k\,\d x\right)^{\frac2d}
\left(\int_{\R^d}(u_\tau^k)^{\frac{(p+1)d-2}{d-2}}\,\d x\right)^{\frac{d-2}{d}}= M^{\frac2d}\left(\int_{\R^d}(u_\tau^k)^{\frac{(p+1)d-2}{d-2}}\,\d x\right)^{\frac{d-2}{d}}
\end{equation*}
with \eqref{stimagirata} and \eqref{sobolev}, having noticed that $(p+m_c-1)d=(p+1)d-2.$
\end{proof}
\RRR

%
%
%

\RRR

\subsection{Discrete differential inequalities for $L^p$ norms}

\RRR

In view of the next statement, for every $1\le s<p<+\infty$ (if $m=m_c$ we add the restriction $s>1$) we define the following quantities
\begin{equation}\label{sigmafg}
\begin{aligned}
\sigma(p,s)&:=\frac{d(p-s)+2s+d(m-1)}{d(p-s)},\qquad r(p,s):=\frac{d(p-s)+2s+d(m-1)}{2s+d(m-2)},\\
f(p,s)&:=\frac{2p+d(m-1)}{d(p-s)},\qquad
g(p,s):=\frac{2p+2+d(m-2)}{2s+d(m-2)}.
\end{aligned}
\end{equation}
\begin{proposition}\label{psprop}
Let $\tau>0$ and let $\{u_\tau^k:k=0,1,2,\ldots\}$ be a sequence given by  {\rm Proposition \ref{prop:existenceMM}}, starting from $u_0\in\PP_2^M(\R^d)$. 
Then for every $k=1,2,3,\ldots$ and every $s,p\in[1,+\infty)$ such that $s<p$ (with the further restriction $s>1$ in the case $m=m_c$), there holds
\begin{equation}\label{basicdisc2}\frac{\|u_\tau^k\|_{L^p(\R^d)}^p-\|u_\tau^{k-1}\|_{L^p(\R^d)}^p}{\tau}\le -\bar K(p,s)\,\frac{\left(\|u_\tau^k\|_{L^p(\R^d)}^p\right)^{\sigma(p,s)}}{\left(\|u_\tau^k\|_{L^s(\R^d)}^s\right)^{f(p,s)}}+\chi \,c(p,s) \left(\|u_\tau^k\|_{L^s(\R^d)}^s\right)^{g(p,s)}
\end{equation}
where, recalling the definition of $K_{p,s}$ from \eqref{kps} and of $C_{2,d}$ from \eqref{sobolev}-\eqref{sq}, 
\begin{equation}\label{kappabar}\bar K(p,s):= \frac{2mp(p-1)}{C^*_{p,s}(p+m-1)^2},\qquad\mbox{with}\quad C^*_{p,s}:=\left\{\begin{array}{ll}K_{p,s}\quad\mbox{if $d=2$}\\ C_{2,d}\quad\mbox{if $d\ge 3$,}\end{array}\right.\end{equation}
and
\begin{equation}\label{cipiesse} 
c(p,s):=\frac{p-1}{r(p,s)}\left(\frac{\chi(p-1)}{\bar K(p,s)}\frac{(p-s+1)d}{(p-s)d+2s+d(m-1)}\right)^{r(p,s)-1}.
\end{equation}
These quantities satisfy
\begin{equation}\label{cfrattocbar}
\frac1{\bar K(p,s)}\le \frac{pmC^*_{p,s}}{2(p-1)},\qquad\frac{\chi c(p,s)}{\bar K(p,s)}\le \left(\tfrac12\,\chi C^{*}_{p,s}mp\right)^{r(p,s)}
\end{equation}
\end{proposition}

\begin{remark}\rm We stress that \eqref{basicdisc2} is valid also for $\chi=0$, providing a discrete dissipation estimate of $L^p$ norms for the porous media equation. \end{remark}
\RRR

\begin{proof} 
We preliminarily observe that \eqref{cfrattocbar} is an immediate consequence of the definitions of $c(p,s)$ and $\bar K(p,s)$, by taking into account the simple inequalities $r(p,s)\ge 1$ and $p+m-1\le pm$, along with the fact that $m\ge m_c$ implies $\frac{(p-s+1)d}{(p-s)d+2s+d(m-1)}\le 1$. 
For the proof of \eqref{basicdisc2} we treat separately the case $d\ge 3$ and $d=2$.

{\bf Case \boldmath$d\ge 3$.} 
For notational ease, we let $\lambda:=\frac{d}{d-2}$.
We start from the basic estimate that we have obtained in Theorem \ref{generalp}. Indeed, from \eqref{stimagirata} we make use of the Sobolev inequality \eqref{sobolev} as done in \eqref{fromsobolev}, and we get 
\begin{equation}\label{basicdisc}
\frac{\|u_\tau^k\|_{L^p(\R^d)}^p-\|u_\tau^{k-1}\|_{L^p(\R^d)}^p}{\tau}\le -\frac{4mp(p-1)}{C_{2,d}(p+m-1)^2}\|u_\tau^k\|^{p+m-1}_{L^{(p+m-1)\lambda}(\R^d)}+\chi(p-1)\|u_\tau^k\|_{L^{p+1}(\R^d)}^{p+1}
\end{equation}
for every $k=1,2,3.\ldots$
Notice that $(p+m-1)\lambda>p+1$ since $m\ge m_c$.
For $s\in[1,p)$ we estimate the $L^{p+1}(\R^d)$ norm by using $L^s(\R^d)-L^{p+1}(\R^d)-L^{(p+m-1)\lambda}(\R^d)$ interpolation and the Young inequality $ab\le r^{-1}\delta^ra^r+(r'\delta^{r'})^{-1}b^{r'}$,
 for $\delta>0$ and $r\in(1,+\infty)$, with $r'=\frac r{r-1}$, we have
\begin{equation}\label{sinterp}
\|u_\tau^k\|_{L^{p+1}(\R^d)}^{p+1}\le \frac{\delta^r}r\,\|u_\tau^k\|_{L^s(\R^d)}^{(1-\zeta)(p+1)r}+\frac{1}{r'\delta^{r'}}\,\|u_\tau^k\|_{{L^{(p+m-1)\lambda}(\R^d)}}^{\zeta(p+1)r'},\quad \zeta=\frac{(p+m-1)(p+1-s)\lambda}{(p+1)[p\lambda-s+(m-1)\lambda]}.
\end{equation}
We choose 
$
r=r(p,s)$,
where $r(p,s)$ is defined in \eqref{sigmafg},
implying $\zeta(p+1)r'=p+m-1$. We notice that the denominator in the definition of $r(p,s)$ is positive for every $s\in[1,p)$ if $m>m_c$  and for every $s\in(1,p)$ if $m=m_c$, and moreover  we have  $r(p,s)>1$. 
We also choose
\[
\delta=\left(\frac{\chi C_{2,d}(p+m-1)^2(p+1-s)\lambda}{2mp\,[p\lambda-s+(m-1)\lambda]}\right)^{\frac{(p+1-s)\lambda}{p\lambda-s+(m-1)\lambda}}
\]
so that $$\frac{\chi}{r'\delta^{r'}}=\frac{2mp}{C_{2,d}(p+m-1)^2}.$$
By inserting \eqref{sinterp} in \eqref{basicdisc}, with these choices of $r,\delta$ we get for $k=1,2,3.\ldots$
\begin{equation}\label{basicdisc3}\begin{aligned}
&\frac{\|u_\tau^k\|_{L^{p}(\R^d)}^p-\|u_\tau^{k-1}\|_{L^{p}(\R^d)}^p}{\tau}\le -\frac{2mp(p-1)}{C_{2,d}(p+m-1)^2}\|u_\tau^k\|^{p+m-1}_{L^{(p+m-1)\lambda}(\R^d)}\\&\;\;\;+\frac{\chi(p-1)\delta^r}r\,\|u_\tau^k\|_{L^{s}(\R^d)}^{(1-\zeta)(p+1)r}= -\bar K_{p,s} \,\|u_\tau^k\|^{p+m-1}_{L^{(p+m-1)\lambda}(\R^d)}+\chi\,c(p,s)\,\|u_\tau^k\|_{L^{s}(\R^d)}^{(1-\zeta)(p+1)r},
\end{aligned}\end{equation}
since $c(p,s)=r^{-1}(p-1)\delta^r$, being $c(p,s)$ defined by \eqref{cipiesse}.

By means of the $L^s(\R^d)-L^p(\R^d)-L^{(p+m-1)\lambda}(\R^d)$ interpolation
\begin{equation}\label{eta}
\|u_\tau^k\|_{L^{p}(\R^d)}^{\frac{p+m-1}{\eta}}\le \|u_\tau^k\|_{L^{s}(\R^d)}^{\frac{(p+m-1)\eta}{1-\eta}}\,\|u_\tau^k\|_{L^{(p+m-1)\lambda}(\R^d)}^{p+m-1},\qquad \eta=\frac{(p-s)(p+m-1)\lambda}{p[p\lambda-s+(m-1)\lambda]},
\end{equation}
to be inserted in \eqref{basicdisc3}, we obtain the result, having noticed that 
 \eqref{sigmafg}, \eqref{sinterp} and \eqref{eta} entail
\begin{equation*}\begin{aligned}&\frac{p+m-1}{p\eta}=\frac{p\lambda-s+(m-1)\lambda}{(p-s)\lambda}=\sigma(p,s),\\&\frac{(p+m-1)(1-\eta)}{s\eta}=\frac{p(\lambda-1)+(m-1)\lambda}{(p-s)\lambda}=f(p,s),\\
&\frac{(1-\zeta)(p+1)r}{s}=\frac{p(\lambda-1)+(m-1)\lambda-1}{s(\lambda-1)+(m-2)\lambda}=g(p,s).
\end{aligned}\end{equation*}

{\bf Case \boldmath$d=2$.} 
From  \eqref{stimagirata} 
we have for $k=1,2,3\ldots$
\begin{equation}\label{basicdisc2d}
\frac{\|u_\tau^k\|_{L^{p}(\R^2)}^p-\|u_\tau^{k-1}\|_{L^{p}(\R^2)}^p}{\tau}\le -\frac{4mp(p-1)}{(p+m-1)^2}\int_{\R^2}\left|\nabla(u_\tau^k)^{\frac{p+m-1}{2}}\right|^2\,\d x+\chi(p-1)\int_{\R^2}(u_\tau^k)^{p+1}\,\d x.
\end{equation}
Let $r=r(p,s)$ from \eqref{sigmafg}, thus
 $$ r':=\frac{r}{r-1}=\frac{p+m-1}{p-s+1}.$$
By applying  \eqref{nuovagns} with $s\in[1,p)$ ($s>1$ if $m=1$) and Young inequality we get for $\delta>0$
\[
\|u_\tau^k\|_{L^{p+1}(\R^2)}^{p+1}\le  \frac{\delta^r}{r}\|u_\tau^k\|_{L^{s}(\R^2)}^{sr}+\frac{K_{p,s}}{r'\delta^{r'}}\|\nabla(u_\tau^k)^{\frac{p+m-1}{2}}\|_{L^{2}(\R^2)}^2,
\]
so that from \eqref{basicdisc2d} we obtain
\[\begin{aligned}
\frac{\|u_\tau^k\|_{L^{p}(\R^2)}^p-\|u_\tau^{k-1}\|_{L^{p}(\R^2)}^p}{\tau}&\le\left(\frac{\chi(p-1) K_{p,s}}{r'\delta^{r'}}-\frac{4mp(p-1)}{(p+m-1)^2}\right)\|\nabla(u_\tau^k)^{\frac{p+m-1}{2}}\|_{L^{2}(\R^2)}^2\\&\qquad+\frac{\chi(p-1)\delta^r}{r}\|u_\tau^k\|_{L^{s}(\R^2)}^{sr}.
\end{aligned}\]
If we choose 
\[
\delta=\left(\frac{\chi K_{p,s}(p+m-1)(p-s+1)}{2mp}\right)^{\frac{p-s+1}{p+m-1}},\qquad\mbox{so that}\quad
\frac{\chi K_{p,s}}{r'\delta^{r'}}=\frac{2mp}{(p+m-1)^2},
\]
we get
\[
\frac{\|u_\tau^k\|_{L^{p}(\R^2)}^p-\|u_\tau^{k-1}\|_{L^{p}(\R^2)}^p}{\tau}\le-\frac{2mp(p-1)}{(p+m-1)^2}\|\nabla(u_\tau^k)^{\frac{p+m-1}{2}}\|_{L^{2}(\R^2)}^2+\frac{\chi(p-1)\delta^r}{r}\|u_\tau^k\|_{L^{s}(\R^2)}^{sr},
\]
and a further application of \eqref{nuovagns} entails
\[
\frac{\|u_\tau^k\|_{L^{p}(\R^2)}^p-\|u_\tau^{k-1}\|_{L^{p}(\R^2)}^p}{\tau}\le-\frac{2mp(p-1)}{K_{p,s}(p+m-1)^2}\frac{\left(\|u_\tau^k\|_{L^{p+1}(\R^2)}^{p+1}\right)^{	r'}}{\left(\|u_\tau^k\|_{L^{s}(\R^2)}^s\right)^{r'}}+\frac{\chi(p-1)\delta^r}{r}\|u_\tau^k\|_{L^{s}(\R^2)}^{sr}.
\]
By using $L^s(\R^2)-L^p(\R^2)-L^{p+1}(\R^2)$ interpolation we deduce
\begin{equation*}
\frac{\|u_\tau^k\|_{L^{p}(\R^2)}^p-\|u_\tau^{k-1}\|_{L^{p}(\R^2)}^p}{\tau}\le-\bar K(p,s)\frac{\left(\|u_\tau^k\|_{L^{p}(\R^2)}^{p}\right)^{\textstyle\frac{p+m-1}{p-s}}}{\left(\|u_\tau^k\|_{L^{s}(\R^2)}^s\right)^{\textstyle\frac{p+m-1}{p-s}}}+{\chi\,c(p,s)}\left(\|u_\tau^k\|_{L^{s}(\R^2)}^s\right)^{\textstyle\frac{p+m-1}{m+s-2}},
\end{equation*}
which is \eqref{basicdisc2}.
\end{proof}



We next give a discrete second moment estimate, based on flow interchange with the Fokker-Planck equation.
\begin{proposition}\label{momentgrowth}
Let $\tau>0$ and $\{u_\tau^k: k=0,1,2,\ldots\}$ be the sequence given by {\rm Proposition \ref{prop:existenceMM}}. 
Then for every $k\ge 1$ there holds
\begin{equation}\label{derivativeofm2}
\frac{\mom(u_\tau^k)-\mom(u_\tau^{k-1})}\tau\le 2d\int_{\R^d}(u_\tau^k)^m\,\d x. \end{equation}
Moreover, the piecewise constant interpolation $u_\tau(\cdot)$ satisfies 
\begin{equation}\label{mommom}
\mom(u_\tau(t))\le \mom(u_\tau^0)+2d\int_0^{t+\tau}\int_{\R^d}(u_\tau(s))^m\,\d x\,\d s, \qquad \forall\, t>0. 
\end{equation}

\end{proposition}
\begin{proof} Let $\eps>0$, $\mathcal V:=\frac12\mom+\eps\mathcal F_1$ and let $u_t:=S_t(u_\tau^k)$ the gradient flow of $\mathcal V$ starting from $u_\tau^k$, $k\ge 1$.  We notice that $u_\tau^k\in L^p(\R^d)$ for every $p\in[1,+\infty)$ by Theorem \ref{generalp}. Thus by Proposition \ref{thenewprop2} and Lagrange theorem, for every $t>0$ there exists $\theta(t)\in(0,t)$ such that
\[
\frac{\mathcal E_m(u_\tau^k)-\mathcal E_m(S_t(u_\tau^k))}{t}=-\frac{d}{ds}\mathcal E_m(u_s){\Bigg{|}}_{s=\theta(t)}\ge-d\int_{\R^d}u_{\theta(t)}^m\,\d x-\chi\eps\int_{\R^d} u_{\theta(t)}^2\,\d x.
\]
By invoking Proposition \ref{prop:FI} we get
\[
\frac{\frac12 \mom(u_\tau^{k-1})+\eps\mathcal F_1(u_\tau^{k-1})-\frac12\mom(u_\tau^k)-\eps\mathcal F_1(u_\tau^k)}{\tau}\ge \limsup_{t\downarrow 0} \left(-d\int_{\R^d}u_{\theta(t)}^m\,\d x-\chi\eps\int_{\R^d} u_{\theta(t)}^2\,\d x\right).
\]
 We may use the representation formula \eqref{convtrans} for $u_t$ and  Proposition \ref{initialdatum} to see that the limit in the above right hand side is equal to $$-d\int_{\R^d}(u_{\tau}^k)^m\,dx-\chi\eps\int_{\R^d} (u_{\tau}^k)^2\,\d x.$$
The arbitrariness of $\eps>0$ yields \eqref{derivativeofm2}.

Let $t>0$.
By summing \eqref{derivativeofm2} from $k
=1$ to $k=\lceil t/\tau\rceil$ we get, since the sum on the left hand side is telescopic, 
\begin{equation*}
\mom(u_\tau(t))\le \mom(u_\tau^0)+2d\int_0^{\lceil t/\tau\rceil\tau}\int_{\R^d}(u_\tau(s))^m\,\d x\,\d s
\end{equation*}
thus proving \eqref{mommom}.
\end{proof}
\RRR

\subsection{Discrete hypercontractivity estimates for finite $p$}

Here, we show $L^p$ hypercontractivity estimates at the discrete level. 
We preliminarily state some comparison principles and error estimates for the implicit-Euler discretization 
of nonlinear differential equations with gradient flow structure. 

\begin{proposition}\label{prop:DI}
Let $\phi:\R\to\R$ be a convex $C^1$ function. Let $\tau>0$, $a_0$, $b_0\in\R$ and
 $a_k$ and $b_k$ satisfying
$$ a_k-a_{k-1} \leq -\tau\phi'(a_k), \qquad b_k-b_{k-1} = -\tau\phi'(b_k), \quad \forall k\in \N.$$
If $a_0\leq b_0$, then $a_k\leq b_k$ for any $k\in\N$.

Let $b:[0,+\infty)\to\R$ the solution of the Cauchy problem
\begin{equation}\label{CauchyPb}
 b'(t)=-\phi'(b(t)), \qquad b(0)=b_0.
\end{equation}
Then $|b_k-b(k\tau)|\leq \frac{1}{\sqrt{2}}|\phi'(b_0)|\tau$.
\end{proposition}

\begin{proof}
By induction, assuming that $a_{k-1}\leq b_{k-1}$ we have that
$$ a_k+\tau\phi'(a_k) \leq a_{k-1}\leq  b_{k-1} = b_k+\tau\phi'(b_k).$$
Since the function $r\mapsto r+\tau\phi'(r)$ is strictly increasing we conclude using its inverse function.

The last inequality is the error estimate for the Euler implicit discretization scheme. 
See for instance the general expression derived by Nochetto-Savar\'e-Verdi \cite{NSV} and 
\cite[Theorem 4.0.7]{AGS}.
\end{proof}

Building on the estimates of Proposition \ref{sscritical}, Theorem \ref{generalp} we may apply the above comparison principles to deduce estimates on the $L^p$ norms of $u_\tau^k$. We start with the small mass case, recalling the definition of $M_{sc}(p)$ from \eqref{defMscp}. 

\begin{theorem}[\bf Decay of $L^p$ norms in the small mass regime]\label{lemmaLp} 

Let $\tau>0$ and $\{u_\tau^k:k=0,1,2,\ldots\}$ be a sequence given by  {\rm Proposition \ref{prop:existenceMM}}.
Let  $m=m_c$,  $p\in(1,+\infty)$ and $M<M_{sc}(p)$.
Then for every integer $k\ge 1$ there holds
\begin{equation}\label{discretesubsub}
    \|u_\tau^k\|^p_{L^p(\R^d)} \leq \min \left\{ \|u_\tau^0\|^p_{L^p(\R^d)},\,M\left({C\, k \tau}\right)^{1-p}\right\}  
    + \frac{\tau(p-1)}{\sqrt{2}}C\,M^{\frac1{1-p}}\,\|u_\tau^0\|^{\frac{p^2}{p-1}}_{L^p(\R^d)}   
\end{equation}
where 
\begin{equation}\label{defCpMd}
	C=C_{\chi,p,M,d}:=\begin{cases}\dfrac{4pm_c}{ C_{2,d}(p+m_c-1)^2M^{2/d}}-\chi &\text{if }d\geq3\vspace{0.2cm}\\
					\dfrac{4}{pMG_p}-\chi &\text{if }d=2.
			\end{cases}
\end{equation}
\end{theorem}

\begin{proof}
During the proof we use the notation $a_\tau^k:=\|u_\tau^k\|_{L^p(\mathbb{R}^d)}^p$.

Let $d=2$.
 By the  interpolation of $L^p$ norms we obtain \RRR
\begin{equation}\label{1pp+1}
\left(\|u\|_{L^p(\mathbb{R}^2)}^p\right)^{\frac{p}{p-1}}\le \|u\|_{L^1(\mathbb{R}^2)}^{\frac{1}{p-1}}\|u\|_{L^{p+1}(\mathbb{R}^2)}^{p+1}.
\end{equation}
Using \eqref{1pp+1} in \eqref{subsub} we obtain letting $C_{\chi,M,p}:=C_{\chi,M,p,2}$, where $C_{\chi,M,p,2}$ is defined in \eqref{defCpMd},
\begin{equation*}
(p-1)C_{\chi,p,M}M^{\frac1{1-p}} \,(a_\tau^k)^{\frac{p}{p-1}}\le \frac{a_\tau^{k-1}-a_\tau^k}\tau.
\end{equation*}
We apply Proposition \ref{prop:DI} with the choice of the convex function 
\[
	\phi(x):=\frac{(p-1)^2}{2p-1}\,C_{\chi,p,M}\, M^{\frac{1}{1-p}}\,x^{\frac{2p-1}{p-1}},\qquad x\ge 0.
\]
The solution of the Cauchy problem \eqref{CauchyPb} with initial datum $b(0)\ge 0$ is 
$$b(t)=\left(b(0)^{\frac{1}{1-p}}+{C_{\chi,p,M} M^{\frac1{1-p}}t}\right)^{1-p}\le\min\left\{b(0), \left({C_{\chi,p,M} M^{\frac1{1-p}}t}\right)^{1-p}\right\}$$
Choosing $b_\tau^0=a_\tau^0=\|u_\tau^0\|^p_{L^p(\R^d)}$, 
by  Proposition \ref{prop:DI} we obtain
$$a_\tau^k\le b_\tau^k\leq b(k\tau)+|b_\tau^k-b(k\tau)|\leq b(k\tau) + \frac{\tau}{\sqrt{2}}\phi'(b_\tau^0),$$
and \eqref{discretesubsub} follows using the definitions of $b$ and $a_\tau^k$.

If $d\ge 3$, again by interpolation of $L^p$ norms we have
\begin{equation*}\label{interpolazione}
	\left(\int_{\R^d} u^p\,\d x\right)^{\frac{(p+m_c-2)d+2}{d(p-1)}}\le \left(\int_{\R^d} u\,\d x\right)^{\frac{(m_c-1)d+2p}{d(p-1)}}
	\left(\int_{\R^d}u^{\frac{(p+m_c-1)d}{d-2}}\,\d x\right)^{\frac{d-2}d}
\end{equation*} and
computing the exponents $\frac{(p+m_c-2)d+2}{d(p-1)}= \frac{p}{p-1}$ and $\frac{(m_c-1)d+2p}{d(p-1)}=\frac{1}{p-1}+\frac{2}{d}$,
 from \eqref{ssd} we obtain
\begin{equation*}
	(p-1) \left(\frac{4pm_c}{C_{2,d}(p+m_c-1)^2}-\chi  M^{2/d}\right)M^{\frac{1}{1-p}-\frac2d} \,(a_\tau^k)^{\frac{p}{p-1}}\le \frac{a_\tau^{k-1}-a_\tau^k}\tau,
\end{equation*}
that is
\begin{equation*}
	(p-1) C_{\chi,p,M,d} M^{\frac{1}{1-p}} \,(a_\tau^k)^{\frac{p}{p-1}}\le \frac{a_\tau^{k-1}-a_\tau^k}\tau.
\end{equation*}
We conclude exactly as in the case $d=2$.
 \end{proof}

Analogous results in the case $d=2, m=1$ are obtained in the next lemma,  based on Proposition \ref{sscritical}, whose proof includes a simple comparison principle for ODEs that will be used again later on.

\begin{lemma}\label{generalsubcritical} Let $d=2$, $m=1$. Let $\tau>0$ and $\{u_\tau^k:k=0,1,2,\ldots\}$ be a sequence given by  {\rm Proposition \ref{prop:existenceMM}}. Let $p\in[2,+\infty)$.
 For  any  $T\ge 1$ and any integer $k\ge1$ such that $k\tau\le T$, if $K>K_{p,\tau}(T)$, where $K_{p,\tau}(\cdot)$ is defined by \eqref{kappat},   
then
\begin{equation}\label{compara}\begin{aligned}
 \int_{\R^2}(u_\tau^k-K)_+^p\,\d x&\le\min\left\{\int_{\R^2}(u_\tau^0-K)_+^p\,\d x,\,\left(\frac{C_1(K)(k\tau)}{p-1}\right)^{1-p}\right\}+\bar\varphi(K)\\
 &\qquad+\frac\tau{\sqrt2}\left(C_1(K)\Big(\int_{\R^2}(u_\tau^0-K)_+^p\,\d x\Big)^{p/(p-1)}+Y(K)\right),
\end{aligned}\end{equation}
where
\begin{equation}\label{C1K}
	C_1(K):=\frac{p-1}{2}C_*(K)M^{\frac{1}{1-p}}
\end{equation}
and $C_*(K)$ is given by \eqref{CKappa}, while
 $Y(K)$  and $\bar\varphi(K)$ are explicit quantities defined through the proof, respectively in  \eqref{YK}-\eqref{zeta} and in \eqref{barphiK}.
\end{lemma}

\begin{proof} 
Let us fix $T\ge1$, $\tau>0$ and $k\ge 1$ such that  $k\tau\le T$.

 We preliminary observe that for every $K\ge 0$ there holds
$\int_{\R^2}(u_\tau^k-K)_+\leq M$. Moreover, by $L^1(\R^2)-L^{p-1}(\R^2)-L^{p}(\R^2)$ interpolation  and Young inequality we get
\begin{equation}\begin{aligned}\label{auxp>2}
K^2\int_{\R^2}(u_\tau^k-K)_+^{p-1}\,\d x & \le K^2M^{\frac{1}{p-1}}\left(\int_{\R^2}(u_\tau^k-K)_+^{p}\,\d x\right)^{\frac{p-2}{p-1}}\\
&\le \frac{K^pM^{\frac{1}{p-1}}}{p-1}+\frac{p-2}{p-1}M^{\frac{1}{p-1}}\,K\int_{\R^2}(u_\tau^k-K)_+^p\,\d x,
\end{aligned}\end{equation}
and similarly   $L^1(\R^2)-L^{p}(\R^2)-L^{p+1}(\R^2)$ interpolation  yields 
\begin{equation}\label{auxgen}
\int_{\R^2}(u_\tau^k-K)_+^{p+1}\,\d x \geq M^{-\frac{1}{p-1}}\left(\int_{\R^2}(u_\tau^k-K)_+^p\,\d x\right)^{\frac{p}{p-1}}.
\end{equation}

Let now  $K>K_{p,\tau}(T)$, so that we can take advantage of \eqref{extraterms1}-\eqref{CKappa}.
Inserting the estimates  \eqref{auxp>2} and \eqref{auxgen} in \eqref{extraterms1} we obtain
\begin{equation}\label{record}\begin{aligned}
& \frac1{\tau} \left(\int_{{\R}^2}(u_\tau^{k-1}-K)_+^p\,\d x -\int_{{\R}^2}(u_\tau^k-K)^p_+\,\d x\right)\ge -Z(K)\int_{\R^2}(u_\tau^k-K)_+^p\,\d x
\\&\qquad
-\frac{p}{p-1}\,\chi M^{\frac{1}{ p-1}} K^{ p}
+(p-1)C_*(K)M^{-\frac{1}{p-1}}\left(\int_{{\R}^2}(u_\tau^k-K)_+^{p}\,\d x\right)^{\frac{p}{p-1}},
\end{aligned}\end{equation}
having introduced the shorthand 
\begin{equation}\label{zeta}
Z(K):=(2p-1)\chi K+ \frac{p\chi}{p-1} \,(p-2)\,M^{\frac1{p-1}}\,K.
\end{equation}
By the Young inequality 
\[
	Z(K)\int_{\R^2}(u_\tau^k-K)_+^p\,\d x\le \frac1p \frac{Z^p}{\eps^p}+\frac{p-1}{p}\left(\eps \int_{\R^2}(u_\tau^k-K)_+^p\,\d x\right)^{\frac{p}{p-1}},
\]
and choosing 
$\eps=\left(\frac{p}{2}C_*(K)M^{\frac{1}{1-p}}\right)^{\frac{p-1}{p}},$
we rewrite \eqref{record} as
\begin{equation}\label{record2}\begin{aligned}
& \frac1{\tau} \left(\int_{{\R}^2}(u_\tau^{k-1}-K)_+^p\,\d x-\int_{{\R}^2}(u_\tau^k-K)^p_+\,\d x\right)\\
&\qquad\ge \frac{(p-1)}{2}C_*(K)M^{\frac{1}{1-p}}\left(\int_{{\R}^2}(u_\tau^k-K)_+^{p}\,\d x)\right)^{\frac{p}{p-1}}-Y(K),
\end{aligned}\end{equation}
where 
\begin{equation}\label{YK}\begin{aligned}
	Y(K)&:=\frac1p Z(K)^p \left(\frac{p}{2}C_*(K)M^{\frac{1}{1-p}}\right)^{{1-p}}
	 +\frac{p}{p-1}\,\chi M^{\frac{1}{p-1}} K^{ p}
\end{aligned}\end{equation}

Introducing the notation 
\begin{equation}\label{ataukappa} a_\tau^k:=\|(u_\tau^k-K)_+\|_{L^p(\mathbb{R}^2)}^p, \qquad k=0,1,2,\ldots,\end{equation}
the inequality \eqref{record2} can be written as
\begin{equation}\label{constantstep}
	\frac{a_\tau^k-a_\tau^{k-1}}{\tau} \leq -C_1(K) (a_\tau^k)^{\frac{p}{p-1}}+Y(K),
\end{equation}
where $C_1(K)$ is defined by \eqref{C1K}.
We define the convex function $\varphi:[0,+\infty)\to \R$ by
 \begin{equation}\label{phi4}\begin{aligned}
	\varphi(x):&=\frac{p-1}{2p-1}C_1(K)\, x^{1+\frac{p}{p-1}}-Y(K)\,x.
\end{aligned}\end{equation}
With this notation \eqref{record2} rewrites as
$
	\dfrac{a_\tau^k-a_\tau^{k-1}}{\tau}\le-\varphi'(a_\tau^k).
$
We denote by
\begin{equation}\label{barphiK}
	\bar\varphi(K):=\left(\frac{Y(K)}{C_1(K)}\right)^{\frac{p-1}{p}}
\end{equation} 
the unique minimum point of $\varphi$.

We let $b:[0,+\infty)\to \mathbb R$ the unique solution to the initial value problem
\[
\left\{\begin{array}{ll}
b'(t)=-\varphi'(b(t)), \qquad t>0\vspace{0.08cm}\\
y(0)=a_\tau^0.
\end{array}\right.
\]
It is not difficult to show that
 \begin{equation}\label{finite0}
	 b(t)\le a_\tau^0+\bar\varphi(K), \qquad \forall\,t>0.
 \end{equation}
Let $b_\infty$  the unique solution to the problem
 \[
\left\{\begin{array}{ll}
b_\infty(t)'=-\varphi'(b_\infty(t)), \qquad t>0\vspace{0.08cm}\\
\lim_{t\downarrow0}b_\infty(t)=+\infty.
\end{array}\right.
\]
By comparison it holds that $b(t)\le b_{\infty}(t)$ for every $t> 0$. On the other hand,
the unique solution of the problem 
 \[
\left\{\begin{array}{ll}
\tilde b_\infty'(t)=-\varphi'(\tilde b_\infty(t))-Y(K)=-C_1(K) (\tilde b_\infty(t))^{\frac{p}{p-1}}, \qquad t>0\vspace{0.08cm}\\
\lim_{t\downarrow0}\tilde b_\infty(t)=+\infty,
\end{array}\right.
\]
is given by
\[
\tilde b_\infty(t)=\left(\frac{C_1(K)t}{p-1}\right)^{1-p},\qquad t>0.
\]
Moreover, the function $w(t):=\tilde b_\infty(t)+\bar\varphi(K)$ satisfies
$$ w'(t) \geq -\varphi'(w(t)).$$
and a comparison principle shows that $w(t)\ge b_\infty(t)$ for every $t>0$. Therefore
\begin{equation}\label{infinite0}
	b(t)\le b_\infty(t)\le w(t)=\left(\frac{C_1(K)t}{p-1}\right)^{1-p}+\bar\varphi(K).
\end{equation}
From Proposition \ref{prop:DI} we obtain
\[
a_\tau^k\le b_\tau^k\le b(k\tau)+\frac\tau{\sqrt2}|\varphi'(a_\tau^0)|,
\]
and thus \eqref{finite0} and \eqref{infinite0} yield the result.
\end{proof}

\begin{theorem}\label{theop}  Let $d=2$, $m=1$. 
Let $\{u_\tau^k:k=0,1,2,\ldots\}$ be a sequence given by  {\rm Proposition \ref{prop:existenceMM}}. Let $p\in[2,+\infty)$.
Then there exists a constant $C_{p,M,\chi}$ depending only on $p$, $M$ and $\chi$ such that 
\begin{equation}\label{DLpEst2}
\|u_\tau^k\|_{L^p({\R}^2)}^p \le C_{p,M,\chi} K_{p,\tau}(T)^{2p-2} +2^{p-1}\min\left\{\|u_\tau^0\|_{L^p(\R^2)}^p ,\,M^p\left(\frac{1}{pG_p} k\tau \right)^{1-p}\right\} + \tilde R_\tau
\end{equation}
for any $T\ge 1$ and any integer $k\ge 1$ with $k\tau\leq T$, where $K_{p,\tau}(T)$ is defined in \eqref{kappat} and \begin{equation}\label{tildeerre}
\tilde R_\tau:=2^{p-1}\,\frac\tau{\sqrt2}\left(\frac{p-1}{pG_p}\,M^{\frac{p}{1-p}}\,\Big(\int_{\R^2}(u_\tau^0)^p\,\d x\Big)^{p/(p-1)}+C_{p,M,\chi}K_{p,\tau}(T)^{2p}\right).
\end{equation}
If $u_0\in D(\mathcal E_1)$, we also have  
\begin{equation}\label{kappat4} K_{p,\tau}(T)\le K_p(T):= \exp\left(\frac{1}{4}\chi p G_p C_0 (1+\log^+(8C_3T))\right)\qquad\forall\tau\ge 0,\end{equation} 
where $C_0$ is the constant from {\rm Proposition \ref{logequi}} that does not depend on $\tau$, and in this case the remainder satisfies $\tilde R_\tau\to 0$ as $\tau\to0$. 
\end{theorem}

\begin{proof}
First of all we observe that for any $v\in L^1_+\cap L^p(\mathbb{R}^d)$ and $K\ge 0$ it holds
\begin{equation}\label{duep}
	\int_{{\R}^d} v^p \,\d x \le K^{p-1}\int_{{\R}^d} v\,\d x + 2^{p-1}\int_{{\R}^d} (v-K)_+^p\,\d x
	+2^{p-1}K^{p-1}\int_{{\R}^d} v\,\d x.
\end{equation}
Indeed, we have
\begin{equation*}\label{2^p-1}\begin{aligned}
	\int_{{\R}^d} v^p\,\d x &=\int_{\{v\le K\}}v^p\,\d x+\int_{\{v>K\}} v^p\,\d x
	= \int_{\{v\le K\}}v^p\,\d x+\int_{\{v>K\}} (v-K+K)^p\,\d x   \\
	&\le K^{p-1}\int_{{\R}^d} v\,\d x +2^{p-1}\int_{\{v>K\}}(v-K)_+^p\,\d x+2^{p-1}\int_{\{v>K\}} K^p\,\d x\\
	&\le K^{p-1}\int_{{\R}^d} v\,\d x+2^{p-1}\int_{{\R}^d} (v-K)_+^p\,\d x+2^{p-1}K^{p-1}\int_{{\R}^d} v\,\d x.
\end{aligned}\end{equation*}

Let $T\ge1$, $p\in[2,+\infty)$, $K>K_{p,\tau}(T)$.
For any integer $k\ge 1$ such that $k\tau\le T$ we use \eqref{duep} so that
$$\|u_\tau^k\|_{L^p({\R}^2)}^p\le (2^{p-1}+1)MK^{p-1}+2^{p-1}\|(u_\tau^k-K)_+\|_{L^p({\R}^2)}^p,$$ 
 and by \eqref{compara} we obtain
\begin{equation}\label{intermest}
\begin{aligned}
 	\int_{{\R}^2}(u_\tau^k)^p\,\d x &\le (2^{p-1}+1)MK^{p-1} +2^{p-1}\min\left\{\|u_\tau^0\|_{L^p(\R^2)}^p,\,\left(\frac{C_1(K)k\tau}{p-1}\right)^{1-p}\right\}\\
	 &\qquad+2^{p-1}\bar\varphi(K)+2^{p-1}\,\frac\tau{\sqrt2}\left(C_1(K)\Big(\int_{\R^2}(u_\tau^0)^p\,\d x\Big)^{p/(p-1)}+Y(K)\right).
\end{aligned}
\end{equation}
Here, $C_1(K)$, $Y(K)$, $\bar\varphi(K)$ are defined as in Lemma \ref{generalsubcritical}.
Choosing $K=K_{p,\tau}(T)^2$ (observe that $K_{p,\tau}(T)^2>K_{p,\tau}(T)$ because $K_{p,\tau}(T)>1$) we have from \eqref{kappat}, \eqref{CKappa} and \eqref{C1K} that 
\begin{equation}\label{comoda}C_*(K_{p,\tau}(T)^2)=\frac{2}{pG_p}M^{-1}\qquad\mbox{and}\qquad\frac{C_1(K_{p,\tau}(T)^2)}{p-1}=\frac{1}{pG_p}M^{\frac{p}{1-p}},\end{equation}
therefore we obtain from \eqref{intermest}
\[\begin{aligned}
	\int_{\R^2}(u_\tau^k)^p\,\d x &\le (2^{p-1}+1)M K_{p,\tau}(T)^{2p-2} +2^{p-1}\min\left\{\|u^0\|_{L^p(\R^2)}^p ,\,M^p\left(\frac{1}{pG_p} k\tau \right)^{1-p}\right\} \\
	 &\qquad
	+ 2^{p-1}\bar\varphi(K_{p,\tau}(T)^2)+\tilde R_\tau,
\end{aligned}\]
Taking \eqref{comoda} into account, from the definition of $\bar\varphi(\cdot)$ in \eqref{barphiK} and of $Y(\cdot)$ in \eqref{YK}-\eqref{zeta} it is not difficult to check that 
$\bar\varphi(K_{p,\tau}(T)^2)\le D\,K_{p,\tau}(T)^{2p-2}$ and $Y(K_{p,\tau}(T)^2)\le D\,K_{p,\tau}(T)^{2p}$
for a suitable constant $D$ depending only on $M$, $p$, $\chi$.
Therefore we obtain \eqref{DLpEst2}, for a suitable constant $C_{p,M,\chi}$. 

Under the additional assumption $u_0\in D(\mathcal E_1)$, from Proposition \ref{logequi} we may invoke \eqref{notau} and deduce \eqref{kappat4}. Then, the remainder $\tilde R_\tau$ vanishes as $\tau\to0$ due to the at most logarithmic explosion of $\|u_\tau^0\|_ {L^p(\R^2)}$ which is in turn a consequence of the definition of $u_\tau^0$ from \eqref{utau0}, see \eqref{heatcontraction}.
\end{proof}

Of course, if $u_0\in L^p(\R^2)$, $p\ge 2$, the above statement provides a uniform bound for $\|u_\tau^k\|_{L^p(\R^2)}$ in terms of $\|u_0\|_{L^p(\R^2)}$. In order to have a similar property for $1<p<2$, we include the following
\begin{proposition}\label{very} 
Let $d=2$, $m=1$. Let $u_0 \in D(\mathcal E_1)\cap L^p(\R^2)$ for some $p\in (1,2)$ and
 $\{u_\tau^k:k=0,1,2,\ldots\}$ be a sequence given by  {\rm Proposition \ref{prop:existenceMM}}.
 For every $T\ge 1$, there exist $\overline C^*_T$ and $\tau^*$, depending on $T,\chi, p,M, \alpha, \mom(u_0), \mathcal E_1(u_0)$, such that 
\[
\|u_\tau^k\|_{L^p(\R^2)}\le \overline C_T^* \|u_0\|_{L^p(\R^2)}
\]
for every  $\tau<\tau^*$ and every integer $k\ge1$ such that $k\tau\le T$.\RRR
\end{proposition}
\begin{proof}
 For every integer $k\ge1$ such that $k\tau\le T$, Proposition \ref{logequi} gives, as $u_0\in D(\mathcal E_1)$,
\[
\int_{\R^2}u_\tau^k|\log u_\tau^k|\,\d x\le C_\tau^0(1+\log^+(8C_3T))\le C_0(1+\log^+(8C_3T))=:H_T,
\]
where $C_0$ is defined by \eqref{notau}.
  Then, by Proposition \ref{trunc} we get for every $K_*>1$
  \begin{equation*}\label{Kstar}
  \int_{\R^2}(u_\tau^k)^{p+1}\,\d x\le 2K_*\int_{\R^2}(u_\tau^k)^p\,\d x+\frac{A_p H_T}{\log K_*}\int_{\R^2}|\nabla (u_\tau^k)^{p/2}|^2\,\d x,
  \end{equation*}
  and by choosing $K_*=K_*(T)$, where $K_*(T)$ is defined by the relation $2\log K_*(T)=p\chi A_p H_T$,  Theorem \ref{generalp} yields
  $
  0\le  (\mathcal F_p(u_\tau^{k-1})-\mathcal F_p(u_\tau^k))+2\chi \,\tau K_*(T) \|u_\tau^k\|_{L^p(\R^2)}^p,
  $
that is
\begin{equation}\label{itering}
\left(1-2\chi(p-1)\,\tau \,K_*(T)\right)\|u_\tau^k\|_{L^p(\R^2)}^p\le \|u_\tau^{k-1}\|_{L^p(\R^2)}^p.
\end{equation}
Choosing now $\tau$ small enough, such that $2\chi(p-1)\,\tau \,K_*(T)<1/2$, by the elementary inquality $-\log(1-x)\le 2x$ holding for every $x\in[0,1/2]$ we deduce
\[
(1-2\chi(p-1)\,\tau \,K_*(T))^{-k}=e^{-k\log(1-2\chi(p-1)\,\tau \,K_*(T))}\le e^{4k \chi(p-1)\,\tau \,K_*(T)}\le e^{4\chi(p-1)\,T \,K_*(T)},
\]
so that by iterating \eqref{itering} we obtain
\[
\|u_\tau^k\|_{L^p(\R^2)}^p\le (1-2\chi(p-1)\,\tau \,K_*(T))^{-k} \|u_\tau^{0}\|_{L^p(\R^2)}^p\le e^{4\chi(p-1)\,T \,K_*(T)}\,\|u_0\|_{L^p(\R^2)}^p.
\]
where we have also used \eqref{jens} in the last inequality.
\end{proof}
\RRR

\begin{theorem}\label{newlpdge3} Let $\tau>0$ and let $\{u_\tau^k:k=0,1,2,\ldots\}$ be a sequence given by  {\rm Proposition \ref{prop:existenceMM}}.

{\rm\textbf{(I)}}
Let $m>m_c$. Let $p>1$. 
Then, 
for every $k=1,2,3,\ldots$ we have the discrete $L^p(\R^d)$ estimate
\begin{equation}\label{longdiscrete}\begin{aligned}
\|u_\tau^k\|_{L^p(\R^d)}^p&\le \min\left\{\|u_\tau^0\|_{L^p(\R^d)}^p,\,\left(\bar K(p,1)\,M^{-f(p,1)}(\sigma(p,1)-1)\,k\tau\right)^{-\frac{1}{\sigma(p,1)-1}}\right\}\\&+\left(\frac{\chi\, c(p,1)\, M^{g(p,1)}}{\bar K(p,1)\,M^{-f(p,1)}}\right)^{\frac1{\sigma(p,1)}}+\frac{\tau}{\sqrt2}\left|\frac{\bar K(p,1)}{M^{f(p,1)}}\left(\|u_\tau^0\|_{L^p(\R^d)}^p\right)^{\sigma(p,1)}-\chi\, c(p,1)\, M^{g(p,1)}\right|,
\end{aligned}\end{equation}
where $\bar K$, $c$, $\sigma$, $g$, $f$ are defined in \eqref{kappabar}, \eqref{cipiesse},  \eqref{sigmafg}.

{\rm\textbf{(II)}}
Let $d\ge 3$, $m=m_c$. Let $p>m_c$.
For every $k=1,2,3,\ldots$ we have the discrete $L^p$ estimate
\begin{equation*}\begin{aligned}
\|u_\tau^k\|_{L^p(\R^d)}^p&\le \min\left\{\|u_\tau^0\|_{L^p(\R^d)}^p,\,\left(\bar K(p,m_c)\,(C^0_\tau)^{-f(p,m_c)}(\sigma(p,m_c)-1)\,k\tau\right)^{-\frac{1}{\sigma(p,m_c)-1}}\right\}\\&\quad+\left(\frac{\chi\, c(p,m_c)\, (C^0_\tau)^{g(p,m_c)}}{\bar K(p,m_c)\,(C^0_\tau)^{-f(p,m_c)}}\right)^{\frac1{\sigma(p,m_c)}}
\\&\quad+\frac{\tau}{\sqrt2}\left|\frac{\bar K(p,m_c)}{(C^0_\tau)^{f(p,m_c)}}\left(\|u_\tau^0\|_{L^p(\R^d)}^p\right)^{\sigma(p,m_c)}-\chi\, c(p,m_c)\, (C^0_\tau)^{g(p,m_c)}\right|
\end{aligned}\end{equation*}
where $C_\tau^0$ is the constant from {\rm Proposition \ref{equi}}, also depending on $\mathcal E_m(u_\tau^0)$. 
If $u_0\in D(\mathcal E_m)$,  we can replace $C_\tau^0$ with the constant $C_0$ independent of $\tau$, still from {\rm Proposition \ref{equi}}. 
\end{theorem}
\begin{proof}
Let $m>m_c$ and $p>1$.
If we choose $s=1$ in \eqref{basicdisc2}, which is possible since $m>m_c$, we obtain as a particular instance that for any $k=1,2,3\ldots$
\begin{equation}\label{discrder}
\frac{\|u_\tau^k\|_{L^p(\R^d)}^p-\|u_\tau^{k-1}\|_{L^p(\R^d)}^p}{\tau}\le -\bar K(p,1)\,M^{-f(p,1)}{\left(\|u_\tau^k\|_{L^p(\R^d)}^p\right)^{\sigma(p,1)}}+\chi\, c(p,1)\, M^{g(p,1)},
\end{equation}
showing that the sequence $(\|u_\tau^k\|_{L^p(\R^d)}^p)_{k\in\mathbb N}$
satisfies the discrete version of the differential inequality $$y'\le -\bar K(p,1)\,M^{-f(p,1)}y^{\sigma(p,1)}+\chi\, c(p,1)\, M^{g(p,1)},$$ 
the right hand side being the opposite derivative of a convex function of $y$, since $\sigma(p,1)>1.$  Notice indeed that the right hand side has exactly the same form of the function $\varphi$ in \eqref{phi4} so that we may apply the same argument. Indeed, by the same comparison principle used in the proof of Lemma \ref{generalsubcritical}, and taking advantage of Proposition \ref{prop:DI} as done therein,  we deduce the desired discrete $L^p$ estimate. This proves {\bf{(I)}}.\RRR 

Else if $m=m_c$ and $p>m_c$, we choose $s=m_c$ in \eqref{basicdisc2} and we apply the same argument, also taking Proposition \eqref{equi} into account, to get 
\begin{equation}\label{discrder2}\begin{aligned}
\frac{\|u_\tau^k\|_{L^p(\R^d)}^p-\|u_\tau^{k-1}\|_{L^p(\R^d)}^p}{\tau}&\le -\bar K(p,m_c)\,(C_\tau^0)^{-f(p,m_c)}{\left(\|u_\tau^k\|_{L^p(\R^d)}^p\right)^{\sigma(p,m_c)}}\\&\qquad+\chi\, c(p,m_c)\, (C_\tau^0)^{g(p,m_c)},
\end{aligned}\end{equation}\EEE
and the result follows in the same way. In the case $u_0\in D(\mathcal E_m)$ the latter inequality also holds with $C_0$ in place of $C_\tau^0$, since $C_\tau^0\le C_0$ as seen in Proposition \ref{equi}.
\end{proof}

\RRR

\RRR

\RRR

\subsection{Euler-Lagrange equation for discrete minimizers}

We close this section
by proving a necessary condition on the minimizers, making variations along the flow of a smooth vector field. 
This is a standard argument in order to obtain a discrete formulation of the equation.
\begin{proposition}\label{prop:EL}
Let  $\tau>0$ and $\{u_\tau^k: k=0,1,2,\ldots\}$ be a sequence given by {\rm Proposition \ref{prop:existenceMM}}.
 Then, for any integer $k\geq 1$, $(u_\tau^k)^m\in W^{1,1}(\R^d)$ and
 \begin{equation}\label{eq:ELdeq}
 	\int_{\R^d} (\nabla (u_\tau^k)^m - \chi u_\tau^k \nabla v_\tau^k)\cdot \eta \, \d x = \frac{1}{\tau} \int_{\R^d} (T_{u_\tau^k}^{u_\tau^{k-1}} - I)\cdot \eta u_\tau^k \, \d x ,
 \end{equation}
for any $\eta\in C^1_c(\R^d;\R^d)$, where $ v_\tau^k = B_\alpha *  u_\tau^k$ and $T_{u_\tau^k}^{u_\tau^{k-1}}$ is the optimal transport map from ${u_\tau^k}$ to $u_\tau^{k-1}$, and $I$ denotes the identity map on $\R^d$.
Moreover
\begin{equation}\label{eq:DDeq}
 	\int_{\R^d} \left| \frac{\nabla (u_\tau^k)^m - \chi u_\tau^k \nabla v_\tau^k}{u_\tau^k}\right|^2 u_\tau^k \, \d x = 
	\frac{1}{\tau^2}W^2(u_\tau^k, u_\tau^{k-1}). 
 \end{equation}
\end{proposition}
\begin{proof}
Let us fix $k\geq 1$ and $\eta\in C^1_c(\R^d;\R^d)$.  We first notice that $(u_\tau^k)^m \in W^{1,1}(\R^d)$ directly follows from Theorem \ref{generalp}\RRR, which also yields $(u_\tau^k)\in L^p(\R^d)$ for every $p\in[1,+\infty)$. For $\eps>0$ sufficiently small, the map $T_\eps:\R^d\to \R^d$ defined by $T_\eps(x)=x+\eps\eta(x)$ is a $C^1$ diffeomorphism.
For notational convenience we set $u:=u_\tau^k$ and $\bar u:=u_\tau^{k-1}$.
Let $u_\eps=(T_\eps)_\#u$.
Since $u$ is a minimizer for \eqref{minmov1}, then
$$ \EE_m(u)+\frac{1}{2\tau}W^2(u,\bar u) \leq \EE_m(u_\eps)+\frac{1}{2\tau}W^2(u_\eps,\bar u),$$
consequently
\begin{equation}\label{eulerineq} \frac{\EE_m(u)-\EE_m(u_\eps)}{\eps} \leq \frac{1}{2\tau} \frac{W^2(u_\eps,\bar u)-W^2(u,\bar u)}{\eps}.\end{equation}
By a standard computation,  see for instance \cite{BCKKLL, HMVV}, \RRR we have
\begin{equation}\label{twolim}\begin{aligned} \lim_{\eps \to 0} \frac{W^2(u_\eps,\bar u)-W^2(u,\bar u)}{\eps} &= 2  \int_{\R^d} (T_{u}^{\bar u} - I)\cdot \eta u \, \d x,\\ \lim_{\eps \to 0} \frac{\FF_m(u)-\FF_m(u_\eps)}{\eps} &=  \int_{\R^d} {u}^m \div\eta \, \d x.\end{aligned}\end{equation}
Denoting by $v=B_\alpha *  u$, we have to prove the following
\begin{equation}\label{Blim} \lim_{\eps \to 0} \frac{\BB_\alpha(u_\eps)-\BB_\alpha(u)}{\eps} =  \int_{\R^d} \nabla v \cdot \eta u \, \d x.\end{equation}
Let $v_\eps=B_\alpha * u_\eps$.
Since $u\in L^p(\R^d)$, for $p\in[1,+\infty)$, and $u_\eps(x)=u(T^{-1}_\eps(x))\det(\nabla(T^{-1}_\eps(x)))$, \RRR
by dominated convergence, it holds that $\lim_{\eps\to 0}\|u_\eps-u\|_{L^p(\R^d)}=0$.
 By the Calder\'on-Zygmund inequality, see for instance \cite[Section 7.2.4]{GM}, we have $D^2v, \ D^2v_\eps \in L^p(\R^d)$ for $1<p<+\infty$.  \RRR By Proposition \ref{loggrowth}, $v_\eps, \ v \in W^{2,p}_{loc}(\R^d)$, and since $p>d$ is admissible, by Morrey's Theorem $v_\eps, v \in C^1(\R^d)$. Furthermore, still invoking Proposition \ref{loggrowth}, we have $\nabla v_\eps,  \nabla v \in L^\infty(\R^d) $ and
 $\lim_{\eps\to 0}\|\nabla (v_\eps-v)\|_{L^{\infty}(\R^d)}=0$. \RRR
Writing
$$ \frac{\BB_\alpha(u_\eps)-\BB_\alpha(u)}{\eps} =  \frac{1}{2\eps} \int_{\R^d} (v_\eps u_\eps - v u ) \, \d x=
	\frac{1}{2\eps}\int_{\R^d} v_\eps (u_\eps -u)  \, \d x +  \frac{1}{2\eps}\int_{\R^d} (v_\eps-v) u  \, \d x,$$
we compute separately the two terms.
The first one can be written as
\begin{equation}\label{vepsueps} \frac{1}{2\eps}\int_{\R^d} v_\eps (u_\eps -u)  \, \d x =  \frac{1}{2\eps}\int_{\R^d} (v_\eps-v) (u_\eps-u)\,\d x + \frac{1}{2\eps}\int_{\R^d} v (u_\eps-u)\,\d x.\end{equation}
The limit of the last term of \eqref{vepsueps} can be computed using the definition of $u_\eps$  as
\begin{equation}\label{firstlim}\lim_{\eps\to 0}
\frac{1}{\eps}\int_{\R^d} v (u_\eps-u)\,\d x=\lim_{\eps\to0}  \frac{1}{\eps}\int_{\R^d} (v\circ T_\eps -v) u\,\d x = \int_{\R^d} \nabla v\cdot\eta\, u\,\d x.
\end{equation}
The last limit holds by dominated convergence, because $v\in C^1(\R^d)$ so that 
$$\frac{1}{\eps}(v(x+\eps\eta(x))- v(x))\to \nabla v(x)\cdot\eta(x)$$ pointwise as $\eps\to0$ and  there holds  $\frac{1}{\eps}\|v\circ T_\eps -v\|_{L^\infty(\R^d)}\leq\|\nabla v\|_{L^\infty(\R^n)}\|\eta\|_{L^\infty(\R^d)}$. 
Similarly, for the first term in the right hand side of \eqref{vepsueps} we have
\[
\frac{1}{2\eps}\int_{\R^d} (v_\eps-v) (u_\eps-u)\,\d x\le \frac1{2\eps}\|\nabla(v_\eps-v)\|_{L^{\infty}(\R^d)}\|\eps\eta\|_{L^\infty(\R^d)}\int_{\R^d} u\,\d x
\]
so that
 $$\lim_{\eps\to 0}\frac{1}{2\eps}\int_{\R^d} (v_\eps-v) (u_\eps-u)\,\d x =0.$$
\RRR
On the other hand, using the definition of convolution, the symmetry $B_\alpha(y-x)=B_\alpha(x-y)$ and Fubini's Theorem, we obtain
\begin{equation*}
\begin{aligned}
	\frac{1}{\eps}\int_{\R^d} (v_\eps-v) u \,\d x & =  \frac{1}{\eps}\int_{\R^d} \int_{\R^d} B_\alpha(x-y)(u_\eps(y)-u(y))\,\d y \, u(x) \,\d x \\
	&=\frac{1}{\eps}\int_{\R^d} \int_{\R^d} (B_\alpha(x-T_\eps(y))- B_\alpha(x-y))u(y)\,\d y \, u(x) \,\d x \\ 
	&=\frac{1}{\eps}\int_{\R^d} \int_{\R^d} (B_\alpha(T_\eps(y)-x)- B_\alpha(y-x))u(x)\,\d x \, u(y) \,\d y \\
	&=\frac{1}{\eps}\int_{\R^d}  (B_\alpha*u(T_\eps(y))- B_\alpha*u(y))\, u(y) \,\d y =  \frac{1}{\eps}\int_{\R^d} (v\circ T_\eps -v) u\,\d x,
\end{aligned}
\end{equation*}
and the limit is thus the same as \eqref{firstlim}. 
Summing up, this proves \eqref{Blim}. By \eqref{eulerineq}, \eqref{twolim}, \eqref{Blim}  and using the definition of the functional $\EE_m$ we obtain
$$   \int_{\R^d} {u}^m \div\eta \, \d x -\chi   \int_{\R^d} \nabla v \cdot \eta u \, \d x \leq    \frac{1}{\tau}\int_{\R^d} (T_{u}^{\bar u} - I)\cdot \eta u \, \d x.$$
Since the same result holds for $-\eta$ instead of $\eta$, 
we obtain
\begin{equation*}\label{eq:wfi}
  \int_{\R^d} {u}^m \div\eta \, \d x -\chi   \int_{\R^d} \nabla v \cdot \eta u \, \d x =  \frac{1}{\tau}\int_{\R^d} (T_{u}^{\bar u} - I)\cdot \eta u \, \d x, 
 \end{equation*}
 and integrating by parts  \eqref{eq:ELdeq} is proved.
Finally, from  \eqref{eq:ELdeq} we obtain that 
$$\frac{\nabla (u_\tau^k)^m - \chi u_\tau^k \nabla v_\tau^k}{u_\tau^k} =\frac{1}{\tau} (T_{u_\tau^k}^{u_\tau^{k-1}} - I), \text{ in }L^2_{u_\tau^k}(\R^d;\R^d)$$
and the equality of the norms yields \eqref{eq:DDeq}.
\end{proof}

\section{Convergence of the JKO scheme}\label{convergingsection}
 This section is devoted to the main convergence properties of the discrete approximation scheme. 
 We shall provide existence of gradient flows of functional $\mathcal E_m$ \EEE also for general initial data in $\mathscr P_2^M(\R^d)$ in the case $m>m_c$. Then we prove that such gradient flows are distributional solutions to \eqref{PE} and that they satisfies an energy inequality. As a first step, we start by proving that discrete  solutions have stronger convergence properties than \eqref{0conv}  under the assumptions $u_0\in D(\mathcal E_m)$, $m=m_c$. \RRR

\subsection{Improved convergence and solution of the Cauchy problem}

\begin{theorem} \label{prop:IC} Let  $m=m_c$ and \EEE
$u_0\in D(\EE_m)$. For every $\tau>0$, let $\{u_\tau^k:k=0,1,2,\ldots\}$ be a sequence from {\rm Proposition \ref{prop:existenceMM}} and
let $u_\tau$ be the piecewise constant curve defined in \eqref{floor}.
 Suppose that the vanishing sequence $(\tau_n)_n$ and the curve  $u\in AC^2_{loc}([0,+\infty);\Space)$ are given by {\rm Theorem \ref{th:convergence1}}, and  thus satisfying \eqref{0conv}. Define 
 $
v_\tau(t):=B_{\alpha}*u_\tau(t)$ and $v(t):=B_{\alpha}*u(t)$, 
 for every $t\ge 0$. Then, 
 for every $t>0$ we have
 \begin{equation}\label{secondasudue}\begin{aligned}
&\mbox{$u_{\tau_n}(t)\rightharpoonup u(t)$ weakly in $L^p(\R^d)$ for  any $p\in[1,+\infty)$, as $n\to\infty$,}\\
&\mbox{
$\nabla v_\tau (t), \nabla v (t) \in W^{1,p}(\mathbb{R}^d)$ for any $p\in (d/(d-1),+\infty)$ and any $\tau>0$,}\\
& \mbox{$\nabla v_{\tau_{n}}(t)\to \nabla v(t)$ strongly in  $L^p_{loc}({\R}^d)$ for any  $p\in [1,+\infty)$,  as $n\to\infty$.}
\end{aligned}\end{equation}
 Moreover, as $n\to \infty$,
\begin{equation}\label{tx}
\begin{aligned}
&\mbox{$u_{\tau_n}\to u$ strongly in $L^p_{loc}((0,+\infty);L^q(\R^d))$ for any $1\le p,q<+\infty$,}\\
&\mbox{$\nabla v_{\tau_n}\to \nabla v$ strongly in $L^p_{loc}((0,+\infty);L^q(\R^d))$ for any $p\in[1,+\infty)$, $q\in(\frac{d}{d-1},+\infty]$}
\end{aligned}
\end{equation}
and
\begin{equation}\label{forpde}
\begin{aligned}
&\mbox{$\nabla u_{\tau_n}^m\to\nabla u^m$ weakly in $L^2_{loc}((0,+\infty);L^2(\R^d))$,}\\
&\mbox{$u_{\tau_n}\nabla v_{\tau_n}\to u \nabla v $ strongly in $L^2_{loc}((0,+\infty);L^2(\R^d))$.}
\end{aligned}
\end{equation}

\EEE
\end{theorem}

\begin{proof} 

The proof makes us of the estimates on discrete solutions that have been obtained in the previous section. Indeed, 
from Theorem \ref{theop} in the case $d=2$ and with the notation therein,  we obtain for every $p\in[ 2,+\infty)$ that
\begin{equation}\label{0interval}
\|u_\tau(t)\|_{L^p(\R^d)}^p\le (2pG_p)^{p-1}M^p\,t^{1-p}+C_{p,M,\chi}\, K_p(t+1)+R_\tau(t+1)\quad\mbox{$\forall t>0,\,\forall \tau\in(0,1)$, }
\end{equation}
where 
\[
 R_\tau(t):=2^{p-1}\,\frac\tau{\sqrt2}\left(\frac{p-1}{pG_p}\,M^{\frac{p}{1-p}}\,\Big(\int_{\R^2}(u_\tau^0)^p\,\d x\Big)^{p/(p-1)}+C_{p,M,\chi}K_{p}(t)^{2p}\right),
\]
and $\sup _{t\in[0,T]}R_\tau(t)\to 0$ if $\tau\to0$ even if $u_0\notin L^p(\R^d)$ due to the at most logarithmic blow up of $\|u_\tau^0\|_{L^p(\R^d)}$ as $\tau\to0$, see \eqref{heatcontraction}. Similarly, if $d\ge 3$ then point {\bf{(II)}} of Theorem \ref{newlpdge3} shows that there exist $\hat C^*, \hat C_*$, only depending on $p,\chi,M, m,d,\alpha, T, \mathcal E_m(u_0), \mom(u_0)$ such that for every $p\in(m_c,+\infty)$ there holds
\begin{equation}\label{intervaloft}
\|u_\tau(t)\|_{L^{p}(\R^d)}^{p}\le \hat C^* t^{-\textstyle{\frac{d(p-m_c)}{(d+2)m_c-d}}}+\hat C_*+ \hat R_\tau\quad\mbox{$\forall t>0,\,\forall \tau\in(0,1)$},
\end{equation}
where the remainder $\hat R_\tau$ can be read from Theorem \ref{newlpdge3}, it  depends on $p, M,m,d,\chi, \alpha$, $\|u_\tau^0\|_{L^p(\R^d)}$,  and again $\hat R_\tau\to0$ as $\tau\to0$.

\smallskip

{\bf {Step 1.}} \EEE In this step we prove the pointwise estimates \eqref{secondasudue}
Let us fix $t\in (0,+\infty)$. By \eqref{0interval}-\eqref{intervaloft}, we have 
\begin{equation}\label{eq:blp}
	\sup_{n}\| u_{\tau_n}(t) \|_{L^p(\R^d)}<+\infty\qquad\mbox{for every $p\in[1,+\infty)$}
\end{equation}
and since the limit $u(t)$ is already determined by \eqref{0conv}, we obtain indeed  that $ u_{\tau_n}(t)\to u(t)$ weakly in $L^p(\R^d)$ for every $p\in[1,+\infty)$ (for the weak $L^1(\R^d)$ convergence we are also taking into account the tightness of the sequence).
By \eqref{eq:blp} and by the Calder\'on-Zygmund inequality, see for instance \cite[Section 7.2.4]{GM}, we have 
\begin{equation}\label{calderon}\sup_{n}\| D^2 v_{\tau_n}(t) \|_{L^{p}(\R^d)}<+\infty\qquad\mbox{ for every $p\in(1,+\infty).$}\end{equation}
 From \eqref{calderon} and Proposition \ref{loggrowth} we deduce that $\nabla v_{\tau_n}$ is bounded and continuous (by Morrey's theorem), 
and 
 vanishing at infinity as directly obtained  from \eqref{vanishingatinfty}, in particular we deduce that $\nabla v_{\tau_n}$ belongs to the homogeneous Sobolev space $\dot W^{1,q}(\R^d)$ for every $1<q<d$ 
and then by the Sobolev embedding $\dot W^{1,q}(\R^d)\hookrightarrow L^{q^*}(\R^d)$, where $q^*=qd/(d-q)$,  we get\RRR 
\begin{equation*}\label{twosup}
 \sup_n\|\nabla v_{\tau_n}(t)\|_{W^{1,p}(\R^d)}<+\infty\qquad\mbox{ for every $p\in(d/(d-1),+\infty)$}.
\end{equation*}
 Thus,
 by Rellich-Kondrakov compactness theorem and the previous identification of the limit, we obtain that
$ \nabla v_{\tau_n}(t)\to \nabla v(t)$ strongly in $L^p_{loc}(\R^d)$ for and every $p\in [1,+\infty)$, thus proving \eqref{secondasudue}. 
%

\smallskip

{\bf {Step 2.}} In this step we prove the spacetime convergence result \eqref{tx}.
This is done by invoking Theorem \ref{generalp}.   We fix $T\ge 1$ and $t_0\in (0,T)$. We assume that $2\tau<t_0$ and that $\tau<1$. 
By taking the sum in \eqref{stimagirata} for $k$ from $\lceil t_0/\tau\rceil$ to $\lceil T/\tau\rceil$, letting for simplicity $c_{m,p}:=\tfrac{4pm}{(p+m-1)^2}$, and recalling the definition of $u_\tau(\cdot)$ from \eqref{floor}, we have
\begin{equation*}
c_{m,p}\int_{t_0}^T\|\nabla(u_\tau(t))^{\frac{p+m-1}{2}}\|_{L^2(\R^d)}^2\,\d t\le \mathcal F_{p}(u_\tau(t_0-\tau))-\mathcal F_p(u_\tau(T))+\chi\int_{t_0-\tau}^{\lceil\frac T\tau\rceil\,\tau}\|u_\tau(t)\|_{L^{p+1}(\R^d)}^{p+1}\,\d t
\end{equation*}
for every $p\ge 1$. 
Assuming $p>1$, from the latter we obtain by $L^1-L^p-L^{p+1}$ interpolation
\begin{equation}\label{equigradient00}\begin{aligned}
&c_{m,p}\int_{t_0}^T\|\nabla(u_\tau(t))^{\frac{p+m-1}{2}}\|_{L^2(\R^d)}^2\,\d t\\&\qquad\qquad\le  \frac{M^{1/p}}{p-1}\|u_\tau(t_0-\tau)\|_{L^{p+1}(\R^d)}^{(p^2-1)/p}+\chi\int_{t_0/2}^{T+1}\|u_\tau(t)\|_{L^{p+1}(\R^d)}^{p+1}\,\d t,
\end{aligned}
\end{equation}
 By choosing $p=3-m_c$ in \eqref{equigradient00}, and by taking into account \eqref{0interval}-\eqref{intervaloft},
 we obtain
\[\int_{t_0}^T\|\nabla u_\tau(t)\|_{L^2(\R^d)}^2\,\d t\le \hat Q(t_0,T)\]
for a suitable $\hat Q(t_0,T)$, also depending on $\chi,M, m,d,\alpha, T, \mathcal E_m(u_0), \mom(u_0)$, but not on $\tau$. 
By adding the estimate of the $L^2(\R^d)$ norm obtained from \eqref{0interval}-\eqref{intervaloft} with $p=2$, we deduce
\begin{equation}\label{boundL2H1}\int_{t_0}^T\|u_\tau(t)\|_{H^1(\R^d)}^2\,\d t\le \hat Q(t_0,T)\end{equation}
for a suitably updated $\hat Q(t_0,T)$, i.e., the family  $\{u_\tau(t):t\in[t_0,T], 0<\tau<t_0/2\}$ is bounded in $L^2((t_0,T);H^1(\R^d))$. It is also bounded in $L^2((t_0,T);H^1(\R^d)\cap\Space)$, since uniform boundedness of second moments has already been observed through the proof of Theorem \ref{th:convergence1}.
It is easy to check that $H^1(\R^d)\cap \mathscr P_2^{M}(\R^d)$ compactly embeds into $L^2(\R^d)$, i.e., that $\mathcal G:L^2(\R^d)\to[0,+\infty]$ defined by $$\mathcal G(\rho)=\|\rho(t,\cdot)\|^2_{H^1(\R^d)}+\int_{\R^d}(1+|x|^2)|\rho(t,x)|\,\d x$$ has compact sublevels in $L^2(\R^d)$. Indeed, we have the Sobolev embedding $\|\rho\|_{L^q(\R^d)}\le S_q\|\rho\|_{H^1(\R^d)}$  for every $2\le q< 2^*$, where $2^*=\frac{2d}{d-2}$ if $d\ge 3$ and $2^*=+\infty$ if $d=2$; since the embedding of $H^1$ in $L^2$ is compact on balls $B_R$, by choosing $q\in(2,2^*)$ and making use of interpolation with $\vartheta_q=\frac{q-2}{2q-2}$, if $G_K:=\{\rho\in L^2(\R^d);\mathcal G(\rho)\le K\}$ we obtain the following tail estimate that provides the desired compact embedding of $H^1(\R^d)\cap \mathscr P_2^{M}(\R^d)$ into $L^2(\R^d)$
\[
\lim_{R\to+\infty}\sup_{\rho\in G_K}\|\rho\|_{L^2(B_R^c)}\le \lim_{R\to+\infty}\sup_{\rho\in G_K}\|\rho\|_{L^1(B_R^c)}^{\vartheta_q}\|\rho\|_{L^q(B_R^c)}^{1-\vartheta_q}\le \lim_{R\to 0} \frac{S_q^{1-\vartheta_q} K^{(1+\vartheta_q)/2}}{R^{2\vartheta_q}}=0.
\]  

If we go back to considering the vanishing sequence $(\tau_n)_n$ from Theorem \ref{th:convergence1},
we notice from \eqref{boundL2H1} that $(u_{{\tau}_n})_n$ is bounded in $L^2((t_0,T);L^2(\R^d))$. Since we already know from \eqref{secondasudue} that $u_{\tau_n}(t)$ weakly converge to $u(t)$ in $L^2(\R^d)$ for every $t>0$, we conclude that $u\in L^2((t_0,T);L^2(\R^d))$ and that $u_{\tau_n}$ converge to $u$ weakly in $L^2((t_0,T);L^2(\R^d))$. But \eqref{0interval}-\eqref{intervaloft} also yield a bound in $L^p((t_0,T);L^2(\R^d))$ for every $p=(1,+\infty)$ so that $u\in L^p((t_0,T);L^2(\R^d))$ for every $p=(1,+\infty)$. 
 Furthermore, we can take advantage of \eqref{equicontinuityofWasserstein} to get
\begin{equation}\label{integralequicontinuity}
\lim_{h\downarrow0}\limsup_{n\to+\infty}\int_{t_0}^{T-h}W(u_{\tau_n}(t+h),u_{\tau_n}(t))\,\d t=0.
\end{equation}
The latter equicontinuity property, together with the uniform bound for $\mathcal G(u_{\tau_n})$, allows to apply the compactness result from \cite[Theorem 2]{Rossi}, since $\mathcal G$ is a normal coercive integrand as defined therein, and deduce that  $(u_{\tau_n})_n$  converges to $u$ also in $\mathcal M((t_0,T);L^2(\R^d))$, i.e., 
$$\lim_{n\to+\infty}\left|\{t\in(t_0,T):\|u_{\tau_n}(t)-u(t)\|_{L^2(\R^d)}>\eps\}\right|=0\qquad \mbox{for every $\eps>0$}.$$
 By taking $q$ as before and $\vartheta_q=\frac{q-2}{2q-2}$, by the same Sobolev embedding we deduce
\[\begin{aligned}
\int_{t_0}^T\|u_{\tau_n}(t)\|_{L^2(\R^d)}^{2/(1-\vartheta_q)}\,\d t&\le\int_{t_0}^{T}\|u_{\tau_n}(t)\|_{L^1(\R^d)}^{2\vartheta_q/(1-\vartheta_q)}\|u_{\tau_n}(t)\|_{L^q(\R^d)}^{2}\,\d t\\&\le M^{2\vartheta_q/(1-\vartheta_q)}S_q^2\int_{t_0}^T\|u_{\tau_n}(t)\|_{H^1(\R^d)}^2\,\d t
\end{aligned}\]
where the right hand side is uniformly bounded with respect to $n$, thanks to \eqref{boundL2H1}. Since $\vartheta_q\in (0,1)$ the latter shows that $\|u_{\tau_n}\|_{L^2(\R^d)}^2$ is uniformly integrable on $(t_0,T)$. Since $u\in L^{p}((t_0,T);L^2(\R^d))$ for every $p\in (1,+\infty)$, we deduce that $\|u_{\tau_n}-u\|_{L^2(\R^d)}^2$ is uniformly integrable on $(t_0,T)$ as well. Therefore $u_{\tau_n}$ converge to $u$ also strongly in $L^2((t_0,T);L^2(\R^d))$ by Vitali's theorem. Finally,  the uniform $L^p((t_0,T);L^q(\R^d))$ bound (for every $1\le p,q<+\infty$) for the sequence $(u_{\tau_n})_n$ that follows from \eqref{0interval}-\eqref{intervaloft} and the uniform bound on second moments allow to conclude that $u_{\tau_n}\to u$ strongly in $L^p((t_0,T);L^q(\R^d))$, for every  $1\le p,q<+\infty$. The arbitrariness of $t_0,T$ allows to conclude that the strong convergence of $u_{\tau_n}$ to $u$ in \eqref{tx} holds true. By Sobolev inequality \eqref{sobolev} and the Calder\'on-Zygmund inequality as in the previous step we obtain that for a suitable constant $Z_p$
\[
\|\nabla v_{\tau_n}(t)-\nabla v(t)\|_{L^{p^*}(\R^d)}\le Z_p \|u_{\tau_n}(t)-u(t)\|_{L^p(\R^d)}\quad \mbox{for every $t>0$ and every $1<p<d$},
\]
where $p^*=pd/(d-p)$, and by fixing instead $p\in(d,+\infty)$ an application of point {\bf {(i)}} of Proposition \ref{loggrowth} yields
\[
\|\nabla v_{\tau_n}(t)-\nabla v(t)\|_{L^{\infty}(\R^d)}\le c_1 \|u_{\tau_n}(t)-u(t)\|_{L^p(\R^d)}+c_2 \|u_{\tau_n}(t)-u(t)\|_{L^p(\R^d)}\quad \mbox{for every $t>0$}.
\]
Integrating from $t_0$ to $T$ and using the strong convergence of $u_{\tau_n}$ to $u$ we conclude therefore the validity of \eqref{tx}.

\smallskip

{\bf{Step 3.}} We prove \eqref{forpde}.
By the strong convergence from \eqref{tx}  and the canonical identification of $L^2((t_0,T);L^2(\R^d))$ with $L^2((t_0,T)\times\R^d)$, \RRR $u_{\tau_n}$ converge to $u$ pointwise a.e. in $(0,+\infty)\times \R^d$ along a suitable subsequence.
Therefore the weak $L^2_{loc}((0,+\infty);L^2(\R^d))$ limit of $u_{\tau_n}^m$ is $u^m$, by identification of weak and pointwise limit. Since the sequence $(\nabla u_{\tau_n}^m)_n$ is also bounded in $L^2((t_0,T);L^2(\R^d))$, thanks to \eqref{equigradient00} with $p=m+1$ and to \eqref{0interval}-\eqref{intervaloft}, we conclude that $\nabla u_{\tau_n}^m\to\nabla u^m$ weakly in $L^2_{loc}((0,+\infty);L^2(\R^d))$.

Concerning the product $u_{\tau_n}\nabla v_{\tau_n}$, we observe that if we fix $p\in (d,+\infty)$ and apply point {\bf {(i)}} of Proposition \ref{loggrowth}, we have by suitably updating the constants $c_1,c_2$ therein
\[
\int_{t_0}^T\|\nabla(v_{\tau_n}(t)-v(t))\|_{L^\infty(\R^d)}^6\,\d t\le \int_{t_0}^T\left(c_1\|u_{\tau_n}(t)-u(t)\|_{L^p(\R^d)}^6+c_2\|u_{\tau_n}(t)-u(t)\|_{L^1(\R^d)}^6\right)\d t.
\]
The strong convergence of $u_{\tau_n}$ to $u$ in $L^{6}((t_0,T);L^p(\R^d)\cap L^1(\R^d))$ allows to conclude, thanks to the arbitrariness of $0<t_0<T$, that $\nabla v_{\tau_n}\to \nabla v$ strongly in $L^6_{loc}((0,+\infty);L^{\infty}(\R^d))$. This fact and the strong $L^3_{loc}((0,+\infty);L^2(\R^d))$ convergence of $u_{\tau_n}$  to $u$ allow to conclude, by making use of the the H\"older inequality
\[
\int_{t_0}^T\|U(t)V(t)\|_{L^2(\R^d)}^2\,\d t\le \left(\int_{t_0}^T\|U(t)\|_{L^2(\R^d)}^3\,\d t\right)^{2/3}\left(\int_{t_0}^T\|V(t)\|_{L^\infty(\R^d)}^6\,\d t\right)^{1/3},
\]
 that $u_{\tau_n}\nabla v_{\tau_n}\to u \nabla v $ strongly in $L^2_{loc}((0,+\infty);L^2(\R^d))$ thus proving \eqref{forpde}.
%
%
%
%
%
\EEE
\end{proof}

 The following is the most important result of this section.
Using the available $L^p$ estimates, we prove convergence of the JKO scheme  also in the case of a general initial datum $u_0\in \Space$, in the case $m>m_c$,  providing a gradient flow of functional $\mathcal E_m$ starting from $u_0$.
The key point is the proof of the $W$-continuity of any limit curve of the JKO scheme up to $t=0$. \EEE

%
%


\begin{theorem}\label{esistenzagenerale}
 Let $m>m_c$ and $u_0\in \mathscr P_2^M(\R^d)$. For every $\tau>0$, let  $\{u_\tau^k:k=0,1,2,\ldots\}$ be a sequence from {\rm Proposition \ref{prop:existenceMM}} and let $u_\tau$ be the corresponding discrete solution  defined in \eqref{floor}.  \smallskip
  
  {\bf{(I)}} 
For every vanishing sequence of positive numbers $(\tau_n)_n$,  
 there exist a (not relabeled) subsequence and a curve
$u\in AC^2_{loc}((0,+\infty);\Space)$ such that \RRR
\begin{equation}\label{narrowconv00}
u_{\tau_{n}}(t) \to u(t)\quad \mbox{ narrowly as $n\to+\infty$, \quad for any $t\in(0,+\infty)$}.
\end{equation}

{\bf {(II)}}
Suppose that the sequence $(\tau_n)_n$ and the curve $u$ are given by {\bf{(I)}} above, thus satisfying \eqref{narrowconv00}. Then  \eqref{secondasudue} holds for every $t>0$, and \eqref{tx}-\eqref{forpde} hold as well.
Moreover, there exists a constant $\tilde K\ge 0$ only depending on $\chi,m,d,\alpha,M,$ $\mom(u_0)$ such that 
\begin{equation}\label{LAST}
W(u(t),u_0)\le \tilde K\,t^{\textstyle{\frac{1}{2+d(m-1)}}}, \qquad \forall\, t\in(0,1),
\end{equation}
so that defining $u(0)=u_0$, the curve $u$ is continuous in $(\mathscr P_2(\R^d),W)$ up to $t=0$,  and thus it is a gradient flow of functional $\mathcal E_m$, starting from $u_0$, according to {\rm Definition \ref{GFdefinition}}.\RRR
\end{theorem}
\begin{proof}
We  invoke the first part of Theorem \ref{newlpdge3}, where  $p>1$ and  $\sigma(p,1)-1=\frac{2+d(m-1)}{d(p-1)},$ \EEE and get for every $t>0$ and $\tau\in(0,1)$
\begin{equation}\label{rightpowerbis}
\|u_\tau(t)\|_p^p\le \tilde C^* t^{-\textstyle{\frac{d(p-1)}{2+d(m-1)}}}+\tilde C_{*}+ R_\tau, 
\end{equation}
where the constants $\tilde C_*,\tilde C^*$ can be read from Theorem \ref{newlpdge3}, in particular they depend only on $p, M,m,d,\chi$, and the remainder $R_\tau$ is also explicit  and can be deduced from Theorem \ref{newlpdge3}, it depends $p, M,m,d,\chi$, $\|u_\tau^0\|_{L^p(\R^d)}$,  and $R_\tau\to0$ as $\tau\to0$.

\smallskip

{\bf{Step 1}.} In this step we prove {\bf{(I)}}.
 By combining \eqref{rightpowerbis} with $p=m$ and \eqref{mommom} we see that for every $t>0$
\begin{equation}\label{newmomestim}
\mom(u_\tau(t))\le \mom(u_\tau^0)+2d(\tilde C_*+R_\tau)(t+\tau)+2d\tilde C^*\int_0^{t+\tau}s^{-\textstyle{\frac{d(m-1)}{2+d(m-1)}}}\,\d s.
\end{equation}
This shows that for fixed $T>0$, the family $\{u_\tau(t):\; t\in[0,T], \tau\in(0,1)\}$ is bounded in $\mathscr P_2^M(\R^d)$, i.e., there exists
 $Q(T)>0$, only depending on $\chi,M,d,m,T$, such that 
 \begin{equation}\label{boundedP2}\mom(u_\tau(t))\le Q(T)\qquad\mbox{for every $t\in[0,T]$ and every $\tau\in(0,1)$}.
 \end{equation} 
 
Let us denote
\[\xi_\tau^k:= \frac{\nabla (u_\tau^k)^m - \chi u_\tau^k \nabla v_\tau^k}{u_\tau^k},\qquad  k=1,2,3,\ldots \RRR \]
with the convention that $\xi_\tau^k=0$ in the set where $u_\tau^k=0$. 
Let $\xi_\tau:(0,+\infty)\times \R^d\to \R^d$ the corresponding piecewise constant (w.r.t. $t$) vector field, i.e. 
$\xi_\tau(t):=\xi_\tau^{\lceil t/\tau\rceil}.$
 Let us fix $t_0>0$ and $T>t_0$. \EEE 
By using \eqref{eq:DDeq}, we obtain for every $\tau<(t_0/2)\wedge1$,
\begin{equation}\label{sqrtutau}\begin{aligned}
&\frac12\int_{t_0}^T\|\xi_\tau(t)\sqrt{u_\tau(t)}\|_{L^2(\R^d)}^2\,\d t=\frac1{2\tau^2}\int_{t_0}^T W^2(u_\tau(t), u_\tau(t-\tau))\,\d t\\
&\qquad\le \frac1{2\tau^2}\sum_{k=\lceil t_0/\tau\rceil}^{\lceil T/\tau\rceil}\int_{(k-1)\tau}^{k\tau} W^2(u_\tau^k,u_\tau^{k-1})\,\d t\le \sum _{k=\lceil t_0/\tau\rceil}^{\lceil T/\tau\rceil} \frac{W^2(u_\tau^k,u_\tau^{k-1})}{2\tau}\\
&\qquad\le \mathcal E_m(u_\tau(t_0-\tau))-\mathcal E_m(u_\tau(T))\le \mathcal E_m(u_\tau(t_0/2))-\mathcal E_m(u_\tau(T)),
\end{aligned}\end{equation}
where we used the basic estimate of minimizing movements, see \eqref{basicjko}. 
By  Proposition \ref{F1boundprop}, in particular from \eqref{Fmbound12} and \eqref{Fmbound2},
there holds
\begin{equation}\label{dopo5.17}
-\mathcal E_m(u_\tau(T))\le C_2+C_3\mom(u_\tau(T)),
\end{equation}
where $C_2,C_3$ are the constants therein (only depending on $\chi, m,d,M,\alpha$\EEE). 
On the other hand, we notice that $\mathcal B_\alpha\ge 0$ if either $\alpha>0$ or $d\ge 3$, else if $d=2$, $\alpha=0$ we have the straightforward estimate
\eqref{-mom}.
Thus we get from \eqref{sqrtutau}
\begin{equation}\label{sqrtutau2}
\frac12\int_{t_0}^T\|\xi_\tau(t)\sqrt{u_\tau(t)}\|_{L^2(\R^d)}^2\,\d t
\le \FF_m(u_\tau(\tfrac {t_0}2))+\frac{\chi M}\pi\mom(u_\tau(\tfrac{t_0}2))+C_2+C_3\mom(u_\tau(T)).
\end{equation}
 By inserting \eqref{rightpowerbis} and \eqref{newmomestim}
into \eqref{sqrtutau2} we obtain
\begin{equation}\label{boundwitht0}
\frac12\int_{t_0}^T\|\xi_\tau(t)\sqrt{u_\tau(t)}\|_{L^2(\R^d)}^2\,\d t\le \tilde Q(t_0,T)
\end{equation}
for suitable $\tilde Q(t_0,T)>0$ only depending on $\chi,m,\alpha,M,d,t_0,T, \mom(u_0)$. In particular, if $(\tau_n)_n$ is any vanishing sequence of positive numbers, there exists a not relabeled subsequence  and a function $g\in L^2(t_0,T)$ such that $\|\xi_{\tau_n}(\cdot)\sqrt{u_{\tau_n}(\cdot)}\|_{L^2(\R^d)}$ weakly converge to $g$ in $L^2(t_0,T)$ as $n\to+\infty$.
By triangle inequality and \eqref{eq:DDeq} we have, for any $t_1,t_2$ such that $t_0\le t_1<t_2\le T$,
\[\begin{aligned}
W(u_\tau(t_1),u_\tau(t_2))&\le \sum_{k=\lceil t_1/\tau\rceil+1}^{\lceil t_2/\tau\rceil} W(u_\tau^k,u_\tau^{k-1})=\sum_{k=\lceil t_1/\tau\rceil+1}^{\lceil t_2/\tau\rceil}\int_{(k-1)\tau}^{k\tau}\frac{W(u_\tau(t),u_\tau(t-\tau))}{\tau}\d t\\
&\le\int_{t_1}^{t_2+\tau}\|\xi_\tau(t)\sqrt{u_\tau(t)}\|_{L^2(\R^d)}\,\d t.
\end{aligned}
\]
Thus we obtain the equicontinuity estimate
\begin{equation}\label{ecutau}
\limsup_{n\to+\infty}W(u_{\tau_n}(t_1),u_{\tau_n}(t_2))\le \int_{t_1}^{t_2}g(t)\,\d t.
\end{equation}
 By \eqref{ecutau}  and the uniform moment estimate \eqref{boundedP2}, we can apply the generalized Ascoli-Arzel\`a theorem \cite[Proposition 3.3.1]{AGS} to the sequence
$(u_{\tau_n})_n$, where $u_{\tau_n}:[t_0,T]\to \mathscr P_2^M(\R^d)$ for every $n$, 
and conclude that there exists $u\in AC^2([t_0,T];\mathscr P_2^M(\R^d))$ such that (along a not relabeled subsequence) $u_{\tau_n}(t)$ narrowly converges to $u(t)$ for every $t\in[t_0,T]$. 
Since $t_0$ and $T$ are arbitrary, by a standard diagonal argument we obtain a not relabeled subsequence  and $u\in AC^2_{loc}((0,+\infty);\mathscr P_2^M(\R^d))$ such that \eqref{narrowconv00} holds.
Of course, if $u_0\in D(\mathcal E_m)$ then we can take $t_0=0$ in \eqref{sqrtutau} and subsequent formulae, and the above argument allows to obtain a limit curve that is absolutely continuous with values in $\Space$ up to $t=0$, as already stated in Theorem \ref{th:convergence1}.

\EEE

\smallskip
 {\bf Step 2.} 
 Let us now prove \eqref{secondasudue}-\eqref{tx}-\eqref{forpde}.
Let $0<t_0<T$. Let $\tau<(t_0/2)\wedge 1$ and $p>1$. 
 Theorem \ref{generalp} can be invoked as done in the proof of Theorem \ref{prop:IC} to get \eqref{equigradient00}. We  apply \eqref{equigradient00} for $p=m+1$, which together with \eqref{rightpowerbis} entails
\begin{equation}\label{Miro}\int_{t_0}^T\| u_\tau(t)^m\|_{H^1(\R^d)}^2\,dt\le \overline Q(t_0,T)\end{equation}
for suitable $\overline Q(t_0,T)$ also depending on $\chi,M,d,m$, but not on $\tau$.
Since $m>1$, a theorem by Mironescu \cite{M} shows that for every $\rho\ge 0$, if $\rho^m\in H^1(\R^d)$ then $\rho$ belongs to the homogeneous Sobolev space  $\dot W^{1/m,2m}(\R^d)$ along with the estimate
$
\|\rho\|^m_{\dot W^{1/m,2m}(\R^d)}\le c(d,m)\|\nabla \rho^m\|_{L^2(\R^d)},
$
where the left hand side is the Gagliardo seminorm. Since $\|\rho\|^m_{L^{2m}(\R^d)}=\|\rho^m\|_{L^2(\R^d)}$ we deduce that 
$\|\rho\|^m_{W^{1/m,2m}(\R^d)}\le \bar c(d,m)\| \rho^m\|_{H^1(\R^d)}$ for every nonnegative $\rho\in H^1(\R^d)$. 
 Thanks to this fact and to Jensen inequality, from \eqref{Miro} we deduce that
\[
\int_{t_0}^T \|u_\tau(t)\|^2_{W^{1/m,2m}(\R^d)}\,dt\le \overline Q(t_0,T),
\]
for a suitably updated $\overline Q(t_0,T)$. By standard Sobolev embedding theorems, $W^{1/m,2m}(\R^d)$ compactly embeds into $L^2_{loc}(\R^d)$, and then $W^{1/m,2m}(\R^d)\cap \Space$ compactly embeds into $L^2(\R^d)$, the tail control being similar to the one in the proof of Theorem \ref{prop:IC}. 

Therefore, if the vanishing sequence $(\tau_n)_n$ and the curve $u\in AC_{loc}^2((0,+\infty);\Space)$ are constructed as in the previous step, thus satisfying \eqref{narrowconv00}, we can repeat Step 1 of the proof of Theorem \ref{prop:IC} to deduce the validity of \eqref{secondasudue} for every $t>0$, and we can repeat Step 2 of the proof of Theorem \ref{prop:IC}, taking advantage of the results of \cite{Rossi}, to conclude that 
$u_{\tau_n}$ converge to $u$ strongly in $L^p_{loc}((0,+\infty);L^q(\R^d))$ for every $1\le p,q<+\infty$. The convergence properties \eqref{forpde} are also obtained as in the proof of Theorem \ref{prop:IC}.

%

\smallskip

{\bf Step 3.}
In order to  prove \eqref{LAST} we observe that, 
by \eqref{rightpowerbis} with $p=m$, by \eqref{newmomestim} and \eqref{-mom}, we have for every $t>0$ and $\tau\in(0,1)$, letting for simplicity $\theta:=\frac{d(m-1)}{2+d(m-1)}<1$,
\begin{equation}\label{zerolimit1}\begin{aligned}
&\mathcal E_m(u_\tau(t))\le \frac1{m-1}\|u_\tau(t)\|^m_{L^m(\R^d)}+\chi M \mom(u_\tau(t))\\
&\le \frac{\tilde C^*t^{-\theta}+\tilde C_*+R_\tau}{m-1}+\chi M\mom(u_\tau^0)+2d\chi M(\tilde C_*+R_\tau)(t+\tau)+\frac{2d\chi M\tilde C^*}{1-\theta}(t+\tau)^{1-\theta}
\end{aligned}\end{equation}
and similarly from \eqref{dopo5.17} and \eqref{newmomestim} we have
\begin{equation}\label{zerolimit2}
-\mathcal E_m(u_\tau(t))\le C_2+C_3\left[\mom(u_\tau^0)+2d(\tilde C_*+R_\tau)(t+\tau)+2d\tilde C^*(1-\theta)^{-1}(t+\tau)^{1-\theta}\right].
\end{equation}
We apply  \eqref{zerolimit1} for $t=s$ and \eqref{zerolimit2} for $t=2s$: by assuming $\tau\le s\le 1$ we directly deduce  
the existence of constants $\tilde K_1$, $\tilde K_2$, $\tilde K_3$, only depending on $\chi,m,d,\alpha,M$ such that
\begin{equation}\label{doubling}
\mathcal E_m(u_\tau(s))-\mathcal E_m(u_\tau(2s))\le (\tilde K_1+\tilde K_2\mom(u_\tau^0)+\tilde K_3R_\tau)\,s^{-\theta}.
\end{equation}
From the basic minimizing movements scheme estimate \eqref{basicjko} and from \eqref{doubling} we have, still assuming $\tau\le s\le 1$,
\begin{equation}\label{s2s}\begin{aligned}
W^2(u_\tau(2s),u_\tau(s))&\le 2(s+\tau)\left(\mathcal E_m(u_\tau(s))-\mathcal E_m(u_\tau(2s))\right)\\&\le 4 (\tilde K_1+\tilde K_2\mom(u_\tau^0)+\tilde K_3R_\tau)\,s^{1-\theta}.\end{aligned}
\end{equation}

We fix $t\in(0,1)$. We next recursively apply \eqref{s2s}. We assume $\tau<t$ (we will send $\tau$ to zero with fixed $t$), and we let $t_j:=2^{-j}t$ for every integer $j\ge0$. Furthermore, we let $j(\tau):=\max\{j\in\{0,1,2,\ldots\}:2^{-j}t\ge \tau\}$, so that 
\begin{equation}\label{tjtau}
\tau\le t_{j(\tau)}<2\tau.
\end{equation}
Taking into account that $t=t_0\ge t_j\ge \tau$ for every $j=0,1,\ldots j(\tau)$, and that $t_j=2t_{j+1}$, we can apply \eqref{s2s} for $s=t_{j+1}$ and deduce
\begin{equation}\label{firlim}\begin{aligned}
W(u_\tau(t),u_\tau(t_{j(\tau)}))&\le \sum_{j=0}^{j(\tau)-1}W(u_\tau(t_j), u_\tau(t_{j+1}))\\&\le 2\sqrt{ \tilde K_1+\tilde K_2\mom(u_\tau^0)+\tilde K_3R_\tau}\;\sum_{j=0}^{j(\tau)-1}t_{j+1}^{\frac{1-\theta}2}\\&\le 2\Theta_0\sqrt{\tilde K_1+\tilde K_2\mom(u_\tau^0)+\tilde K_3R_\tau}\;t^{\frac{1-\theta}{2}},
\end{aligned}\end{equation}
where $\Theta_0:=\sum_{j=0}^{+\infty}(2^{(1-\theta)/2})^{-j-1}<+\infty$ since $\theta\in(0,1)$.
On the other hand, if we apply \eqref{zerolimit2} with $t=t_{j(\tau)}$, taking \eqref{tjtau} into account we find
\[
-\mathcal E_m(u_\tau(t_{j(\tau)}))\le C_2+C_3\left[\mom(u_\tau^0)+6d(\tilde C_*+R_\tau)\tau+2d\tilde C^*(1-\theta)^{-1}(3\tau)^{1-\theta}\right],
\]
while, as before, \eqref{-mom} yields 
\[\mathcal E_m(u_\tau^0)\le \frac1{m-1}\|u_\tau^0\|^m_{L^m(\R^d)}+\chi M \mom(u_\tau^0),\]
 and since $\|u_\tau^0\|_{L^m(\R^d)}$ grows at most logarithmically as $\tau\to0$, and $\mom(u_\tau^0)\to\mom(u_0)$,    we conclude that 
$$\lim_{\tau\to0}\tau(\mathcal E_m(u_\tau^0)-\mathcal E_m(u_\tau(t_{j(\tau)})))= 0.$$ 
Therefore the basic minimizing movements estimate 
\[W^2(u_\tau(t_{j(\tau)}),u_\tau^0)\le2(t_{j(\tau)}+\tau)(\mathcal E_m(u_\tau^0)-\mathcal E_m(u_\tau(t_{j(\tau)})))\]
and \eqref{tjtau} show that
\begin{equation}\label{seclim}
\lim_{\tau\to0}W(u_\tau(t_{j(\tau)}),u_\tau^0)=0.
\end{equation}
Finally, since $W(u_\tau^0,u^0)\to 0$ as $\tau\to 0$, and since
\[
W(u_{\tau}(t), u_0)\le W(u_\tau(t),u_\tau(t_{j(\tau)}))+W(u_\tau(t_{j(\tau)}),u_\tau^0)+W(u_\tau^0,u_0),
\]
by passing to the limit along a vanishing sequence $(\tau_n)_n$ such that \eqref{narrowconv00} holds, by the narrow lower semicontinuity of the distance $W$, by 
\eqref{firlim} and \eqref{seclim}, we deduce 
\[
W(u(t),u_0)\le 2\Theta_0\sqrt{\tilde K_1+\tilde K_2\mom(u^0)}\;t^{\frac{1-\theta}{2}},
\]
as desired.
\EEE
\end{proof}

Building on the previous results we can easily check that any gradient flow of functional $\mathcal E_m$  is a distributional solution  to the PDE in \eqref{PE}.

\begin{theorem}\label{weaktheorem}
Suppose that $ u\in AC_{loc}^2((0,+\infty);\mathscr P_2^M(\R^d))\EEE$ is a curve obtained  by either {\rm Theorem \ref{esistenzagenerale} or Theorem \ref{th:convergence1}}. Then, it satisfies $$\partial_tu=\Delta u^m-\chi\mathrm{div}(u\nabla B_\alpha\ast u)\quad\mbox{in }\mathcal D'((0,+\infty)\times\R^d).$$ 
\end{theorem}
\begin{proof}
Let us fix $\zeta\in C^\infty_c((0,+\infty))$ and $\varphi\in C^\infty_c(\R^d)$.
From \eqref{eq:ELdeq}, taking into account the definition of the piecewise constant curve $u_\tau$ and that of $\xi_\tau$ from the proof of Theorem \ref{esistenzagenerale},  we have
\begin{equation}\label{integratedeuler}
\frac1\tau\int_0^{+\infty}\int_{\R^d}\zeta(t)(T_\tau(t)-I)\cdot\nabla \varphi\,u_\tau(t)\,\d x\,\d t=\int_0^{+\infty}\int_{\R^d}\zeta(t)\xi_\tau(t)\cdot\nabla\varphi\,u_\tau(t)\,\d x\,\d t,
\end{equation}
where $T_\tau(t)$ is the optimal transport map from $u_\tau(t)$ to $u_\tau(t-\tau)$, hence $T_\tau(t)=T_{u_\tau^k}^{u_\tau^{k-1}}$ if $t\in((k-1)\tau,k\tau]$.
We pass to the limit as $\tau\to0$ along a sequence $(\tau_n)_n$ provided by  either {\rm Theorem \ref{esistenzagenerale} or Theorem \ref{th:convergence1}}. Along such a vanishing sequence, we have \eqref{narrowconv00} and  we have the space-time distributional convergence of $\xi_\tau u_\tau$ to $\nabla u^m-\chi u \nabla B_\alpha\ast u$, see \eqref{forpde}, so that the right hand side in \eqref{integratedeuler} converges to
\[
\int_0^{+\infty}\int_{\R^d}\zeta(t)\xi(t)\cdot\nabla \varphi \,u(t)\,\d x\,\d t=\int_0^{+\infty}\int_{\R^d}\zeta(t)\left(\nabla u(t)^m-\chi\,u(t)\nabla B_\alpha\ast u(t)\right)\cdot\nabla \varphi\,\d x\,\d t.
\]
The first term in the right hand side of \eqref{integratedeuler} can be treated with a classical JKO scheme argument (see for instance \cite{JKO}).
Since $u_\tau(t-\tau)={T_{\tau}}_{\#} u_\tau(t)$ we have by Taylor expansion
\[
\frac1\tau\int_0^{+\infty}\int_{\R^d}\varphi\,\zeta(t)\, (u_\tau(t)-u_\tau(t-\tau))\,\d x\,\d t=\frac1\tau\int_0^{+\infty}\int_{\R^d}\zeta(t)(I-T_\tau(t))\cdot\nabla\varphi\,u_\tau(t)\,\d x\,\d t+R^*_\tau
\]
where the remainder term $R^*_\tau$ satisfies, setting $0<t_0<T<+\infty$ such that the support of $\zeta$ is contained in $(t_0,T)$,
\[\begin{aligned}
{|R^*_\tau|}&\le\frac1\tau\, \max_{\R^d}|D^2\varphi|\int_0^{+\infty}\int_{\R^d}|\zeta(t)|\,|I-T_\tau(t)|^2\,u_\tau(t)\,\d x\, \d t\\&\le \tau \max_{\R^d}|D^2\varphi|\int_{t_0}^{T}\int_{\R^d}|\zeta(t)|\,|\xi_\tau|^2u_\tau(t)\,\d x\,\d t\le 2\tau \max_{\R^d}|D^2\varphi|\left(\max_{(0,+\infty)}|\zeta|\right)\,\tilde Q(t_0,T),
\end{aligned}\]
where we have used \eqref{boundwitht0} and \eqref{eq:DDeq}. In particular $R^*_\tau$ vanishes as $\tau\to0$. On the other hand,  discrete integration by parts entails
\[\begin{aligned}
&\lim_{\tau\to0}\frac1\tau\int_0^{+\infty}\int_{\R^d}\varphi\,\zeta(t)\, (u_\tau(t)-u_\tau(t-\tau))\,\d x\,\d t\\&\qquad=-\lim_{\tau\to0}\int_0^{+\infty}\int_{\R^d}\varphi\,\frac{\zeta(t+\tau)-\zeta(t)}{\tau}\,u_\tau(t)\,\d x\,\d t=-\int_0^{+\infty}\int_{R^d}\varphi\,\zeta'(t)\, u_\tau(t)\,\d x\,\d t.
\end{aligned}\]
Therefore the result follows by passing to the limit in \eqref{integratedeuler} as $\tau\to0$, along the sequence $(\tau_n)_n$.
\end{proof}

\subsection{Hypercontractivity  estimates for finite $p$}

We next provide  estimates of $L^p(\R^d)$ norms, $1<p<+\infty$, of the constructed limit curve of the scheme, that are immediate consequence of their discrete counterparts obtained in Section \ref{discretemin}.

\begin{proposition}\label{firstcontinuousestimates} Let $m>m_c$, $u_0\in\mathscr P_2^M(\R^d)$ and $p\in(1,+\infty)$. 
Suppose that $ u\in AC_{loc}^2((0,+\infty);\mathscr P_2^M(\R^d))\EEE$ is a curve obtained by {\rm Theorem \ref{esistenzagenerale}}.  We have 
\begin{equation}\label{badp}
\|u(t)\|_{L^p(\R^d)}\le \min\left\{\|u_0\|_{L^p(\R^d)},\, \kappa(p)\,M^{\textstyle\frac{2p+d(m-1)}{p(2+d(m-1))}}\;t^{\textstyle -\frac{p-1}{p}\frac{d}{2+d(m-1)}}\right\}+\kappa_2(p)\,M^{\textstyle\frac{2p+d(m-2)}{p(2+d(m-2))}}
\end{equation}
for every $t>0$.

 Else, let $d\ge 3$, $m= m_c$, $u_0\in D(\mathcal E_{m_c})$ and $p>m_c$. Let $ u\in AC_{loc}^2([0,+\infty);\mathscr P_2^M(\R^d))\EEE$ be a curve obtained by {\rm Theorem \ref{th:convergence1}}. We have
\begin{equation}\label{badpbis}
\|u(t)\|_{L^p(\R^d)}\le \min\left\{\|u_0\|_{L^p(\R^d)},\, \ell(p)\,C_0^{\frac{f(p,m_c)}{p(\sigma(p,m_c)-1)}}\;t^{\textstyle -\frac{(p-m_c)d}{p(2m_c+d(m_c-1))}}\right\}+\ell_2(p)\,C_0^{\textstyle\frac{f(p,m_c)+g(p,m_c)}{p\sigma(p,m_c)}}
\end{equation}
for every $t>0$, where $C_0$ is the constant from {\rm Proposition \ref{equi}}, also depending on $\mathcal E_m(u_0)$.
Here, $\kappa(p)$, $\ell(p)$ are explicit constants depending only on $p,d,m$, and $\kappa_2(p), \ell_2(p)$ are  explicit constants depending only on $p,\chi,d,m$.
\end{proposition}
\begin{remark}\rm In particular, $\ell_2(p)=\kappa_2(p)=0$ if $\chi=0$, thus recovering  estimates for the porous media equation.\EEE
\end{remark}\RRR
\begin{proof}
Let $m>m_c$ and $p>1$.  Let us fix $t>0$. Since,  for $k=\lceil t/\tau\rceil$,  $u_\tau(t)=u_\tau^k$ for $t\in((k-1)\tau,k\tau]$ and $u_\tau(t)$ narrowly converges to $u(t)$ as $\tau_n\to0$ (along a suitable vanishing sequence $(\tau_n)_n$),
since $\|u_\tau^0\|_{L^p(\R^d)}$ converges to $\|u_0\|_{L^p(\R^d)}$ and 
 \eqref{utau0}-\eqref{heatcontraction} show that $\|u_\tau^0\|_{L^p(\R^d)}$ diverges at most logarithmically w.r.t. $\tau$ as $\tau\to0$, we pass to the limit in \eqref{longdiscrete}  obtaining
\[\begin{aligned}
\|u(t)\|_{L^p(\R^d)}^p&\le \min\left\{\|u_0\|_{L^p(\R^d)}^p,\,\left(\bar K(p,1)\,M^{-f(p,1)}(\sigma(p,1)-1)\,t\right)^{-\frac{1}{\sigma(p,1)-1}}\right\}\\&\qquad\qquad\qquad\qquad\qquad\qquad\qquad\qquad+\left(\frac{\chi\, c(p,1)\, M^{g(p,1)}}{\bar K(p,1)\,M^{-f(p,1)}}\right)^{\frac1{\sigma(p,1)}}.\end{aligned}\]
By the concavity of the $1/p$ power and by writing some coefficients more explicitly with \eqref{sigmafg}, \eqref{kappabar} and \eqref{cipiesse}, 
we deduce \eqref{badp},
where
\begin{equation}\label{kappa12}
\kappa(p):=\left(\frac{dC^*_{p,1}(p+m-1)^2}{2mp(2+d(m-1))}\right)^{\textstyle\frac{d(p-1)}{p(2+d(m-1))}} \quad\mbox{and}\quad \kappa_2(p):=\left(\frac{\chi c(p,1)}{\bar K(p,1)}\right)^{\textstyle\frac1{p\,\sigma(p,1)}}.
\end{equation}
The estimate  \eqref{badp} can not be passed to the limit as $p\to+\infty$, since both $\kappa(p)$ and $\kappa_2(p)$ diverge as powers of $p$, as seen by their expressions.

The estimate \eqref{badpbis}  for the case $m=m_c$, $d\ge 3$, $p>m_c$ can be obtained by the same argument, 
starting from the corresponding discrete estimate from Proposition \ref{newlpdge3}, in the case $u_0 \in D(\mathcal E_{m_c})$.
\end{proof}

\begin{proposition}\label{newLpd=2} Let $d=2, m=1$, $u_0\in D(\mathcal E_1)$. Suppose  $ u\in AC_{loc}^2([0,+\infty);\mathscr P_2^M(\R^d))\EEE$ is a curve obtained by {\rm Theorem \ref{th:convergence1}}.
For any $t>0$ and $p\in[2,+\infty)$, there holds 
\[
\|u(t)\|_{L^p({\R}^2)}^p \le C_{p,M,\chi}  K_p(t\vee1)^{2p-2} +2^{p-1}\min\left\{\|u_0\|_{L^p(\R^2)}^p ,\,M^p\left(\frac{1}{pG_p} t \right)^{1-p}\right\},
\]
where 
$K_p(\cdot)$ is defined by \eqref{kappat4}, $C_{p,M,\chi}$ is the constant from {\rm Theorem \ref{theop}}, $G_p$ is the constant from \eqref{GNS}.    
\end{proposition}
\begin{proof}
We consider \eqref{DLpEst2}-\eqref{tildeerre}, for given $t\in(0,T)$, $T\ge 1$, and with  $k=\lceil t/\tau\rceil$ (so that $k\tau\le T$ for small enough $\tau$). We apply the inequality $K_{p,\tau}(T)\le K_p(T)$, which holds thanks to the fact that $C_\tau^0\le C_0$ as seen in Proposition \ref{logequi}, taking advantage of the assumption $u_0\in D(\mathcal E_1)$. 
Then we pass to the limit as $\tau\to 0$ (along a sequence $(\tau_n)_n$ such that \eqref{0conv} holds).
The choice of $u_\tau^0$ from \eqref{utau0} implies that
the $L^p$ norms of $u_\tau^0$ diverge at most logarithmically as $\tau\to 0$, see \eqref{heatcontraction},
and consequently the remainder vanishes as $\tau\to 0$ as usual.
By the lower semi continuity of the $L^p$ norm with respect to the narrow convergence and
the limits
\[
\lim_{\tau\to 0}\|u_\tau^0\|_{L^p(\R^2)}^p=\|u_0\|_{L^p(\R^2)}^p, \qquad \lim_{\tau\to 0}k\tau=\lim_{\tau\to0}\lceil t/\tau\rceil\tau=t,
\]
valid also in the case $\|u_0\|_{L^p(\R^2)}=+\infty$,
we can pass to the limit as $\tau\to 0$ along the sequence $(\tau_n)_n$ and get
\[
	\|u(t)\|_{L^p({\R}^2)}^p \le C_{p,M,\chi}  K_p(T)^{2p-2} +2^{p-1}\min\left\{\|u_0\|_{L^p(\R^2)}^p ,\,M^p\left(\frac{1}{pG_p} t \right)^{1-p}\right\}. 
\]
By passing to the limit for $T\to t\vee1$, by the continuity of $T\mapsto  K_p(T)$, we conclude.
\end{proof}

\RRR


%


\begin{theorem}[\bf Decay of $L^p$ norms in the small mass regime for $m=m_c$]\label{subsubtheorem}  
Let $m=m_c$, $u_0\in D(\mathcal E_{m_c})$, $p\in(1,+\infty)$ and $M<M_{sc}(p)$, where $M_{sc}(p)$ is defined by \eqref{defMscp}. 
Suppose that $ u\in AC_{loc}^2((0,+\infty);\mathscr P_2^M(\R^d))\EEE$ is a curve obtained by {\rm Theorem \ref{th:convergence1}}. Then 
\begin{equation*}\begin{aligned}
    \|u(t)\|^p_{L^p(\R^d)} &\leq \min \left\{ \|u_0\|^p_{L^p(\R^d)},\,M\left({C_{\chi,p,M,d}\, t}\right)^{1-p}\right\}  , \qquad \forall\, t>0,
\end{aligned}\end{equation*}
where $C_{\chi,p,M,d}$ is defined in \eqref{defCpMd}.
 \end{theorem}
\begin{proof} By Lemma \ref{lemmaLp}, we have \eqref{discretesubsub}.
We fix $t\in(0,+\infty)$ and define $k=\lceil t/\tau\rceil$. We pass to  the limit  as $\tau\downarrow 0$ (along a suitable sequence $(\tau_n)_n$, such that there is narrow convergence of $u_\tau^{\lceil t/\tau\rceil}$ to $u(t)$, by {\rm Theorem \ref{th:convergence1}}). 
The remainder term in \eqref{discretesubsub} is vanishing, for the same reasons discussed in the proof of Theorem \ref{newLpd=2}. 
Since $	\|u_\tau^0\|_{L^p(\R^d)}\to \|u^0\|_{L^p(\R^d)}$  as $\tau\to 0$, the narrow lower semicontinuity of the $L^p$ norms yields the result.
\end{proof}

We observe that, since $\lim_{p\to+\infty} M_{sc}(p) =0$, it is not possible to deduce from Theorem \ref{subsubtheorem} an analogous result for the $L^\infty$ norm because the condition on $M$ should be $M=0$.

As a final estimate we obtain the continuous version of Proposition \ref{momentgrowth}.


\begin{proposition}\label{momgrowth}
The following second moment estimates hold.
\begin{itemize}
\item[\textbf{(I)}] Let  $m=m_c$, $u_0\in D(\mathcal E_m)$. Suppose that $ u\in AC_{loc}^2([0,+\infty);\mathscr P_2^M(\R^d))\EEE$ is a curve obtained by {\rm Theorem \ref{th:convergence1}}. We have for every $t>0$
\[\begin{aligned}
\mom(u(t))&\le \mom(u_0)+4Mt,\qquad\mbox{if $d=2$}\\
\mom(u(t))&\le \mom(u_0)+2dC_0t,\qquad\mbox{if $d\ge 3$,}
\end{aligned}\]
where $C_0$ is the constant from {\rm Proposition \ref{equi}}.
\item[\textbf{(II)}] Let  $m>m_c$. Suppose that $ u\in AC_{loc}^2((0,+\infty);\mathscr P_2^M(\R^d))\EEE$ is a curve obtained by {\rm Theorem \ref{esistenzagenerale}}. We have for every $t>0$
\[
\mom(u(t))\le \mom(u_0)+ C_*(1+t),
\]
where $C_*$ is an explicit constant, only depending on $m,d,M,\chi$.
\end{itemize}
\end{proposition}
\begin{proof}
Let $t>0$. By taking \eqref{mommom} into account, the proof of  \textbf{(I)} immediately follows from Proposition \ref{initialdatum} and Proposition \ref{prop:FP}, also using Proposition \ref{equi} in the case $d\ge3$. Concerning point \textbf{(II)}, 
by  inserting \eqref{rightpowerbis} in \eqref{mommom}  and by making use of  Proposition \ref{prop:FP} and of Proposition \ref{initialdatum}, letting $\tau\to0$ along a sequence $(\tau_n)_n$ such that \eqref{narrowconv00} holds, we obtain 
\[
\mom(u(t))\le \mom(u^0)+2d\tilde C_*t+2d\tilde C^*\int_0^t s^{-\textstyle{\frac{d(m-1)}{2+d(m-1)}}}\,ds
\]
 and the result follows.
\end{proof}

\subsection{Energy dissipation inequality}

The energy dissipation estimate that we prove here is based on the De Giorgi variational interpolant, see for instance \cite[Section 3.2]{AGS}. Later in Section \ref{lastsection}, we will improve this result showing that the energy dissipation inequality is in fact an equality.
For $\tau>0$ the De Giorgi interpolant $\tilde u_\tau:[0,+\infty)\to \PP_2^M(\R^d)$ is defined by $\tilde u_\tau(k\tau)=u_\tau^k$ for any $k=0,1,2,\ldots$ and
\begin{equation*}
\tilde u_\tau(t) \in  {\rm Argmin}\left\{\frac{1}{2(t-(k-1)\tau)}W^2(u,u_\tau^{k-1}) +\EE_m(u)\right\}, \qquad \text{for } t\in ((k-1)\tau,k\tau).
\end{equation*}

\begin{lemma}
For any $t>0$, we have $\tilde u_\tau(t)\in L^p(\R^d)$ for any $p\in [1,+\infty)$ and $(\tilde u_\tau(t))^m\in W^{1,1}(\R^d)$. Denoting by  $\tilde v_\tau(t) = B_\alpha * \tilde u_\tau(t)$,
 the following discrete energy inequality holds for every integers $0\le N_1<N_2$:
\begin{equation}\label{discreteEI}
\begin{aligned}
 &\frac{1}{2}\int_{N_1\tau}^{N_2\tau}\int_{\R^d}\left|\frac{\nabla (u_\tau)^m-\chi u_\tau \nabla v_\tau }{u_\tau}\right|^2 u_\tau \,\d x\,\d t
\\&\qquad+ \frac{1}{2}\int_{N_1\tau}^{N_2\tau}\int_{\R^d}\left|\frac{\nabla (\tilde u_\tau)^m-\chi \tilde u_\tau \nabla \tilde v_\tau }{\tilde u_\tau}\right|^2 \tilde u_\tau \,\d x\,\d t  +\EE_m (u_\tau^{N_2}) \leq \EE_m(u_\tau^{N_1}).
\end{aligned}
\end{equation}
Moreover for any $T>0$ there exists a constant $\hat C(T)$ such that
\begin{equation}\label{DGversusPC}
W^2(\tilde u_\tau(t),u_\tau(t))\leq \hat C(T) \tau, \qquad \forall t\in[0,T].
\end{equation}
\end{lemma}

\begin{proof} For any $t>0$, we have from Theorem \ref{generalp} that $ u_\tau(t)\in L^p(\R^d)$ for any $p\in [1,+\infty)$ and $( u_\tau(t))^m\in W^{1,1}(\R^d)$. The same then holds for $\tilde u_\tau(t)$, since it is still defined through the JKO scheme for functional $\mathcal E_m$, applied with a different time step.
As in the proof of \cite[Lemma~3.2.2]{AGS}, for every integer $k\ge 1$, we have the one step energy estimate 
\begin{equation}\label{eq:DIed}
\begin{aligned}
 &\frac{1}{2}\frac{W^2(u_\tau^k,u_\tau^{k-1})}{\tau}
+ \frac{1}{2}\int_{(k-1)\tau}^{k\tau}\frac{W^2(\tilde u_\tau(t),u_\tau^{k-1})}{(t-(k-1)\tau)^2} \d t
 +\EE_m(u_\tau^k) \le \EE_m(u_\tau^{k-1}).
\end{aligned}
\end{equation}
For $t\in  ((k-1)\tau,k\tau)$, by the same argument of the proof of Proposition \ref{prop:EL}, we obtain the analogous of \eqref{eq:DDeq}, i.e.,
$$\frac{W^2(\tilde u_\tau(t),u_\tau^{k-1})}{(t-(k-1)\tau)^2} = \int_{\R^d}\left|\frac{\nabla (\tilde u_\tau(t))^m-\chi \tilde u_\tau(t) \nabla \tilde v_\tau(t) }{\tilde u_\tau(t)}\right|^2 \tilde u_\tau(t) \,\d x.$$
Using \eqref{eq:DDeq} and the previous identity, \eqref{eq:DIed} can be rewritten as
\begin{equation}\label{eq:DIed2}
\begin{aligned}
 &\frac{1}{2}\tau \int_{\R^d} \left| \frac{\nabla (u_\tau^k)^m - \chi u_\tau^k \nabla v_\tau^k}{u_\tau^k}\right|^2 u_\tau^k \, \d x\\&\quad
+ \frac{1}{2}\int_{(k-1)\tau}^{k\tau}\int_{\R^d}\left|\frac{\nabla (\tilde u_\tau(t))^m-\chi \tilde u_\tau(t) \nabla \tilde v_\tau(t) }{\tilde u_\tau(t)}\right|^2 \tilde u_\tau(t) \,\d x\, \d t +\EE_m(u_\tau^k) \leq \EE_m(u_\tau^{k-1}).
\end{aligned}
\end{equation}
Summing the identity \eqref{eq:DIed2} for $k$ from $N_1+1$ to $N_2$ and recalling the definition of $u_\tau$ from \eqref{floor} we obtain \eqref{discreteEI}.
Finally, the estimate \eqref{DGversusPC} follows from the same argument of \cite[Lemma~3.2.2]{AGS}.
\end{proof}

\begin{proposition}[{\bf Energy dissipation inequality}]\label{gfede}
Let $u_0\in\Space$. If $m=m_c$, assume in addition $u_0\in D(\mathcal E_m).$ Let $ u\in AC_{loc}^2((0,+\infty);\mathscr P_2^M(\R^d))\cap C([0,+\infty);\mathscr P_2^M(\R^d))\EEE$ a curve obtained by either  {\rm Theorem \ref{th:convergence1}} or {\rm Theorem \ref{esistenzagenerale}}. 
Then, for $t_0=0$ and for a.e. $t_0\in(0,+\infty)$, there holds 
\begin{equation}\label{EDI2}
\EE_m(u(T))+\int_{t_0}^T\int_{\mathbb{R}^d}\left | \frac{\nabla u^{m}(t) - \chi u(t) \nabla B_\alpha\ast u(t)} {u(t)}\right|^2u(t)\,\d x\,\d t \le \EE_m(u(t_0)), 
\quad \forall \,T > t_0.
\end{equation}
\end{proposition}
\begin{proof}
{\bf{Step 1}.}
For $t\in (0,+\infty)$, we denote by $\tilde u_{\tau}(t)$ the De Giorgi interpolant. 
We shall take advantage of the fact that $\tilde u_\tau(t)$ is defined by the JKO scheme for functional $\mathcal E_m$, 
thus it satisfies similar estimates as $u_\tau(t)$.
For instance in the case $m>m_c$,  $u_\tau(t)$ satisfies from \eqref{discrder} the one step estimate
\[
\|u_\tau(t)\|_{L^p(\R^d)}^p\le \|u_\tau(t-\tau)\|_{L^p(\R^d)}^p+\tau\,\chi\, c(p,1)\,M^{g(p,1)},
\]
due to the fact that $u_\tau(t)$ is the solution to the one step JKO minimization starting from $u_\tau(t-\tau)$ with time step $\tau$ (here, for definiteness, $u_\tau(t-\tau)=u_\tau^0$ if $0<t\le \tau$\RRR). Then in the same way, since $\tilde u_\tau(t)$ is the solution to the one step minimization from the JKO scheme starting from $u_\tau(t-\tau)$ with step $t-(\lceil t/\tau\rceil-1)\tau$ (or step $\tau$, in case $t$ is an integer multiple of $\tau$\RRR), we get
\[\begin{aligned}
\|\tilde u_\tau(t)\|_{L^p(\R^d)}^p&\le \|u_\tau(t-\tau)\|_{L^p(\R^d)}^p+\left[\tau\vee (t-(\lceil t/\tau\rceil-1)\tau)\right]\RRR\, \chi\, c(p,1)\, M^{g(p,1)}\\&\le \|u_\tau(t-\tau)\|_{L^p(\R^d)}^p+\tau\, \chi\, c(p,1)\, M^{g(p,1)}.
\end{aligned}\]
Similarly for $d=2, m=1$,
 from \eqref{ataukappa}-\eqref{constantstep} we deduce that for $p\in[2,+\infty)$
\[
\|(\tilde u_\tau(t)-K)_+\|_{L^p(\R^2)}^p\le \|(u_\tau(t-\tau)-K)_+\|_{L^p(\R^2)}^p+\tau\, Y(K)
\]
as soon we choose $T$ such  that $T>t$ and $K$ such that $K>K_p(T)$. Else if $m=m_c$, $d\ge 3$, $p>m_c$, we may apply \eqref{discrder2} (in the version with $C_\tau^0$ replaced by the constant $C_0$ which does not depend on $\tau$) and deduce
\[
\|\tilde u_\tau(t)\|_{L^p(\R^d)}^p\le \|u_\tau(t-\tau)\|_{L^p(\R^d)}^p+\tau\, \chi\, c(p,m_c)\, C_0^{g(p,m_c)}.
\]
Therefore,
supposing now that $(\tau_n)_n$ is a vanishing sequence such that \eqref{narrowconv00} holds (as provided by either Theorem \ref{th:convergence1} or Theorem \ref{esistenzagenerale})
 by invoking  \eqref{0interval}, \eqref{intervaloft} and \eqref{rightpowerbis}  to get uniform bounds on  $\|u_{\tau_n}(t-\tau_n)\|_{L^p(\R^d)}$, 
we deduce that for every $p\in[1,+\infty)$ we have
 \begin{equation*}\label{eq:blpDG}
	\sup_{n}\|\tilde u_{\tau_n}(t) \|_{L^p(\R^d)}<+\infty,\qquad \mbox{for every $t>0$.}
\end{equation*}
 This implies together with \eqref{DGversusPC} that $\tilde u_{\tau_n}(t)\rightharpoonup u(t)$ narrowly and weakly in $L^p(\R^d)$ for every $t>0$ and every $p\in[1,+\infty)$, in particular the limit of $\tilde u_{\tau_n}$ identifies with the limit $u$ of $u_{\tau_n}.$
The identification of the limits is due to the fact that for every $t>0$ and every $\varphi\in C^\infty_c(\R^d)$
 \[\begin{aligned}
 \left|\int_{\R^d}\varphi\,u_{\tau_n}(t)\,\d x-\int_{\R^d}\varphi \tilde u_{\tau_n}(t)\, \d x\right|&\le \mathrm{Lip}(\varphi)\, W_1(u_{\tau_n}(t),\tilde u_{\tau_n}(t))\\&\le \mathrm{Lip}(\varphi)\,\sqrt M\, W(u_{\tau_n}(t),\tilde u_{\tau_n}(t)),
 \end{aligned}\]
having used the dual formulation of the Wasserstein distance $W_1$ of order $1$, and then we may notice that the above right hand side vanishes as $n\to\infty$ thanks to \eqref{DGversusPC}.\RRR

  Similarly we obtain uniform space time bounds: for every $0<t_0<T$ and every $1\le p<+\infty$, we have indeed 
 \begin{equation}\label{pqbound}
 \int_{t_0}^T\|\tilde u_{\tau_n}(t)\|^p_{L^p(\R^d)}\,\d t\le Q^*(t_0,T),
 \end{equation}
 where $Q^*(t_0,T)$ depends only on $p, \chi, d, m, M, \alpha, \mom(u_0)$ (and also on $\mathcal E_m(u_0)$ in the case $m=m_c$).
 
  \EEE

{\bf{Step 2}.}
For every $t\ge 0$ we let $\tilde v_\tau(t) = B_\alpha * \tilde u_\tau(t)$ and
$$
\tilde \xi_{\tau_n}(t):= \frac{\nabla (\tilde u_{\tau_n}(t))^m-\chi \tilde u_{\tau_n}(t) \nabla \tilde v_{\tau_n}(t) }{\tilde u_{\tau_n}(t)}
$$
with the convention that $\tilde \xi_{\tau_n}(t)=0$ in the set where $\tilde u_{\tau_n}(t)=0$.
Let us fix $T>t_0>0$ and $N_2=\lceil T/\tau_n\rceil$. Let $N_1=\lceil t_0/\tau_n\rceil-1$. 
By \eqref{discreteEI} we have that 
\begin{equation}\label{perdopo}\begin{aligned}&\frac12\int_{t_0}^{T}\int_{\R^d} | \xi_{\tau_n}(t)|^2 u_{\tau_n}(t) \,\d x\,\d t +\frac12\int_{t_0}^{T}\int_{\R^d} |\tilde \xi_{\tau_n}(t)|^2\tilde u_{\tau_n}(t) \,\d x\,\d t 
\\&\qquad\qquad\leq \mathcal E_{m}(u_{\tau_n}(t_0-\tau_n))- \mathcal E_m(u_{\tau_n}(T))
\le \mathcal E_{m}(u_{\tau_n}(t_0/2))- \mathcal E_m(u_{\tau_n}(T)), \end{aligned} \end{equation}
having assumed wlog that $\tau_n<(t_0/2)\wedge1$ if $t_0>0$. We estimate the above right hand side, in the case $m>m_c$,  as done in \eqref{sqrtutau}-\eqref{dopo5.17}-\eqref{sqrtutau2}, i.e., we observe that 
\begin{equation}\label{t-tau}\frac12\int_{t_0}^{T}\int_{\R^d} |\tilde \xi_{\tau_n}(t)|^2\tilde u_{\tau_n}(t) \,\d x\,\d t
\le \tilde Q(t_0,T),\end{equation}
where $\tilde Q(t_0,T)$ is the same quantity appearing in \eqref{boundwitht0}, only depending on $\chi, d, m, M,\alpha$, $\mom(u_0)$, bu not on $\tau_n$. Else if $m=m_c$ and $u_0\in  D(\mathcal E_m)$, then we estimate the right hand side of \eqref{perdopo} by $\mathcal E_{m}(u_0)- \mathcal E_m(u_{\tau_n}(T))$, which in turn gets uniformly estimated by $C(T)$, also depending on $\mathcal E_m(u_0)$ from the proof of Proposition \ref{th:convergence1}, thus we still have the uniform bound \eqref{t-tau} up to changing the quantity in the right hand side. 

Suppose now that $1< p<2$. By H\"older inequality we have
\[
\int_{t_0}^T\|\tilde \xi_{\tau_n}(t)\tilde u_{\tau_n}(t)\|_{L^p(\R^d)}^p\,\d t\le\left(\int_{t_0}^T\int_{\R^d}|\tilde \xi_{\tau_n}|^2\tilde u_{\tau_n}\, \d x\,\d t \right)^{\frac p2}\left(\int_{t_0}^T\int_{\R^d}\tilde u_{\tau_n}^{\frac{p}{2-p}}\,\d x\,\d t\right)^{1-\frac{p}{2}}
\]
providing a uniform bound in $L^p((t_0,T)\times \R^d)$ for $\tilde \xi_{\tau_n}\tilde u_{\tau_n}$, thanks to \eqref{pqbound} and \eqref{t-tau}. We claim that the same bound holds for $\tilde u_{\tau_n}\nabla  \tilde v_{\tau_n}$. Indeed,  point {\bf{(i)}} of Proposition \ref{loggrowth} implies by suitably updating the constants therein that
\[
\int_{t_0}^T\int_{\R^d}|\tilde u_{\tau_n}\nabla  v_{\tau_n}|^p\,\d x\,\d t\le \int_{t_0}^T\left(c_1\|\tilde u_{\tau_n}\|_{L^p(\R^d)}^{2p}+c_2M^p\|\tilde u_{\tau_n}\|_{L^p(\R^d)}^p\right)\,\d t,
\]
so that the claim follows from \eqref{pqbound}, and it implies,  since $$\tilde \xi_{\tau_n}\tilde u_{\tau_n}=\nabla\tilde u_{\tau_n}^m-\chi\tilde u_{\tau_ n}\nabla \tilde v_{\tau_n},$$ that also the sequence $(\nabla\tilde u_{\tau_n}^m)_n$ is uniformly bounded in $L^p((t_0,T)\times \R^d)$. 
 That is, by taking advantage once more of \eqref{pqbound},
\[
\int_{t_0}^T\|\tilde u_{\tau_n}(t)^m\|_{W^{1,p}(\R^d)}^p\,\d t\le Q_*(t_0,T)
\] 
for a suitable $Q_*(t_0,T)$ depending on $p,\chi,m,d, \alpha, M, \mom(u_0)$ (and also $\mathcal E_m(u_0)$ if $m=m_c$).
Now since $m\ge 1$, by the inequality $\|u_{\tau_n}\|^m_{W^{1/m,mp}(\R^d)}\le \bar c(p,m,d)\|u_{\tau_n}^m\|_{W^{1,p}(\R^d)}$ from \cite{M} and by Jensen inequality we get
\[
\int_{t_0}^T\|\tilde u_{\tau_n}(t)\|_{{W^{1/m,pm}}(\R^d)}^p\,\d t\le Q_*(t_0,T),
\]
having suitably updated the quantity $Q_*(t_0,T)$. Then we may reason as done in the proof of Theorem \ref{esistenzagenerale}: from the one hand, $W^{1/m,pm}(\R^d)\cap \Space$ compactly embeds into $L^p(\R^d)$; on the other hand, the time equicontinuity property \eqref{integralequicontinuity} holding for $u_{\tau_n}$ as seen in \eqref{equicontinuityofWasserstein} (or by \eqref{ecutau} in the case $u_0\notin D(\mathcal E_m)$, $m>m_c$), immediately extends to $\tilde u_{\tau_n}$ in view of \eqref{DGversusPC}.
Therefore we obtain strong $L^p((t_0,T)\times \R^d)$ convergence by \cite[Theorem 2]{Rossi}, thus pointwise a.e. $(t_0,T)\times\R^d$ convergence, up to subsequences, of $\tilde u_{\tau_n}$ to $u$. This pointwise convergence allows to identify the weak $L^p((t_0,T)\times\R^d)$ limit of $\nabla \tilde u_{\tau_n}^m$ with $\nabla u^m$. All in all, we have that $\tilde \xi_{\tau_n}\tilde u_{\tau_n}$ weakly converge in $L^p_{loc}((0,+\infty);L^p(\R^d))$ to $\nabla u^m-\chi u\nabla  v$, along the original sequence $(\tau_n)_n$.
By \eqref{tx} and \eqref{forpde}, $ \xi_{\tau_n} u_{\tau_n}$ weakly converge in $L^p_{loc}((0,+\infty);L^p(\R^d))$ to $\nabla u^m-\chi u\nabla  v$ as well, where $\xi_\tau$ is defined as in the proof of Theorem \ref{esistenzagenerale}. Moreover, \eqref{tx} implies that for every $1\le p<+\infty$
\[\begin{aligned}
&\int_{t_0}^T\|u_{\tau_n}(t-\tau_n)-u(t)\|_{L^p(\R^d)}^p\,\d t \le
\\&\qquad2^{p-1} \int_{t_0}^T\|u_{\tau_n}(t-\tau_n)-u(t-\tau_n)\|_{L^p(\R^d)}^p\,\d t +2^{p-1} \int_{t_0}^T\|u(t-\tau_n)-u(t)\|_{L^p(\R^d)}^p\,\d t \le
\\&\qquad 2^{p-1} \int_{t_0/2}^T\|u_{\tau_n}(t)-u(t)\|_{L^p(\R^d)}^p\,\d t +2^{p-1} \int_{t_0}^T\|u(t-\tau_n)-u(t)\|_{L^p(\R^d)}^p\,\d t
\end{aligned}\]
where the first term in the right hand side vanishes as $n\to+\infty$ by the convergence of $u_{\tau_n}$ to $u$ in $L^p_{loc}((0,+\infty);L^p(\R^d))$ and the second one vanishes as well by the continuity of the $L^p$ norm under translations. We deduce that also $ u_{\tau_n}(\cdot-\tau_n)$ converge to $u$ in $L^p_{loc}((0,+\infty);L^p(\R^d))$, by the arbitrariness of $0<t_0<T$, hence there exists  $I\subset(0,+\infty)$ such that $(0,+\infty)\setminus I$ is a null set and  $u_{\tau_n}(t-\tau_n)$ strongly converge in $L^{p}(\R^d)$ to $u(t)$ for every $t\in I$ and for every $1\le p<+\infty$.


{\bf {Step 3}.}
We conclude the proof. Suppose first that $m>m_c$ and $u_0\notin D(\mathcal E_m)$. If $t_0=0$ there is nothing to prove. Else,
we pass to the limit in the first inequality of \eqref{perdopo}, choosing $t_0\in I$, where $I$ is the set defined in the previous step, so that $u_{\tau_n}(t_0-\tau_n)$ converge to $u(t_0)$ in $L^m(\R^d)$ on a suitable subsequence and thus $\mathcal E_m(u_{\tau_n}(t_0-\tau_n))\to \mathcal E_m(u(t))$ in view of Proposition \ref{prop:FP}, which implies on the other hand that $\mathcal E_m(u(T))\le \liminf_{n\to\infty}\mathcal E_m(u_{\tau_n}(T))$ for every $T>0$. 
  Recalling the uniform bounds \eqref{t-tau} and \eqref{boundwitht0},
  and recalling that both $\xi_{\tau_n}u_{\tau_n}$ and $\tilde \xi_{\tau_n}\tilde u_{\tau_n}$ converge to $\xi u$ in $\mathcal D'((0,+\infty)\times\R^d)$ as seen in the previous step,
   where $\xi=0$ in the set $\{(t,x)\in(t_0,T)\times \R^d: u(t,x)=0\}$ and
\begin{equation}\label{defxi}
\xi:=\frac{\nabla u^m-\chi u\nabla v}{u}
\end{equation}
otherwise, an application of  \cite[Theorem 5.4.4]{AGS} ensures the semicontinuity properties
\begin{equation*}\int_{t_0}^T\int_{\R^d} |\xi|^2u \,\d x\,\d t \leq \liminf_{n\to+\infty}\int_{t_0}^T\int_{\R^d} |\xi_{\tau_n}|^2u_{\tau_n} \,\d x\,\d t,\end{equation*}
\begin{equation*}\int_{t_0}^T\int_{\R^d} |\xi|^2u \,\d x\,\d t \leq \liminf_{n\to+\infty}\int_{t_0}^T\int_{\R^d} |\tilde \xi_{\tau_n}|^2\tilde u_{\tau_n} \,\d x\,\d t\end{equation*}
  which allow to pass to the limit in the first inequality of \eqref{perdopo} and conclude.
  
 Suppose now that $m\ge m_c$ and $u_0\in D(\mathcal E_m)$. If $t_0>0$ the proof is the same as above. Else if $t_0=0$, we may take  $N_2=\lceil T/\tau_n\rceil$ and $N_1=0$   in
 \eqref{discreteEI} and deduce that that 
\begin{equation*}\frac12\int_{0}^{T}\int_{\R^d} | \xi_{\tau_n}(t)|^2 u_{\tau_n}(t) \,\d x\,\d t +\frac12\int_{0}^{T}\int_{\R^d} |\tilde \xi_{\tau_n}(t)|^2\tilde u_{\tau_n}(t) \,\d x\,\d t \le
\mathcal E_{m}(u_{\tau_n}^0)- \mathcal E_m(u_{\tau_n}(T)). \end{equation*}
The latter passes to the limit in the same way,
also taking
\eqref{energytau} into account.
\EEE
\end{proof}

\section{Smoothing effect}\label{hypersection}

In this section we prove a key result of this paper: the sharp $L^\infty(\R^d)$ hypercontractivity  for  gradient flows\RRR. The main argument is given in Theorem  \ref{Linftydge3}, and the necessary variants are discussed later on.   Here we briefly comment on the main steps of the proof of Theorem \ref{Linftydge3}. In the following discussion we refer to the case $m>m_c$, but the strategy will be similar in the case $m=m_c$. 
The starting point is the estimate of discrete derivatives of $L^p(\R^d)$ norms provided by Proposition \ref{psprop},
where $1\le s<p<+\infty$. 
Thanks to the convergence properties of the scheme that we proved in Section \ref{convergingsection}, we may pass to the limit as $\tau\to0$ to see that for a curve  $u \in AC^2_{loc}((0,+\infty);\Space)$ provided by Theorem \ref{esistenzagenerale}, the $L^p(\R^d)$ norm  satisfies a corresponding differential inequality, for $1<p<+\infty$.
This has already been done in Proposition  \ref{firstcontinuousestimates} in the case $s=1$, obtaining a hypercontractivity estimate for the $L^p(\R^d)$ norm in the case of finite $p$. The exponent in the estimate \eqref{badp}  is sharp but the coefficients $\kappa(p)$, $\kappa_2(p)$ do not stay bounded as $p\to+\infty$, and,
as classically happens for the porous medium equation, an $L^\infty$ estimate 
 cannot be obtained directly passing to the limit as $p\to+\infty$ in the $L^p$ estimates.  
 Therefore, we need to apply a Moser-Alikakos iteration method by setting $p=p_j=2^j, s=2^{j-1}$, and $Q_j=\|u(t)\|_{L^{p_j}(\R^d)}^{p_j}$, 
 ending up with a recursive estimate of the form
 \begin{equation}\label{recursivintro}
Q_j\le \left(\bar G^j t^{-Z_j}\,Q_{j-1}^{\gamma_j}\right)\vee \left(\bar B(\chi)^j\,Q_{j-1}^{\beta_j}\right)\qquad \mbox{for every  $j=1,2,3,\ldots$}
 \end{equation}
 where 
 \[
 Z_j:=\frac{2^{j-1}d}{2^{j-1}+d(m-1)},\quad\gamma_j:=\frac{2^j+d(m-1)}{2^{j-1}+d(m-1)},\quad\beta_j:=\frac{2^{j}+d(m-2)}{2^{j-1}+d(m-2)},
 \]
 and where the  coefficients $\bar G, \bar B(\chi)$ are explicit and will be given through the proof of Theorem \ref{Linftydge3}. 
  The computation works also for the simpler case $\chi=0$, where  $\bar B(\chi)=0$, 
  and the recursive relation allows to easily deduce the $L^\infty$ estimate \eqref{smoothingporous} for the solutions of the porous medium equation. In the case $\chi>0$, considered in this paper, \EEE unfortunately we have $\gamma_j<\beta_j$ so that the second term in \eqref{recursivintro} is not immediately negligible, even for small $t$. Still by carefully taking into account some recursive relations that are satisfied by the exponents $Z_j,\gamma_j,\beta_j$ (see Lemma \ref{lemmata} below) we shall deduce that the  
  $L^\infty$ hypercontractivity rate
  is the same obtained for the porous medium equation. \EEE

\subsection{The general $L^\infty$ hypercontractivity result}

Let $p>1$, $s\in[1,p)$,  $m\ge m_c$, and assume in addition $s>1$ if $m=m_c$.    Letting $\sigma(p,s), f(p,s), g(p,s)$ be defined by \eqref{sigmafg}, we define
the following quantities that will play a key role in the iterative argument of  Theorem \ref{Linftydge3}:
\begin{equation}\begin{aligned}\label{gammabeta}
Z(p,s)&:=\frac{1}{\sigma(p,s)-1}=
\frac{d(p-s)}{2s+d(m-1)},\\
\gamma(p,s)&:=Z(p,s)\,f(p,s)
=\frac{2p+d(m-1)}{2s+d(m-1)},\\
\beta(p,s)&:=\frac{f(p,s)+g(p,s)}{\sigma(p,s)}=\frac{2p+d(m-2)}{2s+d(m-2)}
\end{aligned}\end{equation}
where the last equality follows after an elementary computation. Here, the condition $m\ge m_c$, with $s>1$ if $m=m_c$,  
ensures that all the denominators in \eqref{gammabeta} are positive. It is immediate to verify that \begin{equation}\label{g<b}\gamma(p,s)< \beta(p,s)\end{equation} and that the following identities hold
\[
\gamma(p,s)\gamma(s,1)=\gamma(p,1),\qquad  Z(p,s)+\gamma(p,s)Z(s,1)=Z(p,1), \qquad \text{for } m\ge m_c,
\]
\[
\beta(p,s)\beta(s,1)=\beta(p,1), \qquad \text{for } m>m_c.
\]
These identities can be iterated to produce the following
\begin{lemma}\label{lemmata}
Let $m\ge m_c$  and let $(p_j)_{j=0,1,2,\ldots}\subset\R$ be a strictly increasing sequence such that $p_0=1$. 
Defining $\gamma$, $\beta$, $Z$ in \eqref{gammabeta}, let 
\[
\gamma_j:=\gamma(p_j,p_{j-1}), \quad \beta_j:=\beta(p_j,p_{j-1}), \quad Z_j:=Z(p_j,p_{j-1}), \quad \bar Z_j:=Z(p_j,1), \quad  j=1,2,3,\ldots
\]
with $\beta_1$ not defined in the case $m=m_c$.
Then
\begin{equation}\label{gjbj}
\gamma_j\gamma_{j-1}\ldots\gamma_{j-k}=\frac{2p_j+d(m-1)}{2p_{j-k-1}+d(m-1)},\qquad \beta_j\beta_{j-1}\ldots\beta_{j-k}=\frac{2p_j+d(m-2)}{2p_{j-k-1}+d(m-2)}
\end{equation}
for every $j=1,2,3,\ldots$ and $k=0,1,2,\ldots,j-1$ (with the further restriction $k\neq j-1$ in the second formula if $m=m_c$).

There also holds
\begin{equation}\label{Zgamma}\begin{aligned}
\bar Z_j&=Z_j+Z_{j-1}\gamma_j+Z_{j-2}\gamma_{j-1}\gamma_j+Z_{j-3}\gamma_{j-2}\gamma_{j-1}\gamma_j+\ldots+Z_1\gamma_2\gamma_3\ldots\gamma_j
\end{aligned}
\end{equation}
for every $j=1,2,3,\ldots$ 
(reduced to $Z_1=\bar Z_1$ if $j=1$).

Furthermore,  there holds
\begin{equation}\label{crucial}
\beta_j\le \gamma_j+\frac{Z_j}{\bar Z_{j-1}}\qquad\mbox{for every $j=2,3,4,\ldots$}
\end{equation}
\end{lemma}
\begin{proof} The telescopic identities \eqref{gjbj} are straightforward. We prove  \eqref{Zgamma} by induction. We have $Z_1=\bar Z_1$. If \eqref{Zgamma} holds with given $j\ge 1$, then
\begin{equation*}\begin{aligned}
&Z_{j+1}+Z_{j}\gamma_{j+1}+Z_{j-1}\gamma_j\gamma_{j+1}+\ldots+(Z_1\gamma_2\gamma_3\ldots\gamma_{j+1})=Z_{j+1}+\gamma_{j+1}\bar Z_j\\&\qquad=\frac{d(p_{j+1}-p_j)}{2p_j+d(m-1)}+\frac{[2p_{j+1}+d(m-1)][d(p_j-1)]}{[2p_j+d(m-1)][2+d(m-1)]}\\
&\qquad=\frac{2dp_j(p_{j+1}-1)+d^2(m-1)(p_{j+1}-1)}{[2p_j+d(m-1)][2+d(m-1)]}=\frac{d(p_{j+1}-1)}{2+d(m-1)}=\bar Z_{j+1}
\end{aligned}
\end{equation*}
so that \eqref{Zgamma} holds with $j+1$. Let us now check the crucial estimate \eqref{crucial}, which explicitly reads
\[
\frac{2p_j+d(m-2)}{2p_{j-1}+d(m-2)}\le \frac{[2p_j+d(m-1)][p_{j-1}-1]+(p_{j}-p_{j-1})[2+d(m-1)]}{[2p_{j-1}+d(m-1)](p_{j-1}-1)}
\]
The previous inequality is equivalent to
\[\begin{aligned}
&(p_{j-1}-1)(2p_{j-1}+d(m-1))(2p_j+d(m-2))\\&\qquad\qquad\le (2p_{j-1}+d(m-2))(2p_j+d(m-1))(p_{j-1}-1)\\
&\qquad\qquad\quad+(p_{j}-p_{j-1})(2p_{j-1}+d(m-2))(2+d(m-1)).
\end{aligned}\]
By expanding the products, after a computation the latter proves to be equivalent to
\[
(p_j-p_{j-1})(2dp_{j-1}-2d)\le (p_j-p_{j-1})(4p_{j-1}+2d(m-2)+2dp_{j-1}(m-1)+d^2(m-1)(m-2))
\]
and hence to $0\le (2p_{j-1}+d(m-1))(2+d(m-2))$, which holds true since $m\ge m_c$. 
\end{proof}

Taking advantage of  Lemma \ref{lemmata} we can prove the following crucial result. \EEE

\begin{theorem}\label{Linftydge3}

{\rm\textbf{(I)}}
Let  $m>m_c$, $u_0\in \PP^M_2({\R}^d)$ and
 $u\in AC_{loc}^2((0,+\infty);\PP^M_2({\R}^d))$ a  curve given by {\rm Theorem \ref{esistenzagenerale}}.
Then there exists a constant $C_{\chi,d,m}$ depending only on $\chi$, $d$ and $m$ such that
\[
\|u(t)\|_{L^\infty({\R}^d)} \le C_{\chi,d,m} \,(1\vee M\RRR)^{\textstyle\frac{2}{2+d(m-2)}}\,\left(\frac{1}{t}\right)^{\textstyle\frac{d}{2+d(m-1)}}+C_{\chi,d,m} \,(1\vee M\RRR)^{\textstyle\frac{2}{2+d(m-2)}}\qquad\mbox{ $\forall \ t>0$.}
\]

{\rm\textbf{(II)}}
Let $m=m_c$, $u_0\in D(\mathcal E_{m_c})$ and $u\in AC_{loc}^2([0,+\infty);\PP^M_2({\R}^d))$ a  curve given by {\rm Theorem \ref{th:convergence1}}.  If $d\ge 3$, then there exists a constant $\hat C$,  depending only on $\chi,$ $d$, $M$, $\|u_0\|_{L^{m_c}({\R}^d)}$ and $\mom(u_0)$, \EEE
 such that
\[
\|u(t)\|_{L^\infty({\R}^d)}\le \hat C\left(\frac1t\right)^{\textstyle\frac {d^2}{d(d+2)-4}}+\hat C\qquad\mbox{for every $t>0$}.
\]
If $d=2$, then  
 there exist  constants $\bar{\mathcal C}_{M,\chi},\mathcal C^*_\chi$  such that
\begin{equation*}\label{kappa2t}
\|u(t)\|_{L^\infty({\R}^2)}\le \bar{\mathcal C}_{M,\chi}\left(K_{2}(2t+1)\right)^{2}+\mathcal C^*_\chi\frac{(4MG_2)\vee1}{t}\qquad\mbox{for every $t>0$},
\end{equation*}
where $K_2(\cdot)$ is given by \eqref{kappat4} 
and $G_2$ is from \eqref{GNS}.

\end{theorem}

\begin{proof} 

{\bf Step 1.} Throughout this step, we let $1\le s<p<+\infty$. If $m=m_c$  we also let $s>1$, as in Proposition \ref{psprop}.

 From Theorem \ref{prop:IC} and Theorem \ref{esistenzagenerale}, we have $u_{\tau_n}\to u$ strongly in $L^p_{loc}((0,+\infty);L^p(\R^d))$ for every $1\le p<+\infty$, thus there exists a set $I\subset (0,+\infty)$ such that $(0,+\infty)\setminus I$ is a null set ad such that $u_{\tau_n}(t)\to u(t)$ strongly in $L^p(\R^d)$ for every $t\in I$ and every $1\le p<+\infty$, possibly along a not relabeled subsequence. \RRR

We fix $t_1,t_2 \in I$, $t_1<t_2$. From \eqref{basicdisc2} by summing from $k=\lceil t_1/\tau\rceil+1$ to $k=\lceil t_2/\tau\rceil$, we obtain
 \[\begin{aligned}
& \|u_\tau^{\lceil t_2/\tau\rceil}\|_{L^p({\R}^d)}^p-\|u_\tau^{\lceil t_1/\tau\rceil}\|_{L^p({\R}^d)}^p\\&\qquad\le \tau \sum_{k=\lceil t_1/\tau\rceil-1}^{\lceil t_2/\tau\rceil}
 \left(-\bar K(p,s) \frac{(\|u_\tau^k\|_{L^p({\R}^d)}^p)^{\sigma(p,s)}}{(\|u_\tau^k\|_{L^s({\R}^d)}^s)^{f(p,s)}}+\chi\,c(p,s)(\|u_\tau^k\|_{L^s({\R}^d)}^s)^{g(p,s)}\right).
 \end{aligned}\]
Since $u_\tau(t)=u_\tau^k$ for $t\in((k-1)\tau,k\tau]$, we have
\begin{equation}\label{integralversion}\begin{aligned}
&\|u_\tau(t_2)\|_{L^p({\R}^d)}^p-\|u_\tau(t_1)\|_{L^p({\R}^d)}^p\le -\bar K(p,s) \int_{t_1}^{t_2}\frac{(\|u_\tau(t)\|_{L^p({\R}^d)}^p)^{\sigma(p,s)}}{(\|u_\tau(t)\|_{L^s({\R}^d)}^s)^{f(p,s)}}\,\d t\\&\qquad+\chi\,c(p,s)\int_{t_1}^{t_2}(\|u_\tau(t)\|_{L^s({\R}^d)}^s)^{g(p,s)}\,\d t+ \varepsilon_\tau(t_1,t_2),
\end{aligned}\end{equation}
where the remainder $\varepsilon_\tau(t_1,t_2)$ is 
 \[\begin{aligned}
 \varepsilon_\tau(t_1,t_2)&=\bar K(p,s) \int_{t_1}^{\lceil t_1/\tau\rceil\tau}\frac{(\|u_\tau(t)\|_{L^p({\R}^d)}^p)^{\sigma(p,s)}}{(\|u_\tau(t)\|_{L^s({\R}^d)}^s)^{f(p,s)}}\,\d t+\chi\,c(p,s)\int_{t_2}^{\lceil t_2/\tau\rceil\tau}(\|u_\tau(t)\|_{L^s({\R}^d)}^s)^{g(p,s)}\,\d t\\&
 =\tau\bar K(p,s)\, \frac{(\|u_\tau(t_1)\|_{L^p({\R}^d)}^p)^{\sigma(p,s)}}{(\|u_\tau(t_1)\|_{L^s({\R}^d)}^s)^{f(p,s)}}+\tau \chi\,c(p,s)(\|u_\tau(t_2)\|_{L^s({\R}^d)}^s)^{g(p,s)}.
 \end{aligned}
 \]
 Since $t_1,t_2\in I$ we see that $\eps_{\tau_n}(t_1,t_2)\to0$ as $n\to\infty$, 
 and moreover we may pass to the limit as $\tau\to 0$ in \eqref{integralversion} along the above sequence $(\tau_n)_n$. 
 For the integral terms we make use of Fatou Lemma and dominated convergence, recalling that Theorem \ref{newlpdge3} and 
 Theorem \ref{theop} directly imply that $t\mapsto \sup_{\tau\in(0,1)}\|u_\tau(t)\|_{L^s({\R}^d)}$ belongs to $L^\infty(t_1,t_2)$, 
 since the last term in the estimates therein vanishes as usual as $\tau\to0$. 
 We get
\begin{equation}\label{continuousps}\begin{aligned}
&\|u(t_2)\|_{L^p({\R}^d)}^p-\|u(t_1)\|_{L^p({\R}^d)}^p\le -\bar K(p,s) \int_{t_1}^{t_2}\frac{(\|u(t)\|_{L^p({\R}^d)}^p)^{\sigma(p,s)}}{(\|u(t)\|_{L^s({\R}^d)}^s)^{f(p,s)}}\,\d t\\&\qquad+\chi\,c(p,s)\int_{t_1}^{t_2}(\|u(t)\|_{L^s({\R}^d)}^s)^{g(p,s)}\,\d t.
\end{aligned}
\end{equation}
We introduce the function $H\in AC_{loc}((0,+\infty);\R)$ defined by
\[
H(t):=\int_1^t\left(-\bar K(p,s) \frac{(\|u(\bar t)\|_{L^p({\R}^d)}^p)^{\sigma(p,s)}}{(\|u(\bar t)\|_{L^s({\R}^d)}^s)^{f(p,s)}}+\chi\,c(p,s)(\|u(\bar t)\|_{L^s({\R}^d)}^s)^{g(p,s)}\right)\,\d\bar t.
\]
Since in \eqref{continuousps} $t_1<t_2$ are arbitrary,
then the function $G:I \to \R$ defined by $G(t)= \|u(t)\|_{L^p({\R}^d)}^p-H(t)$ is nonincreasing. 
We define   
\[
y_p(t):=
\left\{\begin{array}{ll}\|u(t)\|_{L^p({\R}^d)}^p\qquad&\mbox{if $t\in I$}\vspace{0.2cm}\\
\displaystyle\lim_{{\bar t\to t^+\!,\;\bar t\in I}}\|u(\bar t)\|_{L^p({\R}^d)}^p\qquad&\mbox{if $t\in (0,+\infty)\setminus I$}
\end{array}
\right.
\]
and $y_1(t):=\|u(t)\|_{L^1({\R}^d)}=M$ for every $t\ge 0$.
Since $t\mapsto u(t)$ is narrowly continuous, by Theorem \ref{th:convergence1} and Theorem \ref{esistenzagenerale}, and 
its $L^p(\R^d)$ norm is bounded on any compact interval of $(0,+\infty)$, by Proposition \ref{firstcontinuousestimates} and Proposition \ref{newLpd=2},
then the map $(0,+\infty)\ni t\mapsto \|u(t)\|_{L^p({\R}^d)}^p$ is lower semicontinuous.
By the definition of $y_p$ and the previous observation we have
\begin{equation}\label{jump?}
y_p(t)\ge \|u(t)\|_{L^p({\R}^d)}^p\qquad\mbox{for every $t>0$},
\end{equation}
and the strict inequality can  occur only for $t\in (0,+\infty)\setminus I$.
Furthermore, using the definition of $y_p$, we can prove that the function $\tilde G:(0,+\infty) \to \R$ defined by $\tilde G(t)= y_p(t)-H(t)$ is nondecreasing.
Thus  for $t_1,t_2\in(0,+\infty)$, $t_1<t_2$, we have
\begin{equation*}
y_p(t_2)-y_p(t_1)\le -\bar K(p,s) \int_{t_1}^{t_2}\frac{(y_p(t))^{\sigma(p,s)}}{(y_s(t))^{f(p,s)}}\,\d t+\chi\,c(p,s)\int_{t_1}^{t_2}(y_s(t))^{g(p,s)}\,\d t.
\end{equation*}
In particular, $y_p(t)=H(t)+\tilde G(t)$ for any $t\in(0,+\infty)$, where $H\in AC_{loc}(0,+\infty)$ and $\tilde G$ is nonincreasing on $(0,+\infty)$.
Then $y_p$ is differentiable almost everywhere in $(0,+\infty)$.
Dividing by $t_2-t_1$ the previous inequality and passing to the limit as $t_2-t_1\to0$ we get, 
\[
y_p'(t)\le -\bar K(p,s)\,\frac{(y_p(t))^{\sigma(p,s)}}{\left(y_s(t)\right)^{f(p,s)}}+\chi\,c(p,s)\left(y_s(t)\right)^{g(p,s)}, \qquad \text{for a.e. } t\in (0,+\infty).
\]
\EEE
 If we fix  $\bar t_0\in (0,T)$, we deduce that
\begin{equation}\label{tconzero}
\begin{aligned}
y_p'(t)&\le -\bar K(p,s)\,\frac{(y_p(t))^{\sigma(p,s)}}{\sup_{\bar t\in(\bar t_0,T)}\left(y_s(\bar t)\right)^{f(p,s)}}\\&\qquad\qquad\qquad\qquad+\chi\,c(p,s)\, \sup_{\bar t\in (\bar t_0,T)}\left(y_s(\bar t)\right)^{g(p,s)}\qquad \mbox{for a.e. $t\in(\bar t_0,T)$}.
\end{aligned}
\end{equation}
Notice that the supremum in \eqref{tconzero} is finite, since $\bar t_0>0$ and since $y_p(\cdot)$, by its definition, 
satisfies the same estimates as $\|u(\cdot)\|_{L^p({\R}^d)}$ from Proposition \ref{newLpd=2}  and from Proposition \ref{firstcontinuousestimates}. 
By comparison with the unique solution of the corresponding differential equality coupled with the initial condition $y(\bar t_0)=+\infty$ (similarly as the comparison principle used in the proof of Proposition \ref{newlpdge3}), taking into account that the distributional time derivative of $y_p$ might have a singular part, which is however a nonpositive measure, we deduce that 
\[
y_p(t)\le  \left(\frac{(\sigma(p,s)-1)\,(t-\bar t_0)\,\bar K(p,s)}{\sup_{\bar t\in(\bar t_0,T)}\left(y_s(\bar t)\right)^{f(p,s)}} \right)^{-\frac{1}{\sigma(p,s)-1}}+\left(\frac{\chi c(p,s)}{\bar K(p,s)}\right)^{\frac{1}{\sigma(p,s)}}\left(\sup_{\bar t\in (\bar t_0,T)}\left(y_s(\bar t)\right)\right)^{\frac{f(p,s)+g(p,s)}{\sigma(p,s)}}
\]
for every $t\in (\bar t_0,T)$, which we rewrite with the notation \eqref{gammabeta} as
\begin{equation}\label{foriterationbasic}\begin{aligned}
y_p(t)&\le \left(\frac{Z(p,s)}{\bar K(p,s)\,(t-\bar t_0)}\right)^{Z(p,s)}\left(\sup_{\bar t\in(\bar t_0,T)}\left(y_s(\bar t)\right)\right)^{\gamma(p,s)}\\&\qquad+\left(\frac{\chi c(p,s)}{\bar K(p,s)}\right)^{\frac{1}{\sigma(p,s)}} \left(\sup_{\bar t\in(\bar t_0,T)}\left(y_s(\bar t)\right)\right)^{\beta(p,s)}
\qquad \forall\,t\in(\bar t_0,T).\end{aligned}
\end{equation}

{\bf Step 2.}
In order to apply an iteration method, we fix $t_*>0$, we let $t_j:=(1-2^{-j-1})t_*$, we let $(p_j)_{j=0,1,2,\ldots}\subset(0,+\infty)$ be a strictly increasing diverging sequence with $p_0=1$, and we apply \eqref{foriterationbasic} with $p=p_j$, $s=p_{j-1}$, $T=2t_*$ and $\bar t_0=t_{j-1}$, for $j=1,2,3,\ldots$. Since $p_0=1$, we have to start from $j=2$ in the case that $m=m_c$, in order to avoid $s=1$ (recalling that $\beta(p,s)$ is undefined if $m=m_c$ and $s=1$). Noticing that $$\inf_{\bar t\in (t_{j},T)}(\bar t-t_{j-1})=t_j-t_{j-1},$$ taking the supremum on $(t_j,T)$ we get
\begin{equation}\label{pjpj-1}\begin{aligned}
\sup_{\bar t\in (t_j,T)}\, (y_{p_j}(\bar t))&\le \left(\frac{Z(p_j,p_{j-1})}{\bar K(p_j,p_{j-1})\,(t_j-t_{j-1})}\right)^{Z(p_j,p_{j-1})}\,\left(\sup_{\bar t\in(t_{j-1},T)}\left(y_{p_{j-1}}(\bar t)\right)\right)^{\gamma(p_j,p_{j-1})}\\&\qquad+\left(\frac{\chi c(p_j,p_{j-1})}{\bar K(p_j,p_{j-1})}\right)^{\frac{1}{\sigma(p_j,p_{j-1})}} \left(\sup_{\bar t\in(t_{j-1},T)}\left(y_{p_{j-1}}(\bar t)\right)\right)^{\beta(p_j,p_{j-1})}
\end{aligned}\end{equation}
It is natural to let $p_j=R^j$ for some fixed $R>1$, for every $j=0,1,2,3,\ldots$. Therefore, \eqref{pjpj-1} holds for $j=1,2,3,\ldots$ if $m>m_c$,
while it holds for $j=2,3,4,\ldots$ if $m=m_c$. We next estimate the constants in both terms of \eqref{pjpj-1}. Concerning the first one, it is easy to check from \eqref{gammabeta} that for every $j=1,2,3,\ldots$ there holds
\begin{equation}\label{stimadiz}
\frac{d(R-1)}{2+d(m-1)}\le Z(R^j,R^{j-1})\le \frac{d(R-1)}2. 
\end{equation}
 Since $t_j-t_{j-1}=2^{-j-1}t_*$, we get
\begin{equation}\label{maincoeff}
2\left(\frac{Z(R^j,R^{j-1})}{(t_j-t_{j-1})\bar K(R^j,R^{j-1})}\right)^{Z(R^j,R^{j-1})}\le G_j:=\left(\frac{2^j\tilde G}{t_*}\right)^{Z(R^{j},R^{j-1})} 
\end{equation}
for every $j=1,2,3,\ldots$,
where we have set
\[
\tilde G=\tilde G(R,d,m):=\tfrac12(dmRC_{R}^*)\,2^{\frac{2+d(m-1)}{d(R-1)}},
\]
having used the first estimate of \eqref{cfrattocbar}, the elementary estimate $\frac{R^j}{R^j-1}\le \frac{R}{R-1}$ that holds for every $j=1,2,3,\ldots,$ and having  used \eqref{kappabar} which in turn yields  
\begin{equation}\label{Cstarerre}C^*_{R^j,R^{j-1}}\le C^*_R:=\left\{\begin{array}{ll}C_{2,d}\quad&\mbox{if $d\ge 3$}\\C_{1,2}R/(R-1)\quad&\mbox{if $d=2$}\end{array}\right.\qquad \mbox{for every $j=1,2,3,\ldots$}\end{equation}
Similarly, for the second term, recalling the definitions of $r(\cdot,\cdot)$ and $\sigma(\cdot,\cdot)$ from  \eqref{sigmafg},  we make use of the elementary estimates
\[
\frac{(R-1)d}{2+d\,(m-2)_+}\le \frac{r(R^j,R^{j-1})}{\sigma(R^j,R^{j-1})}\le \lambda(R),\qquad \mbox{ $j=1,2,3$},  
\]
where $j$ starts from $2$ and not from $1$ in the case $m=m_c$, and
with  the notation  $x_+:=x\vee0$, $x_-:=-(x\wedge0)$, and
\[
\lambda(R):=\left\{\begin{array}{ll}\vspace{0.2 cm}\frac{(R-1)d}{2-d\,(m-2)_-}\quad&\mbox{if $m>m_c$}\\
\frac{R(R-1)d}{2R-d\,(m-2)_-}\quad&\mbox{if $m=m_c$}.
\end{array}\right.
\]  
Therefore, from the second estimate in \eqref{cfrattocbar} and from \eqref{Cstarerre} we deduce
\begin{equation}\label{nextcoeff}\begin{aligned}
2\left(\frac{\chi c(R^j,R^{j-1})}{\bar K{(R^j,R^{j-1)}}}\right)^{\frac{1}{\sigma(R^j,R^{j-1})}}\le B_j:&= 2b_\chi \tilde B^j \quad \mbox{for every $j=1,2,\ldots$},
\end{aligned}\end{equation}
where $j$ starts from $2$ if $m=m_c$, and where
 \begin{equation}\label{Btilde} b_\chi=b_{\chi}(R,d,m):=\left(\tfrac12\chi m C_{R}^*\right)^{\lambda(R)}\vee \left(\tfrac12\chi m C^*_{R}\right)^{\frac{(R-1)d}{2+d\,(m-2)_+}},\;\;\tilde B=\tilde B (R,d,m):=R^{\lambda(R)}.\end{equation}

Next,   we define
\begin{equation}\label{Qj}
Q_j:=\sup_{\bar t\in(t_j,T)}\, (y_{p_j}(\bar t))\qquad \mbox{for  $j=0,1,2,3,\ldots\quad$ (note that $Q_0=M$)}.
\end{equation}
In this way, taking into account \eqref{nextcoeff} and \eqref{maincoeff}, and since $p_j=R^j$,  \eqref{pjpj-1} rewrites
\begin{equation}\label{recursiv2}
Q_j\le \left(G_j\,Q_{j-1}^{\gamma_j}\right)\vee \left(B_j\,Q_{j-1}^{\beta_j}\right)\qquad \mbox{for every  $j=1,2,3,\ldots,\quad $ ($j\ge 2$ if $m=m_c$), }
\end{equation}
where  $\gamma_j:=\gamma(R^{j},R^{j-1})$, $\beta_j:=\beta(R^{j},R^{j-1})$, $Z_j:=Z(R^{j},R^{j-1})$.

(Incidentally
we notice that if we had $\chi=0$ then we would get $B_j=0$ for every $j$ and the following argument would be much simpler and it would give the standard estimate \eqref{smoothingporous} of the porous media equation).
\EEE

{\bf Step 3.} Here we let $m>m_c$.
By applying \eqref{recursiv2} recursively we obtain that for every  $j=1,2,3,\ldots$ there holds
\begin{equation}\label{badrecursiv}
Q_j\le \bigvee_{2^j} \left(A_jA_{j-1}^{\eta_j}A_{j-2}^{\eta_j\eta_{j-1}}A_{j-3}^{\eta_j\eta_{j-1}\eta_{j-2}}\ldots A_{2}^{\eta_j\eta_{j-1}\ldots\eta_{3}}A_1^{\eta_j\eta_{j-1}\ldots \eta_2}\right) \, Q_0^{\eta_j\eta_{j-1}\eta_{j-2}\ldots\eta_2\eta_1}
\end{equation}
where, for each $k=1,2,\ldots, j$, $\eta_k$   is equal either to $\gamma_k$ or to $\beta_k$, and accordingly the maximum (denoted $\vee_{2^j}$) runs over all the $2^j$ possible values of the vector $(\eta_j,\eta_{j-1},\eta_{j-2},\ldots,\eta_2,\eta_1)$. Moreover, in each of the $2^j$ terms among which the maximum is taken, the $A_k$'s are defined by the following rule: $A_{k}=G_k$ if $\eta_k=\gamma_k$, see \eqref{maincoeff}, while $A_k=B_k$ if $\eta_k=\beta_k$, see \eqref{nextcoeff}, for any $k=1,2,\ldots,j.$
Indeed, \eqref{badrecursiv} is proved by induction: given that in this notation the one step iteration \eqref{recursiv2} is equivalent to $Q_j\le\vee_{\eta_j\in\{\gamma_j,\beta_j\}} A_jQ_{j-1}^{\eta_j}$, assuming that the $k$ steps inequality
\begin{equation}\label{krecursiv}
Q_j\le \bigvee_{2^k} A_jA_{j-1}^{\eta_j}A_{j-2}^{\eta_j\eta_{j-1}}\ldots A_{j-k+1}^{\eta_j\eta_{j-1}\ldots\eta_{j-k+2}}Q_{j-k}^{\eta_j\eta_{j-1}\ldots\eta_{j-k+1}}
\end{equation}
 holds true,  by an application \eqref{recursiv2} we obtain
\[
\begin{aligned}
Q_j
&\le \bigvee_{2^k} A_jA_{j-1}^{\eta_j}A_{j-2}^{\eta_j\eta_{j-1}}\ldots A_{j-k+1}^{\eta_j\eta_{j-1}\ldots\eta_{j-k+2}}\left((G_{j-k}Q_{j-k-1}^{\gamma_{j-k}})\vee (B_{j-k}Q_{j-k-1}^{\beta_{j-k}})\right)^{\eta_j\eta_{j-1}\ldots\eta_{j-k+1}}\\ 
&=\bigvee_{2^k}(A_j\ldots A_{j-k+1}^{\eta_j\ldots\eta_{j-k+2}}G_{j-k}^{\eta_j\ldots\eta_{j-k+1}}Q_{j-k-1}^{\eta_j\ldots\eta_{j-k+1}\gamma_{j-k}})\vee(A_j\ldots A_{j-k+1}^{\eta_j\ldots\eta_{j-k+2}}B_{j-k}^{\eta_j\ldots\eta_{j-k+1}}Q_{j-k-1}^{\eta_j\ldots\eta_{j-k+1}\beta_{j-k}})\\
&=\bigvee_{2^{k+1}}A_jA_{j-1}^{\eta_j}A_{j-2}^{\eta_j\eta_{j-1}}\ldots A_{j-k}^{\eta_j\eta_{j-1}\ldots\eta_{j-k+1}}Q_{j-k-1}^{\eta_j\eta_{j-1}\ldots\eta_{j-k}},
\end{aligned}
\]
which is the $k+1$ steps version of \eqref{krecursiv}. Therefore \eqref{krecursiv} holds true for every $k=1,2,\ldots, j$, and thus \eqref{badrecursiv} is obtained by letting $k=j$ in \eqref{krecursiv}.

We notice that for every $k=1,\ldots, j$ we have by \eqref{maincoeff} and \eqref{nextcoeff}
\begin{equation}\label{t*}
A_k\le D_k\left(\frac{1}{t_*\wedge1}\right)^{W_k}\qquad\mbox{where}\quad 
W_k:=\left\{\begin{array}{ll}Z_k\quad&\mbox{if $\eta_k=\gamma_k$}\\0\quad&\mbox{if $\eta_k=\beta_k$}\end{array}\right.
\end{equation}
and where \begin{equation}\label{Dk}D_k:=1\vee(2^k\tilde G)^{Z_k}\vee(2b_\chi\tilde B^{k}).\end{equation} Indeed, $A_k$ depends on $t_*$ if and only if $A_k=G_k$, i.e., $\eta_k=\gamma_k$. 
From  \eqref{t*} we obtain
\begin{equation}\label{Westimate}\begin{aligned}
& A_jA_{j-1}^{\eta_j}A_{j-2}^{\eta_j\eta_{j-1}}A_{j-3}^{\eta_j\eta_{j-1}\eta_{j-2}}\ldots A_{2}^{\eta_j\eta_{j-1}\ldots\eta_{3}}A_1^{\eta_j\eta_{j-1}\ldots \eta_2}\\&\quad\le \left(\frac{1}{t_*\wedge1}\right)^{W_j+W_{j-1}\eta_j+W_{j-2}\eta_{j-1}\eta_j+\ldots+(W_1\eta_2\ldots\eta_j)} D_jD_{j-1}^{\beta_j}D_{j-2}^{\beta_j\beta_{j-1}}\ldots 
D_1^{\beta_j\beta_{j-1}\ldots \beta_2}
\end{aligned}\end{equation}
where since $D_k\ge 1$ we have used the inequality $\gamma_k\le \beta_k$, see \eqref{g<b}.

We claim that  for any of the $2^j$ possible values of the vector $(\eta_j,\eta_{j-1},\ldots\eta_1)$ the following estimate holds true
\begin{equation}\label{WZ}
{W_j\!+W_{j-1}\eta_j+W_{j-2}\eta_{j-1}\eta_j+W_{j-3}\eta_{j-2}\eta_{j-1}\eta_j+\!\ldots\!+(W_1\eta_2\ldots\eta_j)}\le \frac{d(R^j-1)}{2+\!d(m-1)}.
\end{equation}
Indeed, if $\eta_1=\beta_1$, the quantity $W_j+W_{j-1}\eta_j+W_{j-2}\eta_{j-1}\eta_j+\ldots+(W_1\eta_2\ldots\eta_j)$
has $W_1=0$ and thus increases by replacing the value of $\eta_1$ with $\gamma_1$. 
More generally, let $j\ge 2$ and let $k=\max\{i\in\{0,1,2,\ldots j-2\}: \eta_{j-i}=\beta_{j-i}\}$.  Since $\eta_{j-k}=\beta_{j-k}$ so that $W_{j-k}=0$, and since $\eta_{j-i}=\gamma_{j-i}$ so that $W_{j-i}=Z_{j-i}$ for every $i\in\{k+1,k+2,\ldots j-1\}\neq\emptyset$,  we have
\[
\begin{aligned}
&{W_j+W_{j-1}\eta_j+W_{j-2}\eta_{j-1}\eta_j+\ldots+(W_1\eta_2\ldots\eta_j)}\\&\qquad
=W_j+W_{j-1}\eta_j+\ldots+ (W_{j-k+1}\eta_{j-k+2}\ldots\eta_j)+(Z_{j-k-1}\beta_{j-k}\eta_{j-k+1}\eta_{j-k+2}\ldots\eta_j)\\
&\qquad\quad+ (Z_{j-k-2}\gamma_{j-k-1}\beta_{j-k}\eta_{j-k+1}\eta_{j-k+2}\ldots\eta_j)+\ldots+(Z_1\gamma_2\ldots\gamma_{j-k-1}\beta_{j-k}\eta_{j-k+1}\ldots\eta_j)\\
&\qquad=W_j+W_{j-1}\eta_j+W_{j-2}\eta_{j-1}\eta_j+\ldots+(W_{j-k+1}\eta_{j-k+2}\ldots\eta_j)\\
&\qquad\quad+ \beta_{j-k}\left(\eta_{j-k+1}\eta_{j-k+2}\ldots\eta_j\right)\left[Z_{j-k-1}+Z_{j-k-2}\gamma_{j-k-1}+\ldots+(Z_1\gamma_2\gamma_3\ldots\gamma_{j-k-1})\right].
\end{aligned}
\]
By introducing the estimate
\[
\beta_{j-k}\le \gamma_{j-k}+\frac{Z_{j-k}}{Z_{j-k-1}+Z_{j-k-2}\gamma_{j-k-1}+Z_{j-k-3}\gamma_{j-k-2}\gamma_{j-k-1}+\ldots+(Z_1\gamma_2\gamma_3\ldots\gamma_{j-k-1})},
\]
i.e., by using \eqref{crucial}, we thus deduce
\[
\begin{aligned}
&W_j+W_{j-1}\eta_j+W_{j-2}\eta_{j-1}\eta_j+W_{j-3}\eta_{j-2}\eta_{j-1}\eta_j\ldots+(W_1\eta_2\ldots\eta_j)\\
&\;\;\le W_j+W_{j-1}\eta_j+W_{j-2}\eta_{j-1}\eta_j+\ldots+(W_{j-k+1}\eta_{j-k+2}\ldots\eta_j)+ (Z_{j-k}\eta_{j-k+1}\eta_{j-k+2}\ldots\eta_j)\\
&\qquad+(Z_{j-k-1}\gamma_{j-k}\eta_{j-k+2}\eta_{j-k+2}\ldots\eta_j)+\ldots+ (Z_1\gamma_2\ldots\gamma_{j-k-1}\gamma_{j-k}\eta_{j-k+1}\eta_{j-k+2}\ldots\eta_j)
\end{aligned} 
\]
where the right hand side is precisely  $W_j+W_{j-1}\eta_j+W_{j-2}\eta_{j-1}\eta_j+\ldots+(W_1\eta_2\ldots\eta_j)$ in  case  $\beta_{j-k}$ were replaced by $\gamma_{j-k}$ in the vector $(\eta_j,\eta_{j-1},\ldots,\eta_1)$. That is, we can reduce the number of $\beta$'s in such a vector without decreasing the exponent of $(1\wedge t_*)^{-1}$ that appears in \eqref{Westimate}. By applying this procedure recursively, 
we obtain that for any of  the $2^j$ possible values of the vector $(\eta_j,\eta_{j-1},\ldots\eta_1)$ there holds
\[\begin{aligned}
&W_j+W_{j-1}\eta_j+W_{j-2}\eta_{j-1}\eta_j+\ldots+(W_1\eta_2\ldots\eta_j)\\&\qquad\le Z_j+Z_{j-1}\gamma_j+Z_{j-2}\gamma_{j-1}\gamma_j+\ldots+(Z_2\gamma_3\ldots\gamma_j)+(Z_1\gamma_2\ldots\gamma_j),
\end{aligned}\]
where the right hand side is equal to $\frac{d(R^j-1)}{2+d(m-1)}$ by \eqref{Zgamma}. The claim is proven and it implies 
\begin{equation}\label{poweroft}
\begin{aligned}
& A_jA_{j-1}^{\eta_j}A_{j-2}^{\eta_j\eta_{j-1}}A_{j-3}^{\eta_j\eta_{j-1}\eta_{j-2}}\ldots A_{2}^{\eta_j\eta_{j-1}\ldots\eta_{3}}A_1^{\eta_j\eta_{j-1}\ldots \eta_2}\\&\qquad\le \left(\frac{1}{t_*\wedge1}\right)^{\textstyle\frac{d(R^j-1)}{2+d(m-1)}} D_jD_{j-1}^{\beta_j}D_{j-2}^{\beta_j\beta_{j-1}}\ldots D_{2}^{\beta_j\beta_{j-1}\ldots\beta_{3}}D_1^{\beta_j\beta_{j-1}\ldots \beta_2}
\end{aligned}
\end{equation}

By inserting \eqref{poweroft} into \eqref{badrecursiv}, we get for every $j=1,2,3,\ldots$
\begin{equation}\label{lastwitht}
Q_j\le  \left(\frac{1}{t_*\wedge1}\right)^{\textstyle\frac{d(R^j-1)}{2+d(m-1)}}  D_jD_{j-1}^{\beta_j}D_{j-2}^{\beta_j\beta_{j-1}}\ldots D_{2}^{\beta_j\beta_{j-1}\ldots\beta_{3}}D_1^{\beta_j\beta_{j-1}\ldots \beta_2}\;\bigvee_{2^j} Q_0^{{\eta_j}\eta_{j-1}\ldots\eta_1}.
\end{equation}

Making use of \eqref{stimadiz} and \eqref{Btilde}, we estimate the $D_k$'s defined by \eqref{Dk}: we get
\[
D_k:=1\vee(2^k\tilde G)^{Z_k}\vee(2b_\chi\tilde B^{k})=1\vee (2^k)^{d(R-1)/2}(\tilde G\vee 1)^{d(R-1)/2}\vee 2b_\chi \,R^{\lambda(R)k}\le a L^k,
\] 
having defined 
\begin{equation}\label{aL}
a=a_{\chi}(R,d,m):=(2b_\chi)\vee (1\vee\tilde G)^{d(R-1)/2},\qquad L=L(R,d,m):=(2\vee R)^{\lambda(R)}.
\end{equation}
Therefore 
\begin{equation}\label{iterativeD}
\begin{aligned}
&D_jD_{j-1}^{\beta_j}D_{j-2}^{\beta_j\beta_{j-1}}\ldots D_{2}^{\beta_j\beta_{j-1}\ldots\beta_{3}}D_1^{\beta_j\beta_{j-1}\ldots \beta_2}\\&\quad\le
a^{1+\beta_j+\beta_j\beta_{j-1}+\ldots+(\beta_j\beta_{j-1}\ldots\beta_2)}\;L^{j+(j-1)\beta_j+(j-2)\beta_j\beta_{j-1}+(j-3)\beta_j\beta_{j-1}\beta_{j-2}+\ldots+(\beta_j\beta_{j-1}\ldots\beta_2)}.
\end{aligned}\end{equation}
Making use of the telescopic properties \eqref{gjbj} of the $\beta_k$'s, an easy estimate shows that 
\[
\beta_{j}\beta_{j-1}\ldots\beta_{j-k}=\frac{2R^j+d(m-2)}{2R^{j-k-1}+d(m-2)}\le \frac{2R^j}{2R^{j-k-1}-d(m-2)_-}
\]
for every $k=0,1,\ldots, j-2$. Hence, letting
\begin{equation}\label{S1S2}
S_i=S_i(R,d,m):= \sum_{k=1}^{\infty}\frac{2k^i}{2R^k-d(m-2)_-},\qquad i=0,1,
\end{equation}
we obtain
\[
{j+(j-1)\beta_j+(j-2)\beta_j\beta_{j-1}+\ldots+(\beta_j\beta_{j-1}\ldots\beta_2)}\le j+S_1R^j
\]
and similarly
\[
{1+\beta_j+\beta_j\beta_{j-1}+\ldots+(\beta_j\beta_{j-1}\ldots\beta_2)}\le 1+S_0R^j, 
\]
which can be inserted into \eqref{iterativeD} to get
\begin{equation}\label{iterativeD2}
D_jD_{j-1}^{\beta_j}D_{j-2}^{\beta_j\beta_{j-1}}\ldots D_{2}^{\beta_j\beta_{j-1}\ldots\beta_{3}}D_1^{\beta_j\beta_{j-1}\ldots \beta_2}\le a^{1+S_0R^j}\,L^{j+S_1R^j}.
\end{equation}

Eventually, by recalling the definition of $Q_j$ from \eqref{Qj} and since $Q_0=M$, by using that $T>t_*>t_j$ for every $j$,  and by combining \eqref{iterativeD2} with \eqref{lastwitht} and \eqref{jump?}, we obtain
\[
\|u(t_*)\|_{L^{R^j}(\R^d)}^{R^j}\le a^{1+S_0R^j}\,L^{j+S_1R^j}\left(\frac{1}{t_*\wedge1}\right)^{\textstyle\frac{d(R^j-1)}{2+d(m-1)}}\,(1\vee M)^{\textstyle\frac{2R^j+d(m-2)}{2+d(m-2)}},
\]
having exploited \eqref{g<b} and \eqref{gjbj} as well.
Taking the $1/R^j$ power and passing to the limit as $j\to+\infty$ we infer
\[
\|u(t_*)\|_{L^\infty(\R^d)}
\le \left(a^{S_0}\,L^{S_1}\left(\frac{1}{t_*}\right)^{\textstyle\frac{d}{2+d(m-1)}}+a^{S_0}\,L^{S_1}\right)(1\vee M)^{\textstyle\frac{2}{2+d(m-2)}}.
\]
By arbitrariness of $t_*>0$, we obtain the desired inequality with $C_{\chi,d,m}:=\inf_{R>1} \left(a^{S_0}\,L^{S_1}\right)$ only depending on $\chi, d,m$, and  obtained by recalling the definition of $S_0,S_1$ from \eqref{S1S2} and the definition of $a, L$ from \eqref{aL}.

{\textbf{Step 4.}}
%
%
Here we let $m=m_c$ and $u_0\in D(\mathcal E_m)$.
We repeat the iteration argument of the previous step iterating down to $Q_1$ instead of $Q_0$. In fact, we recall that for $m=m_c$, in \eqref{recursiv2} we need $j\ge 2$. Therefore, we make use of \eqref{krecursiv} with $k=j-1$ instead of $k=j$. In this case we let $W_1$ be defined as $W_1:=Z_1$, and thus we have 
from \eqref{WZ}
\begin{equation*}
{W_j\!+W_{j-1}\eta_j+W_{j-2}\eta_{j-1}\eta_j+\ldots+(W_{2}\eta_{3}\eta_{4}\ldots\eta_j)\le \frac{d(R^j-1)}{2+\!d(m-1)}-(Z_1\eta_2\ldots\eta_j)}.
\end{equation*}
In this way we obtain  the following estimate which replaces \eqref{lastwitht}
\begin{equation}	\label{recursivto1}
Q_j\le    D_jD_{j-1}^{\beta_j}D_{j-2}^{\beta_j\beta_{j-1}}\ldots D_{2}^{\beta_j\beta_{j-1}\ldots\beta_{3}}\;\bigvee_{2^{j-1}} \left(\frac{1}{t_*\wedge1}\right)^{{\textstyle\frac{d(R^j-1)}{2+d(m-1)}}-Z_1\eta_2\eta_3\ldots\eta_j}\,Q_1^{{\eta_j}\eta_{j-1}\ldots\eta_2},
\end{equation}
%
 and by estimating the $D_k$'s with \eqref{iterativeD2}, and since $\gamma_k\le \eta_k\le \beta_k$ for every $k\ge2$, from \eqref{recursivto1} we also deduce 
 \begin{equation}\label{maxQ1}
 \|u(t_*)\|_{L^{R^j}(\R^d)}^{R^j}\le Q_j\le a^{1+S_0R^j}\,L^{j+S_1R^j}(1\vee Q_1)^{\beta_j\beta_{j-1}\ldots\beta_2}\left(\frac{1}{t_*\wedge1}\right)^{{\textstyle{\frac{d(R^j-1)}{2+d(m-1)}}}-Z_1\gamma_2\gamma_3\ldots\gamma_j}.
 \end{equation}
 
 Let $d\ge 3$.
 We  take the $1/R^j$ power in \eqref{maxQ1} and send $j$ to $+\infty$,
 by taking \eqref{gjbj} into account so that
 \[
 \beta_j\beta_{j-1}\ldots\beta_2=\frac{2R^j+d(m-2)}{2R+d(m-2)},\qquad Z_1\gamma_2\gamma_3\ldots\gamma_j=\frac{d(R-1)(2R^j+d(m-1))}{(2+d(m-1))(2R+d(m-1))},
 \]
  by choosing $R=m_c$ and by using  Proposition \ref{equi}, 
  which together with the lower semicontinuity of the $L^{m_c}$ norm with respect to the narrow convergence, \EEE
  entails $Q_1\le C_0$ (where $C_0$ is the constant depending on $\mathcal E_{m_c}(u_0)$ from Proposition \ref{equi}):  we get for arbitrary $t_*>0$
 \[
 \|u(t_*)\|_{L^{\infty}(\R^d)}\le a^{S_0}L^{S_1}\left(1\vee C_0\right)^{\frac{2}{2m_c+d(m_c-2)}}
 \left(\frac{1}{t_*\wedge1}\right)^{\textstyle\frac{d}{2+d(m_c-1)}(1-\frac{2(m_c-1)}{2m_c+d(m_c-1)})},
 \]
and the desired estimate follows after having noticed that $$\frac{d}{2+d(m_c-1)}=1,\qquad1-\frac{2(m_c-1)}{2m_c+d(m_c-1)}=\frac{d^2}{d(d+2)-4}.$$

Else if $d=2$, $m=1$, we make use of 
   Proposition \ref{newLpd=2} with $p=2$ and the notation therein, and  recalling that $T=2t_*$ and that $t_1=3t_*/4$ we get
\[
Q_1=\sup_{\bar t\in (t_1,T)}y_2(\bar t)\le C_{2,M,\chi}K_2{(\tilde t_*)}^{2}+2M\left(\frac{1}{2G_2} t_* \right)^{-1}\le C_{2,M,\chi}K_2{(\tilde t_*)}^{2}+\frac{(4MG_2)\vee 1}{t_*\wedge 1},
\]
where $\tilde t_*:=(2t_*)\vee1,$
and where $K_2(\cdot)$ is defined by \eqref{kappat4} and it also depends on $\mathcal E_1(u_0)$ and $\mom(u_0)$.
This can be inserted into \eqref{recursivto1} along with the estimate of the $D_k$'s from \eqref{iterativeD2} to obtain
 \[Q_j\le a^{1+S_0R^j}\,L^{j+S_1R^j} \;\bigvee_{2^{j-1}} \left(C_{2,M,\chi}K_2{(\tilde t_*)}^2+\frac{(4G_2M)\vee1}{t_*\wedge1}\right)^{{\textstyle\frac{d(R^j-1)}{2+d(m-1)}}-Z_1\eta_2\eta_3\ldots\eta_j+{\eta_j}\eta_{j-1}\ldots\eta_2}.
\]
Since in this case $Z_1=1$ and $\tfrac{d}{2+d(m-1)}=1$, we deduce
\[\|u(t_*)\|_{L^{R^j}(\R^2)}^{R^j}\le Q_j\le a^{1+S_0R^j}\,L^{j+S_1R^j}\left(C_{2,M,\chi}(K_2{((2t_*)\vee1)})^2+\frac{(4G_2M)\vee1}{t_*\wedge1}\right)^{R^j-1}.
\] Dividing by $R^j$ and passing to the limit as $j\to+\infty$ the result follows.
\end{proof}

\subsection{$L^\infty$ hypercontractivity under stronger assumptions on the initial data}

The same argument in the last step of the proof of Theorem \ref{Linftydge3} can be used in the case of more integrability assumptions on the initial datum. 
We have the following
\begin{corollary}\label{moreintegrability} 
{\rm\textbf{(I)}}
Let $m>m_c$. Suppose that $u_0\in L^p(\R^d)$ for some $p\in(1,+\infty).$ Then there exists a constant $\tilde C(p,\chi,d,m)$ such that 
 \[
 \|u(t)\|_{L^{\infty}(\R^d)}\le \tilde C(p,\chi,d,m)\left(1+\|u_0\|_{L^{p}(\R^d)}^{\textstyle\frac{2p}{2p+d(m-2)}}+M^{\textstyle\frac{2}{2+d(m-2)}}\right)
 \left(\frac{1}{t\wedge1}\right)^{\textstyle\frac{d}{2p+d(m-1)}}\;\forall \,\,t>0.
 \]
 {\rm\textbf{(II)}}
 Let $m=m_c$. If $d\ge 3$ and $u_0\in L^p(\R^d)$ for some $p\in(m_c,+\infty)$, then there exists a constant 
 $\tilde{\mathcal C}=\tilde{\mathcal C}(p,d,\chi,M, \|u_0\|_{L^{p}(\R^d)} , \mom(u_0))$ such that
  \[
 \|u(t)\|_{L^{\infty}(\R^d)}\le \tilde {\mathcal C}
 \left(\frac{1}{t\wedge1}\right)^{\textstyle\frac{d}{2(p-1)+d}}\qquad{\forall \ t>0}.
 \]
If $d=2$ and $u_0\in L^p(\R^2)$ for some $p\in[2,+\infty)$, there exists $\tilde{\mathsf C}=\tilde{\mathsf C}(p,M,\chi)$ such that, denoting with $K_p(\cdot)$ the quantity defined by \eqref{kappat4}, there holds
 \[
  \|u(t)\|_{L^{\infty}(\R^2)}\le \tilde {\mathsf C}\,(K_p(t\vee1)+\|u_0\|_{L^{p}(\R^2)}^p)^{\textstyle\frac1{p-1}}
 \left(\frac{1}{t\wedge1}\right)^{\textstyle\frac1p}\qquad\forall \ t>0,
 \]
 else if $d=2$ and $p\in(1,2)$, there exists $\tilde{\mathsf C}=\tilde{\mathsf C}(p,M,\chi)$ such that,   \[
  \|u(t)\|_{L^{\infty}(\R^2)}\le \tilde {\mathsf C} \left((\overline C^*_{t\vee 1})^p\|u_0\|_{L^{p}(\R^2)}^p\right)^{\textstyle\frac1{p-1}}
 \left(\frac{1}{t\wedge1}\right)^{\textstyle\frac1p}\qquad\forall \ t>0,
 \]
 where  $\overline C^*_T$ is  the quantity defined in {\rm Proposition \ref{very}}.

 \RRR
\end{corollary}
\begin{proof}
Let $m>m_c$.
We start from \eqref{maxQ1}.
An application of \eqref{badp} with $p=R$ entails 
\begin{equation*}\label{Q_1}Q_1=\sup_{\bar t\in(t_*/2,T)}\|u(\bar t)\|_{L^{R}(\R^d)}^R\le  \|u_0\|_{L^{R}(\R^d)}^R+\kappa_2(R)^R\,M^{\textstyle\frac{2R+d(m-2)}{2+d(m-2)}}=:\bar\kappa,
\end{equation*}
where $\kappa_2(\cdot)$ is defined by \eqref{kappa12}. 
Thanks to the latter estimate, 
 we take the $1/R^j$ power in \eqref{maxQ1} and send $j$ to $+\infty$.
  We get for arbitrary $t_*>0$
 \[
 \|u(t_*)\|_{L^{\infty}(\R^d)}\le a^{S_0}L^{S_1}\left(1\vee \bar\kappa\right)^{\frac{2}{2R+d(m-2)}}
 \left(\frac{1}{t_*\wedge1}\right)^{\textstyle\frac{d}{2+d(m-1)}(1-\frac{2(R-1)}{2R+d(m-1)})},
 \]
and since $R=p$  the desired estimate follows. 

Else suppose that $m=m_c$. If $d\ge 3$, we apply \eqref{badpbis} and insert it into \eqref{maxQ1}, then repeat the same argument. If $d=2$ and $p\ge 2$, we take Proposition \ref{newLpd=2} into account, entailing $\|u(t)\|_{L^{p}(\R^d)}^p\le C_{p,M,\chi} K_p(t\vee1)^{2p-2}+2^{p-1}\|u_0\|_{L^{p}(\R^d)}^p$ for every $t>0$, which can be inserted into \eqref{maxQ1} to conclude in the same way.
 Else if $d=2$ and $1<p<2$, by the narrow lower semicontinuity of $L^{p}$ norms and Proposition \ref{very}, we have  $\|u(t)\|_{L^{p}(\R^d)}^p\le \overline C^*_{1\vee t}\,\|u_0\|_{L^p(\R^2)}^p$ for every $t>0$, which can also be inserted into \eqref{maxQ1} to conclude.   \RRR
\end{proof}

Uniform bounds on the solution are granted by bounded initial data, as shown in the next
\begin{theorem}\label{boundedsolution} 
Let $u_{0}\in L^\infty(\mathbb R^d)$  and let $u\in AC_{loc}^2([0,+\infty);(\PP^M_2({\R}^2),W))$ a  curve given by {\rm Theorem \ref{th:convergence1}}. \EEE\\
 If $m>m_c$, then there exists a  constant $\bar C_{\chi,d,m}$ depending only on $\chi$, $d$ and $m$ such that
\begin{equation}\label{boundeddata}
\|u(t)\|_{L^\infty({\R}^d)} \le \bar C_{\chi,d,m}\left(1\vee M^{\textstyle\frac{2}{2+d(m-2)}}\right)(1\vee \|u_0\|_{L^\infty(\R^d)}) \qquad\mbox{for every $t>0$.}
\end{equation}
If $m=m_c$ and $d\ge 3$, then there exists a constant $\mathcal A=\mathcal A(\chi, M,d, \|u_0\|_{L^\infty({\R}^d)}, \mom(u_0))$ such that 
\[
\|u(t)\|_{L^\infty({\R}^d)}\le\mathcal A \qquad \mbox{for every $t>0$}.
\]
If $d=2$, and $m=1$, then there exists a constant $\bar{\mathcal A}=\bar{\mathcal A}(\chi, M,\|u_0\|_{L^\infty({\R}^2)}, \mom(u_0))$ such that
\[
\|u(t)\|_{L^\infty({\R}^2)}\le \bar{\mathcal A}\,(K_2(t\vee1))^2+\bar{\mathcal A}\qquad\mbox{for every $t>0$},
\]
where $K_2(\cdot)$ is defined by \eqref{kappat}, and in particular it grows as a power of $t$ for large $t$. 
\end{theorem}
\begin{proof}
 Let $1\le s<p$ (with $s>1$ if $m=m_c$). In the assumption $u_0\in L^\infty(\R^d)$, from Step 1 in the proof  Theorem \ref{Linftydge3} we may  notice that for fixed $T>0$, \eqref{tconzero}  also holds for $\bar t_0=0$ since the supremum therein is finite thanks to Proposition \ref{firstcontinuousestimates} and Proposition \ref{newLpd=2}. 
 By comparison with the solution of the corresponding differential equality coupled with the initial datum $y_p(0):=\|u_0\|_{L^p({\R}^d)}^p$,  we deduce that
\begin{equation*}\label{foriteration}
y_p(t)\le y_p(0)\vee\left(\left(\frac{\chi c(p,s)}{\bar K(p,s)}\right)^{\frac{1}{\sigma(p,s)}} \left(\sup_{\bar t\in(0,T)}\left(y_s(\bar t)\right)\right)^{\beta(p,s)}\right)\qquad\mbox{for every $t\in(0,T)$}.
\end{equation*}
Therefore, letting  $R>1$, we define for $j=0,1,2,\ldots$ 
$$Y_j:=y_{R^j}(0)=\|u_0\|_{L^{R^j}({\R}^d)}^{R^j}, \qquad
\bar Q_j:=\sup_{\bar t \in(0,T)} y_{R^j}(\bar t)=\sup_{\bar t\in(0,T)}\|u(\bar t)\|_{L^{R^j}({\R}^d)}^{R^j},$$
noticing that $Y_0=\bar Q_0=M$. Letting $B_j$ be defined by \eqref{nextcoeff} and $\beta_j:=\beta(R^{j},R^{j-1})$, 
 we get
\[
\bar Q_j\le Y_j\vee(B_j \bar Q_{j-1}^{\beta_j})\qquad \mbox{ for every $j=1,2,3,\ldots$}\quad\mbox{($j\ge 2$ if $m=m_c$).}
\]

Let $m>m_c$ and let us  prove \eqref{boundeddata}.
Iterating the latter, a simple induction argument shows that for every $j=1,2,3,\ldots$
\begin{equation}\label{simpleiteration}\begin{aligned}
\bar Q_j&\le Y_j\vee (B_j Y_{j-1}^{\beta_j})\vee (B_jB_{j-1}^{\beta_j}Y_{j-2}^{\beta_j\beta_{j-1}})\vee (B_jB_{j-1}^{\beta_j}B_{j-2}^{\beta_j\beta_{j-1}}Y_{j-3}^{\beta_j\beta_{j-1}\beta_{j-2}})\vee\ldots\\
& \vee (B_jB_{j-1}^{\beta_j}B_{j-2}^{\beta_j\beta_{j-1}}\ldots B_2^{\beta_j\beta_{j-1}\ldots\beta_3}\,Y_1^{\beta_j\ldots\beta_2})\vee(B_jB_{j-1}^{\beta_j}\ldots B_1^{\beta_j\beta_{j-1}\ldots\beta_2}\bar Q_0^{\beta_j\beta_{j-1}\ldots\beta_1}),
\end{aligned}
\end{equation}
where the right hand side is the maximum among $j+1$ numbers, the generic one being
\begin{equation}\label{gene}
B_jB_{j-1}^{\beta_j}B_{j-2}^{\beta_j\beta_{j-1}}B_{j-3}^{\beta_j\beta_{j-1}\beta_{j-2}}\ldots B_{j-k+1}^{\beta_j\beta_{j-1}
\ldots\beta_{j-k+2}}Y_{j-k}^{\beta_j\beta_{j-1}\ldots\beta_{j-k+1}},
\end{equation}
$k=0,1,\ldots,j$ (understood to be reduce to $Y_j$ if $k=0$). We next give an estimate for \eqref{gene}.  
For every $k=0,1,2,\ldots, j,$ there holds by interpolation 
$$Y_{j-k}=y_{R^{j-k}}(0)=\|u_0\|_{L^{R^{j-k}}({\R}^d)}^{R^{j-k}}\le M\|u_0\|_{L^\infty(\R^d)}^{R^{j-k}-1},$$ and it is not difficult to verify that  there holds using \eqref{gjbj}
$${(R^{j-k}-1)\beta_j\beta_{j-1}\ldots\beta_{j-k+1}}=(R^{j-k}-1)\frac{2R^j+d(m-2)}{2R^{j-k}+d(m-2)}\le R^j-1,$$ 
so that 
\begin{equation}\label{Yj-k}
Y_{j-k}^{\beta_j\beta_{j-1}\ldots\beta_{j-k+1}}\le M^{\beta_j\beta_{j-1}\ldots\beta_{j-k+1}}\|u_0\|_{L^\infty(\R^d)}^{R^{j-k}-1}.
\end{equation}
Moreover,  for every $k=1,\ldots, j,$ 
there holds by \eqref{nextcoeff}
\begin{equation}\label{nuovasomma}\begin{aligned}
&B_jB_{j-1}^{\beta_j}B_{j-2}^{\beta_j\beta_{j-1}}B_{j-3}^{\beta_j\beta_{j-1}\beta_{j-2}}\ldots B_{j-k+1}^{\beta_j\beta_{j-1}
\ldots\beta_{j-k+2}}\\
&\qquad\le (1\vee 2b_\chi)^{1+\beta_j+\beta_{j}\beta_{j-1}+\ldots+(\beta_j\beta_{j-1}\ldots\beta_2)}\,\tilde B^{j+(j-1)\beta_j+(j-2)\beta_j\beta_{j-1}+\ldots+(\beta_j\beta_{j-1}\ldots\beta_2)}\\
&\qquad\le (1\vee 2b_{\chi})^{1+S_0R^j}\,\tilde B^{j+S_1R^j},
\end{aligned}\end{equation}
where $S_0$ and $S_1$ are defined by \eqref{S1S2}. Thanks to \eqref{Yj-k} and \eqref{nuovasomma} we see that   \[\begin{aligned}&B_jB_{j-1}^{\beta_j}B_{j-2}^{\beta_j\beta_{j-1}}B_{j-3}^{\beta_j\beta_{j-1}\beta_{j-2}}\ldots B_{j-k+1}^{\beta_j\beta_{j-1}
\ldots\beta_{j-k+2}}Y_{j-k}^{\beta_j\beta_{j-1}\ldots\beta_{j-k+1}}\\&\qquad\le(1\vee 2b_{\chi})^{1+S_0R^j}\,\tilde B^{j+S_1R^j}\, M^{\beta_j\beta_{j-1}\ldots\beta_{j-k+1}}\|u_0\|_\infty^{R^{j-k}-1}\\&\qquad\le (1\vee 2b_{\chi})^{1+S_0R^j}\,\tilde B^{j+S_1R^j}\, (1\vee M)^{\beta_j\beta_{j-1}\ldots\beta_1}(1\vee\|u_0\|_{L^\infty(\R^d)})^{R^{j}-1}
\end{aligned}\]
for every $k=0,1,2,\ldots,j$. Therefore from \eqref{simpleiteration} we get
\[
\bar Q_j=\sup_{ t\in(0,T)}\|u(t)\|_{L^{R^j}(\R^d)}^{R^j}\le (1\vee 2b_{\chi})^{1+S_0R^j}\,\tilde B^{j+S_1R^j}\,(1\vee M)^{\beta_j\beta_{j-1}\ldots\beta_1}(1\vee\|u_0\|_{L^\infty(\R^d)})^{R^{j}-1}
\]
{for every $j=1,2,\ldots$}
Taking the $1/R^j$ power and passing to the limit as $j\to+\infty$ we deduce
\[
\|u(t)\|_{L^\infty(\R^d)}\le   (1\vee 2b_{\chi})^{S_0}\,\tilde B^{S_1}\,(1\vee M)^{\textstyle\frac{2}{2+d(m-2)}}(1\vee\|u_0\|_{L^\infty(\R^d)})\qquad\mbox{for every $t>0$}
\]
so that \eqref{boundeddata} holds with $\bar C_{\chi,d,m}=\inf_{R>1} (1\vee 2b_{\chi})^{S_0}\,\tilde B^{S_1}.$

Let $m=m_c$. In this case we iterate down to $\bar Q_1$ instead of $\bar Q_0$, that is we replace \eqref{simpleiteration} with
\begin{equation*}\begin{aligned}
\bar Q_j&\le Y_j\vee (B_j Y_{j-1}^{\beta_j})\vee (B_jB_{j-1}^{\beta_j}Y_{j-2}^{\beta_j\beta_{j-1}})\vee (B_jB_{j-1}^{\beta_j}B_{j-2}^{\beta_j\beta_{j-1}}Y_{j-3}^{\beta_j\beta_{j-1}\beta_{j-2}})\vee\ldots\\
& \vee (B_jB_{j-1}^{\beta_j}B_{j-2}^{\beta_j\beta_{j-1}}\ldots B_2^{\beta_j\beta_{j-1}\ldots\beta_3}\,\bar Q_1^{\beta_j\ldots\beta_2}),\qquad \mbox{ for $j=2,3,\ldots$}
\end{aligned}
\end{equation*}
Exploiting \eqref{Yj-k} and \eqref{nuovasomma} with $k=0,1,\ldots, j-1$,  we then deduce
\[
\bar Q_j\le(1\vee 2b_{\chi})^{1+S_0R^j}\,\tilde B^{j+S_1R^j}\left((1\vee M)^{\beta_j\beta_{j-1}\ldots\beta_3}(1\vee\|u_0\|_{L^\infty(\R^d)})^{R^j-1}\vee \bar Q_1^{\beta_j\beta_{j-1}\ldots\beta_2}\right). 
\]
If $d\ge 3$
we choose $R=m_c$ we estimate $\bar Q_1$ with the constant $C_0$ from Proposition \ref{equi}.  Therefore, dividing by $R^{j}$ and letting $j\to+\infty$ we get
\[
\|u(t)\|_{L^\infty(\R^d)}\le (1\vee 2b_{\chi})^{S_0}\,\tilde B^{S_1}\left((1\vee M)^{\frac1{m_c^2-1}}(1\vee\|u_0\|_{L^\infty(\R^d)})\,C_0^{\frac1{m_c-1}}\right)\quad\mbox{for every $t>0$}
\]
and the desired result follows, noticing that $C_0$ itself can be estimated in terms of $M,\chi$, $\mom(u_0)$ and $\|u_0\|_{L^\infty(\R^d)}$. \\
Else if $d=2$ we choose $R=2$ and estimate $\bar Q_1$ with the estimate of Proposition \ref{newLpd=2} and we get
\[
\|u(t)\|_{L^\infty(\R^2)}\le (1\vee 2b_{\chi})^{S_0}\,\tilde B^{S_1}\left((1\vee M)^{1/3}(1\vee\|u_0\|_{L^\infty(\R^2)})\,\left(C_{2,M,\chi}K_{2}(t\vee1)^{2}+2\|u_0\|^2_{L^2(\R^2)}\right)\right)
\]
for every $t>0$, thus concluding the proof.
\end{proof}

Eventually, we may check that there is uniform convergence to $0$ of the solution as $t\to+\infty$, in the small mass case for $m=m_c$. In the following estimate the rate of extinction depends on the smallness of the mass, but later we shall improve this in Section \ref{sectionproof} by means of a simple scaling argument.
\begin{proposition}\label{extinct}
Let $m=m_c$, let $u_0\in D(\mathcal E_{m_c})$ and let $u\in AC_{loc}^2([0,+\infty);(\PP^M_2({\R}^2),W))$ a  curve given by {\rm Theorem \ref{th:convergence1}}. Let $M<M_{sc}(p)$ for some $p\in (1,+\infty)$, where $M_{sc}(p)$ is defined in \eqref{defMscp}. Then there exists a constant $\mathcal C^*=\mathcal C^*_{\chi,d,p,M}$ such that 
\[
\|u(t)\|_{L^\infty(\R^d)}\le \mathcal C^*_{\chi,d,p,M}\,\left(\frac1t\right)^{\textstyle\frac{2(p-1)}{2(p-1)+d}}\qquad\mbox{ for every $t\ge1$.}
\]
\end{proposition}
\begin{proof}
From \eqref{recursivto1}-\eqref{maxQ1}, with the notation therein,
since for $t_*\ge 1$ we have $1\wedge t_*=1$, we take $R=p$ and we have
\begin{equation}\label{maxmax}
\|u(t_*)\|_{L^{R^j}(\R^d)}^{R^j}\le Q_j\le   a^{1+S_0R^j}\,L^{j+S_1R^j}\;\bigvee_{2^{j-1}} Q_1^{{\eta_j}\eta_{j-1}\ldots\eta_2}.
\end{equation}
We estimate $Q_1$ recalling that $T=2t_*$ and that $t_1=3t_*/4$, and taking advantage of Theorem \ref{subsubtheorem}, so that
\[
Q_1=\sup_{\bar t\in(t_1,T)}y_p(t)\le \breve C t_*^{1-p}
\]
for every large enough $t_*$, where $\breve C$ is a constant depending on $M,\chi,d,p$ and obtained by the application of Theorem \ref{subsubtheorem}. Taking into account that $\eta_k\ge \gamma_k$,  we get therefore for  $t_*\ge 1$
\[
Q_1\le (\breve C\vee1)^{\beta_j\beta_{j-1}\ldots\beta_2}\,(t_*^{1-p})^{\gamma_j\gamma_{j-1}\ldots\gamma_{2}}.
\]
and thus from  \eqref{maxmax} we infer
\[
\|u(t_*)\|_{L^{R^j}(\R^d)}^{R^j}\le a^{1+S_0R^j}\,L^{j+S_1R^j} (\breve C\vee 1)^{\beta_j\beta_{j-1}\ldots\beta_2} \left( t_*^{1-p}\right)^{\gamma_j\gamma_{j-1}\ldots\gamma2}.
\]
Taking the $1/R^j$ power, we conclude by noticing that $R=p$ and \eqref{gjbj} imply
\[
\lim_{j\to+\infty}\frac{\gamma_j\gamma_{j-1}\ldots\gamma_2}{R^j}=\frac2{2p+d(m_c-1)}=\frac{2}{2(p-1)+d},\]\[\lim_{j\to+\infty}\frac{\beta_j\beta_{j-1}\ldots\beta_2}{R^j}=\frac{1}{p-1}.
\] 
thus getting the desired power.
\end{proof}

\section{Uniqueness of gradient flows} \label{lastsection}

\subsection{Evolution variational inequalities}

In this section we will check that our constructed gradient flows satisfy a family of evolution variational inequalities introduced in \cite{CLM} and featuring a log-Lipschitz modulus of continuity.
By recalling point {\bf (iv)} of Proposition \ref{loggrowth},  for every $u\in L^1(\R^d)\cap L^\infty(\R^d)$ we have that $v=B_\alpha\ast u$ satisfies
\begin{equation}\label{loglipschitz}
|\nabla v(x)-\nabla v(y)|^2\le C_d^2 (M+\|u\|_\infty)^2\varphi(|x-y|^2)\qquad\mbox{for every $x,y\in \R^d$}
\end{equation}
for some new positive constant $C_d$, where  $\varphi:[0,\infty)\to[0,\infty)$ is the
concave function defined by
\begin{equation*}\label{varphi}
\varphi(r):=\left\{
\begin{array}{ll}
0\quad&\mbox{if }r=0,\\
r\log^2r\quad&\mbox{if }r\in(0,e^{-1-\sqrt{2}}],\\
r+2(1+\sqrt2)e^{-1-\sqrt2}\quad&\mbox{if } r\in(e^{-1-\sqrt{2}},+\infty).
\end{array}
\right.
\end{equation*}
We also define, following \cite{CLM}, the strictly increasing surjective function $\omega:[0,+\infty)\to [0,+\infty)$
by \begin{equation}\label{omegadef}\omega(r)=\sqrt{M r\varphi(M^{-1}r)},\end{equation} 
which behaves like $r|\log r|$ for small $r>0$ and linearly for large $r$. 
Furthermore, we let $G_\omega:[0,+\infty)\to[-\infty,+\infty)$ be the strictly increasing surjective function defined as
\[
G_\omega(t):=\int_1^t\frac{1}{\omega(s)}\,\d s.
\]

Some properties of every curve $u\in AC^2_{loc}((0,+\infty);\Space)$ from Theorem \ref{th:convergence1} and Theorem \ref{esistenzagenerale} that will be useful in the next statement are the following:

 \begin{equation}\label{perevi1}u(t)^m\in W^{1,1}_{loc}(\mathbb{R}^d)  \mbox{ for
a.e. $t\in (0,+\infty)$},\qquad \int_0^T\mom(u(t))\,\d t<+\infty \mbox{ for every $T>0$,}  
\end{equation}
\begin{equation}\label{perevi2}
 \int_{t_1}^{t_2}\int_{\mathbb R^d}\left|\frac{\nabla u(t)^m}{u(t)}\right|^2u(t)\,\d x\,\d t<+\infty\qquad\mbox{for every $0<t_1<t_2<+\infty$}.
\end{equation}
\begin{equation}\label{perevi3}
\mbox{  the first equation of
\eqref{PE} holds in the sense of distributions on $(0,+\infty)\times
\mathbb{R}^d$}.\end{equation}
Indeed, \eqref{perevi1} follows from Proposition \ref{momgrowth}, from Theorem \ref{prop:IC} and from Theorem \ref{esistenzagenerale}. 
   By taking advantage of Proposition \ref{loggrowth} and letting $c_1,c_2$ be the constants therein, we see that
   \[\begin{aligned}
\int_{t_1}^{t_2}\int_{\mathbb R^d}\left| \nabla B_\alpha\ast u(t)\right|^2u(t)\,\d x\,\d t&\le M(t_2-t_1)\sup_{t\in(t_1,t_2)}\|\nabla B_\alpha\ast u(t)\|_{L^\infty(\R^d)}^2\\&\le M(t_2-t_1)\sup_{t\in(t_1,t_2)}\left( c_1\|u\|_{L^\infty(\R^d)}+c_2M\right)^2
\end{aligned}
\]
where the right hand side is finite for every $0<t_1<t_2<+\infty$ thanks to Theorem \ref{Linftydge3}. By combining this with Proposition \ref{gfede}, we get \eqref{perevi2}. Finally \eqref{perevi3} is shown in Theorem \ref{weaktheorem}.



\RRR

The key step towards uniqueness of  solutions is the derivation of the evolution variational inequality formulation, which is done in the next proposition. 
This result  is borrowed from \cite{CLM}, only minor modifications are needed as detailed through the  proof.
\begin{proposition}[{\bf EVI}]\label{thebasicevi} Suppose that $u\in AC^2_{loc}((0,+\infty);\Space)$ is a curve given either by {\rm Theorem \ref{th:convergence1}} or by {\rm Theorem \ref{esistenzagenerale}}.
 Then $u$ satisfies the following evolution variational inequalities:
 for every $\bar u\in\PP_2^M(\R^d)$ there holds  
\begin{equation}\label{EVI}\begin{aligned}
   & \frac{1}2\frac{d}{dt} W_2^2(u(t),\bar u) \le\,
\EE_m(\bar u) - \EE_m(u(t))\\&\qquad+ \chi C_d(M+\|u(t)\|_{L^\infty(\R^d)}+\|\bar u\|_{L^\infty(\R^d)})\,
\omega(W_2^2(u(t),\bar u))  \quad \text{for a.e. }
t>0,\end{aligned}
\end{equation}
where $C_d>0$ is a constant only depending on $d$.
\end{proposition}
\begin{proof}
 In order to obtain \eqref{EVI}, we need to carefully revisit the proof of \cite[Theorem 3.1]{CLM} and \cite[Corollary 6.1]{CLM}, which treat the case $m=1$ and the case $m>1$, respectively (the basic assumptions for applying such results are the validity of \eqref{perevi1}-\eqref{perevi2}-\eqref{perevi3}, which indeed hold true in our case).  The parameter $\eps$ that appears therein is $0$ in our case, since we are considering the parabolic-elliptic case, while \cite{CLM} is more generally including the analysis for the so called parabolic-parabolic problem as well.  Moreover,  it was assumed in \cite{CLM} that 
 $u$ is a bounded solution, i.e., that the map $t\mapsto\|u(t)\|_{L^\infty(\R^d)}$ is in $L^\infty(0,T)$ for every $T>0$ (incidentally, we notice that this last property is indeed a consequence of the boundedness of the initial datum for the case of  solutions constructed by means of JKO scheme, as we have proved in Theorem \ref{boundedsolution}). 
 
 Let us check the validity of \eqref{EVI} under the sole assumption that $u(t)\in L^\infty(\R^d)$ for every $t>0$ and $\bar u\in L^\infty(\R^d)$ (otherwise the result is trivial). 
Let $\bar v:=B_{\alpha}\ast \bar u$ and similarly we denote by $v(t)=B_{\alpha}\ast u(t)$, for every $t>0$. We have
\begin{equation*}\label{evi2}\begin{aligned}
\frac12\frac{d}{dt}W_2^2(u(t),\bar u)&\le
\mathcal F_m(\bar u)-2\chi\mathcal B_\alpha(\bar u)-\mathcal F_m(u(t))+2\chi\mathcal B_\alpha(u(t))\\&\qquad+\chi\int_{\mathbb{R}^d}
(\bar v-v(t))\bar u\,\d x +\chi C_d(M+\|u(t)\|_{L^\infty(\R^d)})\,
\omega(W_2^2(u(t),\bar u))\end{aligned}
\end{equation*}
for a.e. $t\in (0,T)$.
Indeed, this is obtained from Step 1 of the proof of \cite[Theorem 3.1]{CLM}, by making a single modification to that step (besides the fact that here we are not fixing $\chi=1$): in formula (3.8) of \cite{CLM} the constant $C$ is generically depending on the  $L^\infty((0,T);L^\infty(\R^d))$ norm of $u$, while by applying  \eqref{loglipschitz} we get a pointwise in time estimate with $C$ more precisely expressed as $C=C_d(M+\|u(t)\|_{L^\infty(\R^d)})$.  
 Then, by  following Step 2 of the proof of \cite[Theorem 3.1]{CLM} we get 
 \begin{equation}\label{ev0}
\begin{aligned}
&\frac12\frac{d}{dt}W_2^2(u(t),\bar u)\le
\EE_{m}(\bar u)-\EE_{m}(u(t))+  \chi C_d(M+\|u(t)\|_{L^\infty(\R^d)})\,
\omega(W_2^2(u(t),\bar u))
\\&\quad\quad+ \chi\int_{\mathbb{R}^d} (\bar u -u(t))(\bar v-v(t))\,\d x-\frac\chi2 \|\nabla(v(t)-\bar v)\|_{L^2(\R^d)}^{2}-\frac{\chi\alpha}2\|v(t)-\bar v\|_{L^2(\R^d)}^{2}
\end{aligned}
\end{equation}
for a.e. $t>0$.
Indeed, the latter corresponds to formula (3.13) from \cite{CLM}, again up to having written the quantity $C$ therein by  $C_d(M+\|u(t)\|_{L^\infty(\R^d)})$, and having included  the parameter $\chi$ which was set to $1$ in \cite{CLM}.
Finally, by repeating Step 3 in the proof of \cite[Theorem 3.1]{CLM} we can bound the last line of \eqref{ev0} in terms of $\chi(\|u(t)\|_{L^\infty(\R^d)}+\|\bar u\|_{L^\infty(\R^d)})W_2^2(u(t),\bar u)$ and thus obtain the desired estimate \eqref{EVI}, up to uploading the constant $C_d$.
\end{proof}

%
%

\begin{theorem}[{\bf Stability and uniqueness for \boldmath$m>m_c$}] \label{th:unique} Assume  $m>m_c$.
 Let $u_0^i\in \PP_2^M(\mathbb R^d)$, $i=1,2$, and let $ u^i \in AC^2_{loc}((0,+\infty);\Space)\cap  C([0,+\infty);\Space)\EEE$ corresponding curves 
 given by {\rm Theorem \ref{esistenzagenerale}}.
 Then 
 \begin{equation}\label{contractivity}
W_2^2(u^1(t),u^2(t))\le G_\omega^{-1}(G_\omega(W_2^2(u_0^1,u_0^2))+QF(t)) \qquad \mbox{for every $t\ge 0$},
\end{equation}
where $F:[0,+\infty)\to[0,+\infty)$ is defined by $$\displaystyle F(t)=\int_0^t \left(1+s^{-\frac{d}{2+d(m-1)}}\right)\,\d s$$
 and $Q\ge 0$ is a constant only depending on $M,\chi,m,d$. 
In particular, for every $u_0\in\PP_2^M(\R^d)$, there exists a unique gradient flow of functional $\mathcal E_m$ starting from $u_0$. \EEE
\end{theorem}
\begin{proof}
Let  $u_0^1, u_0^2 \in \Space$, and
and $u^1$, $u^2$ be two corresponding curves given by Theorem \ref{esistenzagenerale}, so that
$u^1$, $u^2$ are indeed continuous in $\PP_2^M(\R^d)$ up to $t=0$ and are therefore gradient flows of functional $\mathcal E_m$ respectively starting from $u_0^1$ and $u_0^2$. \EEE
 By combining \eqref{EVI} and Theorem \ref{Linftydge3}  we get the existence of a constant $Q=Q(\chi,M,m,d)$ such that for every $\bar u\in \PP_2^M(\mathbb R^d)\cap L^\infty(\mathbb R^d)$ and for $i=1,2$ there holds
\begin{equation*}\begin{aligned}
   & \frac{1}2\frac{d}{dt} W_2^2(u^i(t),\bar u) \le\,
\EE_m(\bar u) - \EE_m(u^i(t))\\&\qquad+ Q\left(1+\left(\frac 1t\right)^{\frac d{2+d(m-1)}}+\|\bar u\|_\infty\right)\,
\omega(W_2^2(u^i(t),\bar u))  \quad \text{for a.e. } t>0, \quad i=1,2.
\end{aligned}
\end{equation*}
Using the previous inequality and a standard duplication of variables argument for which we refer to \cite[Theorem 11.1.4]{AGS}, we  obtain
\begin{equation*}\begin{aligned}\label{contr1}
    \frac{1}2\frac{d}{ds} W^2_2(u^1(s),u^2(s)){\Big{|}}_{s=t} &\le\,
    \frac{1}2\frac{d}{ds} W^2_2(u^1(s),u^2(t)){\Big{|}}_{s=t} + \frac{1}2\frac{d}{ds} W^2_2(u^1(t),u^2(s)){\Big{|}}_{s=t}
  \\&\le\, \frac12f(t)\,\omega(W^2_2(u^1(t),u^2(t))) \quad \text{for a.e. } t>0,
\end{aligned}\end{equation*}
where $f(t):=Q(1+t^{-\frac{d}{2+d(m-1)}})$ and $Q$ a suitably updated constant only depending on $M,\chi,m,d$. 
Therefore $y(t):=W^2_2(u^1(t),u^2(t))$, which is continuous on $[0,+\infty)$ with $y(0)=W_2^2(u^1_0,u^2_0)$, 
satisfies a differential inequality of the form $y'(t)\le f(t)\omega(y(t))$, $t>0$, where, by the (log)Lipschitz property of $\omega$, 
the Osgood condition $\int_0^1\frac{1}{\omega(r)}\,\d r =+\infty$ holds and $\int_0^1f(t)\,\d t <+\infty$ as a consequence of $m>m_c$. 
These properties imply uniqueness of the solution of the initial value problem for the differential equation $y'(t)= f(t)\omega(y(t))$, obtained by separation of variables. 
In fact, from the differential inequality $y'(t)\le f(t)\omega(y(t))$ we get
\[
y(t)=W_2^2(u_1(t),u_2(t))\le G_\omega^{-1}(G_\omega(y(0))+QF(t)) \qquad \mbox{for every $t\ge 0$},
\]
 where $F$ is continuous, increasing on $[0,+\infty)$ and $F(0)=0$. This shows the validity of \eqref{contractivity}.
 In particular, if $u^1_0=u_0^2$ we get $y\equiv 0$, i.e. $u^1(t)=u^2(t)$ for every $t\ge 0$ (recalling that $F(0)=0$, $G_\omega(0)=-\infty$, $G_\omega^{-1}(-\infty)=0$\EEE).
\end{proof}

\begin{remark}\rm
Thanks to Theorem \ref{th:unique}, in the case $m>m_c$ the solution map $S_t$ associating to every $u_0\in\Space$ the unique gradient flow $S_t(u_0)\in \Space$ of functional $\mathcal E_m$ starting from $u_0$ at time $t\ge 0$ is a semigroup over $\Space$, i.e., $S_t(S_s(u_0))=S_{t+s}(u_0)$, for every $s\ge 0$ and every $t\ge 0$, $S_0$ being the identity map. This is a consequence of the EVI, since if a curve $u(t)$ satisfies \eqref{EVI} then  $u_h(t):=u(t+h)$ satisfies \eqref{EVI} as well, for any $h>0$, with initial datum $u(h)$, and then it is the unique gradient flow of functional $\mathcal E_m$ starting from $u(h)$. 
\end{remark}
\EEE

\begin{theorem}[{\bf Stability and uniqueness for \boldmath$m=m_c$ and \boldmath$u_0\in  D(\mathcal E_m)$}]\label{uniquebis} 
Assume  $m=m_c$.   Let $u_0^i\in D(\mathcal E_m)$, $i=1,2$, and let   $u^i \in AC^2_{loc}([0,+\infty);\Space)$ corresponding curves 
 given by {\rm Theorem \ref{th:convergence1}}. 
If  $d=2$,  assume in addition that $u_0^i\in L^p(\R^2)$ for some $p>1$, \RRR $i=1,2$.
Then, \eqref{contractivity} holds with $Q$ also depending on $\mathcal E_m(u_0)$ and $\mom(u_0)$,
where $$\displaystyle F(t):=\int_0^t (1+s^{-\frac1p})\,\d s\quad \mbox{if $d=2$},\qquad \mbox{and}\quad \displaystyle F(t):=\int_0^t (1+s^{-\frac{d}{2+d(m-1)}})\,\d s\quad
\mbox{if $d\ge 3$}.$$ 
In particular,  for every $u_0\in D(\mathcal E_m)$ \EEE there exists a unique gradient flow  of functional $\mathcal E_m$ starting from $u_0$\EEE.
\end{theorem}
\begin{proof}
We just repeat the first step of the proof of Theorem \ref{th:unique}. The only difference is in the application of Theorem \ref{Linftydge3}, where in the case $m=m_c$ it is required that $u_0\in D(\mathcal E_m)$. In particular,  the constant $Q$ depends also on $\mathcal E_m(u_0)$ and $\mom(u_0)$. We notice that the exponent $\tfrac{d}{2+d(m_c-1)}$ is larger than $1$ if and only if $d\ge 3$ allowing to deduce uniqueness only for $d\ge 3$. In the case $d=2$ we need to assume $u_0\in L^p(\R^2)$ for some $p>1$ in order to invoke Corollary \ref{moreintegrability}  and improve such exponent to $1/p<1$.
\end{proof}

\subsection{Energy dissipation equality}
We conclude by improving the result of Proposition \ref{gfede}, showing that the energy dissipation inequality therein is indeed an equality.

\begin{theorem}[{\bf EDE}]\label{ede}
Let $u_0\in\Space$. If $m=m_c$, assume in addition that $u_0\in D(\EE_m)$. Let $u\in C([0,+\infty);\Space)$ the unique (by {\rm Theorem \ref{th:unique}} and {\rm Theorem \ref{uniquebis}}\EEE) gradient flow of functional $\mathcal E_m$, starting from $u_0$. Then 
the map $ t\mapsto \mathcal E_m(u(t))$ is locally absolutely continuous on $(0,+\infty)$
and the following energy dissipation equality
\begin{equation}\label{EDE}
\EE_m(u(t_2))+\int_{t_1}^{t_2}\int_{\mathbb{R}^d}\left | \frac{\nabla u^{m}(t) - \chi u(t) \nabla B_\alpha\ast u(t)} {u(t)}\right|^2u(t)\,\d x\,\d t = \EE_m(u(t_1))
\end{equation}
holds for every 
$0\le t_1<t_2<+\infty$.
\end{theorem}

\begin{proof}

\RRR

Suppose  that $t\mapsto u(t)\in\Space$ is a gradient flow of functional $\mathcal E_m$ starting from $u_0$. By Proposition \ref{thebasicevi} and Theorem \ref{Linftydge3}, we have \eqref{EVI} and moreover by fixing $0<T_1<T_2<+\infty$ we have 
$$R:=\sup_{t\in[T_1,T_2]}\|u(t)\|_{L^\infty(\R^d)}<+\infty,$$  
where $R$ depends only on $\chi,M,d,m,T_1,T_2$ (also on $\mathcal E_m(u_0)$ and $\mom(u_0)$ if $m=m_c$).

{\textbf {Step 1.}}
Let $\bar u\in \mathscr P_2^M(\R^N)\cap L^\infty(\R^N)$.  For a.e. $t\in(T_1,T_2)$, we have by triangle inequality 
\[
\frac12\frac{d}{dt}W^2(u(t),\bar u)\ge -|u'|(t)\,W(u(t),\bar u)
\]
so that \eqref{EVI} implies
\[
(\mathcal E_m(u(t))-\mathcal E_m(\bar u))_+ \le |u'|(t)W(u(t),\bar u)+\chi C_d(M+\|u(t)\|_{L^\infty(\R^d)}+\|\bar u\|_{L^\infty(\R^d)})\, \omega(W^2(u(t),\bar u)),
\]
where $\omega$ is the function defined in \eqref{omegadef}.

 Let now $\mathscr U_{M,R}:=\{u\in\mathscr P_2^M(\R^d):\|u\|_\infty\le R\}$. If $d_1$ denotes the $L^1(\R^d)$ distance, $\mathscr U_{M,R}$ is separable with respect to the distance $W+d_1$, and  given a countable dense subset $D_0$ we find therefore a null set of times $N$ such that 
\[
(\mathcal E_m(u(t))-\mathcal E(\bar u))_+ \le |u'|(t)W(u(t),\bar u)+ \chi C_d(M+2R)\, \omega(W^2(u(t),\bar u)),
\]
for every $t\in(T_1,T_2)\setminus N$ and every $\bar u \in D_0$. Since  $D_0$ is dense in $\mathscr U_{M,R}$ with respect to the distance $W+d_1$, the latter extends to every $t\in (T_1,T_2)\setminus N$ and every $\bar u\in\mathscr U_{M,R}$. Indeed, if $\bar u\in\mathscr U_{M,R}$, and a sequence $(\bar u_n)_{n}\subset \mathscr U_{M,R}$ satisfies $W(\bar u_n,\bar u)\to 0$ and  $\|\bar u_n-\bar u\|_{L^1(\R^d)}\to0$ as $n\to\infty$, then $\mathcal E_m(\bar u_n)\to\mathcal E_m(\bar u)$ as well.
As a consequence,
\begin{equation}\label{truncatedslope}
\mathcal G_m(u(t)):=\limsup_{W(\bar u, u(t))\to0\atop{\bar u \in \mathscr U_{M,R}}}\frac{(\mathcal E_m(u(t))-\mathcal E(\bar u))_+}{W(u(t),\bar u)}\le |u'|(t)
\end{equation}
for a.e. $t\in(T_1,T_2)$, in particular $\mathcal G_m(u(\cdot))\in L^2(T_1,T_2)$.

{\textbf{Step 2.}}
An estimate from below for $\mathcal G_m(u(\cdot))$ can be obtained by taking into account that the energy functional $\mathcal E_m$ is $\omega$-convex along Wasserstein geodesics connecting $L^\infty$ densities. This property is itself a consequence of the EVI \eqref{EVI}, and has been proved in \cite[Theorem 4.1]{CLM}. Indeed, for $u,\bar u\in \mathscr P_2^M(\R^d)\cap L^\infty(\R^d)$, we denote by $\gamma_t$, $t\in[0,1]$, the Wasserstein geodesic connecting $u$ to $\bar u$, i.e., the curve $\gamma_t:=(T_t)_{\#} u$, where $T_t(x):=(1-t)x+tT(x)$ and $T:\R^d\to\R^d$ is the optimal transport map from $u$ to $\bar u$. In this way, we have \begin{equation}\label{geodesic}W(\gamma_t,u)=t\,W(u,\bar u).\end{equation} 
Then, it is shown in \cite[Theorem 4.1]{CLM} that
\begin{equation}\label{omegaconv}
\mathcal E_m(\gamma_t)\le (1-t)\mathcal E_m(u)+t\mathcal E_m(\bar u)+\widetilde C\left[(1-t)\omega(t^2\,W^2(u,\bar u))+t\omega((1-t)^2\,W^2(u,\bar u))\right],
\end{equation}
 and $\widetilde C$ is a constant depending on $\|u\|_{L^\infty(\R^d)}$ and $\|\bar u\|_{L^\infty(\R^d)}$. We also recall that as a consequence of the geodesical convexity of $L^p(\R^d)$ norms, the set $\mathscr U_{M,R}$ is geodesically convex, i.e.,  if $u,\bar u\in \mathscr U_{M,R}$ then along the geodesic $\gamma_t$ connecting $u$ to $\bar u$, we have $\gamma_t\in \mathscr U_{M,R}$ for every $t\in[0,1]$.
We deduce existence of a constant $\overline C$ (only depending on $\chi,M, d, R$) such that 
\[\begin{aligned}
\mathcal G_m(u(t))\ge \limsup_{s\to 0}\frac{\mathcal E_m(u(t))-\mathcal E_m(\gamma_s^t)}{W(u(t),\gamma_s^t)}\ge \frac{\mathcal E_m(u(t))-\mathcal E_m(\bar u)}{W(u(t),\bar u)}-\overline C\,\frac{\omega(W^2(u(t),\bar u))}{W(u(t),\bar u)}.
\end{aligned}
\]
for every $t\in [T_1,T_2]$ and every $\bar u\in \mathscr U_{M,R}$, having denoted by $\gamma_s^t$, $s\in[0,1]$, the geodesic from $u(t)$ to $\bar u$.
We notice that the last equality is a consequence of \eqref{geodesic}, \eqref{omegaconv} and of the fact that $\omega(x)$ behaves as $x|\log x|$ as $x\to 0$.

Thus there exists a constant $\overline C$ (only depending on $\chi,M,d,R$) such that
\begin{equation}\label{nomax}
\mathcal E_m(u(t))-\mathcal E_m(\bar u)\le \mathcal G_m(u(t))\,W(u(t),\bar u)+\overline C\omega(W^2(u(t),\bar u))
\end{equation}
for every $t\in[T_1,T_2]$ and every $\bar u \in \mathscr U_{M,R}$. We can take $t_1,t_2\in [T_1,T_2]$ and apply \eqref{nomax} with $t=t_1, \bar u=u(t_2)$, and then with $t=t_2, \bar u=u(t_1)$, and taking the maximum of the two resulting inequalities we obtain
\begin{equation}\label{metrictaylor}\begin{aligned}
|\mathcal E_m(u(t_1))-\mathcal E_m(u(t_2))|&\le\max\{\mathcal G_m(u(t_1)),\mathcal G_m(u(t_2))\}\,W(u(t_1),u(t_2))\\&\qquad+\overline C\omega(W^2(u(t_1),u(t_2))).
\end{aligned}\end{equation} 
We introduce the arc-length parametrization $s\mapsto \hat u(s)$ of the curve $t\mapsto u(t)$ on the interval $[T_1,T_2]$, see \cite[Lemma 1.1.4]{AGS}.
It satisfies $u=\hat u\circ \ell$, where $\ell(t):=\int_{T_1}^t|u'|(r)\,\d r$. We have $|\hat u'|=1$ a.e. in $(0,\ell(T_2))$ and given $s_1,s_2\in[0,\ell(T_2)]$ we find $t_1,t_2\in[T_1,T_2]$ such that $u(t_1)=\hat u(s_1)$, $u(t_2)=\hat u(s_2)$,  and
\[
W(\hat u(s_1), \hat u({s_2}))=W(u(t_1), u(t_2))\le \int_{t_1\wedge t_2}^{t_1\vee t_2}|u'|(r)\,\d r=\ell(t_1\vee t_2)-\ell(t_1\wedge t_2)\le |s_2-s_1|.
\] 
Therefore,  we may apply \eqref{metrictaylor} and get
\[\begin{aligned}
|\mathcal E_m(\hat u(s_1))-\mathcal E_m(\hat u(s_2))|&\le \max\{\mathcal G_m(\hat u(s_1)),\,\mathcal G_m(\hat u(s_2))\}\,|s_2-s_1|+\overline C\,\omega(|s_2-s_1|^2)
\end{aligned}\]
for every $s_1,s_2\in[0,\ell(T_2)]$.
The map $s\mapsto \mathcal G_m(\hat u(s))$ belongs to $L^1(0,\ell(T_2))$. Indeed, taking into account \eqref{truncatedslope}, changing variable in the integral we have
\begin{equation}\label{GL1}
\int_0^{\ell(T_2)}  \mathcal G_m(\hat u(s))\,\d s = \int_{T_1}^{T_2}  \mathcal G_m(u(t))|u'|(t)\,\d t \leq \int_{T_1}^{T_2} |u'|^2(t)\,\d t <+\infty.
\end{equation}
\RRR
An application of \cite[Lemma 1.2.6]{AGS} shows that the map $s\mapsto \mathcal E_m(\hat u(s))$ belongs to $W^{1,1}(0,\ell(T_2))$: the argument therein can be indeed repeated even with the additional term $\overline C\,\omega(|s_2-s_1|^2)$, taking advantage of the fact that $\omega(x)$ behaves like $x|\log x|$ for small $x$. 
 We have to prove that $s\mapsto \mathcal E_m(\hat u(s))$ coincides with his continuous representative on $(0,\ell(T_2))$. 
Since this map is already lower semicontinuous, it is sufficient to prove that for any $s\in (0,\ell(T_2))$
$$\limsup_{\eps\to 0}\frac1{2\eps}\int_{-\eps}^{\eps}(\mathcal E_m(\hat u(s+r))-\mathcal E_m(\hat u(s)))\, \d r \le 0. $$
Fixing $s\in (0,\ell(T_2))$, for any $r$ sufficiently small, by  \eqref{nomax} we have
\[
\mathcal E_m(\hat u(s+r))-\mathcal E_m(\hat u(s))\le \mathcal G_m(\hat u(s+r))|r|+\overline C\omega(|r|^2),
\]
therefore
\[\begin{aligned}
&\limsup_{\eps\to 0}\frac1{2\eps}\int_{-\eps}^{\eps}(\mathcal E_m(\hat u(s+r))-\mathcal E_m(\hat u(s)))\, \d r
\le\limsup_{\eps\to0}\frac1{2\eps}\int_{-\eps}^{\eps}\left(\mathcal G_m(\hat u(s+r))\,|r|+\overline C\omega(r^2)\right)\,\d r
\\&\qquad
\le\limsup_{\eps\to0}\frac1{2\eps}\int_{-\eps}^{\eps}\mathcal G_m(\hat u(s+r))\,|r|\,\d r 
\le \limsup_{\eps\to0}\frac1{2}\int_{-\eps}^{\eps}\mathcal G_m(\hat u(s+r))\,\d r =0,
\end{aligned}\]
where in the last equality we used the property \eqref{GL1}.
Thus the map $s\mapsto \mathcal E_m(\hat u(s))$ is absolutely continuous. 
Hence, the map $t\mapsto \mathcal E_m(u(t))$ belongs to $AC([T_1,T_2])$ as the composition of an absolutely continuous function with an absolutely continuous increasing function.
\EEE

{\textbf {Step 3.}}
By making use of \eqref{truncatedslope}, we get
\begin{equation*}\begin{aligned}
	-\frac{d}{dt}\mathcal E_m(u(t))&=\lim_{h\to 0^+}\frac{\mathcal E_m(u(t))-\mathcal E_m(u(t+h))}{h} \\
	&= \lim_{h\to 0^+}\frac{\mathcal E_m(u(t))-\mathcal E_m(u(t+h))}{W(u(t),u(t+h))}\frac{W(u(t),u(t+h))}{h}\\&\le\mathcal G_m(u(t))\,|u'|(t)\le |u'|^2(t)\qquad\quad\mbox{ for a.e. }t\in(T_1,T_2).\end{aligned}
\end{equation*}
Since $u$ satisfies the continuity equation $\partial_tu=\mathrm{div}(\xi u)$ for $\xi$ defined in \eqref{defxi} (thanks Theorem \ref{weaktheorem}), 
by \cite[Theorem 8.3.1]{AGS} we have
 \[|u'|^2(t)\le \int_{\R^d}|\xi(t)|^2u(t)\,\d x\qquad\mbox{ for a.e. }t\in(T_1,T_2).\]
 By combining the latter inequalities, we obtain that
\begin{equation*}\label{diqua}
	-\frac{d}{dt}\mathcal E_m(u(t) \leq  \int_{\R^d}|\xi(t)|^2u(t)\,\d x\qquad\mbox{ for a.e. } t\in(T_1,T_2).
\end{equation*} 
On the other hand, from the energy dissipation inequality \eqref{EDI2} we directly obtain
\begin{equation*}\label{dila}
\frac{d}{dt}\mathcal E_m(u(t))=\lim_{h\to 0^+}\frac{\mathcal E_m(u(t+h))-\mathcal E_m(u(t))}{h} \leq  -\int_{\R^d}|\xi(t)|^2u(t)\,\d x \qquad\mbox{ for a.e. } t\in(T_1,T_2).
\end{equation*}
Therefore, we conclude that
\begin{equation*}
\frac{d}{dt}\mathcal E_m(u(t)) =  -\int_{\R^d}|\xi(t)|^2u(t)\,\d x \qquad\mbox{ for a.e. } t\in(T_1,T_2).
\end{equation*}
By integrating from $T_1$ to $T_2$ and by the arbitrariness of $T_1,T_2$,
we obtain the validity of  \eqref{EDE} for $t_1>0$.

In the case $t_1=0$, by making use of the already proven equality with $0<t_1<t_2$ and passing to the limit as $t_1\to 0^+$, we have
  \begin{equation*}\begin{aligned}
&\mathcal E_m(u(t_2))+\int_{0}^{t_2}\int_{\R^d}\left|\xi(t)\right|^2 u(t)\,\d x\,\d t\\
&\quad=\mathcal E_m(u(t_2))+\lim_{t_1\to 0}\int_{t_1}^{t_2}\int_{\R^d}\left|\xi(t)\right|^2 u(t)\,\d x\,\d t\\&\quad\ge
 \liminf_{t_1\to0}\mathcal E_m(u(t_1))\ge\mathcal E_m(u_0).\end{aligned}
\end{equation*}
If $u_0\in D(\mathcal E_m)$, invoking again Proposition \ref{gfede}, the proof is concluded. Else the equality holds, both terms being $+\infty$.

\EEE

\end{proof}

\RRR


\section{Proof of the main theorems}\label{sectionproof}

By collecting the previous results, we can conclude with the proof of the main theorems.

\begin{proofad1}
Let $m>m_c$.

 Concerning point {\bf {(I)}}, existence of discrete minimizers is proven in Proposition \ref{prop:existenceMM}. \RRR
We notice that  $(u_\tau^k)^{\frac{p+m-1}{2}}\in H^1(\R^d)$  for every $p\in[1,+\infty)$ is proved in Theorem \ref{generalp}, and the property about $B_\alpha\ast u_\tau^k$ is shown in Theorem \ref{esistenzagenerale}. Concerning point {\textbf{(II)}}, the $AC^2$ property is the basic one from Theorem \ref{esistenzagenerale}. The property $u\in C((0,+\infty);L^p(\R^d))$ follows from the weak $L^p(\R^d)$ lower semicontinuity along with the energy dissipation equality from Theorem \ref{ede}, that implies strong continuity in $L^1(\R^d)$, and then the uniform bounds in $L^p(\R^d)$ for every $p\in[1,+\infty)$ yield the desired conclusion. This also implies the last property of point {\textbf{(II)}}, since $u_n\to u$ in $L^p(\R^d)$ implies $B_\alpha\ast u_n\to B_\alpha\ast u$ in $W^{2,p}_{loc}(\R^d)$.  The property $u^{m/2}\in L^2_{loc}((0,+\infty);H^1(\R^d))$ follows from Theorem \ref{generalp}, by integrating on $(t_1,t_2)$, for every $0<t_1<t_2<+\infty$, and passing to the limit,  using the weak lower semicontinuity of the $L^2$ norm,  and the identification of the weak limit of $\nabla u_{\tau_n}^{m/2}$ with $\nabla u^{m/2}$ is done as in the proof of Theorem \ref{esistenzagenerale}, and the same holds for $u^q$ for every $q\in[m/2,+\infty)$. \RRR
The validity of the PDE in weak form is proven in Theorem \ref{weaktheorem}, energy dissipation equality is found in Theorem \ref{ede}, thus proving points {\textbf{(IV), (V)}}. Existence and uniqueness of the gradient flow as stated in point {\textbf{(II)}} are given by Theorem \ref{esistenzagenerale} and Theorem \ref{th:unique}. For every family $(u_\tau)_{\tau>0}$ of discrete solutions of the JKO scheme, the convergence \eqref{wholetau}  of the whole family $(u_\tau)_{\tau>0}$,  is a consequence of the uniqueness of $u$, and the same for the convergence properties of point {\bf{(III)}} that are found in Theorem \ref{prop:IC} and Theorem \ref{esistenzagenerale}.

We are left to prove point {\textbf{(VI)}}.
We notice that the PDE in \eqref{PE} has the following scaling property. If $u(x,t)$ solves 
\begin{equation*}\label{pde1}
\partial_t u_t=\Delta u^m-\chi\mathrm{div}(u\nabla B_\alpha\ast u),
\end{equation*}
then for every $\lambda >0$ we have that $u_\lambda(x,t):=\lambda u(\lambda^{1-\tfrac m2}x,\lambda t)$ solves
\begin{equation}\label{pdelambda}
\partial_t u_t=\Delta u^m-\chi\mathrm{div}(u\nabla B_{\alpha_\lambda}\ast u),
\end{equation}
where $\alpha_\lambda:=\lambda^{2-m}\alpha$. Moreover, if $\int_{\R^d}u=M$, then $u_\lambda$ has mass $1$ if $\lambda=\lambda_M:=M^{-\textstyle\frac{2}{2+d(m-2)}}$.
By taking into account the uniqueness of gradient flow solutions proved in Theorem \ref{th:unique},
if the initial datum $u_0$ has mass $M$ and
 $u$ is the unique gradient flow solution  for \eqref{PE}, then $u_{\lambda_M}$  is the unique gradient flow solution to \eqref{pdelambda} with initial datum $u_{0,\lambda}(x)=\lambda u_0(\lambda^{1-\tfrac m2}x)$, and it has mass $1$. By invoking Theorem \ref{Linftydge3} and applying it to the unit mass solution $u_{\lambda_M}$, we deduce
 \[\begin{aligned}
\|u(t)\|_{\infty}&=\lambda_M^{-1}\|u_{\lambda_M}(\lambda_M^{-1}t)\|_\infty \le C_{\chi,d,m} \,M^{\textstyle\frac{2}{2+d(m-2)}}\,\left(\left(t^{-1}\,{M^{-\textstyle\frac{2}{2+d(m-2)}}}\right)^{\textstyle\frac{d}{2+d(m-1)}}+1\right)
\end{aligned}\]
for every $t>0$,
and we directly infer
\[
\|u(t)\|_{L^\infty({\R}^d)} \le C_{\chi,d,m} \,M^{\textstyle\frac{2}{2+d(m-1)}}\,\left(\frac{1}{t}\right)^{\textstyle\frac{d}{2+d(m-1)}}+C_{\chi,d,m} \,M^{\textstyle\frac{2}{2+d(m-2)}}\qquad\mbox{for every $t>0$.}
\]
In the case  $u_0\in L^\infty(\R^d)$, in the same way we deduce
from Theorem \ref{boundedsolution}, applied to the unit mass solution $u_{\lambda_M}$, 
\[\begin{aligned}
\|u(t)\|_\infty&\le \lambda_M^{-1}\|u_{\lambda_M}(\lambda_M^{-1}t)\|_\infty\le \bar C_{\chi,d,m}\,\lambda_M^{-1}(1\vee \|u_{0,\lambda_M}\|_\infty)\le \bar C_{\chi,d,m}\left(\lambda_M^{-1}\vee \|u_0\|_\infty\right)
\end{aligned}\]
so that
\[
\|u(t)\|_\infty\le \bar C_{\chi,d,m}\left(M^{\textstyle\frac{2}{2+d(m-1)}}\vee\|u_0\|_\infty\right)\qquad\mbox{for every $t>0$}
\]
thus concluding the proof.
\end{proofad1}

\begin{proofad2}
Points {\textbf{(I), (II), (III), (IV), (V)}} are proven in the same way as in Theorem \ref{th:main} and by invoking the same results.
The only difference is the use of Theorem \ref{uniquebis} instead of Theorem \ref{th:unique} for proving uniqueness.  In the case $d=2$, Theorem \ref{uniquebis} requires the further assumption $u_0\in L^p(\R^d)$ for some $p>1$. However, if $d=2$, uniqueness holds true only with the assumption $u_0\in D(\mathcal E_1)$ and can be deduced from the result of \cite[Theorem 1.3]{EM16}. Indeed, the properties of our constructed gradient flow $u$ (in particular the validity of the PDE in distributional sense from Theorem \ref{weaktheorem} and the energy dissipation inequality from Proposition \ref{gfede}) fit the uniqueness theory of \cite{EM16}.\RRR
  
 Point {\textbf{(VI)}} follows from Theorem \ref{Linftydge3} and Theorem \ref{boundedsolution}. Let us finally consider point {\textbf{(VII)}}. Since $M<M_{sc}=\sup_{p>1}M_{sc}(p)$, see \eqref{penonp}, there exists $p_*\in(1,+\infty)$ such that $M<M_{sc}(p_*)$. Fixing such $p_*$, Proposition \ref{extinct} implies that there exists a constant $ C^*$ depending only on $M,\chi,d$, such that for every $t\ge 1$
\begin{equation}\label{badexp}
\|u(t)\|_\infty\le  C^*{t^{-\sigma}},\qquad\mbox{where $\sigma=\dfrac{2(p_*-1)}{2(p_*-1)+d},\;$ so that $0<\sigma<1$}. 
\end{equation}
If we consider the mass invariant scaling $u_\lambda(x,t):=\lambda^d(\lambda x,\lambda^dt)$, $\lambda>0$, we notice that it preserves the equation since $m=m_c$, up to rescaling $\alpha$. Therefore, for every $\lambda>0$, $u_\lambda$ has mass $M$  and it is the unique gradient flow solution of the PDE, with $\alpha$ replaced by $\tilde\alpha_\lambda:=\lambda^2\alpha$, with initial datum $u_{0,\lambda}(x):=\lambda^d u_0(\lambda x)\in D(\mathcal E_m)$. Since the estimate \eqref{badexp} has a constant $C^*$ that depends only on $M,\chi,d$, then $u_\lambda$ satisfies the same estimate for every $\lambda>0$, i.e.,
\[
\lambda^d\|u(\lambda^d t)\|_\infty=\|u_\lambda(t)\|_\infty\le C^*t^{-\sigma}\qquad\mbox{ for every $t\ge 1$}.
\]
By setting $\tilde t:=\lambda^dt$ we deduce
\[
\|u(\tilde t)\|_\infty\le C^*\lambda^{-d+\sigma d}\,\tilde t^{-\sigma}\qquad\mbox{ for every $\tilde t\ge\lambda^d>0$}.
\]
Choosing $\lambda=\tilde t^{\,1/d}$ we deduce $\|u(\tilde t)\|_\infty\le C^*\,\tilde t^{-1}$ for every $\tilde t>0$ thus concluding the proof.
\end{proofad2}


\RRR

\RRR

\subsection*{Acknowledgements} 
 The authors wish to thank Jos\'e Antonio Carrillo for some fruitful discussions on the topics of this paper. \RRR
The authors are members of the GNAMPA group of the Istituto Nazionale di Alta Matematica (INdAM). They are  partially supported by the 2026 INdAM-GNAMPA project `Modelli aggregazione-diffusione con mobilit\`a non lineare: buona positura, comportamento asintotico',
 n. CUP E53C25002010001.
They have been also partially supported  the MUR - PRIN project 202244A7YL. 

\end{document}